\documentclass[11pt,a4paper]{amsart}

      \usepackage{amssymb}
      \usepackage{graphicx}
      \usepackage{dcolumn,indentfirst,cite,color,hyperref}
     \usepackage{pgfplots}
     \usepackage{geometry}
      \usepackage{pgf}
   \usepackage{bm}
    \usepackage{tikz}
    \usetikzlibrary{arrows, decorations.pathmorphing, backgrounds, positioning, fit, petri, automata}

\definecolor{yellow1}{rgb}{1,0.8,0.2} 
      \theoremstyle{plain}
     \newtheorem{thm}{Theorem}[section]
    
\newtheorem{lem}[thm]{Lemma}

\newtheorem{cor}[thm]{Corollary}
\newtheorem{pro}[thm]{Proposition}
\newtheorem{ft}[thm]{Fact}
\newtheorem{rmk}[thm]{Remark}

\newtheorem{defi}[thm]{Definition}

\newcommand{\bess}{\begin{eqnarray*}}
\newcommand{\eess}{\end{eqnarray*}}

\newcommand{\C}{\mathbb{\widehat{C}}}
\newcommand{\N}{\mathbf{N}}
\newcommand{\U}{\mathbf{U}}
\input xy
\xyoption{all}

\def\EEE{{\cal E}}
\def\FFF{{\cal F}}

\def\NNN{{\cal N}}
\def\PPP{{\cal P}}
\def\QQQ{{\cal Q}}
\def\RRR{{\cal R}}

\def\De{\Delta}
\def\de{\delta}

\def\ov{\overline}

\def\H{\mbox{$\mathbb H$}}

\def\D{\mathbb D}

\def\Z{\mbox{$\mathbb Z$}}

\def\lv{ \left(\begin{matrix} }
	\def\rv{\end{matrix}\right)}

\def\t{\tilde}

\def\cal{\mathcal}

\def\dw{{\dw}}

\usepackage[all, pdf]{xy}
\usepackage{fancyhdr} 
   \title{Boundaries  of the bounded  hyperbolic components of polynomials}
   
   \author{Yan Gao}
\address{Yan Gao, School of Mathematical Sciences, Shenzhen University, Shenzhen 518061, China}
\email{gyan@szu.edu.cn}

\author{Xiaoguang Wang}
\address{Xiaoguang Wang, School of Mathematical Sciences, Zhejiang University, Hangzhou, China}
\email{wxg688@163.com}

\author{Yueyang Wang}
\address{Yueyang Wang, School of Mathematics, Shanghai University of Finance and Economics, Shanghai, 200433, China}
\email{yueyangwang@icloud.com}

\begin{document}
   
      \begin{abstract}  In this paper, we study the local connectivity and Hausdorff dimension for the  boundaries of the bounded hyperbolic components in the space $\mathcal P_d$ of polynomials of degree $d\geq 3$.    It is shown that for  any non disjoint-type bounded hyperbolic component $\mathcal H\subset \mathcal P_d$,  the locally connected part of $\partial\mathcal H$, along each regular boundary 
   	strata, has full Hausdorff dimension $2d-2$.

   	An essential innovation in our argument involves analyzing how the canonical parameterization of the hyperbolic component--realized via Blaschke products  over a mapping scheme--extends to the boundary. 	This framework allows us to study three key aspects of  $\partial \mathcal H$: 
   	the local connectivity structure, the perturbation behavior, and the
   	local Hausdorff dimensions.
   	

   	
   \end{abstract}





   \date{\today}


   \maketitle



 \section{Introduction}
 
 Let $\mathcal P_d$ be the space  of degree $d\geq 3$ monic and centered polynomials.
 The  connectedness locus  $\mathcal C_d$, consisting of all maps $f\in \mathcal P_d$ with connected Julia sets, is known to be connected and cellular  \cite{BH,DH, DP, La}. 
 Moreover, a hyperbolic component within $\mathcal P_d$
 is bounded in $\mathcal P_d$ if and only if it is contained in $\mathcal C_d$. 
 
 For a hyperbolic component $\mathcal{H}\subset \mathcal C_d$,  let $m_{\mathcal H}$ denote the number of  attracting cycles for a representative map in $\mathcal H$.  We call $\mathcal H$  of \emph{disjoint-type}  if $m_{\mathcal H}=d-1$. In this case, each attracting cycle for a representative map in $\mathcal H$ attracts exactly one critical orbit.
 
 Milnor \cite{Mil-boundary} shows that the boundaries of disjoint-type hyperbolic components are semi-algebraic. Hence they are locally connected and have  Hausdorff dimension  $2d-3$.   
However, the situation differs fundamentally for non disjoint-type hyperbolic components: 
 
 \begin{thm}\label{boundary-hd}  Let $\mathcal H\subset \mathcal C_d$ be a non  disjoint-type  hyperbolic component, then 
 	$${\rm H.dim}(\partial \mathcal H)=2d-2.$$
 \end{thm}
 
 Theorem \ref{boundary-hd} exhibits the analytic complexity of the boundary $\partial \mathcal H$, and it confirms a stronger version of a conjecture of Milnor \cite[Conjecture 2b]{Mil-boundary} which states that the Hausdorff dimension of $\partial \mathcal H$ is larger than its topological dimension, in our setting.

Theorem \ref{boundary-hd}  is an immediate consequence of the more general result stated below.


Recall that a rational map $f$ is $k$-Misiurewicz if its Julia set $J(f)$ contains exactly $k$ critical points counting multiplicity, if $f$ has no indifferent cycles and if the $\omega$-limit set of any critical point $c\in J(f)$ does not meet the critical set. Let  $\mathfrak M_k$ be the set of all $k$-Misiurewicz polynomials in $\mathcal P_d$.

Let $X$  be a set  in the Euclidean space $\mathbb R^n$ or $\mathbb C^n$.
Recall that $X$ is locally connected at $x\in X$, if there exists a collection $\{U_k\}_{k\geq 1}$ of open and connected neighborhoods of $x$ in  $X$ such that $\lim_k{\rm diam}(U_k)=0$. 
 Define the locally connected part of $\partial \mathcal H$ by 
  $$\partial_{\rm LC}\mathcal H=\big\{f\in \partial\mathcal H;  \partial \mathcal H \text{ is locally connected at } f\big\}.$$
 
We now state the main result of the paper:

\begin{thm}\label{boundary-hd-lc}  Let $\mathcal H\subset \mathcal C_d$ be a  non disjoint-type  hyperbolic component.  For any $1\leq l\leq d-1- m_{\mathcal H}$,   we have 
		\begin{equation} \label{lc-hd-main}
	{\rm H. dim}(\partial_{\rm LC}\mathcal H\cap \mathfrak M_l)=2d-2.
	\end{equation}
\end{thm}

 \vspace{5pt}
\noindent{\bf Remark 1.2.} \emph{When $ d-1-m_{\mathcal H}<l\leq d-1$,  we have $\partial\mathcal H\cap \mathfrak{M}_l=\emptyset$. If this were not the case, any map in $\partial\mathcal H\cap \mathfrak{M}_l$ would have at least one indifferent cycle,  contradicting the  Misiurewicz property.    Therefore the equality  \eqref{lc-hd-main} does not  hold   any more. } 
  \vspace{5pt}

We actually have a   more precise form of Theorem \ref{boundary-hd-lc}  concerning local Hausdorff dimensions.  For a set $X$ in the Euclidean space, and a point $x$ in the closure of $X$, the local Hausdorff dimension of $X$ at $x$ is defined as 
$${\rm H. dim}(X, x)=\lim_{r\rightarrow 0^+} {\rm H.dim}(X\cap B(x,r)),$$
where $B(x,r)$ is the Euclidean ball centered  at $x$ with radius $r$. 

\begin{thm}\label{lchd-gf-section1}  	Let $\mathcal H\subset \mathcal C_d$ be a  non disjoint-type  hyperbolic component. 
There is a geometrically finite map $g_0\in\partial\mathcal H$ with $m_{\mathcal H}$ parabolic cycles, 
such that for any $1\leq l\leq d-1-m_{\mathcal H}$, 
	\begin{equation} \label{local-hd-g0}
	{\rm H. dim}( \partial_{\rm LC}\mathcal H\cap \mathfrak{M}_l, g_0)=2d-2.
	\end{equation}
\end{thm}

It's clear that Theorem \ref{lchd-gf-section1}   implies Theorem \ref{boundary-hd-lc}.

 \vspace{5pt}
\noindent{\bf Remark 1.3.}  
{\it  
	There are infinitely many geometrically finite maps $g_0$ satisfying \eqref{local-hd-g0} (see Section \ref{proof-main}).
	  These maps are not dense on $\partial \mathcal H$.
	  
	  Moreover, unlike  $\partial \mathcal C_d$, the   boundary $\partial \mathcal H$ is not \emph{homogeneous} in the sense that ${\rm H. dim}( \partial \mathcal H\cap U)=2d-2$ for all open sets $U$ with $U\cap \partial \mathcal H\neq \emptyset$. }
   \vspace{5pt}



In order to prove Theorem \ref{lchd-gf-section1}, we develop a new framework, called  {\it Blaschke divisors over a mapping scheme}, to analyze the boundaries of higher-dimensional hyperbolic components.
The strategy proceeds as follows. 

According to Milnor \cite{M2}, any hyperbolic  component $\mathcal H\subset \mathcal C_d$\footnote{For convenience, we   actually use the finite (possibly branched) covering space of $\mathcal H$, without changing the notation,    see  Section \ref{dmhc}.} can be modeled by the space $\mathcal B^S$ of the (normalized) Blaschke products over a mapping scheme $S$, yielding a natural parameterization 
\begin{equation}\label{eq:homeo}
	\Phi: \mathcal B^S\rightarrow \mathcal H.
\end{equation}

  To study $\partial \mathcal H$ or $\overline{\mathcal H}$,  the   Blaschke products are allowed to  degenerate  in a reasonable way. 
  This  leads to the 
concept of {\it Blaschke divisor}. It consists of two parts: the \emph{non-degenerate part}, which consists of Blaschke products with possibly  lower degrees, and  the \emph{degenerate part}, which records the positions of the degeneration.
Note that  
Blaschke products
corresponds to Blaschke divisors without degenerations.
By coupling  these divisors with the mapping scheme $S$ of $\mathcal H$,  we obtain the space ${\rm Div}{(\overline{\mathbb D})}^S$ of 
Blaschke divisors over $S$, which serves as a  model for $\overline{\mathcal H}$.   The space  $\mathcal B^S$ can be  identified as the interior ${\rm Div}{({\mathbb D})}^S$ of ${\rm Div}{(\overline{\mathbb D})}^S$. 


In general, the homeomorphism $\Phi: \mathcal B^S \rightarrow \mathcal H$ can not extend continuously to the whole boundary $\partial\mathcal B^S=\partial{\rm Div}{({\mathbb D})}^S$. To address this, we  define the \emph{impression} of a divisor $ D\in \partial{\rm Div}{({\mathbb D})}^S$ as
$$
I_{\Phi}(D)=\Big\{f\in \partial \mathcal H; \text{ there exists } \{f_n\}_{n\geq 1} \subset \mathcal H \text{ with }  f_n\rightarrow f,  \Phi^{-1}(f_n)\rightarrow  D\Big\}.$$
By definition, the impression $I_{\Phi}(D)$ is  compact, and $\partial\mathcal H$ decomposes as 
\begin{equation}\label{eq:decomposition}
	\partial\mathcal H=\bigcup_{ E\in \partial {\rm Div}{(\mathbb D)}^S}I_{\Phi}(E).
\end{equation}
This  motivates us to explore  $\partial \mathcal H$ through the impressions of divisors.
Then some fundamental questions arise: what can we say about the maps in $I_{\Phi}(D)$? Under what conditions the set  $I_{\Phi}(D)$ is a singleton?

The novelty of this paper is to identify a class of divisors, called  {\it $\mathcal H$-admissible divisors} (see Definition \ref{adm-divisor}), denoted by $\mathcal A$, 
which are dense in $\partial \mathcal B^S=\partial{\rm Div}{({\mathbb D})}^S$ and whose impressions are singletons, such that 
$\Phi$ extends continuously to them. 
By exploring the self-intersections of $\partial \mathcal H$,   we establish the following boundary extension theorem:
\begin{thm}\label{thm:parameterization}
	The map $\Phi:\mathcal B^S\to \mathcal H$ defined in \eqref{eq:homeo}   extends   to  a homeomorphism 
	$$\overline{\Phi}:\mathcal B^S\cup \mathcal A\to \mathcal H\cup \partial_\mathcal{A}\mathcal H,$$
	where $\partial_\mathcal{A}\mathcal H:=\overline{\Phi}(\mathcal A)$, and $\overline{\Phi}(D)$ is defined to be the unique map in  $I_{\Phi}(D)$ for $D\in \mathcal A$.
\end{thm}

Theorem \ref{thm:parameterization} provides a natural parameterization of $\partial_\mathcal A\mathcal H$ by $\mathcal A$. 
It is worth noting that, while the set $\mathcal A$ of $\mathcal H$-admissible divisors is dense in $\partial {\rm Div}(\mathbb{D})^S$,   its  counterpart   $\partial_\mathcal A\mathcal H$ is not dense in $\partial \mathcal H$ (this implies that  $\Phi$ can not extend to the whole boundary $\partial\mathcal B^S$).

The maps in $ \partial_\mathcal{A}\mathcal H$  are called   {\it $\mathcal H$-admissible maps}. 
These maps exhibit two key features: first, they have controlled critical orbits, which allows us to establish the local connectivity of $\partial \mathcal H$ at the maps; second, there are abundance of such maps on  $\partial \mathcal H$, which enables us to estimate the local Hausdorff dimensions  of  $\partial \mathcal H$  (in fact, we have ${\rm H.dim}(\partial_\mathcal A\mathcal H)=2d-2$ by Theorem \ref{lchd-gf}).  These properties are indispensable for the proof of Theorem \ref{lchd-gf-section1}. 






Firstly, we establish  the local connectivity theorem: 


\begin{thm} [Local connectivity]  \label{boundary-lc}  For any  $f\in \partial_\mathcal A\mathcal H$,  the boundary    $\partial\mathcal H$ is locally connected at   $f$.  That is,  $\partial_\mathcal{A} \mathcal H\subset \partial_{\rm LC} \mathcal H$.   
\end{thm}

To show the local   connectivity of $\partial \mathcal H$ at  $f$, we require a basis of open and connected neighborhoods of $f$ on $\partial \mathcal H$, with diameters  shrinking to $0$. However, since $\partial \mathcal H$ is a higher dimensional fractal set\footnote{In general, self-bumps occur on $\partial \mathcal H$, and  $\overline{\mathcal H}$ is not  a topological manifold with boundary \cite{Luo}.}, visualizing such neighborhoods is challenging. 
We resolve this difficulty by modeling  neighborhoods of $f$ through  {\it exotic fibre bundles over $(2d-3)$-cells}.  Specifically, each model $\mathcal M(\mathbf x, r)\subset \partial \mathcal H$ 
can be described as 
$$\mathcal M(\mathbf x, r):=\bigsqcup_{\mathbf y\in \mathbb B(\mathbf x, r)}\mathcal F(\mathbf y),$$
where 
\begin{itemize}
	\item the base point $\mathbf x$ is an $\mathcal H$-admissible divisor with $I_{\Phi}(\mathbf x)=\{f\}$;
	\item the base space  $\mathbb B(\mathbf x, r)$ is an open ball in $\partial {\rm Div}(\mathbb D)^S$,  with real dimension $(2d-3)$;
	
	\item for each $\mathbf y\in \mathbb B(\mathbf x, r)$, the fiber $\mathcal F(\mathbf y)$  corresponds to the impression $I_\Phi(\mathbf y)$.
	
\end{itemize}

The model $\mathcal M(\mathbf x, r)$  is exotic for two reasons: firstly,  the fibres are not homogeneous; secondly, the non-trivial fibres are compact fractals rather than regular manifolds. In our proof of Theorem \ref{boundary-lc} (in Section \ref{lc}), a basis of neighborhoods of $f$ on $\partial \mathcal H$ is given by a collection of models $\{\mathcal M(\mathbf x, r_n)\}_{n\geq 1}$, where $r_n\to 0$ as $n\to\infty$.\vspace{3pt}

Next, we are concerned with the local strata structure of  $\partial_\mathcal{A}\mathcal H$. This requires the notation of \emph{marked Fatou components}. \vspace{3pt}

Let ${\rm Crit}(g)$ denote the critical set of a polynomial $g$ in $\mathbb C$.   Let $f_0$ be the unique postcritically finite map in $\mathcal H$, and
let $V=\bigcup_{l\geq 0}f_0^l({\rm Crit}(f_0))$ be the finite set endowed with a self map $\sigma: V\rightarrow V, \sigma(v)=f_0(v)$. The   Fatou component of $f_0$ containing $v\in V$, denoted by $U_{f_0, v}$, deforms continuously when the base map  $f_0$  moves throughout $\mathcal H$. This naturally  induces a finite collection of  marked Fatou components $(U_{g, v})_{v\in V}$ for each $g\in \mathcal H$.  
Roughly speaking, an $\mathcal H$-admissible map $f$ is a Misiurewicz map on  $\partial \mathcal H$ with exactly $m_{\mathcal H}$ attracting cycles,  so that for any sequence $\{f_k\}_k\subset \mathcal H$ approaching $f$, the  marked Fatou components $(U_{f_k, v})_{v\in V}$  pinch off to form  a finite collection of the  marked Fatou components $(U_{f, v})_{v\in V}$ of $f$, with the following properties 

\begin{itemize}
	\item  $f(U_{f,v})=U_{f, \sigma(v)}$ for all $v\in V$;
	
	\item  ${\rm Crit}(f)\subset \bigcup_{v\in V}\overline{U_{f,v}}$;
	
	\item  each critical point   $c\in \bigcup_{v\in V}\partial{U_{f,v}}$ is simple, 
	disconnecting the Julia set $J(f)$ into two parts, 
	and the  forward orbit of $c$ avoids the boundary marked points and other critical points.  
\end{itemize}

We now estimate  the local Hausdorff dimension for 
each boundary strata of $\partial_\mathcal{A} \mathcal H $  near an  $\mathcal H$-admissible map.  For this end,
for each $\mathcal H$-admissible map $f$, we  associate a vector $\mathbf n(f)=(n_v(f))_{v\in V}$, where $n_v(f)= \#(\partial U_{f,v}\cap {\rm Crit}(f))$ is the number of the $f$-critical points  on $\partial U_{f,v}$. 
Let $\mathcal U$ be a small neighborhood of $f$ so that  the Fatou components   $(U_{g, v})_{v\in V}$  for $g\in \mathcal U$ can also be marked in the way that their centers  move continuously.  Note that $\mathcal U\ni g\mapsto\partial U_{g, v}$ is not Hausdorff continuous.

\begin{thm} [Local Hausdorff dimension] 
	\label{local-hausdorff-dim} 
	Let $f\in \partial\mathcal H$ be an $\mathcal H$-admissible map.
	For any  $ \mathbf n\neq \mathbf 0$ and $\mathbf{0}\leq  \mathbf n:=(n_v)_{v\in V}\leq \mathbf n(f)$, define the $ \mathbf n$-boundary stratum
	$$(\partial_\mathcal{A} \mathcal H)^{\mathbf n}_{\mathcal U}:=\big\{g\in \mathcal U\cap \partial_\mathcal{A} \mathcal H;  \     \mathbf n(g)=\mathbf n
	\big\}.$$
	Then  we have 
	$${\rm H.dim}((\partial_\mathcal{A} \mathcal H)^{\mathbf n}_{\mathcal U}, f)\geq \sum_{v\in V} n_v\cdot {\rm H.dim}(\partial U_{f,v})+2(d-1-|\mathbf n|), \ |\mathbf n|=\sum_{v\in V} n_v.$$
\end{thm}

One consequence of Theorem \ref{local-hausdorff-dim}  is, if    $\min_{v\in V}{\rm H.dim}(\partial U_{f,v})$ is sufficiently close to $2$ (see Proposition \ref{sup-hb-Fatou} for the existence of such $f$), then the boundary  stratum  $(\partial_\mathcal{A} \mathcal H)^{\mathbf n}_{\mathcal U}$ is
substantial, as its Hausdorff dimension approaches  $2d-2$. 


The proof  of Theorem \ref{local-hausdorff-dim}   is based on a perturbation theory on  the boundary $\partial \mathcal H$.  To formulate the theorem, we quickly introduce some notations. Let $f$ be an $\mathcal H$-admissible map,  whose critical points on  Julia set  are marked by $c_1, \cdots, c_n$ (they are different since each one is simple), so that $c_{j}\in \partial U_{f, v_j}$ for each $j$. Shrink $\mathcal U$ if necessary, there is a continuous map $c_j:   \mathcal U\rightarrow \mathbb C$ so that $c_j(g)$ is a $g$-critical point for  all $g\in  \mathcal U$ and all $1\leq j\leq n$.

For each $v\in V$, define 
$$\partial_0 U_{f, v}=\{x\in \partial U_{f,v};  J(f)\setminus\{x\}  \text{ is  connected}\}.$$
An equivalent definition  of $\partial_0 U_{f, v}$ in terms of {\it limb} is  given in \eqref{fatou-boundary-0}. 
Let $V_{\rm p}$ consist of all $\sigma$-periodic indices $v\in V$.

With these notations, the perturbation theorem reads:




\begin{thm}  [Perturbation on  $\partial \mathcal H$] \label{local-perturbation0} 
	Let $f\in \partial\mathcal H$ be an $\mathcal H$-admissible map, whose  critical points on Julia set are marked by  $c_1, \cdots, c_n$.  Suppose  the $f$-orbit of $c_j$ meets $\partial U_{f, u_j}$ for some  $u_j\in V_{\rm p}$.   Given  an $f$-hyperbolic set $X\subset \bigcup_{v\in V_{\rm p}}\partial_0 U_{f, v}$,  a holomorphic motion $h: \mathcal U\times X\rightarrow \mathbb C$  of $X$  based at $f$ (shrink $\mathcal U$ if necessary), an  index set $\emptyset\neq\mathcal J\subset \{1, \cdots, n\}$.

	Given  any number $\varepsilon>0$, and any 
	multipoint $ \mathbf x=(x_{j})_{j\in \mathcal{J}}\in  \prod_{j\in \mathcal J}(X\cap \partial U_{f, u_j})$.
	
	(1).  There exist  
	$g_0\in   \partial_\mathcal{A} \mathcal H\cap  \mathcal N_{\varepsilon}(f)$\footnote{Here $\mathcal N_{\varepsilon}(f)$ is the $\varepsilon$-neighborhood of $f$ in $\mathcal P_d$.},  positive integers $(m_{j})_{j\in  \mathcal{J}}$, 
	so that 
	$$g_0^{m_{j}}(c_{j}(g_0))=h(g_0,  x_{j}) , \forall j\in \mathcal J; \ \ c_{j}(g_0)\in U_{g_0, v_j},  \forall j\in\{1, \cdots, n\}\setminus \mathcal{J}.$$

	(2). Moreover, there is $\rho_0\in (0, \varepsilon)$ so that 
	$$\mathcal E(g_0, \rho_0):=\big\{g\in \mathcal N_{\rho_0}(g_0); \ g^{m_j}(c_j(g))\in h(g, X),  \  \forall j\in \mathcal J\big\}\subset \partial_\mathcal{A} \mathcal H\cap \mathcal  N_{\varepsilon}(f). $$
\end{thm}

The last ingredient  in the proof of  Theorem \ref{lchd-gf-section1} is the theory of parabolic implosion.  With the help of this technique, Shishikura \cite{S} construct \emph{large}\footnote{Here,  a hyperbolic set in $\mathbb C$ is `large', if its Hausdorff dimension is close to $2$}  hyperbolic sets within the Julia sets of rational maps (see  Theorem \ref{thm:shishikura}). In this work, we refine his approach by controlling the orbits of the points in these hyperbolic sets. This improvement formulates a condition on which the large hyperbolic sets lie on the boundaries of attracting Fatou components (see Lemma \ref{lem:on-boundary}). Applying this result to polynomials, we obtain the following:

\begin{thm}\label{thm:boundary-dimension0}
	Let $f\in\mathcal C_d$ be a subhyperbolic\footnote{A rational map $f$ is called  {\it subhyperbolic} if each of its critical orbits either is preperiodic or tends to an attracting cycle .} map. Let $O_1(f),\ldots,O_m(f)$ be a collection of marked geometrically-attracting cycles\footnote{It means that the multiplier of this cycle has modulus  $\in(0,1)$.}, each of which attracts only one critical point. Then there exists a geometrically finite map $g_0\in\mathcal C_d$ so that the following property holds:
	
	For any $\epsilon,\rho>0$, there exists a polynomial $\t f\in\mathcal N_\rho(g_0)$ quasiconformally conjugate to $f$ such that, for every $k\in\{1,\ldots,m\}$,
	\begin{itemize}
		\item the boundary of the immediate attracting basin of $O_k(\t f)$ contains a hyperbolic set with Hausdorff dimension larger than $2-\epsilon$, and\vspace{2pt}
		\item this hyperbolic set is disjoint from the cut-points\footnote{We say $a\in K$ is a cut point of a  compact continuum $K$ if $K\setminus\{a\}$ is disconnected.} of $J(\t f)$.
	\end{itemize}
\end{thm}

The panorama of the proof and implications of the theorems are  summarized as follows: 

\begin{figure}[h]
	\begin{center}
		\includegraphics[height=7cm]{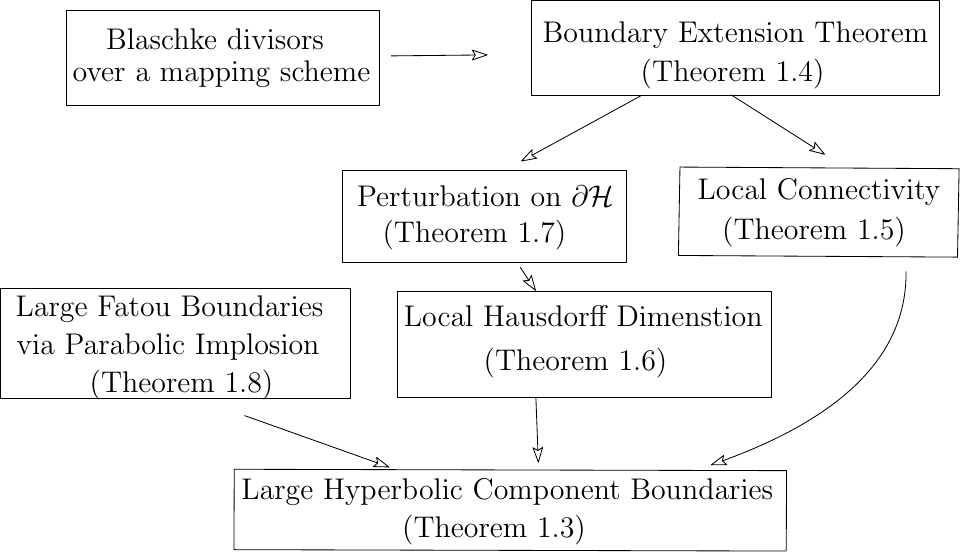}
	\end{center}
\end{figure}

\noindent\textbf{Structure of the paper.}  In Section \ref{perturbation-hyp-set}, we  examine  some perturbation properties  of hyperbolic sets for rational maps.

The theory of Blaschke divisors over a mapping scheme is developed 
 from Section \ref{Blaschke-divisors-over-S} through Section \ref{impression-singleton}. Precisely, Blaschke divisors over a mapping scheme are introduced in Section \ref{Blaschke-divisors-over-S}, and they are related to the hyperbolic components via Milnor's parameterization   in Section \ref{p-d-h-c}.  In Section \ref{imp-divisor},
the impression $I_{\Phi}(D)$ of a divisor $D$ is introduced and the divisor equality for the maps in $I_{\Phi}(D)$ is established.  It is worth noting that the combinatorial information of the hyperbolic component comes into play.    The combinatorial properties for the maps in 
$I_{\Phi}(D)$ are given in Section \ref{comb-pro-imp}. In Section \ref{h-adm-div}, we define and study  {\it $\mathcal H$-admissible divisors}. Especially,  a  criterion for $\mathcal H$-admissibility is given therein.   In Section \ref{impression-singleton}, we shall show that the impression $I_{\Phi}(D)$ of  an $\mathcal H$-admissible divisor $D$ is a singleton. This key property enables the study of the boundary extension of $\Phi$, the local connectivity of $\partial \mathcal H$,  and the perturbation along   $\partial \mathcal H$.


In Section \ref{bet}, we prove the boundary extension theorem--Theorem \ref{thm:parameterization}, based on a  non local self-intersection property of $\partial \mathcal H$.
In Section \ref{lc},  we prove Theorem \ref{boundary-lc}. In Section \ref{pb-boundary},  by  some perturbation properties of hyperbolic sets given  in Section \ref{perturbation-hyp-set},  we prove a slightly general form (Proposition \ref{pertubation-boundary-strata}) of Theorem \ref{local-perturbation0} (1), and  Theorem \ref{local-perturbation0} (2) is proved in  Section \ref{lchd-via-p}.
Based on Theorem \ref{local-perturbation0},  a more precise form of 
Theorem \ref{local-hausdorff-dim}  is proven in Section \ref{lchd-via-p}. 

The theory of parabolic implosion is reviewed in Sections \ref{sec:parabolic-implosion}.  In Section \ref{hd-attracting},  a new property of the large hyperbolic set (under certain condition, see Lemma \ref{lem:on-boundary})  is formulated, and this allows us to construct    Fatou   boundary  with arbitrarily large Hausdorff dimension (Theorem \ref{thm:boundary-dimension0}).

We prove Theorem \ref{lchd-gf-section1}   in Section \ref{proof-main}.






\vspace{5pt}
\noindent\textbf{Notes and references.}   
The Hausdorff dimension of  the bifurcation locus in the families of polynomials or rational maps has been extensively studied  in \cite{S,Tan-dimension, Mc00, G}.

The   hyperbolic components in the connectedness  locus $\mathcal C_d\subset \mathcal P_d$ are shown to be topological cells \cite{M2}.  Their boundaries are complicated \cite{Mc94b, Mil-boundary}. Previous works are focus on settings of the main hyperbolic component $\mathcal H_d$ or some lower dimensional subspaces. 

An analytic coordinate of $\mathcal H_3$ with extension to the regular part $\partial_{\rm reg}\mathcal H_3$   is given in \cite{PT}. 
 The boundary $\partial \mathcal H_3$ is explored from the combinatorial viewpoints by \cite{BOPT14, BOPT18}. 
  A combinatorial classification of geometrically finite polynomials on $\partial \mathcal H_d$ is established in \cite{Luo}, based on the dynamics of  Blaschke products \cite{Mc3, Mc4, Mc2}.
 
 In some typical subspaces of $\mathcal P_d$, the topology of the hyperbolic component boundaries is known \cite{F, R, W21a, CWY}. A rational case is given in \cite{Mil-boundary}.


\vspace{5pt}
  
  \noindent\textbf{Acknowledgments.}  
  The research is supported by the National Key R$\&$D Program of China (Grant No. 2021YFA1003200, 2021YFA1003203), NSFC (Grant No. 12322104), NSFGD (Grant No. 2023A1515010058), and the Fundamental Research Funds for the Central Universities.

\section{Perturbation of hyperbolic sets}\label{perturbation-hyp-set}

Let $\Lambda$ be a  connected topological space (resp. complex manifold).  
A {\it  continuous (resp. holomorphic) motion } of a set $X\subset \C$ with basepoint $\lambda_0\in \Lambda$,  is a continuous mapping 
$h: \Lambda \times X\rightarrow \C$
satisfying that

\begin{itemize}
	\item for any  $\lambda\in \Lambda$, $x\mapsto h(\lambda, x)$ is injective in $X$;
	
	\item  for any  $x\in X$, $\lambda\mapsto h(\lambda, x)$ is continuous (resp. holomorphic);
	
	\item $h(\lambda_0, \cdot)={\rm id}$.
\end{itemize}

Let $\Omega\subset \C$ be a domain, and $f: \Omega \rightarrow \C$ be a holomorphic map. 
A compact set  $X\subset \Omega$  is called a {\it hyperbolic set } for $f$,  if $f(X)\subset X$ and there exist $c>0$ and $\mu>1$ such that 
$\|(f^k)'(x)\|\geq c\mu^k$ for all $x\in X, k\geq 1$, where $\|\cdot \|$ denotes the norm of
the derivative with respective to the spherical metric of $\C$.  The spherical metric can be replaced by any comparable conformal metric  on $X$. 
It is known that if $f$ is a rational map, then $X\subset J(f)$.

The following fact implies that a hyperbolic set admits a locally unique continuous motion in a continuous family of holomorphic maps. This result is well known in the setting of holomorphic family of rational maps (see \cite{MS} Theorem 2.3 page 225 or \cite{S} Section 2).

\begin{pro} \label{hyp-set}  Let $\mathcal{F} $ be a continuous family of holomorphic maps defined on a domain $\Omega\subset \C$, and let $X\subset \Omega$ be a hyperbolic set of $f_0\in \mathcal F$.   
	
	1. (Existence)  There is a neighborhood $\mathcal N\subset\mathcal{F}$ of $f_0$, and a continuous motion $h:\mathcal{N}\times X\rightarrow \C$ of $X$ with basepoint $f_0$, satisfying that 
	\begin{equation}\label{conj-motion}
		f(h(f,x))=h(f, f_0(x)), \ \forall(f, x)\in \mathcal{N}\times X.
	\end{equation}
	Further, if $\mathcal F$ is a holomorphic family, then $h$  is a holomorphic motion.

	
	2. (Local uniqueness, for continuous motions) If  $h_1, h_2:\mathcal{N}\times X\rightarrow \C$ are two continuous motions satisfying (\ref{conj-motion}),  then there is a neighborhood $\mathcal{V}\subset   \mathcal{N}$ of $f_0$ such that 
	$$h_1|_{\mathcal{V}\times X}=h_2|_{\mathcal{V}\times X}.$$
	
	3. (Uniqueness, for holomorphic motions) If $h_1, h_2:\mathcal{N}\times X\rightarrow \C$  are two
	holomorphic motions satisfying   (\ref{conj-motion}),  then $h_1 = h_2$.

	
\end{pro}

\begin{rmk}   The local uniqueness does not imply  uniqueness. 
	The reason is that the motion $h(f, X)$ of the $f_0$-hyperbolic set $X$ might not be a hyperbolic set for a certain $f\in \mathcal N$.  Here is an example. Let 
	$f_t(z)=z^2+t, t\in \mathcal{N}: =(-1,1)$.   Consider two continuous motions of  $X=\{1\}$:
	$$h_1(f_t, 1) =\begin{cases}\frac{1+\sqrt{1-4t}}{2},  & t\in (-1, 1/4]\\
		\frac{1+i\sqrt{4t-1}}{2},  & t\in (1/4, 1)
	\end{cases}, \  h_2(f_t, 1) =\begin{cases}\frac{1+\sqrt{1-4t}}{2},  & t\in (-1, 1/4]\\
		\frac{1-i\sqrt{4t-1}}{2},  & t\in (1/4, 1)
	\end{cases}.$$
	It is clear that $h_k(f_t, 1)$ is a fixed point of $f_t$ for all $t\in \mathcal{N}$. Note that
	$h_1\neq h_2$, but $h_1(f_t, 1)=h_2(f_t,1)$ for $t\in (-1, 1/4)$.
\end{rmk}

\begin{proof}   The argument   combines   the proof of   \cite[Theorem 2.3]{MS}  and  \cite[Section 2]{S}. 
	Without loss of generality, assume $X\subset \mathbb{C}$.
	Since $X$ is a hyperbolic set of $f_0$, there is an integer $m\geq 1$ such that 
	$|(f_0^m)'(x)|>1$ for all $x\in X$. 
 
	Define a conformal metric $\rho=\rho(z)|dz|$ in a neighborhood $U$ of $X$, where
	$$\rho(z)=\sum_{k=0}^{m-1}|(f_0^k)'(z)|, \ z\in U.$$ 
	One may verify that for all $x\in X$,  the $\rho$-derivative of $f_0$
	$$|f_0'(x)|_\rho:=\frac{\rho(f_0(x))}{\rho(x)}|f_0'(x)|=1+\frac{|(f_0^m)'(x)|-1}{\rho(x)}>1.$$ 
	
	Let $D_\rho(x, r)$ be the $\rho$-disk of $x\in X$.    By the  uniform continuity of the map   $(\mathcal F\times U)\ni (f,z)\mapsto |f'(z)|_\rho$ on compact sets,  there exist $r>0, \kappa>1$ and a neighborhood $\mathcal N$ of $f_0$ such that
	$X^r:=\bigcup_{x\in X} D_\rho(x, r)\subset U$, and
	\bess	&|f'(z)|_\rho:=\frac{\rho(f(z))}{\rho(z)}|f'(z)|\geq\kappa, \   \forall (f, z)\in \mathcal{N}\times X^r,&\\
	&	D_{\rho}(f_0(x), r)\Subset f(D_{\rho}(x, r)) , \  \forall (f,x)\in \mathcal{N}\times X.&
	\eess


	
	
	For any $(f,x)\in \mathcal{N}\times X$, define
	$$E^f_n(x)=\bigcap_{0\leq k\leq n}f^{-k}(\overline{D_{\rho}(f_0^k(x), r)}), \  n\geq 0.$$
	By the second property above, we have $E^f_0(x)\supset E^f_1(x)\supset\cdots$. Since  $f^n:  E^f_n(x)\rightarrow D_{\rho}(f_0^n(x), r)$ is conformal  and $|(f^n)'(z)|_\rho\geq \kappa^n$ for all $z\in E^f_n(x)$, we have
	${\rm diam}_{\rho}(E^f_n(x))\leq 2r\kappa^{-n}.$
	Therefore $E^f_\infty(x):=\bigcap E^f_n(x)$ is a singleton, denoted by $h(f,x)$.
	
Note that $h(f_0, \cdot)={\rm  id}$, $h(f, \cdot)$ is injective  and equality (\ref{conj-motion}) holds. To see that $h$ is continuous,    define a sequence of maps 
	$H_n: \mathcal{N}\times X\rightarrow \C$ by
	\begin{equation} \label{inverse-lim}
		H_n(f,x):=f|_{D_{\rho}(x, r)}^{-1}\circ\cdots \circ f|_{D_{\rho}(f_0^n(x), r)}^{-1}(f_0^{n+1}(x)), \ (f,x)\in \mathcal{N}\times X,  \ n\geq 1.
	\end{equation}
	One may verify  that $H_n$ is continuous, and  $H_n(f,x)\in E^f_{n+1}(x)$.
	The fact 
	$$d_\rho(H_n(f,x), h(f,x)) \leq {\rm diam}_{\rho}(E^f_{n+1}(x))\leq 2r\kappa^{-n-1}$$ 
	implies that
	$h$ is the uniform limit of $H_n$'s, hence $h$ is continuous.

	If $\mathcal{F}$ is a holomorphic family,  by equation (\ref{inverse-lim}),  $H_n(f,x)$ is holomorphic in $f$. 
	By Weierstrass Theorem, $h$ is a holomorphic motion.

	2.  Since  $h_1, h_2:\mathcal{N}\times X\rightarrow \C$ is continuous, by the uniform continuity in compact sets, for the $r>0$ given  above, there is a smaller neighborhood $\mathcal{V}\Subset \mathcal{N}$ of $f_0$ such that for any $(f,x)\in \mathcal{V}\times X$,
	\begin{equation} \label{uniform-continuity}
		d_{\rho}(h_j(f, x), h_j(f_0, x))= d_{\rho}(h_j(f, x), x)< r, \ \ j=1,2.
	\end{equation}
	Since $h_1, h_2$ satisfy equation (\ref{conj-motion}),  for all $x\in X$ and $n\geq 1$,  by (\ref{uniform-continuity}) we have
	$f^n(h_j(f,x))= h_j(f, f_0^n(x))\in D_{\rho}(f_0^n(x), r)$. It follows that 
	$h_j(f,x)\in E^f_\infty(x)$.  As is known that  $E^f_\infty(x)$ is a singleton, we have $h_1(f,x)=h_2(f,x)$.
	
	3.  By 2,  
	we have the  local uniqueness:
	$h_1(f,x)=h_2(f,x)$  whenever    $(f,x)\in \mathcal{V}\times X$. Since $f\mapsto h_j(f, x)$ is holomorphic, the Identity Theorem of holomorphic maps
	gives $h_1(f, x)=h_2(f, x)$ for  $(f,x)\in \mathcal{N}\times X$.
\end{proof}

  Let $\mathcal{F} $ be a continuous family of rational maps of degree $d\geq 2$.
  Let $f_0\in \mathcal{F}$ have an attracting periodic point $\alpha_0$, with period $p\geq 1$. 
  Let $U_{f_0}(\alpha_0)$ be the Fatou component containing 
   $\alpha_0$.  
Then there is a neighborhood $\mathcal U$ of $f_0$ and a continuous map $\alpha: \mathcal U\rightarrow \mathbb C$ so that $\alpha(f)$ is an attracting point of $f\in \mathcal U$, with $\alpha(f_0)=\alpha_0$.  
Let $U_f(\alpha(f))$ be the Fatou component containing $\alpha(f)$.

\begin{pro} \label{holo-hyperbolic-set} Let $X\subset \partial U_{f_0}(\alpha_0)$ be a hyperbolic set of $f_0^p$.  Suppose that $\partial U_{f_0}(\alpha_0)$  is connected and locally connected.  Then there exist $\mathcal V\subset \mathcal U$,   and a continuous motion $h: \mathcal V\times X\rightarrow \mathbb C$,   such that  
	$h(f,  X)\subset  \partial U_{f}(\alpha(f))$ and
		$f^p( h(f, z))=h(f, f^p_0(z))$ for any $ (f,z)\in \mathcal V\times X.$
\end{pro}

\begin{proof} Replacing $\mathcal F$ with $\mathcal F^p:=\{f^p;f\in \mathcal F\}$,  it suffices to assume $p=1$. 
	Let $h: \mathcal V\times X\rightarrow \mathbb C$ be the continuous motion of $X$ given by Proposition \ref{hyp-set}. 
	 We first reduce the problem to a special case by quasi-conformal surgery. 	\vspace{5pt}

	 {\it Claim. There exist continuous families of rational maps $\{g_f\}_{f\in \mathcal V}$ 
	 		 and quasiconformal maps 
	 		 $\{\eta_f\}_{f\in\mathcal V}$ on the Riemann sphere, such that $\eta_f(U_f(\alpha(f)))$ has a fixed point of $f$ with local degree $m:={\rm deg}(f_0|_{U_{f_0}(\alpha_{0})})$ and  $g_f\circ\eta_f= \eta_f\circ f$ on $J(f)$, for every $f \in \mathcal{V}$. }\vspace{3pt}
	 
\begin{proof}[Proof of the claim]	
 Take three Jordan disks $U\Subset V\Subset W\subset U_{f_0}(\alpha_0)$, with smooth boundary, so that $f_0: W\rightarrow V$ and  $f_0:V \rightarrow U$ are proper of degree $m$.  	By shrinking $\mathcal{V}$ if necessary, we get   continuous families  of Jordan disks $\{W_f\}_{f\in \mathcal V}$, $\{V_f\}_{f\in \mathcal V}$ such that $U\Subset V_f\Subset W_f$ for all $f\in \mathcal V$, and   degree $m$ proper maps  $\{f: W_f\rightarrow V_f\}_{f\in \mathcal V}$ and  $\{f:V_f \rightarrow U\}_{f\in \mathcal V}$.
	   
	   There are Riemann mappings $\phi_{f}: (V_f, \alpha(f))\rightarrow (\mathbb D, 0)$, $\varphi_{f}: (U, \alpha(f))\rightarrow (\mathbb D, 0)$,  continuous in $f \in \mathcal{V}$  with $\phi_{f}'(\alpha(f))>0, \varphi_{f}'(\alpha(f))>0$.
  Define 
	\begin{equation*}
		\tilde{f}(z)=\begin{cases} \varphi_f^{-1} (\phi_f(z)^m ),  &z \in V_f,\\
				\xi_f (z), &z\in W_f\setminus V_f,\\
				f(z), & \text{elsewhere}.
		\end{cases}
	\end{equation*}
	where $\xi_f:  W_f\setminus V_f\rightarrow V_f\setminus U$ is the  quasiregular interpolation, continuous with respect to $f\in \mathcal V$.
	Let $\sigma_0$ be the standard complex structure.
	For $f\in \mathcal V$, define an $\tilde{f}$-invariant complex structure as follows:
	\begin{equation*}
		\sigma_{f}=
		\begin{cases} 	(\tilde{f}^{k})^{*}(\sigma_{0}),  &\text{ in } \tilde{f}^{-k} (W_f\setminus V_f),  k\geq 1,\\
		\sigma_{0}, & \text{elsewhere}.
		\end{cases}
	\end{equation*}
	The complex structure $\sigma_{f}$ is continuous with respect to $f \in \mathcal{V}$, and the Beltrami coefficient $\mu_f$ of $\sigma_{f}$  satisfies $\|\mu_f\|_\infty<1$.
	
	According to Measurable Riemann Mapping Theorem, there is a  quasiconformal map  $\eta_f:  \C \to \C$ with $0,1,\infty$ fixed so that $\eta_f^*\sigma_0=\sigma_f$. 
	Then $g_f=\eta_f \circ \tilde{f}\circ\eta_f^{-1}$ is a rational map. By the parameter dependence and equicontinuity of quasiconformal maps, the map $(f,z)\mapsto \eta_f(z)$ is continuous on $\mathcal{V}\times \C$. 
	Since $\tilde{f}|_{J(f)}=f|_{J(f)}$, we have $g_f(\eta_f(z))= \eta_f(f(z)), \ \forall z \in J(f)$.
	\end{proof}
	
	 For every $f\in\mathcal V$, set $\beta_f:=\eta_f(\alpha(f))$.  Let $\psi_{g_f}$ be the B\"ottcher coordinate of $g_f$  near $\beta_f$, so that $\psi_{g_f}$ is continuous with respect to $f$. Since $\partial U_{f_0}(\alpha_0)$ is connected, $\psi_{g_{f_0}}$ extends to a conformal map on  $U_{g_{f_0}}(\beta_{f_0})$.
By assumption, $\partial U_{g_{f_0}}(\beta_{f_0})=\eta_{f_0}(\partial U_{f_0}(\alpha_0))$
		is locally connected.  
	Hence for all 
		$t\in \mathbb R/\mathbb Z$, the internal ray $ \psi_{g_{f_0}}^{-1}([0,1)e^{2\pi it})$ lands at a point on  $\partial U_{g_{f_0}}(\beta_{f_0})$.
		
		For each $x\in X$, let 
		$$\Theta(x)=\{t\in \mathbb R/\mathbb Z; \text{ the internal ray } \psi_{g_{f_0}}^{-1}((0,1)e^{2\pi it}) \text{ lands at }\eta_{f_0}(x)\}.$$
		Note that $\Theta(x)$ might be not a singleton.
		The set $\Theta(X):=\bigcup_{x\in X}\Theta(x)$ is compact in $\mathbb R/\mathbb Z$ and satisfies $m \Theta(X) \subset \Theta(X)$.
		Since $X$ is $f_0$-hyperbolic,   it admits an expanding metric $\rho$ and  a family of disks $\{D_{\rho}(x,r)\}_{x\in X}$, as given by Proposition \ref{hyp-set}. Since $\psi_{g_{f_0}}^{-1}: \overline{\mathbb D}\rightarrow \overline{U_{g_{f_0}}(\beta_{f_0})}$ is uniformly continuous, we may choose $T\in(0,1	)$ sufficiently close to $1$ so that $\psi_{g_{f_0}}(\overline{W})\subset \mathbb D(0, T)$ and
		$$\psi_{g_{f_0}}^{-1}([T, T^{1/m}]e^{2\pi it})\subset \eta_{f_0}(D_{\rho}(x,r)), \ \forall t\in \Theta(x), \ \forall x\in X.$$
		Shrink $\mathcal V$ so that $\psi_{g_f}$ can be defined over $W_f$ with $\psi_{g_f}(\overline{W_f})\subset  \mathbb D(0, T)$,  and $\psi_{g_f}^{-1}$ can be defined in $\overline{\mathbb D(0, T^{1/m})}$,  for all $f\in \mathcal V$.  By  the continuity of  	  
$$\mathcal V\times  \overline{\mathbb D(0, T^{1/m})}\ni (f, w)\mapsto \psi_{g_f}^{-1}(w), \ \mathcal V\times \C\ni (f,z)\mapsto \eta_f(z),$$
and shrinking $\mathcal V$ further if necessary, we have 
			$$\psi_{g_f}^{-1}([T, T^{1/m}]e^{2\pi it})\subset \eta_{f}(D_{\rho}(x,r)), \ \forall\, t\in \Theta(x), \ \forall\, (f,x)\in (\mathcal V,X).$$
			For each $f\in \mathcal V$, $x\in X$ and  $t\in \Theta(x)$,  let $L_{f,t}^0=\eta_f^{-1}(\psi_{g_f}^{-1}([T, T^{1/m}]e^{2\pi it}))$, 
			$$L_{f, t}^k=f|_{D_{\rho}(x, r)}^{-1}\circ\cdots \circ f|_{D_{\rho}(f_0^{k-1}(x), r)}^{-1}(L_{f, m^kt}^0), \ k\geq 1.$$
Note that $\tilde{f}=f$ near $J(f)$, hence we have the conjugacy $g_f^k\circ \eta_f= \eta_f\circ f^k$ in a small neighborhood of $J(f)$ (depending on $k$). By the fact $L_{f,\theta}^0\subset U_{f}(\alpha(f))$ for all $\theta\in \Theta(X)$ and induction,  we conclude that $L_{f, t}^k$ and $L_{f, t}^{k+1}$ have one common endpoint, and $L_{f, t}^k\subset D_\rho(x,r)\cap U_{f}(\alpha(f))$, for all $k\geq 0$. Hence 
\[\gamma_{f,t}:=\bigcup_{k\geq 0}L_{f, t}^k\subset D_\rho(x,r)\cap U_{f}(\alpha(f))\]
is a path.
    By the shrinking property shown in the proof of Proposition \ref{hyp-set},  this path has finite $\rho$-length and 
converges to $h(f, x)$.  Since $h(f, x)\in J(f)$, we get  $h(f, x)\in \partial U_{f}(\alpha(f))$.
\end{proof}

\begin{rmk}\label{rem:2.4} 1.  The non-triviality  of Proposition \ref{holo-hyperbolic-set} is that $ \partial U_{f}(\alpha(f))$ might be discontinuous in Hausdorff topology with respect to $f\in \mathcal{V}$.  

2. The assumption that  $\partial U_{f_0}(\alpha_0)$  is connected and locally connected seems unnecessary.  However, we will not discuss the general case, since it suffices for our purpose.

3. The proof of Proposition \ref{holo-hyperbolic-set} implies that the continuous motion  $h: \mathcal V\times X\rightarrow \mathbb C$ has an extension to a continuous motion of $ \bigcup_{t\in\Theta(X)}\overline{\gamma_{f_0,t}}\supset X$:
$$h: \mathcal V\times \bigcup_{t\in\Theta(X)}\overline{\gamma_{f_0,t}}\rightarrow \mathbb C$$
such that $h(f, \bigcup_{t\in\Theta(X)}\overline{\gamma_{f_0,t}})=\bigcup_{t\in\Theta(X)}\overline{\gamma_{f,t}}$ with $ \gamma_{f,t}\subset U_{f}(\alpha(f))$, and that  the motion is compatible with dynamics:
$$h(f, f_0^p(z))=f^p(h(f, z)), \ z\in \bigcup_{t\in\Theta(X)}\overline{\bigcup_{k\geq 1}L_{f_0, t}^k}.$$ 
\end{rmk}


\section{Blaschke divisors over a mapping scheme} \label{Blaschke-divisors-over-S}


In this section, we shall introduce the notion of mapping scheme,  and the space of  Blaschke divisors over it.
This space will serve as a model space for the closure of the  hyperbolic component in the connectedness locus of polynomials (see the next two sections).

\subsection{Divisors}  We recall some basics for divisors.
A {\it divisor} $D$ on a set $\Omega$ is a formal sum
$$D=\sum_{q\in \Omega} \nu(q) \cdot q,$$
where $\nu: \Omega\rightarrow \mathbb Z$ is a map, $\nu(q)\neq 0$ for only finitely many $q\in \Omega$.
 The {\it support}  ${\rm supp}(D)$ of $D$ is the finite set $\{q\in \Omega; \nu(q)\neq 0\}$.  
The divisor $D$ is called {\it integral}, if $\nu\geq 0$. The {\it degree} of an integral divisor $D$ is defined by 
${\rm deg}(D)=\sum_{q\in \Omega} \nu(q)$.

Let $D$ be an integral  divisor  on $\Omega$, and let  $f: \Omega_0\rightarrow \Omega$ be a finite-to-one holomorphic map. The pull-back divisor $f^*D$ on $\Omega_0$ is defined by
$$f^*D=\sum_{p\in \Omega_0}{\rm deg}(f, p)\nu(f(p)) \cdot p,$$
where ${\rm deg}(f, p)$ is the local degree of $f$ at $p$.
In particular, if $f$ is a proper map, then ${\rm deg}(f^*D)={\rm deg}(f)\cdot {\rm deg}(D)$,   
where ${\rm deg}(f)$ denotes the cardinality of  $f^{-1}(\zeta)$ counted with multiplicities for $\zeta\in \Omega$.

Given two integral divisors $D_k=\sum_{q\in \Omega} \nu_k(q) \cdot q, k\in\{1,2\}$, we say $D_1\geq D_2$ or $D_2\leq D_1$, if $\nu_1\geq \nu_2$.  Clearly   $D_2\leq D_1$ implies ${\rm supp}(D_2)\subset {\rm supp}(D_1)$.

In the following, we assume that $\Omega$ is a planar set, so it inherits a natural topology from $\mathbb C$.
For any integer $e\geq 1$,  let ${\rm Div}_e(\Omega)$ be the set of all degree-$e$ integral divisors on $\Omega$. 
There is a natural quotient map from $\Omega^e= {\Omega\times \cdots \times \Omega}$ to ${\rm Div}_e(\Omega)$,   $(z_1, \cdots, z_e)\mapsto \sum_{k=1}^e 1\cdot z_k$.
Therefore,  ${\rm Div}_e(\Omega)$ carries a natural quotient topology.
A  divisor $D\in {\rm Div}_e(\Omega)$ is usually written in two ways:
$D=\sum_{k=1}^{e} 1\cdot q_k$
(here the points $q_k$'s may  coincide) or  $D=\sum_{q\in {\rm supp}(D)}\nu(q)\cdot q$.

For any planar domain $\Omega\subset\mathbb C$,
one may verify that 
$\overline{{\rm Div}_e({\Omega})}={\rm Div}_e(\overline{\Omega})$.
Moreover, we have the following set-theoretic decompositions

$${\rm Div}_e(\overline{\Omega})=\bigsqcup_{d_1+d_2=e; \ d_1,d_2\geq 0}\Big({\rm Div}_{d_1}(\Omega)+
{\rm Div}_{d_2}(\partial\Omega)\Big),$$

$$\partial {\rm Div}_e(\overline{\Omega})=\bigsqcup_{d_1+d_2=e; \ d_1\geq 0,  d_2\geq 1}\Big({\rm Div}_{d_1}({\Omega})+
{\rm Div}_{d_2}(\partial\Omega)\Big).$$

If $\Omega=\mathbb D$, for $D=\sum_{k=1}^{e} 1\cdot q_k\in  {\rm Div}_e(\overline{\mathbb D})$, let $\N^0_{\epsilon}(D)\subset {\rm Div}_e(\mathbb D)$,
$\U^0_{\epsilon}(D)\subset {\rm Div}_e{(\overline{\mathbb D})}$ be defined as follows:
\bess &\N^0_\epsilon(D)=\Big\{E= \sum_{k=1}^{e} 1\cdot {a}_{k};  \
{a}_{k}\in \mathbb{D}(q_{k}, \epsilon)\cap \mathbb D, \  1\leq k\leq e\Big\},&\\
&\U^0_\epsilon(D)=\Big\{E= \sum_{k=1}^{e} 1\cdot {a}_{k};  \
{a}_{k}\in \mathbb{D}(q_{k}, \epsilon)\cap \overline{\mathbb D}, \  1\leq k\leq e\Big\}.&
\eess

 Let $\mathcal B_{\rm fc}^{e+1}$ denote the space of Blaschke products of degree $e+1$ that fix both  $0$ and $1$. Then 
each element $\beta\in\mathcal B_{\rm fc}^{e+1}$ can be written as
$$\beta(z)=e^{i\theta}z\prod_{k=1}^{e}\frac{z-a_k}{1-\overline{a_k}z},$$ 
where $\theta$ is chosen so that $\beta(1)=1$. 
Clearly  $\beta$ is uniquely determined by its {\it free zero divisor}  
${\rm div}(\beta):=\sum_{k=1}^{e} 1\cdot a_k\in {\rm Div}_{e}(\mathbb D)$. 
Conversely, each divisor  $D=\sum_{q\in {\rm supp}(D)}\nu(q)\cdot q\in {\rm Div}_{e}(\mathbb D)$ induces a Blaschke product $\beta=\beta(D)\in \mathcal  B_{\rm fc}^{e+1}$ whose free zero divisor is $D$:  $$\beta(z)=z\prod_{q\in {\rm supp}(D)}\Bigg(\frac{1-\bar{q}}{1-q}\cdot \frac{z-q}{1-\bar{q}z}\Bigg)^{\nu(q)}.$$
This establishes that the mapping   $\beta\mapsto {\rm div}(\beta)$  induces a homeomorphism between $\mathcal B_{\rm fc}^{e+1}$ and  ${\rm Div}_{e}(\mathbb D)$. Thus, we obtain a natural identification  $\mathcal B_{\rm fc}^{e+1}\cong {\rm Div}_{e}(\mathbb D)$ as topological spaces.


 The critical set of $\beta$ induces the {\it ramification divisor}: 
$$R_{\beta}:=\sum_{q\in \mathbb D}({\rm deg}(\beta,q)-1)\cdot q \in {\rm Div}_{e}(\mathbb D).$$ 
The composition of the two maps $D\mapsto \beta=\beta(D)$ and $\beta\mapsto R_{\beta}$ gives a self-mapping $\Psi_{e}$ of ${\rm Div}_{e}(\mathbb D)$: $D\mapsto R_{\beta(D)}$. 

\begin{thm} [Heins \cite{H}, Zakeri \cite{Z}]  \label{z-p}  For any integer $e\geq 1$, the map $$\Psi_{e}:
	\begin{cases} {\rm Div}_e(\mathbb D)\rightarrow {\rm Div}_e(\mathbb D)\\
		D\mapsto R_{\beta(D)}
		\end{cases}$$
	is a homeomorphism.
\end{thm}



The second named author and Yongcheng Yin further studied the boundary extension of the map $\Psi_{e}$:

\begin{thm} [X. Wang and Y. Yin, \cite{W}] \label{zakeri-exten}
	The map $\Psi_{e}$   extends to a homeomorphism 
	$$\overline{\Psi}_{e}: {\rm Div}_e(\overline{\mathbb D})\rightarrow {\rm Div}_e(\overline{\mathbb D}).$$
	The extension is given as follows: for any $D_1\in {\rm Div}_{d_1}(\mathbb D), D_2\in {\rm Div}_{d_2}(\partial\mathbb D)$ such that $d_2\geq 1$ and 
	$d_1+d_2=e$,   
	$$\overline{\Psi}_{e}(D_{1}+D_2)=\Psi_{d_1}(D_1)+D_2,$$
	where $\Psi_{d_1}$ is the map given by Theorem \ref{z-p}.
\end{thm}

\subsection{Blaschke divisors over a mapping scheme}

Following Milnor \cite{M2}, a \emph{mapping scheme} $S$ is a triple $(V,\sigma,\delta)$, where $V$ is a finite set,
$\sigma:V\rightarrow V$ is a self-mapping, and $\delta:V\rightarrow \mathbb N$ is a map 
such that for any periodic point $v\in V$ of $\sigma$, we have
$\prod_{k=0}^{\ell_v-1}\delta(\sigma^{k}(v))\geq2$, 
where $\ell_v$ is the $\sigma$-period of $v$.

The set $V$ has a decomposition $V=V_{\rm p}\sqcup V_{\rm np}$, where $V_{\rm p}$ consists of all $\sigma$-periodic points $v\in V$, and  $V_{\rm np}$ is the set of all aperiodic points in $V$.

Let $S=(V, \sigma, \delta)$ be a mapping scheme, and let $\Omega$ be a planar set.  The space of divisors ${\rm Div}{(\Omega)}^S$ over the mapping scheme $S$,  consists of divisors $D=(D_v)_{v\in V}$ such that for each $v\in V$, the factor $D_v\in {\rm Div}_{\delta(v)-1}(\Omega)$. 

If $\Omega$ is a planar domain, the boundary $\partial {\rm Div}{(\Omega)}^S$ is given by
$$\partial {\rm Div}{(\Omega)}^S=\Big\{D\in {\rm Div}{(\overline{\Omega})}^S; \text{ there is } v\in V \text{ such that } D_v\in \partial {\rm Div}_{\delta(v)-1}(\Omega) \Big\}.$$

This paper focuses on  the case $\Omega=\mathbb D$. For each $D=(D_v)_{v\in V}\in \partial {\rm Div}{(\mathbb D)}^S$, write
$$D_v=\sum_{k=1}^{d_{v,1}} 1\cdot a_{v,k}+\sum_{j=1}^{d_{v,2}} 1\cdot b_{v, j}:=D_v^0+D_v^\partial \in
{\rm Div}_{d_{v,1}}({\mathbb D})+{\rm Div}_{d_{v,2}}(\partial{\mathbb D}),$$ 
where 
$$d_{v,1}+d_{v,2}=\delta(v)-1, d_{v,1},d_{v,2}\geq 0; \ a_{v,1}, \cdots, a_{v,d_{v,1}}\in \mathbb D; \ b_{v,1}, \cdots, b_{v, d_{v,2}}\in \partial\mathbb D,$$ 
the points  $\{a_{v,k}\}_k$ (respectively $\{b_{v,j}\}_j$) need not be distinct. It is possible that $d_{v,2}=0$ for some $v\in V$. 
However, the assumption $D\in\partial{\rm Div}{(\mathbb D)}^S$ requires implicitly that
$\sum_{v\in V} d_{v,2}\geq 1$.
If there is no confusion, we may also write each factor $D_v$ as the pair $(D_v^0, D_v^\partial)\in {\rm Div}_{d_{v,1}}({\mathbb D})\times
{\rm Div}_{d_{v,2}}(\partial{\mathbb D})$, or  $(B_v^0, D_v^\partial)\in \mathcal B^{d_{v,1}+1}_{\rm fc}\times {\rm Div}_{d_{v,2}}(\partial{\mathbb D})$, where $B_v^0$ is the Blaschke product whose free zero divisor  ${\rm div}(B_v^0)=D_v^0$.


Let $D=(D_v)_{v\in V}\in {\rm Div}{(\overline{\mathbb D})}^S$. Write $D_v=\sum_{k=1}^{\delta(v)-1}1\cdot q_{v,k}$ for each $v\in V$. For any $\epsilon>0$, we introduce  two sets of divisors $\N_{\epsilon}(D)\subset {\rm Div}(\mathbb D)^S$,
$\U_{\epsilon}(D)\subset {\rm Div}{(\overline{\mathbb D})}^S$, as follows:
\bess &\N_\epsilon(D)=\Big\{E=(E_v)_{v\in V}; \
E_v \in \N^0_\epsilon(D_v), \  \forall v \in V
\Big\},&\\
&\U_\epsilon(D)=\Big\{E=(E_v)_{v\in V}; \
E_v\in \U^0_\epsilon(D_v), \   \forall v \in V
\Big\}.&
\eess
It's clear that $\N_{\epsilon}(D)\subset \U_{\epsilon}(D)$,  $\N_{\epsilon}(D)$ is an open subset of ${\rm Div}(\mathbb D)^S$,
and $\U_{\epsilon}(D)$ is an open subset of ${\rm Div}{(\overline{\mathbb D})}^S$.
Note that  if $D\in  {\rm Div}{({\mathbb D})}^S$, then $D\in \N_{\epsilon}(D)\subset \U_{\epsilon}(D)$;
 if $D\in   \partial {\rm Div}{({\mathbb D})}^S$, then   $ D\in \U_{\epsilon}(D)\setminus \N_{\epsilon}(D)$.

The boundary $\partial {\rm Div}{(\mathbb D)}^S$ has a distinguished subset $\partial_0 {\rm Div}{(\mathbb D)}^S$:
$$\partial_0 {\rm Div}{(\mathbb D)}^S=\Big\{D=(D_v)_{v\in V}\in \partial {\rm Div}{(\mathbb D)}^S; \sum_{k=0}^{\ell_v-1} d_{\sigma^k(v), 1}\geq 1 ,  \forall \  v\in V_{\rm p}\Big\},$$
where $\ell_v$ is the $\sigma$-period of $v$.  It is easy to verify that $\partial_0 {\rm Div}{(\mathbb D)}^S=\emptyset$ if and only if $V=V_{\rm p}$, and  for each $v\in V_{\rm p}$, we have
$\prod_{k=0}^{\ell_v-1}\delta(\sigma^{k}(v))=2$.

\begin{ft} \label{subset-boundary}  If  $\partial_0 {\rm Div}{(\mathbb D)}^S\neq \emptyset$, then 
 $\partial_0 {\rm Div}{(\mathbb D)}^S$ is an open subset of  $\partial {\rm Div}{(\mathbb D)}^S$.
\end{ft}
\begin{proof} Fix a divisor $D=(D_v)_{v\in V}\in \partial_0 {\rm Div}{(\mathbb D)}^S$. For any $v\in V$,
 write
$$D_v=\sum_{k=1}^{d_{v,1}} 1\cdot a_{v,k}+\sum_{j=1}^{d_{v,2}} 1\cdot b_{v, j}:=D_v^0+D_v^\partial \in
{\rm Div}_{d_{v,1}}({\mathbb D})+{\rm Div}_{d_{v,2}}(\partial{\mathbb D}).$$ 
Take $\epsilon=\min_{v,k}\{1-|a_{v,k}|\}$. Then for any $E=((E_v^0, E_v^\partial))_{v\in V}\in \U_\epsilon(D)$, we have ${\rm deg}(E_v^0)\geq {\rm deg}(D_v^0)=d_{v,1}$. It follows that 
$\sum_{k=0}^{\ell_v-1}{\rm deg}(E_{\sigma^k(v)}^0)\geq \sum_{k=0}^{\ell_v-1} d_{\sigma^k(v),1}\geq 1.$
This implies that  $\U_\epsilon(D)\cap \partial {\rm Div}{(\mathbb D)}^S\subset \partial_0 {\rm Div}{(\mathbb D)}^S$.
\end{proof}

\subsection{Sequence of Blaschke divisors} As noted before,  a divisor $D\in {\rm Div}_{d}(\overline{\mathbb D})$ can be written as $(B, D^\partial)$, where $B\in   \mathcal B^{d_1+1}_{\rm fc} \cong {\rm Div}_{d_1}(\mathbb D)$, $D^\partial\in  {\rm Div}_{d_2}(\partial \mathbb D)$, $d_1+d_2=d$.


Given a sequence of divisors 
$\{(B_n, D_n^\partial)\}\subset  {\rm Div}_{d}(\overline{\mathbb D})$ such that 
$(B_n, D_n^\partial)\rightarrow (B, D^{\partial})\in {\rm Div}_{d}(\overline{\mathbb D})$,  one can   find an 
   integer $l\geq 0$ and a subsequence  $\{(B_{n_k}, D_{n_k}^\partial)\}\subset 
 {\rm Div}_{d-l}({\mathbb D})\times {\rm Div}_{l}({\partial\mathbb D})$ so that  $\lim_{k\rightarrow \infty}D_{n_k}^\partial$ exist. Let $D^\partial_\infty$ be the limit. We may check that   $D^{\partial}\geq D_\infty^{\partial}$ and $B_{n_k}\rightarrow (B, D^{\partial}-D_\infty^{\partial})$.


The following   is a variant of  \cite[Proposition 3.1]{Mc2}, see \cite[Lemma 4.2]{DeM} for a general statement for rational maps.

\begin{lem} \label{degenerate-0}  Let $0\leq l<d$ be integers.  Let $D_n=(B_n, D_n^\partial)\in {\rm Div}_{d-l}({\mathbb D})\times {\rm Div}_l(\partial {\mathbb D}) \subset {\rm Div}_d(\overline{\mathbb D})$ be a sequence of divisors  so that 
	\begin{itemize}
		
		\item $D_n\rightarrow D=(B^0, D^\partial)\in {\rm Div}_d(\overline{\mathbb D})$, 
		
		\item  $D_n^\partial\rightarrow D^\partial_\infty\in {\rm Div}_l(\partial {\mathbb D})$.
			\end{itemize}
	

  If $1\notin {\rm  supp}(D^\partial-D^\partial_\infty)$, then the sequence of Blaschke products $B_n$ converges locally and 
uniformly to $B^0$ in
$\C\setminus {\rm  supp}(D^\partial-D^\partial_\infty)$.
\end{lem}

\begin{proof}  As pointed out by McMullen, the Lemma follows by  two key facts:
	\begin{itemize}
		\item   For any $q\in \partial\mathbb D\setminus\{1\}$, the  Möbius transformation
		$$\beta_b(z)=\frac{|b|}{b}\cdot\frac{b-z}{1-\overline{b}z}, \ b\in \mathbb D\setminus\{1\}$$
		converges to the constant function $1$  uniformly on compact subsets of  $\C\setminus \{ q\}$ as $b\rightarrow q$.  This follows from the  equality
		$$\beta_{b}(z)-1
		=\frac{(|b|-1)(b+z|b|)}{b(1-\overline{b}z)}.$$
		\item The  inequality $|z_1\cdots z_m-w_1\cdots w_m|\leq \sum_{k=1}^m|z_k-w_k|,$
		whenever $|z_k|\leq 1, |w_k|\leq 1$.
		\end{itemize}
\end{proof} 

\begin{lem} \label{composition-law} Given  integers $d\geq d_0\geq 0,e\geq e_0\geq  0$ and two sequences of divisors 
$$\{(F_n, D_n^\partial)\}_{n\geq 1}\subset   {\rm Div}_{d_0}({\mathbb D})\times  {\rm Div}_{d-d_0}({\partial \mathbb D}),$$
$$  \{(G_n, E_n^\partial)\}_{n\geq 1}\subset    {\rm Div}_{e_0}({\mathbb D})\times  {\rm Div}_{e-e_0}({\partial \mathbb D})$$ such that 
\begin{itemize}

\item $(F_n, D_n^\partial)\rightarrow (F, D_F^{\partial})\in {\rm Div}_{d}(\overline{\mathbb D}), \ (G_n, E_n^\partial)\rightarrow (G, E_G^{\partial})\in {\rm Div}_{e}(\overline{\mathbb D})$, 

\item $D_n^\partial\rightarrow D_\infty^\partial$ and $E_n^\partial\rightarrow E_\infty^\partial$, 

\item  $1\notin {\rm supp}(D_F^{\partial}-D_\infty^\partial)$, $1\notin {\rm supp}(E_G^{\partial}-
 E_\infty^\partial)$.
\end{itemize}
 Then the composition $F_n\circ G_n$ converges locally and uniformly to 
$F\circ G$ in $$\C\setminus (G^{-1}({\rm supp}(D_F^{\partial}-D_\infty^\partial))\cup {\rm supp}(E_G^{\partial}-E_\infty^\partial)).$$
\end{lem}
\begin{proof} 
There are four cases:

(1). ${\rm deg}(F)<d_0\leq d, {\rm deg}(G)<e_0\leq e$;

(2). $d_0={\rm deg}(F), {\rm deg}(G)<e_0\leq e$;

(3). ${\rm deg}(F)<d_0\leq d, e_0={\rm deg}(G)$;

(4).  $d_0={\rm deg}(F), e_0={\rm deg}(G)$.

We deal with these cases as follows:

 (1). Write $F_n=R_n\cdot S_n$ and $G_n=r_n\cdot s_n$, such that
$${\rm deg}(R_n)={\rm deg}(F), \ R_n(1)=S_n(1)=1, R_n\rightarrow F, S_n\rightarrow D_F^\partial-D_\infty^\partial,$$
$${\rm deg}(r_n)={\rm deg}(G), \ r_n(1)=s_n(1)=1, \ r_n\rightarrow G, \ s_n\rightarrow E_G^\partial-E_\infty^\partial.$$  Then 
$F_n\circ G_n=R_n(r_n\cdot s_n)\cdot S_n(r_n\cdot s_n)$.
By Lemma \ref{degenerate-0},  $S_n$ converges locally and uniformly in $\C\setminus {\rm supp}(D_F^\partial-D_\infty^\partial)$ to the constant $1$, and 
$s_n$ converges locally and uniformly in $\C\setminus {\rm supp}(E_G^\partial-E_\infty^\partial)$ to the constant $1$.
It follows  that $R_n(r_n\cdot s_n)$ converges locally and uniformly in $\C\setminus {\rm supp}(E_G^\partial-E_\infty^\partial)$ to the map $F\circ G$, and  $S_n(r_n\cdot s_n)$ converges locally and uniformly in
 $\C\setminus (G^{-1}({\rm supp}(D_F^\partial-D_\infty^\partial))\cup {\rm supp}(E_G^\partial-E_\infty^\partial))$
to the constant 1. Therefore,  $F_n\circ G_n$ converges locally and uniformly to 
$F\circ G$ in $\C\setminus (G^{-1}({\rm supp}(D_F^\partial-D_\infty^\partial))\cup {\rm supp}(E_G^\partial-E_\infty^\partial))$.

(2). The assumption $d_0={\rm deg}(F)$ implies that $F_n$ converges uniformly to $F$ on $\C$.
Write
$G_n=r_n\cdot s_n$ with  ${\rm deg}(r_n)={\rm deg}(G), r_n(1)=s_n(1)=1, r_n\rightarrow G$ and $s_n\rightarrow E_G^\partial-E_\infty^\partial$. 
By Lemma \ref{degenerate-0}, 
$s_n$ converges locally and uniformly in $\C\setminus {\rm supp}(E_G^\partial-E_\infty^\partial)$ to the constant $1$.
 It follows that 
$F_n\circ G_n=F_n(r_n\cdot s_n)$ converges locally and uniformly in $\C\setminus {\rm supp}(E_G^\partial-E_\infty^\partial)$ to the map $F\circ G$.

(3).  The assumption $e_0={\rm deg}(G)$ implies $G_n$ converges uniformly to $G$ in $\C$. Similarly as above,  $F_n\circ G_n$ converges locally and uniformly to 
$F\circ G$ in $\C\setminus G^{-1}({\rm supp}(D_F^\partial-D_\infty^\partial))$.

(4). In this case, $F_n\circ G_n$ converges uniformly to 
$F\circ G$ in $\C$.

The lemma follows  immediately.
\end{proof}

	\begin{pro} \label{convergence}  
	Let $D_n=((B_{n,u}, D_{n,u}^\partial))_{u\in V}\in{\rm Div}(\overline{\mathbb D})^S$ be a sequence of divisors converging to $D=\big((B^0_u, D_u^{\partial})\big)_{u\in V}\in {\rm Div}(\overline{\mathbb D})^S$.  Assume
	 that for all $u\in V$,
	
	\begin{itemize}
		\item $(B_{n,u}, D_{n,u}^\partial)\in  {\rm Div}_{\delta(u)-1-l_u}({\mathbb D})\times  {\rm Div}_{l_u}({\partial \mathbb D})$ for some integer $l_u\geq 0$ and for all $n$;
		
		\item $ D_{n,u}^\partial\rightarrow D_{\infty,u}^\partial$ as $n\rightarrow \infty$;
		
		\item $1\notin {\rm  supp}(D_u^\partial-D_{\infty,u}^\partial)$.
		\end{itemize}
%

Then we have the following statements:	
	
	(1). 	For any $v\in V$ and any integer $k\geq 1$,  the sequence $B_{n,{\sigma^{k-1}(v)}}\circ \cdots \circ B_{n,{\sigma(v)}}\circ B_{n,v}$ converges   locally and uniformly  to $B^0_{\sigma^{k-1}(v)}\circ \cdots \circ B^0_{\sigma(v)}\circ B^0_v$ in  $\C\setminus W_v^k$, where
	$$W_v^k=\begin{cases} {\rm supp}(D_v^{\partial}-D_{\infty,v}^\partial) , &  k=1, \\
		\bigcup_{m=1}^{k-1}\big(B^0_{\sigma^{m-1}(v)}\circ \cdots \circ B^0_v\big)^{-1}
		\big({\rm supp}(D_{\sigma^{m}(v)}^{\partial}-D_{\infty,\sigma^{m}(v)}^\partial)\big)\\
		\  \ \bigcup 	{\rm supp}(D_v^{\partial}-D_{\infty,v}^\partial), &  k\geq2.			
	\end{cases}$$
	
	In  particular, if $D_u^\partial=D_{\infty,u}^\partial$ for all $u\in \{\sigma^j(v), j\geq 0\}$, then    $B_{n,{\sigma^{k-1}(v)}}\circ \cdots \circ B_{n,{\sigma(v)}}\circ B_{n,v}$ converges     uniformly  to $B^0_{\sigma^{k-1}(v)}\circ \cdots \circ B^0_{\sigma(v)}\circ B^0_v$ in  $\C$.
	
	(2). Assume $v\in V_{\rm p}$.   Let $\ell\geq 1$ be the $\sigma$-period of $v$, and let
	$$\widehat{B}^0_{v}=B^0_{\sigma^{\ell-1}(v)}\circ \cdots \circ B^0_{\sigma(v)}\circ B^0_v, \ \widehat{B}_{n,v}=B_{n,{\sigma^{\ell-1}(v)}}\circ \cdots \circ B_{n,{\sigma(v)}}\circ B_{n,v}.$$

	Then for any integer $k\geq 1$, the $k$-th iterate
	$\widehat{B}_{n,v}^{k}$ converges locally and uniformly  to  the $k$-th iterate $(\widehat{B}^0_{v})^k$ in
	$\C\setminus \bigcup_{0\leq j\leq k-1} (\widehat{B}^0_v)^{-j}(W_v^\ell)$.
\end{pro}
\begin{proof}  It follows from  Lemma \ref{composition-law} and iteration.
\end{proof}

\subsection{Perturbation of Blaschke divisors} \label{pbp}

For  the convenience of discussion,  sometimes we  need to distinguish the domains and ranges   for the maps $(B_u)_{u\in V}$  of the divisor  $((B_u, D_u^{\partial}))_{u\in V}\in  {\rm Div}(\overline{\mathbb D})^S$.    For each $u\in V$,  we shall view  $B_u$ as a map $B_u:\C_u\rightarrow \C_{\sigma(u)}$, and denote a set $X\subset \C_u$ by $X_u$.

\begin{pro} \label{extension-repelling}
	Let 
	$D=((B_w^0, D_w^{\partial}))_{w\in V}\in \partial_0 {\rm Div}(\overline{\mathbb D})^S$.
	Let  $u\in V$ so that  $v:=\sigma^s(u)\in V_{\rm p}$ for some $s\geq 0$.
	Let   $\ell\geq 1$ be the $\sigma$-period of $v$.  Write $\widehat{B}^0_{v}=B^0_{\sigma^{\ell-1}(v)}\circ \cdots \circ B^0_{\sigma(v)}\circ B^0_v$, and 
	$$Z_v=\begin{cases}
		{\rm supp}(D_v^{\partial}), & \text{ if }  \ell=1,\\
		{\rm supp}(D_v^{\partial})\bigcup 
		\bigcup_{m=1}^{\ell-1}\big(B^0_{\sigma^{m-1}(v)}\circ \cdots \circ B^0_v\big)^{-1}
		\big({\rm supp}(D_{\sigma^{m}(v)}^{\partial})\big)  , & \text{ if }  \ell\geq 2.
	\end{cases}$$
 
	Assume $b\in \partial \mathbb D_u$ is $D$-pre-periodic, in the sense that there are integers 
	$l\geq 0$ and $n\geq 1$, such that  $$\begin{cases}
		(\widehat{B}^0_v)^{l+n}(b)=(\widehat{B}^0_v)^{l}(b), & \text{ if }  s=0,\\
		(\widehat{B}^0_v)^{l+n}(B^0_{\sigma^{s-1}(u)}\circ \cdots \circ   B^0_u(b))=(\widehat{B}^0_v)^{l}(B^0_{\sigma^{s-1}(u)}\circ \cdots \circ   B^0_u(b)), & \text{ if }  s\geq 1.
	\end{cases}$$ 
	Assume further that  $1\notin {\rm  supp}(D_{\sigma^k(u)}^\partial), \ \forall k\geq 0$, and $b\in  \partial \mathbb D_u\setminus X_{u,s}$, where 
	$$X_{u,s}=\begin{cases}
		  \bigcup_{j\geq 0} (\widehat{B}^0_v)^{-j}(Z_v), & \text{ if }  s=0,\\
	 Y_{u,s}  \bigcup(B^0_{\sigma^{s-1}(u)}\circ \cdots \circ B^0_u\big)^{-1}( \bigcup_{j\geq 0} (\widehat{B}^0_v)^{-j}(Z_v))  , & \text{ if }  s\geq 1,
	\end{cases}$$
	here, $Y_{u,s}= {\rm supp}(D_u^{\partial})\bigcup \bigcup_{1\leq m<s}\big(B^0_{\sigma^{m-1}(u)}\circ \cdots \circ B^0_u\big)^{-1}
	\big({\rm supp}(D_{\sigma^{m}(u)}^{\partial})\big)$.

 Then 
		 there exist $\tau=\tau(D,b)>0$ with ${\U_{\tau}(D)}\cap \partial {\rm Div}(\overline{\mathbb D})^S\subset \partial_0 {\rm Div}(\overline{\mathbb D})^S$ and a unique continuous map 
		 \begin{equation}
		 	\label{rvb}
		 	r_{u,b}:   {\rm Div}({\mathbb D})^S\cup {\U_{\tau}(D)}\rightarrow \partial \mathbb D_u,
		 \end{equation}
	satisfying that  $r_{u,b}(D)=b$, and for any  $E=\big((B_w, E_w^{\partial})\big)_{w\in V}\in   {\rm Div}({\mathbb D})^S\cup \U_{\tau}(D)$,  the point
	$r_{u,b}(E)$ is pre-periodic under $E$, in the sense that
	$$\begin{cases}
		\widehat{B}_v^{l+n}(r_{u,b}(E))=\widehat{B}_v^{l}(r_{u,b}(E)), & \text{ if }  s=0,\\
		\widehat{B}_v^{l+n}(B_{\sigma^{s-1}(u)}\circ \cdots \circ   B_u(r_{u,b}(E)))=\widehat{B}_v^{l}(B_{\sigma^{s-1}(u)}\circ \cdots \circ   B_u(r_{u,b}(E))), & \text{ if }  s\geq 1,
	\end{cases}$$ 
	where $\widehat{B}_v=B_{\sigma^{\ell-1}(v)}\circ \cdots \circ B_{\sigma(v)}\circ B_v$.
\end{pro}

\begin{proof}  We only  treat the case $s=0$, the same argument  works for $s\geq 1$. In this case, $u=v$.
	Let $m=l+n$.  Set $Z_v^m=\bigcup_{0\leq j\leq m-1} (\widehat{B}_{v}^0)^{-j}(Z_v)$.
  Take small $r>0$ so that
	$b$ is the only zero of  $(\widehat{B}_{v}^0)^{l+n}-(\widehat{B}_{v}^0)^{l}$  in $\overline{\mathbb{D}(b, r)}$, and 
	$$\overline{\mathbb{D}(b,r)}\subset \C\setminus \bigcup_{q\in Z_v^m} \mathbb{D}(q,r).$$
	Observe that $b$ is the simple zero  of  $(\widehat{B}_{v}^0)^{l+n}-(\widehat{B}_{v}^0)^{l}$, because $(\widehat{B}_{v}^0)^{l}(b)$ is $(\widehat{B}_{v}^0)^{n}$-repelling (implying that ($(\widehat{B}_{v}^0)^{n})'((\widehat{B}_{v}^0)^{l}(b))\neq 1$), and 
	$$\frac{d}{dz}[(\widehat{B}_{v}^0)^{l+n}-(\widehat{B}_{v}^0)^{l}]|_{z=b}=[((\widehat{B}_{v}^0)^{n})'((\widehat{B}_{v}^0)^{l}(b))-1]((\widehat{B}_{v}^0)^{l})'(b)\neq 0.$$
	Take $\epsilon=\min_{|z-b|=r}|(\widehat{B}_{v}^0)^m(z)-(\widehat{B}_{v}^0)^l(z)|$. By Proposition \ref{convergence}, there is 
	$\tau>0$ with  ${\U_{\tau}(D)}\cap \partial{\rm Div}(\overline{\mathbb D})^S\subset \partial_0 {\rm Div}(\overline{\mathbb D})^S$ such that for any 
	$E=((B_w, E_w^\partial))_{w\in V}\in  {\U_{\tau}(D)}$ and any $z\in \overline{\mathbb{D}(b, r)} \subset \C\setminus \bigcup_{q\in Z_v^m} \mathbb{D}(q,r)$, 
	$$|\widehat{B}_{v}^m(z)-(\widehat{B}_{v}^0)^m(z)|<\epsilon/2, \ |\widehat{B}_{v}^l(z)-(\widehat{B}_{v}^0)^l(z)|<\epsilon/2.$$
	 It follows that 
	\bess &&\max_{|z-b|=r}|(\widehat{B}_{v}^m(z)-\widehat{B}_{v}^l(z))-((\widehat{B}_{v}^0)^m(z)-(\widehat{B}_{v}^0)^l(z))|\\
	&\leq& \max_{|z-b|=r}(|\widehat{B}_{v}^m(z)-(\widehat{B}_{v}^0)^m(z)|+ |\widehat{B}_{v}^l(z)-(\widehat{B}_{v}^0)^l(z)|)\\
	&<&\epsilon=\min_{|z-b|=r}|(\widehat{B}_{v}^0)^m(z)-(\widehat{B}_{v}^0)^l(z)|.
	\eess
	By Rouch\'e's  Theorem, the equation $\widehat{B}_{v}^m(z)-\widehat{B}_{v}^l(z)=0$ has exactly one solution   in $\mathbb{D}(b, r)$,  denoted by $r_{v,b}(E)$.
	The inclusion  ${\U_{\tau}(D)}\cap \partial{\rm Div}(\overline{\mathbb D})^S\subset \partial_0 {\rm Div}(\overline{\mathbb D})^S$ implies that $0$ is  attracting for $\widehat{B}_{v}$, hence  $r_{v,b}(E)$ is  pre-repelling.
	Clearly  $r_{v,b}(D)=b$. The continuity of  $r_{v,b}$ follows by the formula 
	$$r_{v,b}(E)=\frac{1}{2\pi i}\int_{|\zeta-b|=r}\zeta\frac{(\widehat{B}_{v}^m(\zeta)-\widehat{B}_{v}^l(\zeta))'}{\widehat{B}_{v}^m(\zeta)-\widehat{B}_{v}^l(\zeta)} d\zeta, \ \ E\in    {\U_{\tau}(D)}.$$ 
	
	By the Implicit Function Theorem,   $r_{v,b}|_{\N_{\tau}(D)}$ extends continuously to $  {\rm Div}({\mathbb D})^S$ so that
	$r_{v,b}(E)$ is a pre-repelling   point  of  $\widehat{B}_v$, for  
	 $E=(B_w)_{w\in V}\in  {\rm Div}({\mathbb D})^S$. 
	This gives a required map $r_{v,b}:   {\rm Div}({\mathbb D})^S\cup {\U_{\tau}(D)}\rightarrow \partial \mathbb D$.

	Finally, we prove the uniqueness of $r_{v,b}$. 
	Shrink $\tau>0$ if necessary,  let $R_{v,b}:  {\rm Div}({\mathbb D})^S\cup {\U_{\tau}(D)}\rightarrow \partial \mathbb D$ be the continuous map with the same property as $r_{v,b}$.  Then  for each $E=((B_w, E_w^{\partial}))_{w\in V}\in   {\rm Div}({\mathbb D})^S\cup \U_{\tau}(D)$,
	$$r_{v,b}(E), R_{v,b}(E)\in \big\{z\in \partial \mathbb D;  \ \widehat{B}_{v}^m(z)-\widehat{B}_{v}^l(z)=0\big\}.$$
	Hence $r_{v,b}$ and $R_{v,b}$ take discrete values. By the continuity and the fact $r_{v,b}(D)=R_{v,b}(D)=b$, we have $r_{v,b}\equiv R_{v,b}$.
	%
\end{proof}

\begin{rmk} \label{fixed-1} In Proposition \ref{extension-repelling}, the assumption
	$1\notin {\rm  supp}(D_{\sigma^k(u)}^\partial), \ \forall k\geq 0$  
  implies that $1\in  \partial \mathbb D_u\setminus X_{u,s}$.
So if we take $b=1$, then $r_{u,1}:   {\rm Div}({\mathbb D})^S\cup {\U_{\tau}(D)}\rightarrow \partial \mathbb D_u$ satisfies that $r_{u,1}(D)=1$, and  for   $E=((B_w, E_w^{\partial}))_{w\in V}\in   {\rm Div}({\mathbb D})^S\cup \U_{\tau}(D)$, $r_{u,1}(E)$ is $E$-pre-periodic.
   Hence $r_{u,1}\equiv 1$.
\end{rmk}

\section{Polynomial dynamics and hyperbolic components} \label{p-d-h-c}

\subsection{Basics  of polynomial dynamics}

Let $f\in \mathcal P_d$. 
The \emph{filled Julia set} of $f$ is $K(f):=\{z\in\mathbb C; \{f^n(z); n\geq 0\}  \text{ is bounded in } \mathbb C\}$,  and its  complement set $\Omega(f):=\mathbb C\setminus K(f)$ is called  the \emph{basin of infinity} of $f$. It is well-known that the Julia set  $J(f)=\partial K(f)$.  For any point $z$ in the Fatou set $F(f):=\mathbb C\setminus J(f)$,  the Fatou component $U_f(z)$  is  the component of $F(f)$ containing $z$.

 For any $f\in \mathcal C_d$, 
 there is a unique conformal map $\phi_f: \Omega_f\rightarrow \mathbb C\setminus \overline{\mathbb D}$, 
such that $	\phi_f\circ f=\phi_f^d$ in $\Omega_f$,  with normalization  $\lim_{z\rightarrow \infty}\phi_f(z)/z=1$.
This $\phi_f$ is called the \emph{B\"{o}ttcher map} of $f$. 
For any  $\theta\in \mathbb R/\mathbb Z$, the {\it external ray  with angle $\theta$}  is defined as
$R_f(\theta):=\phi_f^{-1}((1,\infty)e^{2\pi i\theta})$.
Clearly, $f(R_f(\theta))=R_f(d\theta)$.
An external ray \emph{lands} at $z$ if its accumulation set on $ J(f)$ is $\{z\}$. By \cite[\S 18]{M},  every periodic external ray  lands at a parabolic or repelling  point; conversely, every parabolic or repelling  point is  landed by finitely many periodic external rays with the same period. 

If $f\in \mathcal P_d\setminus \mathcal C_d$,  the  {B\"{o}ttcher map} $\phi_f$ of $f$ can   be defined similarly near $\infty$. It  has   an analytic continuation to a maximal subdomain of  $\Omega(f)$,  in the way that $	\phi_f\circ f=\phi_f^d$. 
So suitable external rays can also be defined.


We collect some known results on polynomial dynamics  used in the paper.

\begin{lem}
	[Stability of external rays]
	\label{stability-e-r}
	Let $f\in\mathcal{C}_d$ and $\theta\in \mathbb Q/\mathbb Z$.
	Suppose that $R_f(\theta)$ lands at a pre-repelling point $q$. 
	
	1. 	Then there is a neighborhood $\mathcal{N}\subset\mathcal P_d$ of $f$ such that $$\gamma_\theta: \begin{cases}(\mathcal{N}\cap\mathcal{C}_d)\times[1,\infty)\rightarrow\mathbb{C},\\
	(g, r)\mapsto \phi_g^{-1}(re^{2\pi i\theta}) \end{cases}$$ is continuous.  

2. If we further require that the $f$-orbit of $q$ meets no critical point, then there is a neighborhood $\mathcal{N}\subset\mathcal P_d$ of $f$ such that $$\gamma_\theta: \begin{cases}\mathcal{N}\times[1,\infty)\rightarrow\mathbb{C},\\
	(g, r)\mapsto \phi_g^{-1}(re^{2\pi i\theta}) \end{cases}$$ is continuous, here 
$\phi_g$ is   defined in a  maximal subdomain of  $\Omega(g)$.
\end{lem}

Lemma \ref{stability-e-r} (2) is \cite[Proposition 8.5]{DH}.  Lemma \ref{stability-e-r} (1)  follows by  taking preimages,  under the restriction $\mathcal{N}\cap\mathcal{C}_d$, see \cite[Lemma 2.2]{CWY}. 

The following result is from \cite{RY}.

\begin{pro}\label{RY}  Let $U$ be a parabolic  or bounded attracting  Fatou component of  $f\in \mathcal P_d$. Then  $\partial U$ is a Jordan curve.  
	
	Assume further $f\in \mathcal C_d$, we have the following properties.
	
	1. If $x\in \partial U$ is preperiodic, then $x$ is either pre-repelling or pre-parabolic.
	
	2. If  $x\in \partial U$ is wandering, then $J(f)$ is locally connected at $y$, and there is at least one external ray landing at $x$.
	
	3.  For any  $x\in \partial U$, there is a  continuum $L_{U, x}\subset K(f)\setminus U$, called a limb, such that $L_{U, x}\cap \partial U=\{x\}$, and 
	$$K(f)=U\bigsqcup \bigsqcup_{x\in \partial U} L_{U, x}.$$
	
	4. If $L_{U, x}\neq \{x\}$, then  $L_{U, x}\setminus\{x\}$ has finitely many connected components.
	Moreover, there are two external rays $R_{f}(\alpha),R_{f}(\beta)$ landing at $x$, such that  $L_{U, x}\setminus \{x\}$ is contained in one connected component of
	$\mathbb C\setminus(\overline{R_{f}(\alpha)}\cup \overline{R_{f}(\beta)})$.  

	5. $L_{U, x}\neq \{x\}$ if and only if there is an integer $n\geq 0$ such that $L_{f^n(U), f^n(x)}$ contains a critical point.
	
\end{pro}





\begin{rmk} \label{limb-wandering}   The limbs satisfy that $L_{f(U), f(x)}\subset f(L_{U, x})$.
	If $x\in \partial U$ is wandering, then $L_{f^{n}(U),f^n(x)}=\{f^n(x)\}$ for large $n$.
\end{rmk}
\begin{proof} The  property $L_{f(U), f(x)}\subset f(L_{U, x})$ is immediate. Suppose that $x$ is wandering. Replacing $(U, f)$ by  $(f^m(U), f^m)$ for some integer $m\geq 0$ if necessary, we assume  that $f(U)=U$. If $L_{f^n(U), f^n(x)}\neq \{f^n(x)\}$ for all $n\geq 1$, then when $n$ is large, the limb $L_{f^n(U), f^n(x)}$ contains no critical points of $f$, hence
	the angular difference $\Delta_n$ of the external rays landing at $f^n(x)$  grows exponentially in $n$. This is a contradiction.
\end{proof}

\begin{cor} \label{intersection-Fatou-components} Let $f\in \mathcal C_d$ have two different  bounded periodic Fatou components $U$ and $V$, either attracting or parabolic. If $\partial U\cap \partial V\neq \emptyset$, then $\partial U\cap \partial V$ consists of a parabolic or repelling  periodic point.
\end{cor}
\begin{proof}      By Proposition \ref{RY}, $\partial U\cap \partial V$  consists of a singleton $q$.  Let $n_U, n_V$ be the $f$-periods of $U,V$ respectively, and let $m$ be the least common multiple of $n_U$ and $n_V$.  The facts $f^m(U)=U$ and $f^m(V)=V$ imply that $f^m(q)=q$.   
	Hence  $q\in J(f)$ is periodic.  By Proposition \ref{RY}(1),  $q$ is either repelling or parabolic. 
\end{proof}

\subsection{Carath\'eodory kernel convergence}\label{s-ckc}

Let $w_0\in \mathbb C$ be given and let $W_n$ be domains with $w_0\in  W_n\subset \mathbb C$.  We say that 
$W_n\rightarrow W$ as $n\rightarrow \infty$ with respect to $w_0$, in the sense of {\it kernel convergence}, if

\begin{itemize}
	
	\item  either $W=\{w_0\}$, or $W$ is a domain  $\neq \mathbb C$, with $w_0\in W$ such that some
	neighborhood of every $w\in W$ lies in $W_n$ for large $n$.
	
	\item for $w\in \partial W$, there exist $w_n\in \partial W_n$ such that $w_n\rightarrow w$ as $n\rightarrow +\infty$.
\end{itemize}

We use  $(W_n, w_0)\rightarrow (W, w_0)$ to denote the kernel convergence.

It is known that the limit $W$ is uniquely determined and it depends on the choice of the base point $w_0$. Further, if all $W_n$'s are simply connected and  $W\neq \{w_0\}$, then $W$ is also simply connected. 






Let $f\in \mathcal P_d$ have an attracting periodic point $z_f\in \mathbb C$. By the Implicit Function Theorem,
there is a neighborhood $\mathcal N$ of $f$ and a continuous map $g\mapsto z_g$ defined in $\mathcal N$, such that
$z_g$ is  attracting for $g\in \mathcal N$.  By shrinking $\mathcal N$, we assume $z_f\in U_g(z_g)$ for all $g\in \mathcal N$, hence $U_g(z_f)=U_g(z_g)$.

\begin{lem} \label{carathodory-convergence} Let $\{f_n\}\subset \mathcal N$ be a sequence of polynomials such that 
	$f_n\rightarrow f$.  Then we have the kernel convergence:
	$$(U_{f_n}(z_{f_n}), z_f)\rightarrow (U_f(z_f), z_f).$$
\end{lem}








\begin{proof} Let $p\geq 1$ be the $f$-period of $z_f$. For any compact set $K\subset U_f(z_f)$ containing $z_f$, let $V,W$ be a pair of Jordan disks such that $K\subset V\Subset W\Subset  U_f(z_f)$, and $f^p:W\rightarrow V$ be a proper map of degree ${\rm deg}(f^p|_{U_f(z_f)})$.

	There is a neighborhood $\mathcal N_0\subset \mathcal N$ of $f$ such that for all 
	$g\in \mathcal N_0$, 
	\begin{itemize}
		\item the boundary $\partial V$ contains no critical value of $g^p$;
		
		\item the component $W_g$ of  $g^{-p}(V)$ containing $z_g$ is a Jordan disk with $ V\Subset W_g\Subset U_g(z_f)$. 
	\end{itemize}
	
	In particular, for large $n$ with
	$f_n\in \mathcal N_0$, we have  $K\subset U_{f_n}(z_f)$.

	
	On the other hand, take any $w\in \partial U_f(z_f)$.
	By the continuity of the Green function 
	$(g,z)\mapsto G_g(z)$ (see \cite[Proposition 8.1]{DH}), for any $n\geq 1$, there is a neighborhood $\mathcal U_n\subset  \mathcal N_0$ of $f$ such that $\mathbb{D}(w, 1/n)\cap (\mathbb C\setminus K(g))\neq \emptyset$ for all $g\in \mathcal U_n$.
	Combining with the previous paragraph, there is sequence of integers $k_1<k_2<\cdots$, such that 
	$$\mathbb{D}(w, 1/n)\setminus U_{f_k}(z_f)\neq \emptyset, \  \mathbb{D}(w, 1/n)\cap U_{f_k}(z_f)\neq \emptyset, \ \forall k\geq k_n.$$
	This implies that $\mathbb{D}(w, 1/n)\cap \partial U_{f_k}(z_f)\neq \emptyset$ for $k\geq k_n$. Take $w_k\in \mathbb{D}(w, 1/n)\cap \partial U_{f_k}(z_f)$ for $k_n\leq k<k_{n+1}$. Then the sequence $\{w_k\}$ satisfies that
	$w_k\in \partial U_{f_k}(z_f)$ and $w_k\rightarrow w$ as $k\rightarrow \infty$.
\end{proof}

\begin{rmk}  \label{pre-c-t}  Lemma \ref{carathodory-convergence} also holds when $z_f$ is pre-attracting  (that is, $f^{m}(z_f)$ is attracting for some $m\geq 1$) and $U_f(z_f)$ is pre-periodic. 
\end{rmk}

\begin{rmk} By contrast, the sequence of the compact sets  $\overline{U_{f_n}(z_f)}$ might not converge to $\overline{U_f(z_f)}$   in Hausdorff topology.
	
	However, if $\partial{U_f(z_f)}$ contains neither critical point nor parabolic point, then $\overline{U_{f_n}(z_f)}$ converges to $\overline{U_f(z_f)}$ as $n\rightarrow \infty$ in Hausdorff topology.
\end{rmk}

\subsection{Parameterization of hyperbolic components}\label{phc}
In this part,   we   review Milnor's parameterization of a hyperbolic component $\mathcal H\subset \mathcal C_d$ via Blaschke products  over a mapping scheme \cite{M2}. There is a slight difference here: we use the original polynomial space rather than the generalized polynomial space associated with a mapping scheme.

First, by an unpublished theorem of McMullen, there is a unique postcritically finite map $f_0 \in \mathcal H$. Let ${\rm Crit}(f_0)$ be the critical set of $f_0$ in $\mathbb C$. 
  The mapping scheme $S=(V,\sigma,\delta)$ associated with $\mathcal H$ is a triple given by 
$$V=\bigcup_{n\geq 0}f_0^n({\rm Crit}(f_0)); \ \sigma(v)=f_0(v), \delta(v)={\rm deg}(f_0|_{U_{f_0, v}}), \ v\in V,$$
where $U_{f_0, v}$ is the Fatou component of $f_0$ containing $v\in V$.  


 By Milnor \cite{M2}, $\mathcal H$ is a topological cell.  This topology property allows us to get a well-defined  holomorphic motion 
$$h: \mathcal H\times J(f_0)\rightarrow \mathbb C$$
satisfying that $h(f_0, \cdot)={\rm id}$, $h(f, J(f_0))=J(f)$ and 
$f\circ h(f,z)=h(f, f_0(z))$ for all $(f,z)\in \mathcal H\times J(f_0)$.
For all $v\in V$ and all maps $f\in \mathcal H$,  the marked Fatou component $U_{f, v}$ is  the Fatou component of $f$ bounded by $h(f, \partial U_{f_0, v})$.


For each $g\in \mathcal H$, a {\it boundary marking} of $g$ means a function $\nu_g: V\rightarrow \mathbb C$  which assigns to each  $v\in V$ a  point $\nu_g(v)\in\partial U_{g,v}$, satisfying that 
$$ g(\nu_g(v))=\nu_g(\sigma(v)), \  \forall\ v\in V.$$


The boundary marking $\nu_g$ is not unique, but there are only finitely many choices of $\nu_g$ for a given map $g$.
The number of choices is 
$$n_{\mathcal H}:=\prod_{v\in V_{\rm np}} \delta(v) \prod_{\mathcal O\subset V_{\rm p}} (\Delta(\mathcal O)-1),$$
  where  $\mathcal O$ ranges over all $\sigma$-periodic orbits in $V_{\rm p}$, and $\Delta(\mathcal O)=\prod_{v\in \mathcal O} \delta(v)$. Here recall that $V_{\rm p}$ consists of all $\sigma$-periodic points $v\in V$, and $V_{\rm np}$ is the residual part.
  
The pair $(g, \nu_g)$ is called a {\it marked map}.
Let 
$$\widetilde{\mathcal H}=\Big\{(g, \nu_g); g\in \mathcal H, \nu_g \text{ is a boundary marking of } g\Big\}.$$
be the collection of all marked maps. The set $\widetilde{\mathcal H}$ has a natural topology so that each map of $\mathcal H$ has a
neighborhood which is evenly covered  under the projection  $(g, \nu_{g})\mapsto g$.  
 As a topology space,  $\widetilde{\mathcal H}$ might be  disconnected (we will see in Corollary \ref{marking1-1}  that $\widetilde{\mathcal H}$  has exactly $n_{\mathcal H}$ connected components).  A  connected component of $\widetilde{\mathcal H}$
is called a {\it marked hyperbolic component}. 

  Let $\widehat{\mathcal H}$ be a marked hyperbolic component.
For  
$(f, \nu_f)\in \widehat{\mathcal{H}}$ and $v\in V$,  there is a  Riemann mapping 
$h_{f,v}: U_{f,v}\rightarrow \mathbb D$ with normalization:

\begin{itemize}
	\item if $v\in V_{\rm p}$,  we   require  that $h_{f,v}(v(f))=0$ and  $h_{f,v}(\nu_f(v))=1$, where  $v(f)$  is  the attracting periodic point  in $U_{f,v}$.
	

	\item if $v\in V_{\rm np}$, assume $h_{f, \sigma(v)}$ is given, 
	then $h_{f, v}$  is normalized  so that
	$$\sum_{q\in f|_{U_{f,v}}^{-1}(h_{f, \sigma(v)}^{-1}(0))} h_{f,v}(q)=0, \ h_{f,v}(\nu_f(v))=1,$$
	where the summation is taken counting multiplicity.  Such $h_{f, v}$  exists by  Douady and Earle \cite[\S 2]{DE}.
\end{itemize}
The   map $h_{f,v}$ is  unique, and extends to a 
 homeomorphism  $h_{f,v}: \overline{U_{f,v}}\rightarrow \overline{\mathbb D}$.

Now $B_{f,v}=h_{f,\sigma(v)}\circ f\circ h_{f,v}^{-1}$ is a Blaschke product of degree $\delta(v)={\rm deg}(f|_{U_{f,v}})$.
Clearly, if $v\in V_{\rm p}$, then $B_{f,v}(0)=0$ and $B_{f,v}(1)=1$, and we call such $B_{f,v}$ {\it fixed point centered}; if  
 $v\in V_{\rm np}$, then $\sum_{q\in B_{f,v}^{-1}(0)} q=0$ and $B_{f,v}(1)=1$, and we call such  $B_{f,v}$ {\it zeros centered}.
 
 Recall that  $\mathcal B_{\rm fc}^\delta$ is the space
 of fixed point  
 centered  Blaschke products of degree $\delta$. Similarly, we define  $\mathcal B_{\rm zc}^\delta$ to be the space of zeros  centered  Blaschke products of degree $\delta$.  

Let $S=(V, \sigma, \delta)$ be the mapping scheme associated with $\mathcal H$, the model space $\mathcal B^S$ consists of all maps $B: V\times \mathbb D\rightarrow V\times \mathbb D$, defined by
$B(v, z)=(\sigma(v), B_{v}(z))$, where $B_v\in \mathcal B_{\rm fc}^{\delta(v)}$ for
$v\in V_{\rm p}$, and 
 $B_v\in \mathcal B_{\rm zc}^{\delta(v)}$ for $v\in V_{\rm np}$. 
For simplicity, we also write $B$ as $(B_v)_{v\in V}$.

\begin{thm}  [Milnor \cite{M2}] \label{homeoH}  Let $\widehat{\mathcal  H}$ be a  marked hyperbolic component, then  the  map 
$$ \Psi:  
\begin{cases} \widehat{\mathcal  H} \rightarrow \mathcal B^S, \\
(f, \nu_f)\mapsto B_f=(B_{f,v})_{v\in V}
\end{cases}$$
is a homeomorphism.
\end{thm}

\begin{cor} \label{marking1-1}  Let $\widehat{\mathcal  H}$ be a  marked hyperbolic component,  
then the projection 
\begin{equation}
\label{bm-H}
\boldsymbol{p}:  
\begin{cases} \widehat{\mathcal  H} \rightarrow \mathcal  H, \\
(f, \nu_f)\mapsto f
\end{cases}
\end{equation}
is a homeomorphism.   Moreover,   $\widetilde{\mathcal{H}}$ has 
$n_\mathcal{H}$ 
 connected components.
\end{cor}
\begin{proof}    The set  $\widetilde{\mathcal{H}}$ contains exactly $n_\mathcal{H}$  marked maps of the form $(f_0,\nu)$, and all of them are mapped by  $\Psi$ to the same element   $(B_v(z)=z^{\delta(v)})_{v\in V}\in \mathcal B^S.$
 By Theorem \ref{homeoH}, each connected component of  $\widetilde{\mathcal{H}}$ contains only one of  $(f_0,\nu)$'s. 
\end{proof}

Let $\mathcal H$ be a hyperbolic component of disjoint-type. Then each map $f\in \mathcal H$ has
$d-1$ attracting cycles.
The $k$-th attracting cycle is denoted by 
$$\mathcal O_k: \  z_k\mapsto f(z_k)\mapsto \cdots \mapsto f^{p_k-1}(z_k)\mapsto f^{p_k}(z_k)=z_k,$$
with multiplier  $\rho_k(f)=(f^{p_k})'(z_k)\in \mathbb D$, where $p_k\geq 1$ is the period of $z_k$.  
\begin{cor} \label{disjoint-type} If ${\mathcal  H}$ is of disjoint-type,
then the multiplier map 
$$\rho:  
\begin{cases} {\mathcal  H} \rightarrow \mathbb D^{d-1}, \\
f\mapsto (\rho_1(f), \cdots, \rho_{d-1}(f)),
\end{cases}$$
is biholomorphic.
\end{cor}
\begin{proof}  By Theorem \ref{homeoH} and Corollary \ref{marking1-1}, the map $\Psi\circ \boldsymbol{p}^{-1}: \mathcal H\rightarrow \mathcal B^S$ is a homeomorphism.  Write
	$\Psi\circ \boldsymbol{p}^{-1}(f)=(B_{f,v})_{v\in V}$.
	
	Since $\mathcal H$ is of disjoint-type, we have $V=V_{\rm p}$. For each $f\in \mathcal H$ and $v\in V$,
$$\delta(v)=2, \text{ if }  {\rm Crit}(f)\cap U_{f,v}\neq\emptyset; \ 
  \delta(v)=1, \text{ if }  {\rm Crit}(f)\cap U_{f,v}=\emptyset.$$ 
Moreover, for any $v\in V$ with $\sigma$-period say $\ell\geq 1$,  there is only one $u\in \{v, \sigma(v), \cdots, \sigma^{\ell-1}(v)\}$ such that $\delta(u)=2$.
It follows that
$$B_{f,\tilde{v}}(z)= 
\begin{cases} z, & \text{ if }  \tilde{v}\in  \{v, \cdots, \sigma^{\ell-1}(v)\}\setminus\{u\} ,\\
z(z+\lambda)(1+\bar{\lambda})/[(1+\bar{\lambda}z)(1+\lambda)], &\text{ if }   \tilde{v}=u. 
\end{cases}$$
Hence the model maps $B_{f,\sigma^k(v)}, 0\leq k<\ell$ are uniquely determined by  $\lambda\in \mathbb D$,  which is related to the multiplier $\rho_v=\rho_v(f)$ of the attracting cycle by
$$\rho_v=\lambda \frac{1+\bar{\lambda}}{1+\lambda} \Longleftrightarrow \lambda=\rho_v \frac{1-\bar{\rho_v}}{1-\rho_v}.$$
It follows that the multiplier map $\rho$ is a homeomorphism. Since the multipliers change holomorphically in $f$, the map $\rho$ is actually biholomorphic.
\end{proof}

\subsection{A doubly marked hyperbolic component}\label{dmhc}

For the convenience of
  further discussion, we need a refinement of the marked hyperbolic component. In this part, we 
  shall introduce the {\it internal marking}  and the {\it doubly marked hyperbolic component}.
  
  For each $g\in \mathcal H$, an {\it internal marking} of $g$ means an injection ${\boldsymbol a}: V\rightarrow \mathbb C$  which assigns to each  $v\in V$ a  point ${\boldsymbol a}(v)\in  U_{g,v}$, satisfying that 
  $$ g({\boldsymbol a}(v))={\boldsymbol a}(\sigma(v)), \  \forall\ v\in V.$$
  It clear that  ${\boldsymbol a}(v)=v(g)$ for $v\in V_{\rm p}$.  
  If $V_{\rm np}\neq\emptyset$,  the internal marking ${\boldsymbol a}$ is not unique in general, but there are only finitely many choices of ${\boldsymbol a}$ for a given map $g$.
  The number of choices for a generic map (i.e. map without critical relations)  $g$ is exactly
  $\prod_{v\in V_{\rm np}} \delta(v)$.
  
  The triple  $(g, {\boldsymbol a}, \nu_g)$ is called a {\it doubly marked map}.
  Let 
  $$\widetilde{\mathcal H}_*=\left\{(g,  {\boldsymbol a},  \nu_g); g\in \mathcal H,
 \begin{cases} 
  {\boldsymbol a} \text{ is an internal marking of } g \\
   \nu_g \text{ is a boundary marking of } g
\end{cases}\right\}$$
  be the collection of all doubly  marked maps. The set $\widetilde{\mathcal H}_*$ has a natural topology so that each marked map
 of $\widetilde{\mathcal H}$ has a
  neighborhood which is  finitely branched covered \footnote{In topology, a map is a branched covering if it is a covering map everywhere except for a nowhere-dense set known as the branched set.}  under the projection  $\boldsymbol{p}_*: (g, {\boldsymbol a}, \nu_{g})\mapsto (g,  \nu_{g})$.  
Note that $\widetilde{\mathcal H}_*$ might be  disconnected.
  A  connected component   $\widehat{\mathcal H}_*$ of $\widetilde{\mathcal H}_*$
  is called a {\it doubly marked hyperbolic component}.

Let
  $(f, {\boldsymbol a}, \nu_f)\in   \widehat{\mathcal H}_*$. For any $v\in V$,  there is a  unique Riemann mapping 
  $\psi_{f,v}: U_{f,v}\rightarrow \mathbb D$ with normalizations:
   $\psi_{f,v}({\boldsymbol a}(v))=0$ and  $\psi_{f,v}(\nu_f(v))=1$. 
   Then $B_{f,v}=\psi_{f,\sigma(v)}\circ f\circ \psi_{f,v}^{-1}$ is a Blaschke product of degree $\delta(v)$, with $B_{f,v}(0)=0$ and $B_{f,v}(1)=1$.   Hence $B_{f,v}\in \mathcal B_{\rm fc}^{\delta(v)}$ for all
   $v\in V$.
   
  Recall that $S=(V, \sigma, \delta)$ is the mapping scheme associated with $\mathcal H$. 
 Let $\mathcal B^S_*$ be the space of  $B=(B_v)_{v\in V}$ such that $B_v\in \mathcal B_{\rm fc}^{\delta(v)}$ for all
   $v\in V $.   
   
   We get a well-defined map
   $$ \Psi_*:  
   \begin{cases} \widehat{\mathcal H}_* \rightarrow \mathcal B_*^S, \\
   	(f, {\boldsymbol a}, \nu_f)\mapsto B_f=(B_{f,v})_{v\in V}
   \end{cases}$$
   
   
 \begin{ft} \label{t-map} 
 	 There is a  map
 	 $\mathcal T: \mathcal B^S_*\rightarrow \mathcal B^S$ 
 	 so that the following diagram commutes.
 	  $$\xymatrix{ 
 	 	\widehat{\mathcal H}_*  \ar[r]^{ \boldsymbol{p}_* } \ar[d]_{\Psi_*}&   	\widehat{\mathcal H}  \ar[d]^{\Psi} \\ 
 	 	\mathcal B_*^S \ar[r]_{\mathcal T}   &  \mathcal B^S  &\\	
 	 }$$
 	  
 	\end{ft}
 \begin{proof}
For each $B=(B_v)_{v\in V}\in  \mathcal B^S_*$,  set $\beta_v=B_v$ for $v\in V_{\rm p}$. We decompose   $V_{\rm np}$ as $V_{\rm np}=\bigsqcup_{l\geq 1}V_{\rm np}^{(l)}$, where 
\begin{equation} \label{level-vnp}
   V_{\rm np}^{(l)}=\{v\in  V_{\rm np}; l \text{ is the smallest integer so that } \sigma^l(v)\in V_{\rm p}\}. 
   \end{equation}
   For each $v\in  V_{\rm np}^{(1)}$, let  $h_v: \mathbb D\rightarrow \mathbb D$ be  the unique conformal map so that 
   $$h_v(1)=1, \ \sum_{q\in B_v^{-1}(0)} h_{v}(q)=0,$$
   	where the summation is taken counting multiplicity. 
   	Set $\beta_v=B_v\circ h_v^{-1}$. Clearly, $\beta_v\in \mathcal B_{\rm zc}^{\delta(v)}$.
   	Assume by induction that for all $v\in  V_{\rm np}^{(1)}\sqcup \cdots \sqcup  V_{\rm np}^{(k)}$ ($k\geq 1$), the conformal map $h_v: \mathbb D\rightarrow \mathbb D$ with $h_v(1)=1$, and the Blaschke product $\beta_v\in \mathcal B_{\rm zc}^{\delta(v)}$ are defined. Then for $v\in  V_{\rm np}^{(k+1)}$, note that $h_{\sigma(v)}$ is defined,  let  $h_v: \mathbb D\rightarrow \mathbb D$ be  the unique conformal map so that $h_v(1)=1$ and  
   	$$\sum_{q\in (h_{\sigma(v)}\circ B_v)^{-1}(0)} h_{v}(q)=0,$$
   	then $\beta_v= h_{\sigma(v)}\circ B_v\circ h_v^{-1} \in \mathcal B_{\rm zc}^{\delta(v)}$.
   	After finitely many steps, for each $v\in V_{\rm np}$, we get  a zero centered  Blaschke product $\beta_v\in \mathcal B_{\rm zc}^{\delta(v)}$.

   	Set $\mathcal T(B)=(\beta_v)_{v\in V}$. One may verify that  $\mathcal T\circ \Psi_*=\Psi\circ \boldsymbol{p}_*$.
   	\end{proof}

   
   The virtue of the notations $\widehat{\mathcal H}_*, \mathcal B_*^S$ is illustrated by the following proposition, which is a modified version of Theorem \ref{homeoH}.

  \begin{pro}  \label{branched-model}  Let   $\widehat{\mathcal H}_*$ be  a  doubly  marked hyperbolic component. The  map $ \Psi_*:   \widehat{\mathcal H}_* \rightarrow \mathcal B_*^S$  	is a homeomorphism.
  \end{pro}
\begin{proof}  The proof of covering property and surjectivity of  $\Psi_*$ is the same as that of  \cite[Section 5]{M2}, using quasiconformal surgery.  We omit the details. 
	
	To see the injectivity of $\Psi_*$, assume
 $\Psi_*((f_1, {\boldsymbol a}_{1}, \nu_{f_1}))=\Psi_*((f_2, {\boldsymbol a}_{2}, \nu_{f_2}))$.
	It follows that $\mathcal T\circ \Psi_*((f_1, {\boldsymbol a}_{1}, \nu_{f_1}))=\mathcal T\circ \Psi_*((f_2, {\boldsymbol a}_{2}, \nu_{f_2}))\in  \mathcal B^S$.
	By the commutative  diagram of Fact \ref{t-map} and Theorem \ref{homeoH}, we have $(f_1, \nu_{f_1})=(f_2,  \nu_{f_2})$.  Write $(f_1, \nu_{f_1})$ as $(f, \nu)$.  To finish, we need to prove ${\boldsymbol a}_{1}={\boldsymbol a}_2$.  It suffices to show ${\boldsymbol a}_{1}(v)={\boldsymbol a}_2(v)$ for $v\in V_{\rm np}$.
	
	For each $v\in V$ and $k\in\{1,2\}$, let $\psi^{{\boldsymbol a}_{k}}_{f,v}: U_{f,v}\rightarrow \mathbb D$ be the conformal map with
	$\psi^{{\boldsymbol a}_{k}}_{f,v}({\boldsymbol a}_k(v))=0, \psi^{{\boldsymbol a}_{k}}_{f,v}(\nu_f(v))=1$.
	Clearly,  if $v\in V_{\rm p}$, then $\psi^{{\boldsymbol a}_{1}}_{f,v}=\psi^{{\boldsymbol a}_{2}}_{f,v}$. If $v\in V_{\rm np}^{(1)}$ (given by \eqref{level-vnp}), then $\psi^{{\boldsymbol a}_{k}}_{f,\sigma(v)} \circ f|_{U_{f,v}}=B_{f,v}\circ \psi^{{\boldsymbol a}_{k}}_{f,v}$. Hence 
	$$B_{f,v}\circ \psi^{{\boldsymbol a}_{1}}_{f,v}=B_{f,v}\circ \psi^{{\boldsymbol a}_{2}}_{f,v}\Longleftrightarrow B_{f,v}\circ \psi^{{\boldsymbol a}_{1}}_{f,v}\circ (\psi^{{\boldsymbol a}_{2}}_{f,v})^{-1}= B_{f,v}.$$
The normalization    $ \psi^{{\boldsymbol a}_{1}}_{f,v}\circ (\psi^{{\boldsymbol a}_{2}}_{f,v})^{-1}(1)=1$   implies  that  $ \psi^{{\boldsymbol a}_{1}}_{f,v}\circ (\psi^{{\boldsymbol a}_{2}}_{f,v})^{-1}\equiv {\rm id}$   by the  above functional equation. 
Hence  $ \psi^{{\boldsymbol a}_{1}}_{f,v}\equiv \psi^{{\boldsymbol a}_{2}}_{f,v}$, and ${\boldsymbol a}_{1}(v)=(\psi^{{\boldsymbol a}_{1}}_{f,v})^{-1}(0)=(\psi^{{\boldsymbol a}_{2}}_{f,v})^{-1}(0)={\boldsymbol a}_{2}(v)$.  
By induction on the level $l$ of $ V_{\rm np}^{(l)}$, we have    ${\boldsymbol a}_{1}(v)={\boldsymbol a}_2(v)$ for all $v\in V_{\rm np}$.
	\end{proof}

 A  doubly  marked hyperbolic component  $\widehat{\mathcal H}_*$  induces an {\it internal marked hyperbolic component} ${\mathcal H}_*$.  It is the image of   $\widehat{\mathcal H}_*$ under the projection  
 $$  
{\rm proj}:  \begin{cases} \widehat{\mathcal H}_* \rightarrow {\mathcal H}_*, \\
 	(f, {\boldsymbol a}, \nu_f)\mapsto 	(f, {\boldsymbol a}).
 \end{cases}$$ 
The internal marked hyperbolic component ${\mathcal H}_*$  
can be viewed as $\widehat{\mathcal H}_*$ by the following (analogous  to Corollary \ref{marking1-1})

 \begin{ft}  \label{proj-homeo} The projection ${\rm proj}$ is a homeomorphism.
\end{ft}
\begin{proof}   The proof is the same as that of Corollary \ref{marking1-1}.  First note that $f_0$ has only one internal marking, say ${\boldsymbol a}_0$. The set  $\widetilde{\mathcal H}_* $   contains exactly $n_\mathcal{H}$  marked maps of the form $(f_0, {\boldsymbol a}_0, \nu)$, all of which are mapped by   $ \Psi_*$  to the same element   $(B_v(z)=z^{\delta(v)})_{v\in V}\in \mathcal B_*^S$. 
	By Proposition \ref{branched-model}, each component $ \widehat{\mathcal H}_*$ of  $\widetilde{\mathcal H}_* $ contains only one of  $(f_0, \boldsymbol a_0, \nu)$'s. 
\end{proof}

In the rest of the paper, we prefer to study the internally marked hyperbolic component $\mathcal{H}_*$ rather than $\mathcal{H}$ itself, which offers two key advantages:

\begin{itemize}
	\item $\mathcal{H}_*$ is homeomorphic to $\mathcal{B}^S_*$ via $\Psi_* \circ \mathrm{proj}^{-1}$ (Proposition \ref{branched-model} and Fact \ref{proj-homeo}),  and $\mathcal{B}^S_*$ corresponds to $\mathrm{Div}(D)^S$. This identification simplifies the analysis of degenerations.
	
	\item Replacing $\mathcal{H}$ with $\mathcal{H}_*$ preserves all main results, because these results depend solely on analyses near $\mathcal{H}$-admissible maps $f \in \partial\mathcal{H}$. When examining maps in the fiber of $f$ on $\partial\mathcal{H}_*$, we actually use a suitable biholomorphic coordinate chart.
\end{itemize}

	

For simplicity, we hereafter denote $\mathcal{H}_*$ by $\mathcal{H}$ (noting that $\mathcal{H}_* = \mathcal{H}$ when $V_{\mathrm{np}} = \emptyset$). The internal marking ${\boldsymbol a}$ associated with $f \in \mathcal{H}$ is uniformly represented as
${\boldsymbol v}(f)=(v(f))_{v\in V}$, 
a notation that simultaneously encodes the index $v \in V$ and the mapping $f \in \mathcal{H}$.


\section{Impression associated with divisors} \label{imp-divisor} 
By  Proposition \ref{branched-model} and Fact \ref{proj-homeo}, 
 $\Psi_*\circ {\rm proj}^{-1}:  {\mathcal H}_*\rightarrow \mathcal B^S_*$ is a homeomorphism. 
Identifying $\mathcal B^S_*$ with ${\rm Div}(\mathbb D)^S$ by the map $\boldsymbol{i}: \mathcal B^S_*\rightarrow {\rm Div}(\mathbb D)^S$, we get a parameterization 
\begin{equation}
	\label{para-H}
	\Phi={\rm proj}\circ \Psi_*^{-1}\circ \boldsymbol{i}^{-1}: {\rm Div}(\mathbb D)^S\rightarrow   {\mathcal H}_*. 
\end{equation}

The boundary $\partial \mathcal H_*$
consists of the pairs $(f, {\boldsymbol v}(f))$  with $f\in \partial\mathcal H$ and ${\boldsymbol v}(f):=(v(f))_{v\in V}$, so that there is a sequence $\{f_n\}_n\subset \mathcal H$ with the property  
$$f_n\rightarrow f, \  {\boldsymbol v}(f_n)\rightarrow {\boldsymbol v}(f), \ \text{ as } n\rightarrow \infty,$$ 
where the notation  ${\boldsymbol v}(f_n)\rightarrow {\boldsymbol v}(f)$ means that $v(f_n)\rightarrow v(f)$ for all $v\in V$.   We call ${\boldsymbol v}(f)$ the {\it internal marking } for $f$.

We establish the following notational conventions:
\begin{itemize}
	\item $\mathcal{H}$ refers to $\mathcal{H}_*$ unless stated otherwise.
	\item The boundary $\partial\mathcal{H}_*$ is denoted by $\partial\mathcal{H}$.
	\item Elements $(f, \boldsymbol{v}(f)) \in \partial\mathcal{H}$ are abbreviated as $f$.
	\item Convergence $f_n \to f$ implies $(f_n, \boldsymbol{v}(f_n)) \to (f, \boldsymbol{v}(f))$ in parameter space.
\end{itemize}


Let $D=\big((B_v^0, D_v^{\partial})\big)_{v\in V}\in \partial {\rm Div}(\mathbb D)^S$ and let $\epsilon_0>0$ be a small number, 
we define the {\it impression} of $D$ under the map $\Phi$ by
$$I_{\Phi}(D)=\bigcap_{0<\epsilon<\epsilon_0}\overline{\Phi(\N_{\epsilon}(D))}.$$
 As a shrinking sequence of connected compact sets, $I_{\Phi}(D)$ is a connected   compact subset of $\partial \mathcal H$. 
Here is   another explanation of $I_{\Phi}(D)$:

\begin{ft} \label{fact-def} Let $D=\big((B_v^0, D_v^{\partial})\big)_{v\in V}\in \partial {\rm Div}{(\mathbb D)}^S$.  Then $I_{\Phi}(D)$ consists of all maps $f\in \partial \mathcal H$ so that there is a sequence 
$\{f_n\}_{n\geq 1}\subset \mathcal H$ satisfying that 
$$f_n\rightarrow f \ \text{ and } \ \Phi^{-1}(f_n)\rightarrow D  \ \text{ as }  n\rightarrow \infty.$$
\end{ft}
\begin{proof} If $f\in I_{\Phi}(D)$, then $f\in \overline{\Phi(\N_{\epsilon}(D))}$ for all 
$\epsilon\in(0,\epsilon_0)$. Take a decreasing sequence $\epsilon_n\rightarrow0^+$. For each $\epsilon_n$, there is $f_n\in \Phi(\N_{\epsilon_n}(D))\cap \mathcal N_{\epsilon_n}(f)$, where $ \mathcal N_{\epsilon_n}(f)$ is the $\epsilon_n$-neighborhood of $f$ in $\mathcal P_d$. Hence  $f_n\rightarrow f$  and $\Phi^{-1}(f_n)\rightarrow D$.

Conversely, suppose $\{f_n\}_{n\geq 1}\subset \mathcal H$ satisfies  that   $f_n\rightarrow f \in \partial \mathcal H$  and $\Phi^{-1}(f_n)\rightarrow D$. By choosing a subsequence, there is a
  decreasing sequence $\epsilon_n\rightarrow0^+$
  such that $\Phi^{-1}(f_k)\in \N_{\epsilon_n}(D)$ (equivalently $f_k\in \Phi(\N_{\epsilon_n}(D))$) for all $k\geq n$.
  The assumption $f_k\rightarrow f$ implies that $f\in \overline{\Phi(\N_{\epsilon_n}(D))}$. Since $n$ is arbitrary, we have $f\in \bigcap_n \overline{\Phi(\N_{\epsilon_n}(D))}=I_{\Phi}(D)$.
\end{proof}


\begin{rmk} By the compactness of $ {\rm Div}{(\overline{\mathbb D})}^S$, we have
 \begin{equation}\label{decomposition-boundary}
 	\partial\mathcal H=\bigcup_{ E\in \partial {\rm Div}{(\mathbb D)}^S}I_{\Phi}(E), \  \partial_{\rm reg}\mathcal H=\bigcup_{ E\in \partial_0{\rm Div}{(\mathbb D)}^S}I_{\Phi}(E),
 	\end{equation}
where
 \begin{equation}\label{reg-boundary}
  \partial_{\rm reg}\mathcal H=\{f\in \partial \mathcal H; f \text{ has exactly } m_\mathcal{H} \text{ attracting cycles }  \}.
\end{equation}
 To see this,  for any $g\in \partial \mathcal H$, there is a sequence $\{g_n\}\subset \mathcal H$ with $g_n\rightarrow g$.  Since  $ {\rm Div}{(\overline{\mathbb D})}^S$ is compact,  by Heine-Borel, the sequence  $\{\Phi^{-1}(g_n)\}_{n\geq 1}$ has an accumulation divisor $E\in \partial {\rm Div}{(\mathbb D)}^S$.  By Fact \ref{fact-def},  $g\in I_{\Phi}(E)$. This implies the left equality. The right   equality follows from the same reasoning.
 
 By \eqref{decomposition-boundary}, to understand $\partial\mathcal H$ or
 $ \partial_{\rm reg}\mathcal H$, we need to study the structure of $I_{\Phi}(D)$.
\end{rmk}

\begin{defi} \label{subset-boundary}  Let us define an open and dense  subset of $\partial_0 {\rm Div}{(\mathbb D)}^S$:
$$\partial_0^* {\rm Div}{(\mathbb D)}^S=\Big\{D=((B_v^0, D_v^\partial))_{v\in V}\in \partial_0 {\rm Div}{(\mathbb D)}^S; 1\notin {\rm supp}(D_v^\partial) ,  \forall \  v\in V\Big\}.$$
\end{defi}

Let $D=((B_v^0, D_v^\partial))_{v\in V}\in \partial_0^* {\rm Div}{(\mathbb D)}^S$ and $f\in I_{\Phi}(D)$. By Fact \ref{fact-def}, there is a sequence
$\{f_n\}\subset \mathcal H$ satisfying that $f_n\rightarrow f \ \text{ and } \ \Phi^{-1}(f_n)\rightarrow D$. 
Recall that  $f_n\rightarrow f$ means $(f_n,  {\boldsymbol v}(f_n))\rightarrow (f,  {\boldsymbol v}(f))$.
Let $\nu_n=(\nu_n(v))_{v\in V}$ be the boundary marking  of $f_n$ so that 
 $\{(f_n, {\boldsymbol v}(f_n), \nu_n)\}_{n}$ are in the same doubly marked hyperbolic component. 

 For each $v\in V$, since $U_{f_n, v}$ is a Jordan disk (Proposition \ref{RY}),  there is a unique Riemann mapping $\psi_{f_n, v}: U_{f_n,v}\rightarrow \mathbb D$ with  $\psi_{f_n, v}(v(f_n))=0$ and $\psi_{f_n, v}(\nu_n(v))=1$. Let $\phi_{f_n, v}=\psi_{f_n, v}^{-1}$.
Since $f\in \partial_{\rm reg}\mathcal H$,  the point $v(f)=\lim_n v(f_n)$ is   $f$-(pre-)attracting.
By  Lemma \ref{carathodory-convergence} and Carath\'eodory's Kernel Convergence Theorem,  we assume the sequence
$\phi_{f_n, v}$ converges locally and uniformly in $\mathbb D$ to a conformal map $\phi_{f,v}: (\mathbb D, 0)\rightarrow (U_{f,v}, v(f))$. 

\begin{lem} \label{model-map} Let $D=((B_v^0, D_v^\partial))_{v\in V}\in \partial_0^* {\rm Div}{(\mathbb D)}^S$  and
 $\{f_n\}_{n\geq 1}$ be a sequence of maps in $\mathcal H$ such that 
\begin{itemize}
\item  $f_n\rightarrow f\in I_{\Phi}(D)$, $\Phi^{-1}(f_n)\rightarrow D$, and

\item  $\phi_{f_n,v}$ converges locally and uniformly on $\mathbb D$ to  $\phi_{f,v}$, \ $\forall v\in V$.
\end{itemize} 
Then we have ${B}^0_v=\phi^{-1}_{f,\sigma(v)}\circ f\circ \phi_{f,v}, \ \forall v\in V$.
\end{lem}
\begin{proof} The assumption $D=((B_v^0, D_v^\partial))_{v\in V}\in \partial_0^* {\rm Div}{(\mathbb D)}^S$ implies that $1\notin {\rm supp}(D_v^\partial),  \forall  v\in V$.
 By  Lemma \ref{degenerate-0} and  the assumption   
$$\Phi^{-1}(f_n)=(B_{n,v})_{v\in V}\rightarrow D=((B_v^0, D_v^\partial))_{v\in V},$$  the Blaschke products ${B}_{n,v}$ converge locally and uniformly to $B_v^0$ in $\mathbb D$.

The local and uniform convergence $\phi_{f_n,v}\rightarrow \phi_{f,v}$  implies that 
${B}_{n,v}=\phi^{-1}_{f_{n,\sigma(v)}}\circ f_n\circ \phi_{f_{n,v}}$ converges  locally and uniformly  to $\phi^{-1}_{f,\sigma(v)}\circ f\circ \phi_{f,v}$ in 
$\mathbb D$. By the Identity Theorem,  ${B}^0_v=\phi^{-1}_{f,\sigma(v)}\circ f\circ \phi_{f,v}$ for all $v\in V$.
\end{proof}

In the following, we shall introduce some notations (for example, the sets $\partial_0 U_{f,v}, \partial_v^f \mathbb D$, the divisor $D_{v}^\partial(f)$, etc) based on $\phi_{f,v}$ (or its inverse $\psi_{f,v}$).  It seems that   an ambiguity  might arise, because a priori different sequences of $f_n$'s  might result in  different   maps $(\phi_{f,v})_{v\in V}$.    The readers will be relieved to know  that we shall show in the next section (see Proposition \ref{combinatorial-property0}) that different sequence of $f_n$  will give the same   maps $(\phi_{f,v})_{v\in V}$. 
So for the moment, it is  safe  even if  we  keep in mind that the notations  depends on $(\phi_{f,v})_{v\in V}$.

By Carath\'eodory's Boundary Extension Theorem, the map $\phi_{f,v}$ extends to a homeomorphism $\phi_{f,v}: \overline{\mathbb D}\rightarrow \overline{U_{f,v}}$. Hence $\phi_{f,v}(1)$ is well defined and is a preperiodic point of $f$ on $\partial U_{f,v}$.
Let $\psi_{f,v}=\phi_{f,v}^{-1}: \overline{U_{f,v}}\rightarrow \overline{\mathbb D}$, and define
\begin{equation}
\label{fatou-boundary-0}\partial_0 U_{f,v}=\big\{q\in \partial U_{f,v}; L_{U_{f,v}, q}=\{q\}\big\}, \  \partial_v^f \mathbb D= \psi_{f,v}(\partial_0 U_{f,v}).
\end{equation}
\begin{ft} \label{fact-sub-boundary}  The sets $\partial_0 U_{f,v}, \partial_v^f \mathbb D$ satisfy the following properties: 
	$$f(\partial_0 U_{f,v})\subset \partial_0 U_{f,\sigma(v)}, \ B^0_v(\partial_v^f \mathbb D)
	\subset \partial_{\sigma(v)}^f \mathbb D.$$
	Moreover,  the residual set 
	$\partial\mathbb D\setminus  \partial_v^f \mathbb D$ is countable. 
\end{ft}
\begin{proof} The inclusion property is immediate.  By Proposition \ref{RY},  $z\notin \partial_0 U_{f,v}$ if and only if there is a $k\geq 0$ such that the limb 
	$L_{U_{f, \sigma^k(v)}, f^k(z)}$ is non-trivial and contains a critical point.  Since there are  finitely many critical points, the residual set $\partial U_{f,v}\setminus \partial_0 U_{f,v}$(hence
	$\partial\mathbb D\setminus  \partial_v^f \mathbb D$) is countable. 
	\end{proof}

\begin{pro} \label{convergence-repelling} Let $D=\big((B^0_w, D_w^{\partial})\big)_{w\in V}\in \partial_0^* {\rm Div}{(\mathbb D)}^S$, and
 $\{f_n\}_{n\geq 1}$ be a sequence of maps in $\mathcal H$ such that 
\begin{itemize}
\item  $f_n\rightarrow f\in I_{\Phi}(D)$, $\Phi^{-1}(f_n)\rightarrow D$, and

\item  $\phi_{f_n,w}$ converges locally and uniformly on $\mathbb D$ to  $\phi_{f,w}$, \ $\forall w\in V$.
\end{itemize} 

	Let  $u\in V$ so that  $v:=\sigma^s(u)\in V_{\rm p}$ for some $s\geq 0$.  Let   $\ell\geq 1$ be the $\sigma$-period of $v$. Let $\widehat{B}^0_{v}, Z_v, X_{u,s}$  be defined in Proposition \ref{extension-repelling}.
	 
Let $s_0\in \partial \mathbb D_u\setminus X_{u,s}$ be   $D$-pre-periodic,
so that $\phi_{f,u}(s_0)$ is $f$-pre-repelling. Suppose that  the external ray $R_f(\theta)$ lands at $\phi_{f,u}(s_0)$. 
 Then

\begin{itemize}
\item  $\lim_n\phi_{f_n,u}(r_{u,s_0}(\Phi^{-1}(f_n))= \phi_{f,u}(s_0)$, here $r_{u,s_0}$ is given by (\ref{rvb});

\item  for any $g\in \mathcal H$, the external ray $R_g(\theta)$  lands at $\phi_{g,u}(r_{u,s_0}(\Phi^{-1}(g)))$.
\end{itemize}
\end{pro}
\begin{rmk} 
	The existence of $s_0$ in   Proposition \ref{convergence-repelling} is guaranteed by  the density of the  $\widehat{B}^0_{v}$-periodic  points   in $\partial \mathbb D\setminus \bigcup_{j\geq 0}
	(\widehat{B}^0_v)^{-j}(Z_v)$.
	
	 In  Proposition \ref{combinatorial-property0}, we shall show that for any  $D$-pre-periodic point $s_0\in \partial_u^f \mathbb D$ (defined by (\ref{fatou-boundary-0})),
the point $\phi_{f,u}(s_0)$ is $f$-pre-repelling.
\end{rmk}

\begin{proof}  We focus on the case $u=v$(that is, $s=0$), and $s_0$ is $\widehat{B}^0_{v}$-periodic. In this case,  $\phi_{f,u}(s_0)$ is $f$-repelling. The argument also works for general case, by taking preimages.

	Let $p>0$ be the $\widehat{B}^0_{v}$-period of $s_0$. By Koenig's linearization theorem, there exist a neighborhood  $V(s_0)\Subset \mathbb C \setminus \bigcup_{0\leq j<p} (\widehat{B}^0_v)^{-j}(Z_v)$ of $s_0$,  a univalent map $\kappa: (V(s_0), s_0)\rightarrow (\mathbb C, 0)$, such that 
$\kappa \circ (\widehat{B}^0_{v})^{p}=[(\widehat{B}^0_{v})^{p}]'(s_0)\cdot\kappa$ in $V(s_0)$.

 Let $a\in V(s_0) \cap \mathbb D$ and  $\alpha_v(D)$ be an  arc connecting $a$ and  $[(\widehat{B}^0_{v})^{p}]_{V(s_0)}^{-1}(a)$. 
 The arc  $\alpha_v(D)$ generates a path starting from $a$ and converging to $s_0$:
  $$\gamma_{v, s_0}(D)=\bigcup_{j\geq 0} 
[(\widehat{B}^0_{v})^{p}]_{V(s_0)}^{-j}(\alpha_v(D)).$$

By shrinking $V(s_0)$ if necessary, we may further require that 

\begin{itemize}

\item $\phi_{f,v}(V(s_0)\cap \overline{\mathbb D})$ is contained in a linearization neighborhood $W(f)$ of the $f$-repelling periodic point $\phi_{f,v}(s_0)$;

\item there is a number $\tau>0$  such that for any ${B}=(B_w)_{w\in V}\in \N_{\tau}(D)\subset{\rm Div}{(\mathbb D)}^S$,  the set $V(s_0)$ is   in a linearization neighborhood of the $\widehat{B}_v$-periodic  point $r_{v,s_0}(B)$ (see Proposition  \ref{extension-repelling}).
\end{itemize}

For any ${B}=(B_w)_{w\in V}\in \N_{\tau}(D)$, 
let $\alpha_v(B)$ be an arc connecting $a$ and  $(\widehat{B}_v^{p}|_{V(s_0)})^{-1}(a)$, satisfying that
\begin{itemize}

\item  $\gamma_{v,s_0}(B)=\bigcup_{j\geq 0} (\widehat{B}_v^{p}|_{V(s_0)})^{-j}(\alpha_v(B))$ is a path converging to $r_{v,s_0}(B)$;

\item $\gamma_{v,s_0}(B)$ is continuous(in Hausdorff topology) in  $B\in \N_{\tau}(D)\cup \{D\}$.
\end{itemize}

Now we turn to the dynamical plane of $f$: the set $\phi_{f,v}(\gamma_{v,s_0}(D))$ is a path converging to the $f$-repelling point $\phi_{f,v}(s_0)$. By the Implicit Function Theorem, there exist a neighborhood 
$\mathcal U$ of $f$, and a continuous map $\zeta: \mathcal U\rightarrow \mathbb C$, such that $\zeta(g)=g^{\ell p}(\zeta(g))$ is  $g$-repelling, for $g\in \mathcal U$, and $\zeta(f)=\phi_{f,v}(s_0)$.


\begin{figure}[h]
	\begin{center}
		\includegraphics[height=5.3cm]{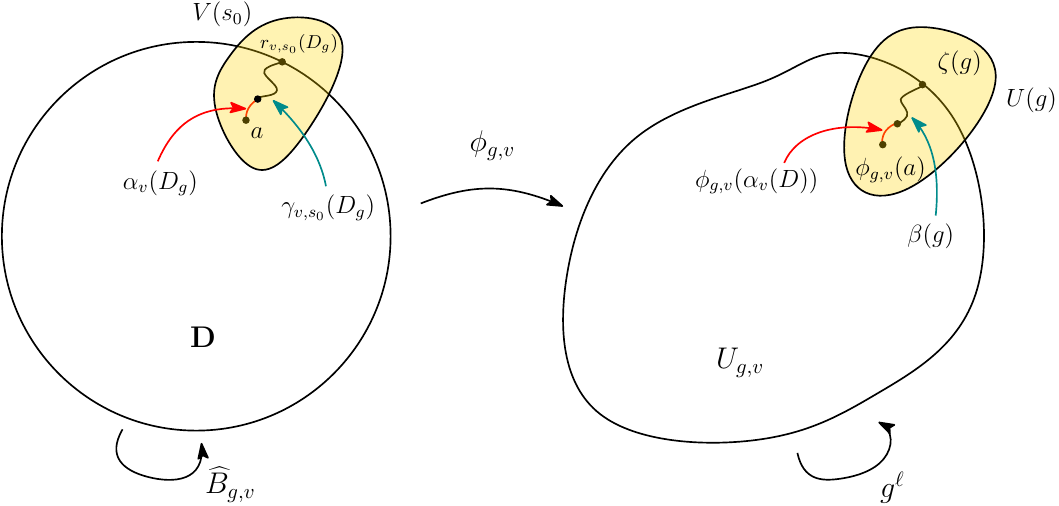}
	\end{center}
	\caption{The model plane and the dynamical plane}
\end{figure}

The assumption $f_n\rightarrow f$ and $\Phi^{-1}(f_n)\rightarrow D$ implies that $f_n\in \Phi(\N_{\tau}(D))\cap \mathcal U$ for large $n$. Without loss of generality,  we assume $\{f_n\}_{n\geq 1}\subset \Phi(\N_{\tau}(D))\cap \mathcal U$.
Let $\mathcal F=\{f_n;n\geq 1\}\cup \{f\}$.
For any $\rho\in(0, 1)$,  the map
$$  
\begin{cases}  \mathcal F \times \overline{\mathbb{D}(0, \rho)}\rightarrow \mathbb C, \\
 (g,z) \mapsto \phi_{g,v}(z),
\end{cases}$$
is continuous.
We  further assume that $|a|$ is chosen sufficiently close to $1$ so that the arc $\phi_{g,v}(\alpha_v(D_g))$ (here $D_f=D$ and $D_{f_n}=\Phi^{-1}(f_n)$) is contained in a linearization neighborhood $W(g)$ of the $g$-repelling point $\zeta(g)$.
Observe that $\phi_{g,v}(\widehat{B}_{g,v}^{p}|_{V(s_0)}^{-1}(a))=g^{\ell p}|_{W(g)}^{-1}(\phi_{g,v}(a))$, hence $\phi_{g,v}(\alpha_v(D_g))$  is a path connecting $\phi_{g,v}(a)$ and $g^{\ell p}|_{W(g)}^{-1}(\phi_{g,v}(a))$. 
It follows that 
$$\beta(g)=\bigcup_{j\geq 0} g^{\ell p}|_{W(g)}^{-j}(\phi_{g,v}(\alpha_v(D_g)))$$ is a path connecting $\phi_{g,v}(a)$ and $\zeta(g)$. 
The relation $\phi_{g,v}\circ \widehat{B}_{g,v}^{p}=g^{\ell p}\circ \phi_{g,v}$
implies that $\beta(g)=\phi_{g,v}(\gamma_{v,s_0}(D_g))$, and $\beta(g)$ converges to   $\phi_{g,v}(r_{v, s_0}(D_g))$. Hence $\phi_{g,v}(r_{v,s_0}(D_g))=\zeta(g)$ and $\lim_{n\rightarrow \infty}\phi_{f_n,v}(r_{v,s_0}(\Phi^{-1}(f_n))))= \phi_{f,v}(s_0)$.



By  the stability of external rays (see Lemma \ref{stability-e-r}),  the landing point $x_g(\theta)$ of   $R_{g}(\theta)$ satisfies $x_g(\theta)=\zeta(g)$ for $g\in \mathcal F$. 
Since the two maps 
$$g\mapsto x_g(\theta)   \text{ and  } g\mapsto \phi_{g,v}(r_{v,s_0}(\Phi^{-1}(g)))$$
 are continuous in  $g\in\mathcal H$, taking values in the finite set $\{w; g^{\ell p}(w)=w\}$, identical when $g=f_n$, we have $x_g(\theta)=\phi_{g,v}(r_{v,s_0}(\Phi^{-1}(g)))$ for all $g\in \mathcal H$. 
\end{proof}


\begin{defi} [The induced divisor $ D_{v}^\partial(g)$] \label{an induced divisor}   
For  $g\in \mathcal H$  and $v\in V$, let
 \begin{equation}
	\label{in-divisor}
 \displaystyle D_{v}^\partial(g)=\sum_{s\in \partial\mathbb D}\Big(\sum_{c\in C_{s,v}}({\rm deg}(g,c)-1)\Big)\cdot s \in {\rm Div}_{d-\delta(v)}(\partial \mathbb D),
\end{equation}
where $d={\rm deg}(g)$ and $C_{s,v}={\rm Crit}(g)\cap L_{U_{g,v}, \phi_{g,v}(s)}$.

The divisor $D_{v}^\partial(\cdot )$ can also be defined for $f\in I_{\Phi}(D)$ with $D\in \partial_0^* {\rm Div}{(\mathbb D)}^S$ in following way:  let 
$\{f_n\}_{n\geq 1}\subset \mathcal H$ be a sequence with
\begin{itemize}
	\item  $f_n\rightarrow f$, $\Phi^{-1}(f_n)\rightarrow D$, and
	
	\item  $\phi_{f_n,v}$ converges locally and uniformly on $\mathbb D$ to  $\phi_{f,v}$, \ $\forall v\in V$.
\end{itemize} 
Then $D_{v}^\partial(f)$ can be defined using (\ref{in-divisor}), and $D_{v}^\partial(f) \in {\rm Div}_{d-{\rm deg}(f|_{U_{f,v}})}(\partial \mathbb D)$, where $U_{f,v}$ is the Fatou component of $f$ containing $v(f)$.
\end{defi}

\begin{pro} \label{divisor-correspondence} Let
$D=\big((B^0_u, D_u^{\partial})\big)_{u\in V}\in \partial_0^* {\rm Div}{(\mathbb D)}^S$ and  $f\in I_{\Phi}(D)$.
Let  $\{f_n\}_{n\geq 1}$ be a sequence of maps in $\mathcal H$ such that 
\begin{itemize}
\item  $f_n\rightarrow f$, $\Phi^{-1}(f_n)\rightarrow D$,  and

\item  $\phi_{f_n,u}$ converges locally and uniformly on $\mathbb D$ to  $\phi_{f,u}$,  \ $\forall u\in V$.
\end{itemize} 
 
Then for any $v\in V$,   the divisor sequence $\{D_{v}^{\partial}(f_n)\}_{n\geq 1}$ converges to a limit   $E_{v}^\partial(f)\in {\rm Div}_{d-\delta(v)}(\partial \mathbb D)$, which satisfies the equality
$$E_{v}^\partial(f)+D_v^\partial=D_{v}^\partial(f).$$
\end{pro}

\begin{rmk} \label{divisor-relation}
 Proposition \ref{divisor-correspondence} implies that $ D_{v}^\partial(f) \geq D_{v}^{\partial}$, hence ${\rm supp}(D_{v}^{\partial})\subset {\rm supp}(D_{v}^\partial(f))$.  
 It is possible that $D_{v}^{\partial}$ is the zero divisor for some $v\in V$.
\end{rmk}


\begin{proof} We focus on the case $v\in V_{\rm p}$. The argument works for $v\in V_{\rm np}$, by taking preimages.
	
Fix  any small   $\epsilon>0$. It suffices to show that for large $n$, 
 $$D_{v}^{\partial}(f_n)+D_v^{\partial}\in \N^0_{4\epsilon}(D_{v}^\partial(f)).$$ 

Assume   $\epsilon$  is small so that the closed disks $\{\overline{\mathbb{D}(q,\epsilon)}; q\in 
{\rm supp}(D_{v}^\partial(f))\}$ are disjoint.
For any $q\in {\rm supp}(D_{v}^\partial(f))$, there are two $\widehat{B}^0_{v}$-repelling periodic points  
$q^+, q^-\in (\mathbb{D}(q,\epsilon)\setminus \overline{\mathbb{D}(q, \epsilon/2)})\cap \partial \mathbb D$, satisfying that
 
\begin{itemize}
\item  $q^+, q^-\in \partial_v^f \mathbb D\setminus \bigcup_{j\geq 0} (\widehat{B}_v^0)^{-j}(Z_v)$ \footnote{For the moment, we have not proven that $\partial_v^f \mathbb D\bigcap \bigcup_{j\geq 0}   (\widehat{B}_v^0)^{-j}(Z_v)=\emptyset$. After we have proven Proposition \ref{divisor-correspondence}, we will see that  $\partial_v^f \mathbb D\bigcap \bigcup_{j\geq 0}  (\widehat{B}_v^0)^{-j}(Z_v)=\emptyset$}, where $Z_v$  is given by Proposition \ref{extension-repelling};

\item $q^+, q, q^-$ are in counter clockwise order;

\item  $\phi_{f,v}(q^+), \phi_{f,v}(q^-)$ are two repelling periodic points of $f$. 
\end{itemize} 

For any $s_0\in \{q^{\pm}; q\in {\rm supp}(D_{v}^\partial(f))\}$, let $p=p(s_0)\geq 1$ be the $\widehat{B}^0_{v}$-period of $s_0$.
With the same notations as those in the proof of Proposition \ref{convergence-repelling}, there exist a  neighborhood 
$V(s_0)$ of $s_0$,  
 an  arc
  $\gamma_{v, s_0}(D)\subset V(s_0)$ connecting $a(s_0)\in V(s_0) \cap \mathbb D$ with  $s_0$,   with $(\widehat{B}_{v}^{0})^{p}|_{V(s_0)}^{-1}(\gamma_{v, s_0}(D))\subset \gamma_{v, s_0}(D)$.

We assume $|a(q^+)|=|a(q^-)|:=r_0\in (1-\epsilon, 1)$ for all $q\in  {\rm supp}(D_{v}^\partial(f))$,
here $r_0$ is independent of $q$, and is chosen so that $B_v^0$ has exactly ${\rm deg}(B_v^0)-1$ critical points in $\mathbb D(0, r_0)$.
 By the proof of Proposition  \ref{convergence-repelling},  there is $\tau>0$ such that for   $s_0\in \{q^{\pm}; q\in {\rm supp}(D_v^{\partial}(f))\}$ and  ${B}=(B_u)_{u\in V}\in \N_{\tau}(D)$, 

\begin{itemize}
 
\item  $V(s_0)$ is    in a linearization neighborhood of the $\widehat{B}_v$-periodic  point $r_{v,s_0}(B)$;

\item $\gamma_{v, s_0}(B)$ is a path
starting from $a(s_0)$ and converging to $r_{v,s_0}(B)$, and moves continuously with respect to 
${B}\in \N_{\tau}(D)\cup\{D\}$.

\end{itemize}
 Let $C_{v,q}\subset \{|z|=r_0\}$ be the shorter  arc connecting $a(q^+)$ and $a(q^-)$,   and let $C_{v,q}(B)\subset \partial \mathbb D$ be  the shorter  arc connecting $r_{v, q^+}(B)$ and $r_{v, q^-}(B)$. The  curve 
$C_{v,q}(B)\cup C_{v,q}\cup 
\gamma_{v, q^-}(B)\cup  \gamma_{v,q^+}(B)$  bounds a disk $\Omega_{v,q}(B)\subset \mathbb D$, for ${B}\in \N_{\tau}(D)\cup\{D\}$. By a careful choice of $r_0$ and $\tau$, we assume
$$\max_{q\in {\rm supp}(D_{v}^\partial(f))} {\rm diam } (\Omega_{v,q}(B))\leq 3\epsilon, \ \forall {B}\in \N_{\tau}(D)\cup\{D\}.$$


\begin{figure}[h]
	\begin{center}
		\includegraphics[height=5cm]{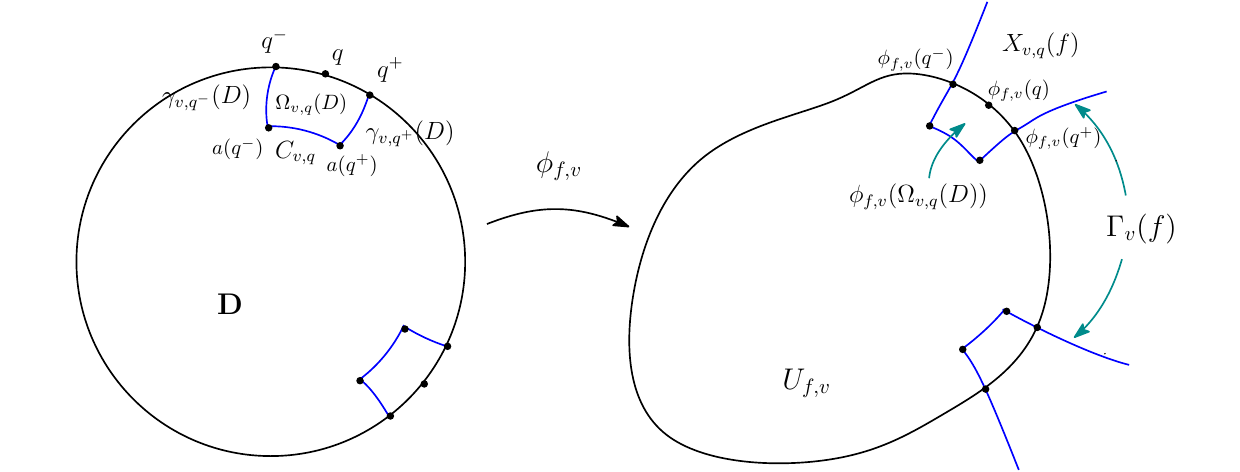}
	\end{center}
	\caption{Graphs in model plane and dynamical plane, for $f$}
\end{figure}

In the dynamical plane of $f$, we may assume $\phi_{f,v}(V(s_0)\cap \mathbb D)$ is contained in a
linearization neighborhood   of $\phi_{f,v}(s_0)$.
By the assumption $s_0\in \{q^{\pm}; q\in {\rm supp}(D_{v}^{\partial}(f))\}\subset \partial_v^f \mathbb D$
 (that is $L_{U_{f,v}, \phi_{f,v}(s_0)}=\{\phi_{f,v}(s_0)\}$),   there is a unique external ray $R_f(\theta_{v,s_0})$ landing at $\phi_{f,v}(s_0)$.

\begin{figure}[h]
	\begin{center}
		\includegraphics[height=5cm]{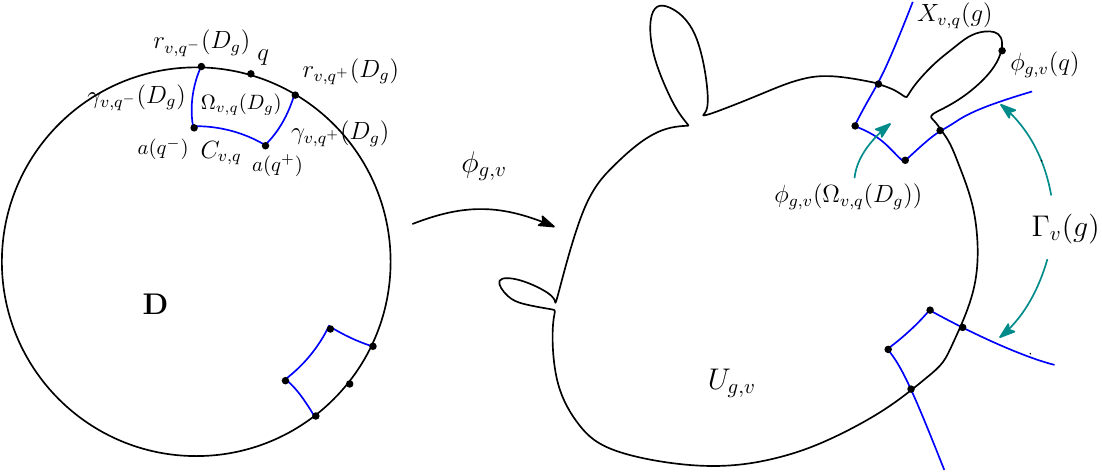}
	\end{center}
	\caption{Graphs in model plane and   dynamical plane, for $g$. Note that $\Omega_{v,q}(D_g)$ is  continuous in $g\in \mathcal F$,  but  $\phi_{g,v}(\Omega_{v,q}(D_g))$    is not continuous in $g\in \mathcal F$ in general.}
\end{figure}

By Proposition \ref{convergence-repelling} and  Lemma \ref{stability-e-r},
 there exist a neighborhood $\mathcal U$ of $f$,   an integer $n_0\geq 1$ such that $  \{f_n; n\geq n_0\} \subset \mathcal U\cap \Phi(\N_{\tau}(D))$, and  
  for  
$g\in  \mathcal F:= \{f_n; n\geq n_0\} \cup\{f\}$, 
\begin{itemize}
\item  the external ray  $R_g(\theta_{v,s_0})$ lands at the $g$-repelling point  $\phi_{g,v}(r_{v,s_0}(D_g))$  (here $D_f=D$ and $D_{f_n}=\Phi^{-1}(f_n)$);

\item the  arc $\phi_{g,v}(\gamma_{v, s_0}(D_g))$ is contained in a  linearization neighborhood of $\phi_{g,v}(r_{v,s_0}(D_g))$;
 
\item the  points  $r_{v,q^+}(D_g),   r_{v,q^-}(D_g)\in \mathbb{D}(q,\epsilon)\setminus\{q\}$;

\item the  graph 
$$\Gamma_v(g)=\bigcup_{q\in {\rm supp}(D_v^{\partial}(f))}\Big[  \phi_{g,v}\big(C_{v,q}\big)\bigcup \bigcup_{s_0=q^{\pm}}\Big(\overline{R_{g}(\theta_{v,s_0})}\cup \phi_{g,v}\big(\gamma_{v,s_0}(D_g)\big)\Big)\Big]$$
 moves continuously in Hausdorff topology with respect to 
$g\in \mathcal F $. 
\end{itemize}

 It follows that
$$\Gamma_v(f_n)\rightarrow \Gamma_v(f), \   \phi_{f_n,v}(r_{v,s_0}(\Phi^{-1}(f_n)))\rightarrow 
\phi_{f,v}(s_0), \ \ \text{ as } n\rightarrow +\infty.$$
For each $q\in {\rm supp}(D_{v}^\partial(f))$ and $g\in  \mathcal F$, 
let $X_{v,q}(g)$ be the component of  $\mathbb C\setminus \Gamma_v(g)$ containing $\phi_{g,v}(q)$.
Then $\phi_{g,v}(\Omega_{v,q}(D_g))= X_{v,q}(g) \cap U_{g,v}.$
Write   
$$D_v^\partial=\sum_{y\in {\rm supp}(D_v^\partial)} \nu(y)\cdot y, \
D_{v}^\partial(f)=\sum_{q\in {\rm supp}(D_{v}^\partial(f))} \mu(q)\cdot q.$$

Note that $\Gamma_v(f)$ avoids the critical set of $f$,  and  $X_{v,q}(f)$ contains exactly $\mu(q)\geq 1$ critical points of $f$. By the continuity of the graphs, for large $n$, 
\begin{itemize}
	\item   (A1) the graph $\Gamma_v(f_n)$ contains no critical points of $f_n$, hence  $X_{v,q}(f_n)$ contains exactly $\mu(q)$ critical points of $f_n$;
	
	\item  (A2) $f_n$ has exactly $d-1- {\rm deg}(D_{v}^\partial(f))={\rm deg}(B_v^0)-1$ critical points in 
	$$\mathbb C\setminus \bigsqcup_{q\in  {\rm supp}(D_{v}^\partial(f))} \overline{X_{v,q}(f_n)}.$$ 
	\end{itemize}
	



Note that $\Omega_{v,q}(B)$ moves continuously in ${B}\in \N_{\tau}(D)\cup\{D\}$.  By shrinking $\tau$,  there is an $\epsilon_*\in (0, \epsilon]$ so that the closed Jordan disks 
$\big\{\overline{ \Omega_{v,q}(B)}\big\}_{q\in  {\rm supp}(D_{v}^\partial(f))}$ are  $\epsilon_*$-apart, for all  ${B}\in \N_{\tau}(D)\cup\{D\}$.

By the assumption $\Phi^{-1}(f_n)=(B_{f_n, u})_{u\in V}\rightarrow D$ and      the behavior of the critical points
	of Blaschke products (see Theorem \ref{zakeri-exten}), when $n$ is large, 
\begin{itemize}
	\item (A3)    ${B}_{f_n,v}$ has exactly  ${\rm deg}(B_v^0)-1$ critical points in
	$\mathbb D(0, r_0)$;
	\item (A4)   ${B}_{f_n,v}$ has exactly $\nu(y)$ critical points  in $\mathbb{D}(y, \epsilon_*/2)\cap \mathbb D$, for   $y\in {\rm supp}(D_v^\partial)$. 
\end{itemize}

By (A1)(A2) and (A3),  $f_n$ has exactly ${\rm deg}(B_v^0)-1$ critical points in 
$$U_{f_n ,v}\cap (\mathbb C\setminus \bigsqcup_{q\in  {\rm supp}(D_{v}^\partial(f))}\overline{X_{v,q}(f_n))},$$ and the critical points of   ${B}_{f_n,v}$  in
$\mathbb D\setminus \mathbb D(0, r_0)$ are actually in $\bigcup_{q\in {\rm supp}(D_{v}^\partial(f))}\Omega_{v,q}(\Phi^{-1}(f_n))$. 
By (A4) and the choice of $\epsilon_*$, for each $y\in {\rm supp}(D_v^\partial)$,  there is a unique $q(y)\in {\rm supp}(D_{v}^\partial(f))$ such that $\mathbb{D}(y, \epsilon_*/2)\cap  \Omega_{v,q(y)}(\Phi^{-1}(f_n))\neq \emptyset$ for all large $n$.  So  the $\nu(y)$ critical points of  ${B}_{f_n,v}$  in  $\mathbb{D}(y, \epsilon_*/2)$ are actually in 
$\Omega_{v,q(y)}(\Phi^{-1}(f_n)	)$.

\begin{figure}[h]
	\begin{center}
		\includegraphics[height=4.8cm]{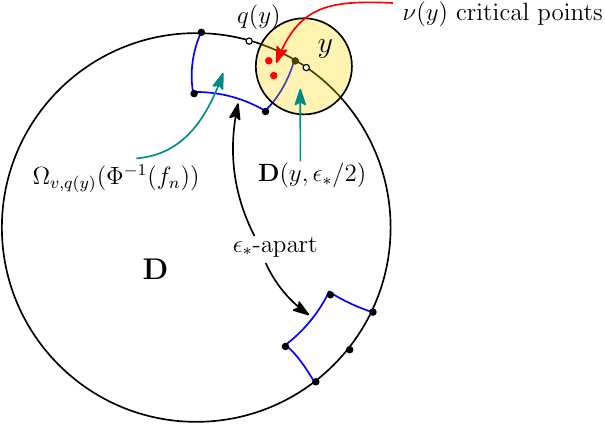}
	\end{center}
\caption{Some near degenerate critical points  in model plane.}
\end{figure}

 This implies  that $D_{v}^{\partial}(f_n)+D_v^{\partial}\in \N^0_{\epsilon_n}(D_{v}^\partial(f))$ for large $n$, where 
$$\epsilon_n={\rm diam }(\mathbb{D}(y, \epsilon_*/2))+\max_{q\in {\rm supp}(D_{v}^\partial(f))} {\rm diam } (\Omega_{v,q}(\Phi^{-1}(f_n)))\leq 4\epsilon.$$ 
 
 Since $\epsilon>0$ is arbitrary, we have 
$\lim_{n\rightarrow +\infty}D_{v}^{\partial}(f_n)+D_{v}^\partial=D_{v}^\partial(f).$
\end{proof}

\section{Combinatorial properties for maps in $I_{\Phi}(D)$} \label{comb-pro-imp}

In this section, we shall study some combinatorial properties of the maps in   $I_{\Phi}(D)$, with $D\in \partial_0^* {\rm Div}{(\mathbb D)}^S$.     These properties enhance our comprehension of  the regular part 
  $\partial_{\rm reg} \mathcal H$ on the boundary of   $\mathcal H$.

Let
$$D=\big((B^0_v, D_v^{\partial})\big)_{v\in V}\in \partial_0^* {\rm Div}{(\mathbb D)}^S  \text{ and } f\in I_{\Phi}(D).$$

We shall establish the following two properties step by step.

\begin{itemize}
\item  (P1).  For any sequence  $\{f_n\}_{n\geq 1}\subset\mathcal H$ satisfying that  $f_n\rightarrow f$ and $\Phi^{-1}(f_n)\rightarrow D$, and any $v\in V$, the  conformal maps 
 $\phi_{f_n,v}: \mathbb D\rightarrow U_{f_n, v}$ converge locally and uniformly on $\mathbb D$ to  a limit map, denoted by $\phi_{f,v}$. This implies that $\{\phi_{f,v}\}_{v\in V}$ are uniquely determined by the pair $(f,D)$ \footnote{A priori, it might happen that  $f\in I_{\Phi}(D)$ is also contained in another impression $I_{\Phi}(D')$ (in this case, by Fact \ref{fact-def}, there is a  sequence $\{g_n\}_{n\geq 1}\subset\mathcal H$ satisfying that  $g_n\rightarrow f$, $\Phi^{-1}(g_n)\rightarrow D'$,  and the conformal maps 
 	$\phi_{g_n,v}: \mathbb D\rightarrow U_{g_n, v}$   have a limit  $\phi_{f,v}^{D'}$. That's
 	the reason  why  the maps  $\{\phi_{f,v}\}_{v\in V}$ are   determined by the pair $(f,D)$ rather than a single map $f$).  In Proposition \ref{map-imp}, we shall show that $f\in I_{\Phi}(D)$ is contained in only one impression, under certain conditions on $D$. }, see Proposition \ref{combinatorial-property0}.

\item (P2.a).    For any $v\in V$, the set $\partial_v^{f} \mathbb D$ (see (\ref{fatou-boundary-0})) and the divisor     
$D_{v}^\partial(f)$ (see (\ref{in-divisor})) depend on $D$, and are independent of $f\in I_{\Phi}(D)$, see Proposition \ref{combinatorial-property-2}. Because of this, we denote $\partial_v^{f} \mathbb D$  by $\partial_v^{D} \mathbb D$,   and  denote $D_v^\partial (f)$ by $D_v^\partial (\mathcal H)$ (to emphasize the dependence of $D_v^\partial (f)$
on $\mathcal H$).

\item (P2.b).  For any $v\in V$ and any $s\in \partial \mathbb D$, the angles $\theta_{f,v}^+(s)$ and $\theta_{f,v}^-(s)$  for which the external rays $R_{f}(\theta^+_{f,v}(s))$ and $R_{f}(\theta^-_{f,v}(s))$ separate the limb $L_{U_{f,v}, \phi_{f,v}(s)}$ and the Fatou component $U_{f, v}$,  depend on $D$ and are independent  of $f\in I_{\Phi}(D)$ (see Proposition \ref{combinatorial-property-2}).
Because of this, we denote  $\theta_{f,v}^{\pm}(s)$ by $\theta_{D,v}^{\pm}(s)$.
 

\end{itemize}

\subsection{Combinatorial property (P1)}

Let $g$ be a polynomial with connected Julia set $J(g)$, and let $q\in J(g)$ be a  (pre-)repelling or  (pre-)parabolic point of $g$.
 Let $\mu_g(q)$ be the number of external rays landing at $q$.
We define an integer $M_g\geq 0$ as follows: if  $g$ has no parabolic cycles, we set $M_g=0$; if $g$ has parabolic cycles,  let  $M_g$ be the maximum of all the periods of the angles for which the external rays landing at parabolic points.

\begin{lem} \label{rays-number} Let $f$ be a polynomial with connected Julia set. Suppose that
\begin{itemize}
\item $q\in J(f)$ is a pre-repelling point of $f$, whose orbit avoids critical points, and meets a repelling cycle with period $l>M_f$;

\item $\mathcal U$ is a neighborhood of $f$ so that
 for all $g\in \mathcal U$, $q(g)$ is a pre-repelling point of $g$, moving continuously with respect to $g\in \mathcal U$, and $q(f)=q$.
\end{itemize}
Then there is a neighborhood $\mathcal N\subset \mathcal U$ of $f$, such that
$$\mu_g(q(g))=\mu_f(q), \ \forall g\in \mathcal N\cap \mathcal C_d.$$
\end{lem}
We remark that $ \mathcal N\cap \mathcal C_d$ can be replaced by $\mathcal N$, see the proof.
\begin{proof}  By taking preimages, it suffices to treat the repelling case.
	
	 Let 
$\Theta_q=\{\theta\in \mathbb R/\mathbb Z; R_f(\theta) \text{ lands at }q \}$.  
By \cite[Lemma 18.12]{M},  all angles in $\Theta_q$ have the same period, say $\ell>0$, which is a multiple  of $l$.

Let $\Theta$ be the set of all angles with period $\ell$. Clearly $\Theta_q\subset \Theta$. By  assumption $\ell\geq l>M_f$,  the external ray  $R_f(\theta)$ lands at a repelling 
point for  $\theta\in \Theta$.
Let $q_0=q, q_1, \cdots, q_m$ be  all landing points of the external rays $R_f(\theta), \ \theta\in \Theta$.  By the Implicit Function Theorem and shrink $\mathcal U$ if necessary, for each $g\in \mathcal U$, there is a repelling point $q_k(g)$ of $g$, moving continuously with respect to $g\in \mathcal U$ and 
$q_k(f)=q_k$.

By  Lemma \ref{stability-e-r}, there is a neighborhood $\mathcal N\subset \mathcal U$ of $f$ such that for any $\theta\in \Theta$,
$$R_f(\theta) \text{ lands at }q_k \Longrightarrow R_g(\theta) \text{ lands at } q_k(g), \ \forall \ g\in \mathcal N\cap \mathcal C_d.$$
 It follows that 
$$\mu_g(q_k(g))\geq \mu_f(q_k), \ 0\leq k\leq m.$$

On the other hand, again by \cite[Lemma 18.12]{M}, for each $0\leq k\leq m$, 
 the angle $\theta$ for which  $R_g(\theta)$ lands at $q_k(g)$   has period $\ell$. Hence 
$$\#\Theta= \sum_{k=0}^m \mu_f(q_k)\leq \sum_{k=0}^m \mu_g(q_k(g))\leq \#\Theta.$$
This implies that $\mu_g(q_k(g))=\mu_f(q_k)$ for all $0\leq k\leq m$ and $g\in \mathcal N\cap \mathcal C(d)$.
\end{proof}



\begin{pro} \label{combinatorial-property0} 
 Let $D=\big((B_u^0, D_u^{\partial})\big)_{u\in V}\in \partial_0^* {\rm Div}{(\mathbb D)}^S$ and  $f\in I_{\Phi}(D)$. 
 Then there exist    unique conformal mappings $\{\phi_{f,v}: (\mathbb D, 0)\rightarrow (U_{f,v}, v(f))\}_{v\in V}$  
  with the property that for any $v\in V$ and any sequence $\{f_n\}_{n\geq 1}\subset \mathcal H$ with
$$f_n\rightarrow f \ \text{ and } \ \Phi^{-1}(f_n)\rightarrow D,$$
the sequence  
$\{\phi_{f_n, v}\}_{n\geq 1}$ converges locally and uniformly in $\mathbb D$ to $\phi_{f,v}$.

 Further,  let  $s_0\in \partial_v^f \mathbb D$
 be a  $D$-pre-periodic point, then
  $\phi_{f,v}(s_0)$ is a pre-repelling  point of $f$.
\end{pro}




The following remark gives a better understanding of  Proposition \ref{combinatorial-property0} . 
\begin{rmk} 
	1. Suppose $v\in V_{\rm p}$. If  $s_0\in \partial \mathbb D\setminus \partial_v^f \mathbb D$ is a  $\widehat{B}^0_{v}$-periodic point,  then $\phi_{f,v}(s_0)$ is a parabolic periodic point of $f$, see  Proposition \ref{para-dym}(2).

2. Suppose $v\in V_{\rm p}$.  A priori, different subsequences of  $\{\phi_{f_n, v}\}_{n\geq 1}$ might have different limit maps. Take two of them $\phi_{f,v}$ and $\phi^*_{f,v}$.  By Lemma \ref{model-map},  
$$\widehat{B}^0_v=\phi^{-1}_{f,v}\circ f^{\ell}\circ \phi_{f,v}=(\phi_{f,v}^*)^{-1}\circ f^{\ell}\circ \phi^*_{f,v},$$
where $\ell$ is the $\sigma$-period of $v$.
 It follows  that 
$$\phi_{f,v}^{-1}\circ \phi_{f,v}^*\in {\rm Aut}(\widehat{B}^0_v):=\Big\{h:\mathbb D\rightarrow \mathbb D \text{ is conformal and } h\circ \widehat{B}^0_v=\widehat{B}^0_v\circ h\Big\}.$$ 
Since ${\rm Aut}(\widehat{B}^0_v)$ is    finite, there are only finitely many possible limits for  $\{\phi_{f_n, v}\}_{n\geq 1}$.
 Proposition \ref{combinatorial-property0} implies a slightly stronger fact: there is exactly one  limit map for  $\{\phi_{f_n, v}\}_{n\geq 1}$ even if ${\rm Aut}(\widehat{B}^0_v)$ is non-trivial.
\end{rmk}

\begin{proof}[Proof of Proposition \ref{combinatorial-property0}.]
First, as discussed  at the beginning of Section \ref{imp-divisor}, there exist a sequence   $\{f_n\}_{n\geq 1}\subset\mathcal H$ and  conformal maps $\big\{\phi_{f,u}: (\mathbb D, 0)\rightarrow (U_{f,u}, u(f))\big\}_{u\in V}$,  such that
\begin{itemize}
\item  $f_n\rightarrow f\in I_{\Phi}(D)$, $\Phi^{-1}(f_n)\rightarrow D$, and

\item  $\phi_{f_n,u}$ converges locally and uniformly on $\mathbb D$ to  $\phi_{f,u}$, \ $\forall u\in V$.
\end{itemize}

It suffices to prove the following two
 statements:

1. For $v\in V_{\rm p}$,  if $s_0\in \partial_v^f \mathbb D$ is  $\widehat{B}^0_{v}$-periodic, then
 $\phi_{f,v}(s_0)$ is $f$-repelling. 

2. For any  $v\in V$ and any other sequence $\{g_n\}_{n\geq 1}\subset \mathcal H$ with $g_n\rightarrow f\in I_{\Phi}(D)$ and $\Phi^{-1}(g_n)\rightarrow D$, the maps  $\phi_{g_n,v}$  converge locally and uniformly in $\mathbb D$ to  $\phi_{f,v}$.

\vspace{5pt}

 1. Note that $\phi_{f,v}(s_0)\in \partial U_{f,v}$ is $f$-periodic. By Proposition \ref{RY},  $\phi_{f,v}(s_0)$ is either repelling or parabolic.
 If  $\phi_{f,v}(s_0)$  is  $f$-parabolic,  then there is an integer $l\geq 0$ and a parabolic critical Fatou component $V$,  such that $f^l(\phi_{f,v}(s_0))\in \partial U_{f, \sigma^l(v)}\cap \partial V$.  
It follows that $s_0\in {\rm supp}(D_{v}^{\partial}(f))$ (if $l=0$)
or $B^0_{\sigma^{l-1}(v)}\circ \cdots \circ B^0_v(s_0)\in  {\rm supp}(D_{\sigma^l(v)}^{\partial}(f))$ (if $l\geq 1$).
In either case, the limb $L_{U_{f,v}, \phi_{f,v}(s_0)}$ is not trivial, implying that $s_0\notin \partial_v^f \mathbb D$.
This contradicts the assumption.


2.  Given another  sequence 
$\{g_n\}_{n\geq 1}\subset \mathcal H$ 
 with  $g_n\rightarrow f\in I_{\Phi}(D)$ and $\Phi^{-1}(g_n)\rightarrow D$. By choosing a subsequence, we assume $\phi_{g_n,u}$ converges locally and uniformly in $\mathbb D$ to  $\phi^*_{f,u} : (\mathbb D, 0)\rightarrow (U_{f, u}, u(f))$ for all $u\in V$. 
In the following, we shall show that for all $u\in V$, $\phi^*_{f,u}=\phi_{f,u} $.

 By Lemma \ref{model-map},
 ${B}^0_v=(\phi^{*}_{f,\sigma(v)})^{-1}\circ f\circ \phi^*_{f,v}=\phi^{-1}_{f,\sigma(v)}\circ f\circ \phi_{f,v}, \ \forall v\in V$.
Recall that $\partial_v^f \mathbb D= \phi^{-1}_{f,v}(\partial_0 U_{f,v})$. Let
$\partial_v^* \mathbb D= (\phi^*_{f,v})^{-1}(\partial_0 U_{f,v})$.  

Let
$s_0\in \partial_v^f \mathbb D\cap \partial_v^* \mathbb D$
be  $D$-pre-periodic. By the choice of $s_0$,  only one external ray lands at
$\phi_{f,v}(s_0)$ or $\phi^*_{f,v}(s_0)$. By  1,    both 
$\phi_{f,v}(s_0)$ and $\phi^*_{f,v}(s_0)$ are $f$-pre-repelling.
 We may further assume $s_0$ is chosen so  that  the $f$-orbits of $\phi_{f,v}(s_0)$ and $\phi^*_{f,v}(s_0)$ eventually meet repelling cycles with periods $>M_f$ (here $M_{f}$ is defined before
 Lemma \ref{rays-number}).  
 Take $\theta, \theta^*\in \mathbb R/ \mathbb Z$ so that  $R_f(\theta)$ lands at 
$\phi_{f,v}(s_0)$, $R_f(\theta^*)$ lands at 
$\phi^*_{f,v}(s_0)$. By Proposition \ref{convergence-repelling}, for $g\in \mathcal H$, the two external rays $R_g(\theta), R_g(\theta^*)$   land at $\phi_{g,v}(r_{v,s_0}(\Phi^{-1}(g)))$.
By Lemma \ref{rays-number},  
 $\phi_{g,v}(r_{v,s_0}(\Phi^{-1}(g)))$ is the landing point of exactly one external ray,  when $g\in \mathcal H$ is sufficiently close to $f$. Hence $\theta=\theta^*$, implying that $\phi_{f,v}(s_0)=\phi^*_{f,v}(s_0)$. 
 It follows that the conformal map
 $\phi^{-1}_{f,v}\circ \phi^*_{f,v}:\mathbb D\rightarrow \mathbb D$  fixes $0$ and $s_0\in \partial \mathbb D$, we have $\phi_{f,v}=\phi_{f,v}^*$. 
\end{proof}


%

\subsection{Combinatorial property (P2)} \label{cp2}
 
Let
$D=\big((B_v^0, D_v^{\partial})\big)_{v\in V}\in \partial_0^* {\rm Div}{(\mathbb D)}^S$.
Let $f\in I_{\Phi}(D)$ and $ v\in V$. 
Taking any point $q\in \partial \mathbb D$, there are two cases:

\textbf{Case 1: $q\in \partial_v^f \mathbb D$.}  In this case, the limb   $L_{U_{f,v}, \phi_{f,v}(q)}=\{\phi_{f,v}(q)\}$.  By Proposition \ref{RY}, there is a unique  $\theta\in \mathbb R/\mathbb Z $ so that $R_f(\theta)$ lands at $\phi_{f, v}(q)$. We set $\theta^+_{f,v}(q)=\theta^-_{f,v}(q)=\theta$.

\textbf{Case 2: $q\in  \partial \mathbb D\setminus \partial_v^f \mathbb D$.}  In this case, the limb $L_{U_{f,v}, \phi_{f,v}(q)}\neq \{\phi_{f,v}(q)\}$. 
By Proposition \ref{RY},  there are  
 $\alpha, \beta\in \mathbb R\setminus \mathbb Z$ satisfying that
\begin{itemize}

\item  
$R_f(\alpha)$ and $R_f(\beta)$ land at $\phi_{f,v}(q)$;

\item $\overline{R_f(\alpha)}$, $L_{U_{f,v}, \phi_{f,v}(q)}$, $\overline{R_{f}(\beta)}$ attach  $\phi_{f,v}(q)$    in     counter-clockwise order.
\end{itemize}
We set $\theta^+_{f,v}(q)=\beta, \theta^-_{f,v}(q)=\alpha$.

\begin{figure}[h]
	\begin{center}
		\includegraphics[height=5cm]{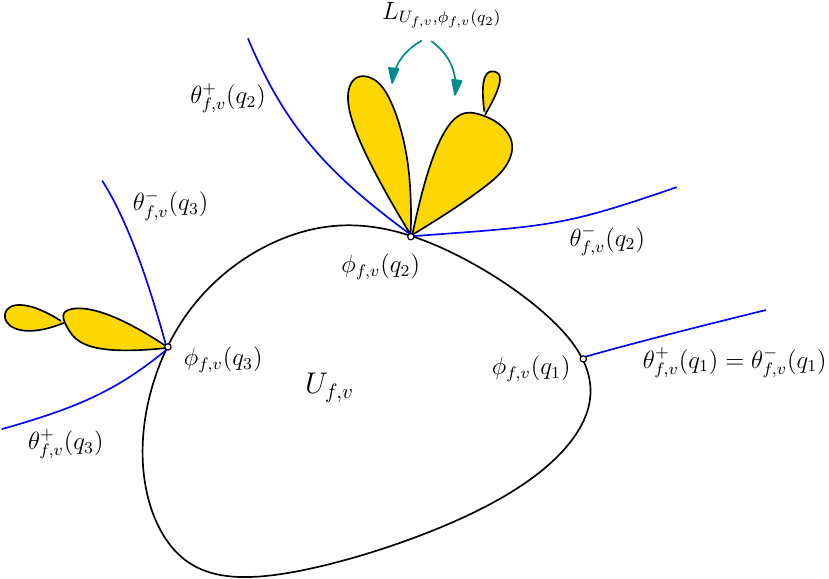}
	\end{center}
	\caption{Limbs and angles $\theta_{f,v}^{\pm}$}
\end{figure}

In this way, for all $q\in \partial \mathbb D$, the angles $\theta^-_{f,v}(q), \theta^+_{f,v}(q)$ are well-defined.   Let $\Delta_{f, v}(q)=\theta^+_{f,v}(q)-\theta^-_{f,v}(q)$,  and let $\nu_{f, v}(q)$ be the number of the critical points in $L_{U_{f,v}, \phi_{f,v}(q)}$ counting multiplicity.

\begin{lem} \label{limb-angles} Let
$D=\big((B_v^0, D_v^{\partial})\big)_{v\in V}\in \partial_0^* {\rm Div}{(\mathbb D)}^S$, $f\in I_{\Phi}(D),  v\in V$.

1.  The two functions $\theta^+_{f,v}, \theta^-_{f,v}: \partial\mathbb D \rightarrow \mathbb R/\mathbb Z$ satisfy
$$\lim_{q\rightarrow q_0^+}\theta^+_{f,v}(q)=\lim_{q\rightarrow q_0^+}\theta^-_{f,v}(q)=\theta^+_{f,v}(q_0), \ \lim_{q\rightarrow q_0^-}\theta^+_{f,v}(q)=\lim_{q\rightarrow q_0^-}\theta^-_{f,v}(q)=\theta^-_{f,v}(q_0),$$
where $q\rightarrow q_0^+$ (resp.  $q\rightarrow q_0^-$) means that $q$ approaches $q_0$  along $\partial \mathbb D$  in clockwise  (resp.   counter-clockwise) order.

2. For any $q\in \partial \mathbb D$,  we have 
$\nu_{f, v}(q)=d \Delta_{f, v}(q)- \Delta_{f, \sigma(v)}(B_v^0(q))$.
\end{lem}
\begin{proof} 1. It is an immediate consequence of Proposition \ref{RY}.
	
	2. See \cite[\S 2]{GM}.
	\end{proof}

\begin{pro} \label{combinatorial-property-2} 
 Let $D=\big((B_v^0, D_v^{\partial})\big)_{v\in V}\in \partial_0^* {\rm Div}{(\mathbb D)}^S$.
For any $f_1, f_2\in I_{\Phi}(D)$ and any $v\in V$, we have 
$$\partial_v^{f_1} \mathbb D=\partial_v^{f_2} \mathbb D, \    \theta^+_{f_1,v}\equiv\theta^+_{f_2,v}, \ \theta^-_{f_1,v}\equiv\theta^-_{f_2,v}, \
D_{v}^\partial (f_1)=D_{v}^\partial(f_2).$$
\end{pro}
\begin{proof} By taking preimages, it suffices to treat the case $v\in V_{\rm p}$ 
	
	Note that 
$\partial_v^{f_k} \mathbb D=\big\{s\in \partial \mathbb D; \theta^+_{f_k,v}(s)=\theta^-_{f_k,v}(s) \big\}$ for $k=1,2$.
To show $\partial_v^{f_1} \mathbb D=\partial_v^{f_2} \mathbb D$, it suffices to show that 
$\theta^+_{f_1,v}\equiv\theta^+_{f_2,v}$ and $\theta^-_{f_1,v}\equiv\theta^-_{f_2,v}$.


 Let $\Theta(f_1, f_2)$ be the set of all $\widehat{B}^0_{v}$-periodic points in $\partial_v^{f_1} \mathbb D \cap \partial_v^{f_2} \mathbb D$, whose periods are greater than $\max\{M_{f_1}, M_{f_2}\}$ (here $M_{f_k}$ is defined before
Lemma \ref{rays-number}).
Clearly,  $\Theta(f_1, f_2)$ is a dense subset of $\partial \mathbb D$. 
For each $s_0\in \Theta(f_1, f_2)$, let $\theta_{f_k, v}(s_0)$ be the unique angle so that $R_{f_k}(\theta_{f_k, v}(s_0))$ lands at $\phi_{f_k,v}(s_0)$.  

\vspace{5pt}
{\it Claim: For any $s_0\in \Theta(f_1, f_2)$, we have $\theta_{f_1,v}(s_0)=\theta_{f_2,v}(s_0)$. }
\begin{proof}
  By Proposition \ref{convergence-repelling}, for all $g\in \mathcal H$, both $R_g(\theta_{f_1, v}(s_0))$ and $R_g(\theta_{f_2,v}(s_0))$ land at $\phi_{g,v}(r_{v,s_0}(\Phi^{-1}(g)))\in \partial U_{g,v}$. Since $s_0\in \Theta(f_1, f_2)\subset\partial_v^{f_1} \mathbb D \cap \partial_v^{f_2} \mathbb D$,  by Lemma \ref{rays-number},
   the external ray $R_g(\theta_{f_k,v}(s_0))$ is the unique one landing at $\phi_{g,v}(r_{v,s_0}(\Phi^{-1}(g)))$ for $g\in \mathcal H$ sufficiently close to $f_k$, hence for all $g\in \mathcal H$ (by the stability).
This shows
  $\theta_{f_1,v}(s_0)=\theta_{f_2,v}(s_0)$.
\end{proof}
By the density of $\Theta(f_1, f_2)$ in $\partial \mathbb D$, for any $s_*\in \partial \mathbb D$, there are two sequences
 $\{s^+_n\}_{n\geq 1}, \{s^-_n\}_{n\geq 1} \subset \Theta(f_1, f_2)$ satisfying that
$$s^+_1>s^+_2>\cdots > s_*> \cdots> s^-_2>s^-_1, \  \lim_{n\rightarrow \infty} s^+_n=\lim_{n\rightarrow \infty} s^-_n=s_*.$$
By Claim, we have $\theta_{f_1,v}(s^+_n)=\theta_{f_2,v}(s^+_n)$ and $\theta_{f_1,v}(s^-_n)=\theta_{f_2,v}(s^-_n)$  for all $n$. By Lemma \ref{limb-angles}(1), we have 
$$\theta^+_{f_1,v}(s_*)=\lim_{n\rightarrow \infty} \theta_{f_1,v}(s^+_n)=\lim_{n\rightarrow \infty} \theta_{f_2,v}(s^+_n)=\theta^+_{f_2,v}(s_*);$$
$$\theta^-_{f_1,v}(s_*)=\lim_{n\rightarrow \infty} \theta_{f_1,v}(s^-_n)=\lim_{n\rightarrow \infty} \theta_{f_2,v}(s^-_n)=\theta^-_{f_2,v}(s_*).$$
Hence $\theta^+_{f_1,v}\equiv\theta^+_{f_2,v}$ and $\theta^-_{f_1,v}\equiv\theta^-_{f_2,v}$.

It remains to show $D_{v}^\partial(f_1)=D_{v}^\partial(f_2)$. By definition,  
$$D_{v}^\partial(f_k)=\sum_{q\in \partial\mathbb D\setminus \partial_v^{f_k}\mathbb D} \nu_{f_k, v}(q)\cdot q, \ k=1,2.$$

By the proven fact $\theta^{\pm}_{f_1,u}\equiv\theta^{\pm}_{f_2,u}$ for all $u\in V_{\rm p}$, we have 
$$\Delta_{f_1, \sigma(v)}(B_v^0(q))=\Delta_{f_2, \sigma(v)}(B_v^0(q)), \ \Delta_{f_1, v}(q)=\Delta_{f_2, v}(q).$$

By Lemma \ref{limb-angles}(2),    for  $q\in  \partial\mathbb D\setminus \partial_v^{f_1}\mathbb D=\partial\mathbb D\setminus \partial_v^{f_2}\mathbb D$, we get $\nu_{f_1, v}(q)=\nu_{f_2, v}(q)$.
This implies that $\nu_{1, v} \equiv\nu_{2, v}$. Hence $D_{v}^\partial(f_1)=D_{v}^\partial(f_2)$.
\end{proof}

By Proposition \ref{combinatorial-property-2}, the notations $\theta^+_{f,v}, \theta^-_{f,v}, D_{v}^\partial(f)$ and $\partial_v^f \mathbb D$ are dependent on $D$ and independent of  $f\in I_{\Phi}(D)$. For this reason, we denote them by 
\begin{equation}\label{notations}\theta^+_{D,v}, \theta^-_{D,v}, D_{v}^\partial(\mathcal H), \partial_v^D \mathbb{D}
	\end{equation}
respectively. 
They will be used in the following part of the paper.



\section{$\mathcal H$-admissible divisors}\label{h-adm-div}

In this section, we shall introduce the  {\it $\mathcal H$-admissible divisors}, which will play a crucial role in our further discussions.

\subsection{A priori characterization}

Let $D\in \partial_0^* {\rm Div}{(\mathbb D)}^S$.  We first give  an a  priori characterization of the maps in $I_{\Phi}(D)$:

\begin{pro} \label{para-dym}
Let $D=\big((B_v^0, D_v^{\partial})\big)_{v\in V}\in \partial_0^* {\rm Div}{(\mathbb D)}^S$. 
Let $f\in I_{\Phi}(D)$, $v\in V$ and $q\in {\rm supp}(D_v^{\partial})$.

1. If  $B_{\sigma^{k-1}(v)}^0\circ \cdots \circ B_{\sigma(v)}^0\circ B_{v}^0(q)\notin 
{\rm supp}(D_{\sigma^k(v)}^\partial(\mathcal H)),   \forall k\geq 1,$
 then the limb  $L_{U_{f,\sigma(v)}, f(\phi_{f,v}(q))}$ is trivial and   $\phi_{f,v}(q)$ is a critical point of $f$.
 
 2. If $v\in V_{\rm p} $ and $q$ is  $\widehat{B}^0_v$-periodic, then $\phi_{f,v}(q)$ is a parabolic point of $f$; 
\end{pro}

\begin{proof} 
 By Proposition \ref{divisor-correspondence},   there is at least one 
$f$-critical point in $L_{U_{f,v}, \phi_{f,v}(q)}$.




1. By the definition of $D_{v}^{\partial}(\mathcal H)$, all $f$-critical points in $\mathbb C\setminus U_{f, v}$ are contained in  $\bigcup_{y\in {\rm supp}(D_{v}^{\partial}(\mathcal H))}L_{U_{f,v}, \phi_{f,v}(y)}$. Translating the assumption  
to the dynamical plane of $f$, we have 
$$L_{f^k(U_{f,v}), f^k(\phi_{f,v}(q))}\cap {\rm Crit}(f)=\emptyset, \ k\geq 1.$$

By Proposition \ref{RY},  the limb $L_{f(U_{f,v}), f(\phi_{f,v}(q))}=L_{U_{f,\sigma(v)}, f(\phi_{f,v}(q))}$ is trivial.  Since $q\in {\rm supp}(D_v^{\partial})$, the limb $L_{U_{f,v}, \phi_{f,v}(q)}$ is non-trivial. Hence $\phi_{f,v}(q)$ is $f$-critical. 

2. Let $\ell\geq 1$ be the $\sigma$-period of $v$. If $q$ is  $\widehat{B}^0_v$-periodic, then it is repelling, with   $\widehat{B}^0_v$-period say $l\geq 1$.
 By Konig's Theorem,  there is a path $\gamma\subset \mathbb D$ converging to $q$ such that
$\gamma\subset(\widehat{B}^0_v)^{l}(\gamma)$. By the relation $f^{l\ell }\circ \phi_{f,v}=\phi_{f,v}\circ (\widehat{B}^0_v)^{l}$ on $\overline{\mathbb D}$, the point
$\phi_{f,v}(q)$ is $f$-periodic. Moreover, $\gamma_f:=\phi_{f,v}(\gamma)$ is a path in $U_{f,v}$ converging to 
$\phi_{f,v}(q)$, and  $\gamma_f\subset f^{l\ell}(\gamma_f)$. By the Snail Lemma \cite[Lemma 16.2]{M},  $\phi_{f,v}(q)$ is either  repelling or parabolic for $f$.

Suppose that $\phi_{f,v}(q)$ is $f$-repelling. 
Note that $R_{f}(\theta_{D,v}^+(q)), R_f(\theta_{D,v}^-(q))$ land at 
$\phi_{f,v}(q)$.
By  Lemma \ref{stability-e-r}, there exist a neighborhood $\mathcal U$ of $f$,  a continuous map $\zeta: \mathcal U\rightarrow \mathbb C$ 
 such that 
\begin{itemize}
\item$\zeta(g)$ is $g$-repelling for all $g\in \mathcal U$,  with $\zeta(f)=\phi_{f,v}(q)$. It follows from Proposition \ref{convergence-repelling} that $\zeta(g)=\phi_{g,v}(r_{v,q}(\Phi^{-1}(g)))$  for  $g\in \mathcal U\cap \mathcal H$.

\item for all $g\in \mathcal U\cap \mathcal C_d$, the   rays
$R_{g}(\theta_{D,v}^+(q)), R_g(\theta_{D,v}^-(q))$ land at  $\zeta(g)$, and  
$\overline{R_{g}(\theta_{D,v}^+(q))}\cup \overline{R_g(\theta_{D,v}^-(q))}$
 separates 
$U_{g,v}$ from $L_{U_{g,v}, \zeta(g)}\setminus \{\zeta(g)\}$.
\end{itemize}

 This implies the number of the $g$-critical points  in  $L_{U_{g,v}, \zeta(g)}$  remains unchanged, for all $g\in \mathcal U\cap \mathcal C_d$.
However this  contradicts that the number of the $f$-critical points  in $L_{U_{f,v}, \phi_{f,v}(q)}$
is strictly larger than that of the $g$-critical points in $L_{U_{g,v}, \zeta(g)}$ when $g\in \mathcal H\cap \mathcal U$.
\end{proof}

\begin{rmk}  
By constrast, if  $B_{\sigma^{k-1}(v)}^0\circ \cdots \circ B_{\sigma(v)}^0\circ B_{v}^0(q)\in {\rm supp}(D_{\sigma^k(v)}^\partial)\big(\subset {\rm supp}(D_{\sigma^k(v)}^\partial(\mathcal H))\big)$ 
 for some $k\geq 1$,
 it might happen that $\phi_{f,v}(q)$ is not $f$-critical.  
\end{rmk}

\subsection{$\mathcal H$-admissible divisor}


For each $v\in V$, let 
\begin{equation} \label{index-v}
V(v)=\{u\in V\setminus \{v\}; v=\sigma^{k}(u) \text{ for some } k\geq 1\}.
\end{equation}
For each $u\in V(v)$, let $r_v(u)\geq 1$ be the smallest integer with $\sigma^{r_v(u)}(u)=v$.

\begin{defi}  [Misiurewicz divisor] \label{Misiurewicz}
	Let $D=\big((B_v^0, D_v^{\partial})\big)_{v\in V}\in \partial_0^* {\rm Div}{(\mathbb D)}^S$. We call  $D$ a Misiurewicz divisor if 
	$$ {\rm supp}(D_{v}^\partial)\cap R_v(D)=\emptyset, \forall v\in V; \ {\rm supp}(D_{v}^\partial)\cap O_v(D)=\emptyset, \forall v\in V_{\rm p},$$  where
	$$R_v(D)=\bigcup_{u\in V(v)}B_{\sigma^{r_v(u)-1}(u)}^0\circ \cdots \circ  B_{u}^0({\rm supp}(D_u^\partial)), \  v\in V,$$
	$$O_v(D)=	\overline{\bigcup_{l\geq 1} \big(\widehat{B}^0_{v}\big)^l\big(R_v(D)\cup {\rm supp}(D_v^\partial)\big)}, \  v\in V_{\rm p}.$$
\end{defi}

\begin{lem} \label{m-w-lim}  Let $D=\big((B_v^0, D_v^{\partial})\big)_{v\in V}\in \partial_0^* {\rm Div}{(\mathbb D)}^S$ be  Misiurewicz.
	
	1.   For any $v\in V$ and any $k\geq 1$,   
	$$B_{\sigma^{k-1}(v)}^0\circ \cdots \circ B_{\sigma(v)}^0\circ B_{v}^0({\rm supp}(D_v^\partial))\bigcap {\rm supp}(D_{\sigma^k(v)}^\partial)=\emptyset.$$

2.   Let $v\in V_{\rm p}$ have period $\ell_v$. If one of  ${\rm supp}(D_{v}^\partial), \cdots, {\rm supp}(D_{\sigma^{\ell_v-1}(v)}^\partial)$ is non-empty, then $O_v(D)$
		is a totally disconnected subset of $\partial \mathbb D$. 
\end{lem}
\begin{proof} 1. It follows from the definition of the Misiurewicz divisor.
	
	2.  Clearly $O_v (D)\neq \emptyset$.  One may check that $ B_{v}^0(O_v (D))\subset  O_{\sigma(v)}(D)$,  and $\widehat{B}^0_{v}( O_v(D))\subset  O_v(D)$.
	Without loss of generality, we  assume ${\rm supp}(D_{v}^\partial)\neq\emptyset$.    Let $I\subset \partial\mathbb D $ be a connected  component of $ O_v(D)$. 
	If $I$ is  not  a singleton, then $I$ is  a closed interval in $\partial \mathbb D$.
	The assumption $D\in \partial_0^* {\rm Div}{(\mathbb D)}^S$ implies that 
	$\widehat{B}^0_v$ has degree at least two, hence  $\widehat{B}^0_v$ is expanding on $\partial \mathbb D$.
	So there is an integer $l\geq 1$ such that $(\widehat{B}_v^0)^{l}(I)=\partial \mathbb D$.
	However,  this    contradicts the fact  
	$$(\widehat{B}_v^0)^{l}(I)\subset  (\widehat{B}_v^0)^{l}(O_v(D))\subset O_v(D)\subset \partial \mathbb D\setminus {\rm supp}(D_{v}^\partial)\neq  \partial \mathbb D.$$
	\end{proof}


	
	

 \begin{defi}[$\mathcal H$-admissible divisor] \label{adm-divisor} Let $D=\big((B_v^0, D_v^{\partial})\big)_{v\in V}\in \partial_0^* {\rm Div}{(\mathbb D)}^S$.  We call $D$  generic, if for any $v\in V$ and any $q\in {\rm supp}(D_v^\partial)$, the multiplicity of $D_v^\partial$ at $q$ is exactly one.  The divisor $D$ is called $\mathcal H$-admissible if it 
 	is  a generic  Misiurewicz divisor, such that
  for any $v\in V$ and any $k\geq 1$, 
 	$$B_{\sigma^{k-1}(v)}^0\circ \cdots \circ B_{\sigma(v)}^0\circ B_{v}^0({\rm supp}(D_v^\partial))\bigcap 
 	{\rm supp}(D_{\sigma^k(v)}^\partial(\mathcal H))=\emptyset.$$
 \end{defi}

The maps in the impression of an $\mathcal H$-admissible divisor enjoy the following properties.

 \begin{pro} \label{h-admissible-characterization}   Let $D=\big((B_u^0, D_u^{\partial})\big)_{u\in V}\in \partial_0^* {\rm Div}{(\mathbb D)}^S$ be an  $\mathcal H$-admissible divisor.  Let  $v\in V$ and  $f\in I_{\Phi}(D)$.
 	
 	1.    We have
 	\begin{itemize}
 		\item  $\phi_{f, v}({\rm supp}(D_{v}^\partial))\subset {\rm Crit}(f)$;  moreover, for any   $q\in  {\rm supp}(D_{v}^{\partial})$, there are exactly two external rays landing at $\phi_{f,v}(q)$;
 		
 		\item  $f$ is a Misiurewicz map;
 		
 		\item  The critical set on $J(f)$ satisfies  
 		  $$J(f)\cap {\rm Crit}(f)=\bigsqcup_{u\in V} (\partial U_{f,u}\cap {\rm Crit}(f))=\bigsqcup_{u\in V}\phi_{f, u}({\rm supp}(D_{u}^\partial));$$
 		\end{itemize}

 
 	2. If $q\in  {\rm supp}(D_{v}^{\partial}(\mathcal H))\setminus {\rm supp}(D_{v}^{\partial})$,
 	 and 
 	 \begin{itemize}
 	 	\item   if  $B_{\sigma^{k-1}(v)}^0\circ \cdots \circ B_{\sigma(v)}^0\circ B_{v}^0(q)\notin 
 	 	{\rm supp}(D_{\sigma^k(v)}^\partial)$ for  $1\leq  k\leq l$, where $l\geq 1$ is an integer, then none of  $\phi_{f,v}(q), \cdots, f^{l}(\phi_{f,v}(q))$ is critical;

 	 	\item   if  $B_{\sigma^{k-1}(v)}^0\circ \cdots \circ B_{\sigma(v)}^0\circ B_{v}^0(q)\notin 
 	 	{\rm supp}(D_{\sigma^k(v)}^\partial)$ for all $ k\geq 1$, then $\phi_{f,v}(q)$ is a  pre-repelling point.
\end{itemize} 	 

 	\end{pro}
 Here,  a rational map $g$ of degree $\geq 2$ is {\it Misiurewicz} if 
 it is non-hyperbolic, without
 parabolic cycles, and  $\omega(c)\cap {\rm Crit}(g)=\emptyset$ for every critical point $c\in J(g)$, where $\omega(c)=\bigcap_{n\geq 0}\overline{\{g^k(c);k>n\}}$ is the $\omega$-limit set of $c$. 
 
 \begin{proof} 
 	1. By the definition of $\mathcal H$-admissibility and  Proposition \ref{para-dym}(1),  we have    $\phi_{f, v}({\rm supp}(D_{v}^\partial))\subset {\rm Crit}(f)$.
 In the following, we show 
   $${\rm Crit}(f)\setminus{\bigcup_{u\in V} U_{f,u}}=C:=\bigcup_{u\in V}\phi_{f, u}({\rm supp}(D_{u}^\partial)).$$
 	 To this end, 
 	 for each $c\in C$, let $Y_c=\{u\in V;  c\in \partial U_{f,u}\}$ and $m(c)=\# Y_c$.\vspace{2pt}
 	
 	{\it Fact 1:  If $m(c)\geq 2$, then for any $u_1, u_2\in Y_c$,  we have $\sigma(u_1)=\sigma(u_2)$.}\vspace{2pt}
 	
 	In fact, if $\sigma(u_1)\neq \sigma(u_2)$, then $U_{f,\sigma(u_2)}\subset L_{U_{f,\sigma(u_1)}, f(c)}$ and $U_{f,\sigma(u_1)}\subset L_{U_{f,\sigma(u_2)}, f(c)}$. Hence neither $L_{U_{f,\sigma(u_1)}, f(c)}$ nor $L_{U_{f,\sigma(u_2)}, f(c)}$  is trivial. However, this contradicts  Proposition \ref{para-dym}(1).\vspace{2pt}
 	
 		{\it Fact 2:  If $m(c)\geq 2$, then the local degree  ${\rm deg}(f, c)\geq m(c)+1$.}\vspace{2pt}
 		
 		\begin{figure}[h]
 			\begin{center}
 				\includegraphics[height=3.5cm]{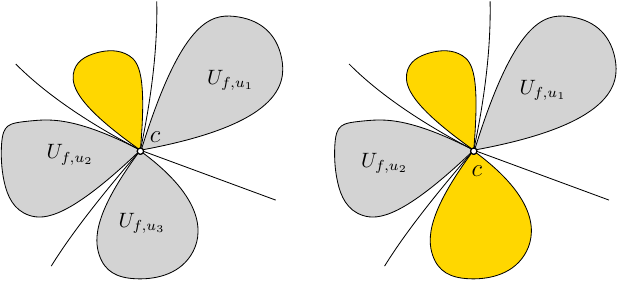}
 			\end{center}
 			\caption{$(m(c), {\rm deg}(f, c))=(3, 4)$ (left) and $(2,4)$ (right).} \label{fig:m-d}
 		\end{figure}  
 	
 	By Fact 1, there is $v_c\in V$ so that $\sigma(u)=v_c$ for all $u\in Y_c$. By Proposition \ref{para-dym}(1), the limb $L_{U_{f, v_c}, f(c)}$ is trivial, hence there is only one external ray landing at $f(c)$. It follows that
 	the local degree ${\rm deg}(f, c)$ equals the number of the external rays landing at $c$.
 		These rays separate the Fatou components $U_{f, u}, u\in Y_c$ as well as at least one extra  Fatou component caused by the critical point $c$ (this Fatou component can be specified as follows: note that $c\in \phi_{f, u}({\rm supp}(D_{u}^\partial))$ for some $u\in V$, there is an additional Fatou component  splitting from $U_{f, u}$).  Hence ${\rm deg}(f, c)\geq m(c)+1$.  See Figure \ref{fig:m-d}.
 	   
 	    By counting the number of critical points, we have 
 	 $$\sum_{c\in C}({\rm deg}(f, c)-1 )\geq \sum_{c\in C} m(c)\geq \sum_{u\in V}{\rm deg}(D_u^\partial).$$
 	 On the other hand, since $C\subset {\rm Crit}(f)\setminus{\bigcup_{u\in V} U_{f,u}}$, we have
 	 $$\sum_{c\in C}({\rm deg}(f, c)-1 )\leq d-1-\sum_{u\in V}({\rm deg}(B_u^0)-1)=\sum_{u\in V}{\rm deg}(D_u^\partial).$$
 	It follows  that ${\rm Crit}(f)\setminus{\bigcup_{u\in V} U_{f,u}}=C$ and  
 	 $$\sum_{c\in C}({\rm deg}(f, c)-1 )=\sum_{u\in V}{\rm deg}(D_u^\partial).$$

 	 	The equality implies  that $f$ has no parabolic point (otherwise there will be other critical points in parabolic basins). Transferring the  Misiurewicz property (see Definition \ref{Misiurewicz})  of   $D$ to $f$,  for any $c\in J(f)\cap {\rm Crit}(f)=C$, its $\omega$-limit set $\omega(c)$ is disjoint from  
 	 $J(f)\cap {\rm Crit}(f).$
 	 Hence $f$ is  Misiurewicz.
 	
 	\begin{figure}[h]
 		\begin{center}
 			\includegraphics[height=3.5cm]{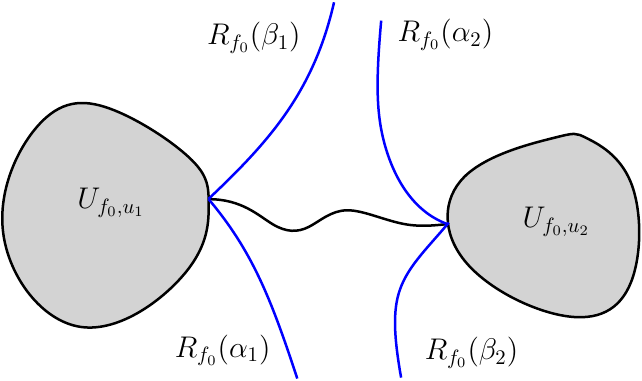}
 		\end{center}
 		\caption{Relative positions of $U_{f_0, u_1}$ and $ U_{f_0, u_2}$.} \label{fig: p-landing-relation}
 	\end{figure}
 
 In the following, we shall derive some further properties.\vspace{2pt}
 	
 		{\it Fact 3:  We actually have  $m(c)=1$ for all $c\in C$.}\vspace{2pt}
 		
 	Suppose $m(c)\geq 2$ for some $c\in C$. By Fact 1,  for any different $u_1, u_2\in Y_c$,  we have $\sigma(u_1)=\sigma(u_2)$.
 	Let $f_0$ be the center of $\mathcal H$. The property  $\sigma(u_1)=\sigma(u_2)$ implies that $\partial U_{f_0, u_1}\cap \partial U_{f_0, u_2}=\emptyset$ (otherwise the intersection point would be a critical point of $f_0$,  contradiction!). For each $j\in \{1,2\}$, suppose $R_{f_0}(\alpha_j), R_{f_0}(\beta_j)$ are external rays landing at the same point, separating $U_{f_0, u_j}$ for its limb containing $U_{f_0, u'}$, here $u'\in \{1,2\}\setminus\{u_j\}$. See Figure \ref{fig: p-landing-relation}.  By the   Misiurewicz property of $f$ and Lemma \ref{stability-e-r},  
 	 the external rays
 	$R_{f}(\alpha_j), R_{f}(\beta_j)$ land at the same point.   These four rays separate
 	$\overline{U_{f,u_1}}$ from  	$\overline{U_{f,u_2}}$.  
 	Hence $\partial U_{f,u_1}\cap \partial U_{f,u_2}=\emptyset$. This contradicts that $c\in \partial U_{f,u_1}\cap \partial U_{f,u_2}$.

By Fact 3, we have 
 	\begin{equation} \label{intersection-property}
 		\partial{U_{f,u_1}}\cap \partial{U_{f,u_2}}\cap {\rm Crit}(f)=\emptyset, \ \forall u_1, u_2\in V, u_1\neq u_2.
 	\end{equation}
 This equality implies that 
 \begin{equation}\label{critical-location}
  J(f)\cap {\rm Crit}(f)=\bigsqcup_{u\in V}\phi_{f, u}({\rm supp}(D_{u}^\partial)).
 \end{equation}

It follows that for any $q\in  {\rm supp}(D_{v}^{\partial})$, there are exactly two external rays landing at $\phi_{f,v}(q)$.

 		

 	2.  
 	Suppose  $q\in  {\rm supp}(D_{v}^{\partial}(\mathcal H))\setminus {\rm supp}(D_{v}^{\partial})$.  If  $B_{\sigma^{k-1}(v)}^0\circ \cdots \circ B_{\sigma(v)}^0\circ B_{v}^0(q)\notin 
 	{\rm supp}(D_{\sigma^k(v)}^\partial)$ for  $1\leq  k\leq l$,  then 
 	 $f^k(\phi_{f, v}(q))\notin \phi_{f, \sigma^k(v)}({\rm supp}(D_{\sigma^k(v)}^\partial))$  for  $1\leq  k\leq l$.  By (\ref{intersection-property}) and (\ref{critical-location}),
 	$f^k(\phi_{f, v}(q))$ is not a critical point.  In particular, if $l=\infty$, then the $f$-orbit of $\phi_{f,v}(q)$ meets no critical points. If this orbit is wandering, then the  angular difference of the external rays landing at $f^k(\phi_{f,v}(q))$ grows  exponentially as $k\rightarrow \infty$. This is a contradiction. Hence $\phi_{f,v}(q)$ is $f$-pre-periodic. Since $f$    is  Misiurewicz,   $\phi_{f,v}(q)$ is pre-repelling.
 	\end{proof}
 
 
 \begin{rmk}  1.    By contrast, if $q\in  {\rm supp}(D_{v}^{\partial}(\mathcal H))\setminus {\rm supp}(D_{v}^{\partial})$ satisfies that 	$B_{\sigma^{k-1}(v)}^0\circ \cdots \circ B_{\sigma(v)}^0\circ B_{v}^0(q)\in 
 	{\rm supp}(D_{\sigma^k(v)}^\partial)$ for some $ k\geq 1$,   it can happen that $\phi_{f,v}(q)$ is  not a pre-periodic point of $f$.
 	
 	2. If $D\in  \partial_0^* {\rm Div}{(\mathbb D)}^S$ is  Misiurewicz but not  $\mathcal H$-admissible, it can happen that  some $f\in I_{\Phi}(D)$  is not Misiurewicz.
 \end{rmk}

While generic Misiurewicz divisors can be readily constructed, determining their $\mathcal{H}$-admissibility remains non-trivial. The subsequent three subsections aim to establish an $\mathcal{H}$-admissibility criterion (Proposition  \ref{criterion-H-adm}), achieved through the following two steps:
  in Section \ref{divisors-p-d},   we construct a `stretching sequence' of divisors $E_n$'s  approaching $D$,   with the help of the higher dimensional  Intermediate Value Theorem (Theorem \ref{higher-ivt}); we then characterize in Section \ref{p-prescribed-dynamics} the limit dynamics  of the polynomial sequence associated with these divisors, inspired by the idea of  Petersen and Ryd \cite{PR}.  
  
  Building on these preparations, we establish in Section  \ref{S-criterion-h-a} an $\mathcal{H}$-admissibility criterion, which implies 
  the density of  $\mathcal H$-admissible divisors on $\partial_0^* {\rm Div}{(\mathbb D)}^S$.

\subsection{Divisors with prescribed dynamics} \label{divisors-p-d}
 
 Let $E=\big((B_u, E_u^{\partial})\big)_{u\in V}\in  \partial_0^* {\rm Div}{(\mathbb D)}^S$. As we did   before,    for each $u\in V$,  we shall view  $B_u$ as a map $B_u:\C_u\rightarrow \C_{\sigma(u)}$, and denote a set $X\subset \C_u$ by $X_u$. 
 
 
 \begin{lem} \label{family-koenig}  
 	Let $D=\big((B_u^0, D_u^{\partial})\big)_{u\in V}\in  \partial_0^* {\rm Div}{(\mathbb D)}^S$. Suppose that   $({B}_u^0)'(0)\neq 0$ for all $u\in V$.
 		There exist $\tau>0$ and for each $v\in V$,  a continuous  family of Koenigs maps 
 	$$\kappa_{E, v}: \mathbb D_v\rightarrow \mathbb C,  \ \text{ parameterized by }  E=((B_u, E_u^{\partial}))_{u\in V}\in \U_\tau(D),$$
 	satisfying that $\kappa_{E, v}(0)=0, \kappa'_{E, v}(0)=1$, and  $\kappa_{E, \sigma(v)}\circ  {B}_v= {B}'_v(0)\cdot \kappa_{E, v}$.
 \end{lem}
\begin{proof}  Take small $\tau>0$  so that   ${B}_u'(0)\neq 0$ for  all $E=((B_w, E_w^{\partial}))_{w\in V}\in \U_\tau(D)$ and all $u\in V$. 
	Take a representative $v\in V_{\rm p}$ in each $\sigma$-cycle.   Recall that $\ell_v\geq 1$ is the $\sigma$-period of $v$, and  
	$ \widehat{B}_{v}=B_{{\sigma^{\ell_v-1}(v)}}\circ \cdots  \circ B_{v}.$  By \cite[\S 8]{M}, there is a unique continuous  family of global Koenigs maps $\big\{\kappa_{E, v}: \mathbb D_v\rightarrow \mathbb C\big\}_{E\in \U_\tau(D)}$,  satisfying that 
	 $\kappa_{E, v}(0)=0$, $\kappa_{E, v}\circ \widehat{B}_v=\widehat{B}'_v(0)\cdot \kappa_{E, v}$, and $\kappa'_{E, v}(0)=1$. 
	 For each $u\in V(v)$ (given by \eqref{index-v}),  set
	 $$\kappa_{E, u}=
	 {\kappa_{E, v}\circ B_{\sigma^{r_v(u)-1}(u)}\circ \cdots \circ B_{u}}/
	 {( B_{\sigma^{r_v(u)-1}(u)}\circ \cdots \circ B_{u})'(0)}.$$
	 In this way, for each $u\in V$, we get a  required   family $\{\kappa_{E, u}\}_{E\in \U_\tau(D)}$.
	\end{proof}

Based on Lemma \ref{family-koenig}, for $E=\big((B_u, E_u^{\partial})\big)_{u\in V}\in \U_\tau(D)$ and $v\in V$,   define the potential function $G_{E, v}: \mathbb D_v\rightarrow  [-\infty, +\infty)$ by
$$G_{E, v}(z)=\log |\kappa_{E, v}(z)|,   \  z\in \mathbb D_v.$$
It satisfies that 
$G_{E, \sigma(v)}(B_v(z))=G_{E, v}(z)+ {\log|{B}'_v(0)|}, \ z\in \mathbb D_v$.

Let $D=\big((B_w^0, D_w^{\partial})\big)_{w\in V}\in \partial_0^* {\rm Div}{(\mathbb D)}^S$ and $u\in V$.  A point $q\in \partial \mathbb D_{u}$  is called  {\it $D$-preperiodic} if  there are two  minimal integers $m\geq 0, l\geq 1$  such that $v:=\sigma^m(u)\in V_{\rm p}$, and 
$$\begin{cases}
	(\widehat{B}_{v}^0)^{l}(q)=q,   &m=0, \\
	(\widehat{B}_{v}^0)^{l}(B_{\sigma^{m-1}(u)}^0\circ \cdots \circ B_{u}^0(q))= B_{\sigma^{m-1}(u)}^0\circ \cdots  \circ B_{u}^0(q), & m\geq 1.
		\end{cases}$$
In this case, we say $q$  is $(m, l)$-$D$-preperiodic.
If   $m\geq 1$, then $q$ is called 
{\it   strictly $D$-preperiodic}; otherwise,   $m=0$, and $q$ is called  
{\it   $D$-periodic}.   Note  that  $m\geq r_{v}(u)$, here recall that $r_{v}(u)$  is the minimal integer such that
$\sigma^{r_{v}(u)}(u)=v$.  

The divisor $D$ is called {\it  preperiodic (resp. strictly preperiodic)} if for any $u\in V$, any   $q\in {\rm supp}(D_u^\partial)$  is   {\it  $D$-preperiodic (resp. strictly $D$-preperiodic)}.

Let $\mathbf{Div}_{\rm spp}$ be the set of all strictly preperiodic divisors $D\in \partial_0^* {\rm Div}{(\mathbb D)}^S$.

\begin{lem} [\textbf{and definition for internal ray}] \label{internal-ray-model}
Let $D=\big((B_w^0, D_w^{\partial})\big)_{w\in V}\in \partial_0^* {\rm Div}{(\mathbb D)}^S$ and $u\in V$.  Suppose   $({B}_{\sigma^k(u)}^0)'(0)\neq 0$ for all $ k\geq 0$.  Let $q\in \partial \mathbb D_{u}$ be $(m, l)$-$D$-preperiodic, with orbit 
 	$$q_{0}:=q\overset{B^0_{u}}{\longrightarrow}  q_{1}  \overset{B^0_{\sigma(u)}}{\longrightarrow} q_{2}   \overset{B^0_{\sigma^2(u)}}{\longrightarrow} \cdots.$$
Then for $k\geq 0$,  there is a ray $\mathbf{R}^D_{\sigma^k(u),  q_k}:  (-\infty, +\infty)\rightarrow \mathbb D_{\sigma^k(u)}$
such that

\begin{itemize}
\item	$G_{D, \sigma^k(u)}(\mathbf{R}^D_{\sigma^k(u),  q_k}(t))=t, \ t\in \mathbb R$.

\item  $\lim_{t\rightarrow +\infty } \mathbf{R}^D_{\sigma^k(u),  q_k}(t)= q_k$ and 
$$\lim_{t\rightarrow -\infty } \mathbf{R}^D_{\sigma^k(u),  q_k}(t)=\begin{cases}  0, & \text{ if }  k\geq m,\\ 
	\zeta_k,  &  \text{ if }    0\leq k< m,
\end{cases}$$
where $\zeta_k$ is a point in $\mathbb D_{\sigma^k(u)}$ with $B_{\sigma^{m-1}(u)}^0\circ \cdots \circ B_{\sigma^k(u)}^0(\zeta_k)=0$.

\item  the dynamical  compatibility holds:
$$B^0_{\sigma^k(u)}(\mathbf{R}^D_{\sigma^k(u),  q_k}(t))=
\mathbf{R}^D_{\sigma^{k+1}(u),  q_{k+1}}\Big(t+\log|(B^0_{\sigma^k(u)})'(0)|\Big).$$

\item the ray  $\mathbf{R}^D_{u,  q}$ is  $D$-preperiodic: 
$$\begin{cases}
	(\widehat{B}_{v}^0)^{l}(\mathbf{R}^D_{u,  q})=\mathbf{R}^D_{u,  q},   &m=0, \\
	(\widehat{B}_{v}^0)^{l}(B_{\sigma^{m-1}(u)}^0\circ \cdots \circ B_{u}^0(\mathbf{R}^D_{u,  q}))= B_{\sigma^{m-1}(u)}^0\circ \cdots  \circ B_{u}^0(\mathbf{R}^D_{u,  q}), & m\geq 1.
\end{cases}$$
	\end{itemize}
\end{lem}

\begin{figure}[h]
	\begin{center}
		\includegraphics[height=4.8cm]{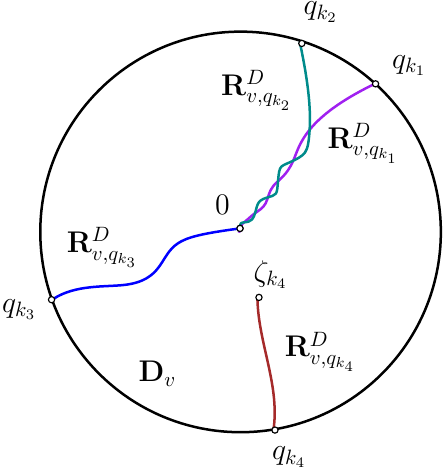}
	\end{center}
	\caption{Some internal rays in $\mathbb D_v$}
\end{figure}

The ray $\mathbf{R}^D_{\sigma^k(v),  q_k}$ in Lemma \ref{internal-ray-model} is called the {\it  internal ray}  landing at $q_k$.  As the proof indicates,  the ray  $\mathbf{R}^D_{\sigma^k(v),  q_k}$  with the required properties is not unique, because of the flexibility in the choice of the fundamental arc.   
 
\begin{proof}  We first assume $q\in \partial \mathbb D_{v}$ is $\widehat{B}_{v}^0$-periodic ($m=0$). We shall construct a  ray 
$\mathbf{R}^D_{v,  q}:  (-\infty, +\infty)\rightarrow \mathbb D_v$ with the first two  properties (for $k=0$), and  
	\begin{equation} \label{invariant-property}
(\widehat{B}_{v}^0)^{l}(\mathbf{R}^D_{v,  q}(t))= \mathbf{R}^D_{v,  q}(t+l \cdot {\log|({\widehat{B}_{v}^0})'(0)|}), \ t\in \mathbb R.
\end{equation}

Let ${\bf go}(0)\subset \mathbb D_v$ be the $\widehat{B}_{v}^0$-grand orbit of $0$. Identifying $z$ and $\widehat{B}_{v}^0(z)$ in $\mathbb D_v\setminus {\bf go}(0)$ yields a complex torus 
$T_{D, v}:=\mathbb C/(2\pi i \mathbb Z \oplus \log (\widehat{B}_{v}^0)'(0)\mathbb Z)$, and a projection $\Pi_{D,v}: \mathbb D_v\setminus {\bf go}(0)\rightarrow T_{D, v}$. 
The $\widehat{B}_{v}^0$-grand orbit  of 
$${\rm Crit} (B_{v}^0)\bigcup\bigcup_{w\in V(v)} \bigcup_{0\leq k<r_{v}(w)}  B_{\sigma^{r_v(w)-1}(w)}^0\circ \cdots  \circ B_{\sigma^k(w)}^0({\rm Crit} (B_{\sigma^k(w)}^0))$$
 projects to a finite  set $C_{D,v}\subset T_{D, v}$, by  $\Pi_{D,v}$.
Let $V$ be the linearization neighborhood of the $(\widehat{B}_{v}^0)^{l}$-fixed point $q$.
Take  $z_0\in V\cap \mathbb D$ so that $(\widehat{B}_{v}^0)^{l}(z_0)\in V$.
There is an arc  $\alpha \subset V$ connecting $z_0$ and   $(\widehat{B}_{v}^0)^{l}(z_0)$, avoiding $C_{D,v}$,  so that $\Pi_{D,v}(\alpha)$ is a loop in $T_{D,v}$.  We use $\alpha$  to generate a set 
\begin{equation} \label{arc-to-ray} \mathbf{R}^D_{v,  q}=\bigcup_{k\geq 0 }(\widehat{B}_{v}^0)^{kl}(\alpha) \bigcup 
\bigcup_{k\geq 1} (\widehat{B}_{v}^0)^{l }|_{V}^{-k}(\alpha).
\end{equation}
 By the choice of $\alpha$,  the set $\mathbf{R}^D_{v,  q}$ is a curve with endpoints $0$ and $q$,   and without self-intersections  because different points in $\alpha$ (can be made to) have different potentials.  It is clear that $\mathbf{R}^D_{v,  q}$  is a component of $\Pi_{D,v}^{-1}\Pi_{D,v}(\alpha)$  containing $\alpha$. Parameterize $\mathbf{R}^D_{v,  q}:  (-\infty, +\infty)\rightarrow \mathbb D_v$ so that 
$G_{D,  v}(\mathbf{R}^D_{v,  q}(t))=t$ for $t\in \mathbb R$,  then
 $\lim_{t\rightarrow +\infty } \mathbf{R}^D_{v,  q}(t)= q$ and $\lim_{t\rightarrow -\infty } \mathbf{R}^D_{v,  q}(t)= 0$.
    It satisfies  (\ref{invariant-property}), because  
    $$G_{D, v}((\widehat{B}^0_v)^{l}(z))=G_{D, v}(z)+l \cdot {\log|({\widehat{B}_{v}^0})'(0)|}, \ z\in \mathbb D_v.$$

Now  suppose $q\in \partial \mathbb D_{u}$ is $(m, l)$-$D$-preperiodic with $m\geq 1$.  Then 
$q_{m}$ is $\widehat{B}_{v}^0$-periodic, with period $l$.
By above construction, there is a periodic ray $\mathbf{R}^D_{v,  q_{m}}$  avoiding all critical orbits. So there is a unique component   $\mathbf{R}^D_{\sigma^{m-1}(u),  q_{m-1}}$ of   $({B}_{\sigma^{m-1}(u)}^0)^{-1}(\mathbf{R}^D_{v,  q_{m_q}})$ with one endpoint  $q_{m-1}$ and the other endpoint   $\zeta_{m-1}\in ({B}_{\sigma^{m-1}(u)}^0)^{-1}(0)$. 
Inductively, by taking preimages step by step, we get rays $\mathbf{R}^D_{\sigma^k(u),  q_k}, 0\leq k< m$, which can be parameterized  as required.

By taking images,  we get the rays $\mathbf{R}^D_{\sigma^k(u),  q_k}, k\geq m$.
	\end{proof}

\begin{lem} [Continuous motion of internal rays] \label{c-m-internal-ray-model}
	Let $D=\big((B_w^0, D_w^{\partial})\big)_{w\in V}\in \partial_0^* {\rm Div}{(\mathbb D)}^S$ and $u\in V$.  Let $q\in \partial \mathbb D_{u}$ be $(m, l)$-$D$-preperiodic, with orbit 
	$$q_{0}:=q\overset{B^0_{u}}{\longrightarrow}  q_{1}  \overset{B^0_{\sigma(u)}}{\longrightarrow} q_{2}   \overset{B^0_{\sigma^2(u)}}{\longrightarrow} \cdots.$$
	

	Let $\mathbf{R}^D_{u,  q}$ be the internal ray given by Lemma \ref{internal-ray-model}.
	Suppose   $({B}_{\sigma^k(u)}^0)'(0)\neq 0$ and  $q_k\notin {\rm supp}(D_{\sigma^k(u)}^\partial)$ for all $k\geq 0$.   

	Then there exist $\tau>0$,  $T\in \mathbb R\cup\{-\infty\}$, and  a continuous motion 
    $$H_{u,q}:  \U_{\tau}(D)\times \mathbf{R}^D_{u,  q}[T, +\infty) \rightarrow \mathbb D_u$$
    satisfying that 
    $$H_{u,q}(E, \mathbf{R}^D_{u,  q}(t))=\mathbf{R}^E_{u,  r_{u, q}(E)}(t), \   \forall (E, t)\in \U_{\tau}(D)\times [T, +\infty) ,$$
    where $ r_{u, q}(E)$ is given by Proposition \ref{extension-repelling}.
\end{lem}
	\begin{proof}  The proof is almost same as that of Lemma \ref{internal-ray-model},  we shall sketch the slight difference.  We mainly treat the case $m=0$ (hence $v:=\sigma^m(u)=u$),  for which  $q$ is $\widehat{B}_{v}^0$-periodic, with period $l$.  The strictly preperiodic case follows by taking preimages.
		
		The assumption  on $(q_k)_{k\geq 0}$  implies that
	$$Q:=\{q_0=q, q_{\ell_v}, \cdots, q_{\ell_v (l-1)}\}\subset \partial \mathbb D_v\setminus  F,$$
	where  $F=\bigcup_{0\leq j\leq l-1} (\widehat{B}^0_v)^{-j}(Z_v)$,  $Z_v$ is   given by Proposition \ref{extension-repelling}, and $\ell_v$ is the $\sigma$-period of $v$.
	Choose $r>0$ so that $Q\subset \partial \mathbb D_v\setminus  F^r$, where 
	$F^r=\bigcup_{\zeta\in F} \overline{\mathbb D(\zeta, r)}$. 
	
	Choose small $\tau>0$ so that   ${B}_{\sigma^k(v)}'(0)\neq 0$  for all  $E=((B_u, E_u^{\partial}))_{u\in V} \in \U_\tau(D)$ and all $0\leq k<\ell_v$. 
	By Propositions  \ref{convergence} and \ref{extension-repelling},  we may assume $\tau$ also satisfies that
	
	\begin{itemize}
		\item  $\psi_E:= \widehat{B}_v^{l}: \C\setminus F^r \rightarrow \C$
		defines a continuous family of holomorphic maps,  parameterized by   
		$E\in \U_\tau(D)$.
		
		\item for $w\in Q$, the   map $r_{v, w}: \U_\tau(D)\rightarrow \partial \mathbb D_v$ (given by Proposition \ref{extension-repelling}) are defined  with image $r_{v, w}(\U_\tau(D))
		\subset
		\partial \mathbb D_v\setminus F^r$.
		
		\item there is a neighborhood $V\subset \C\setminus F^r$ of $q$ so that it is also a linearization neighborhood of the  $\widehat{B}_v^{l}$-fixed point $r_{v, q}(E)$, with $\overline{V}\subset \widehat{B}_v^{l}(V)$, for all $E\in \U_\tau(D)$.
	\end{itemize}
	
	For  $E \in \U_\tau(D)$, let $A(E)=\bigcup_{k\geq 0}\widehat{B}_v^{-k}(F^r)$, and let ${\bf go}_E(z)\subset \mathbb D_v$ be the $\widehat{B}_{v}$-grand orbit of $z\in \mathbb D_v$. Then $\mathbb D_v\setminus(A(E)\cup {\bf go}_E(0))$ modulo  $\widehat{B}_{v}$-action gives a continuous family of 
	complex tori
	$$T_{E, v}=\mathbb C/(2\pi i \mathbb Z \oplus \log (\widehat{B}_{v}'(0))\mathbb Z), \  E\in \U_\tau(D).$$
	There is an induced continuous family of projections:
	$$\Pi_{E,v}: \begin{cases}  \mathbb D_v\setminus(A(E)\cup {\bf go}_E(0))\rightarrow T_{E, v}\\ 
		z\mapsto {\bf go}_E(z)
	\end{cases}.$$ 
	Note also that the $\widehat{B}_{v}$-grand orbit of the  critical set (associated with $v$)    in
	$\mathbb D_v\setminus(A(E)\cup {\bf go}_E(0))$ are mapped by $\Pi_{E,v}$  to a finite set $C_{E,v}\subset T_{E,v}$.  
	 The rest part of the  proof is   same as that of Lemma \ref{internal-ray-model}, 
we omit the details.
\end{proof}

\begin{rmk} \label{c-m-ray-end}   In Lemma \ref{c-m-internal-ray-model}, we  take $T=-\infty$ if  $q\in \partial \mathbb D_{u}$ is	$D$-periodic (note that $\mathbf{R}^D_{u,  q}(-\infty)=0$).
\end{rmk}
 
 For further discussion, we need a fact on limit behavior of critical values:
 
 \begin{lem} \label{cp-cv}  Let  $e\geq 1$ be an integer. Let $\{B_n\}_{n\geq 1}\subset {\rm Div}_e({\mathbb D})$,
 	such that	
 	$$B_n\rightarrow D=(B, D^\partial)\in {\rm Div}_e(\overline{\mathbb D})\   \text{ as  } n\rightarrow \infty.$$ 
 	Suppose $1\notin {\rm  supp}(D^\partial)$.   For each $q\in  {\rm supp}(D^\partial)$, there is a $B_{n}$-critical point    $c_{n}(q)\in\mathbb D$  with the following property	 :
 	$$\lim_{n\rightarrow \infty}c_{n}(q)=q, \  \lim_{n\rightarrow \infty} B_{n }(c_{n}(q))= B(q).$$
 \end{lem}
\begin{proof}
	By Lemma \ref{degenerate-0},  $B_n$ converges locally and 
	uniformly to $B$ in
	$\C\setminus {\rm  supp}(D^\partial)$.  Hence for any $\epsilon>0$, any $r>0$ with 
	$B(\mathbb D\cap \partial \mathbb D(q,r))\subset   \mathbb D(B(q), \epsilon)$, $B_n$ converges 
	uniformly to $B$ in
	$\C\setminus \bigcup_{q\in {\rm  supp}(D^\partial)} \mathbb D(q, r)$,  and 
	$$\alpha_n(q):=B_n(\mathbb D\cap \partial \mathbb D(q,r))\subset  \mathbb D(B(q), \epsilon), \ \ \forall  q\in {\rm  supp}(D^\partial)$$
	for large $n$.
	Note that for small $r$  and large $n$, $\mathbb D\cap   \mathbb D(q,r)$ contains exactly $\nu(q)\geq 1$ (here, $\nu(q)$ is the multiplicity of $D^\partial$ at $q$) critical points of $B_n$, so the preimage $B_n^{-1}(\alpha_n(q))\cap \overline{\mathbb D(q,r)}$ either contains a critical point, or consists of $\nu(q)$ arcs.
	In the former case, let $c_n(q)$ be this critical point. In the latter case,  the component $V_n(q)$ of $(\mathbb D\cap   \mathbb D(q,r))\setminus B_n^{-1}(\alpha_n(q))$ whose boundary contains  $\mathbb D\cap \partial \mathbb D(q,r)$, is simply connected and $\partial V_n(q)$ necessarily contains  a  component of $B_n^{-1}(\alpha_n(q))\cap \mathbb D(q,r)$.
	It follows that ${\rm deg}(B_n|_{V_n(q)})\geq 2$, $B_n(V_n(q))\subset  \mathbb D\cap  \mathbb D(B(q), \epsilon)$, and $V_n(q)$ contains at least one critical point.  We denote one of them    by $c_n(q)$. 
	
	\begin{figure}[h] \label{cv-cp}
		\begin{center}
			\includegraphics[height=3.5cm]{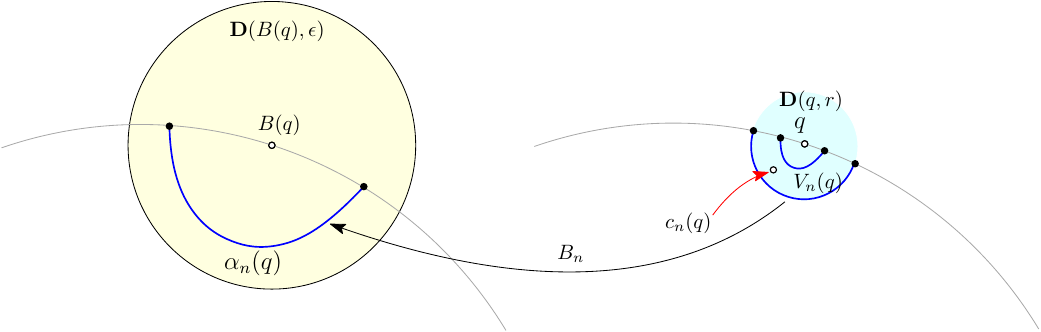}
		\end{center}
 	\caption{Behavior of the critical value.}
	\end{figure}

	In either case, we have $B_n(c_n(q))\in \mathbb D\cap {\mathbb D(B(q), \epsilon)}$ for large $n$.   
	\end{proof}

\begin{rmk} \label{cp-cv-rmk}  Write $D^\partial=\sum_{\zeta\in  {\rm supp}(D^\partial)} \nu(\zeta)\cdot \zeta$. 
	
 If some $q\in  {\rm supp}(D^\partial)$ has multiplicity $\nu(q)=1$, 	
then for large $n$,  there is a  unique $B_{n}$-critical point $c_{n}(q)$  close to $q$. Lemma \ref{cp-cv} implies that  $B_{n }(c_{n}(q))\rightarrow B(q)$ as $n\rightarrow \infty$.

	  If some $q\in  {\rm supp}(D^\partial)$ has multiplicity $\nu(q)\geq 2$, then  for large $n$,   there are  $\nu(q)$  $B_{n}$-critical points  $c_{n}^1(q), \cdots, c_{n}^{\nu(q)}(q)$, which are close to $q$. In general, it is not true that  for all $1\leq k\leq \nu(q)$, 
	 $$B_{n }(c^k_{n}(q))\rightarrow B(q) \ \text{ as  } n\rightarrow \infty.$$
The limit  behavior of some sequence $\{B_{n}(c^k_{n}(q))\}_{n\geq 1}$  might be out of control.
	\end{rmk}


\begin{pro}  \label{c-prescribed-divisors}
	 Let  $D=\big((B_u^0, D_u^{\partial})\big)_{u\in V}\in \mathbf{Div}_{\rm spp}$
 be a generic  Misiurewicz  divisor. Suppose 
 	$({B}_u^0)'(0)\neq 0$ for all $u\in V$.

 For each $v\in V$ and $q\in {\rm supp}(D_v^\partial)$, let $\mathbf{R}^D_{\sigma(v),  q_1}$ (here $q_1=B^0_{v}(q)$) be the internal ray  given by Lemma \ref{internal-ray-model}, and let $H_{\sigma(v),  q_1}:  \U_{\tau}(D)\times  \mathbf{R}^D_{\sigma(v),  q_1}[T, +\infty) \rightarrow \mathbb D_{\sigma(v)}$ be the continuous motion given by Lemma \ref{c-m-internal-ray-model} (here $\tau, T$ are independent of $v\in V$ and $q\in {\rm supp}(D_v^\partial)$).
 
 Then there is a sequence of divisors  $E_n=(B_{n,u})_{u\in V}\in {\rm Div}(\mathbb D)^S$  with $E_n\rightarrow D$ as $n\rightarrow \infty$, such that for all $v\in V$ and all $q\in {\rm supp}(D_v^\partial)$, 
 $$B_{n, v}(c_{n, v}(q))\in  H_{\sigma(v),  q_1}(E_n,  \mathbf{R}^D_{\sigma(v),  q_1}[T, +\infty)),$$
 here $c_{n, v}(q)$ is the critical point of $B_{n,v}$ such that $c_{n, v}(q)\rightarrow q$ as $n\rightarrow \infty$.
 \end{pro}

To prove Proposition \ref{c-prescribed-divisors}, 
we need the following higher dimensional  Intermediate Value Theorem, first proven by Miranda, also  known as Bolzano-Poincar\'e-Miranda theorem.

\begin{thm}  [Miranda \cite{Mir, V}] \label{higher-ivt} 
	Let $F=(f_1, \cdots, f_n): [0,1]^n\rightarrow \mathbb R^n$ be continuous, such that
	for any $1\leq k\leq n$, 
	$$f_k \geq 0  \text{ on  }\{\mathbf x\in [0,1]^n; x_k=0\}; \ f_k \leq 0  \text{ on  }\{\mathbf x\in [0,1]^n; x_k=1\}.$$
		Then there is at least one  $\mathbf x\in  [0,1]^n$  such that $F(\mathbf x)=\mathbf 0$.
  \end{thm}

\begin{proof}
[Proof of Proposition \ref{c-prescribed-divisors}.]
For $\rho\in(0, \tau ]$ 
we define a subset  of $\U_\rho(D)$:
$$\U_\rho^{*}(D)=\Big\{E=\big((B_u, E^\partial_u)\big)_{u\in V}\in {\U_{\rho}(D)}; \ {\rm deg}(E^\partial_u)= \ {\rm deg}(D^\partial_u), \forall u\in V
\Big\}.$$
It's worth observing  that if $E\in \U_\tau^{*}(D)$ takes the form
  $E=\big((B_u^0, E^\partial_u)\big)_{u\in V}$, then for any    $v\in V$ and $q\in {\rm supp}(D_v^\partial)$,  we have $H_{\sigma(v),  q_1}(E,  \cdot )=  {\rm id}$.  That is  
$$H_{\sigma(v),  q_1}(E,  \mathbf{R}^D_{\sigma(v),  q_1} (t))=  \mathbf{R}^D_{\sigma(v),  q_1}(t), \ \forall t\in [T, +\infty). $$

Fix any  small $\delta\in (0, \tau)$.  For  $v\in V$ and $q\in {\rm supp}(D_v^\partial)$, let $\partial \mathbb D(q, \delta)\cap \partial \mathbb D_v=\{a(q), b(q)\}$. We assume $\delta$ is small so that  the  arcs $\{\overline{ \mathbb D(q, \delta)}\cap \partial \mathbb D_v\}_{q\in {\rm supp}(D_v^\partial)}$ are pairwise disjoint. Let 
$$\rho_{v, q}=\min\Big\{|q_1-B_v^0(a(q))|, |q_1-B_v^0(b(q))|\Big\}.$$

Let $\mathcal A=\{E=\big((B_u^0, E^\partial_u)\big)_{u\in V}\in  \overline{\U_\delta^{*}(D)}\}$. 
For any $E\in  \mathcal A$, there is $\delta_E>0$ so that for any $E'\in \U_{\delta_E}(E)$, the landing point $r_{\sigma(v),q_1}(E')$ 
of the  truncated ray $H_{\sigma(v),  q_1}(E',  \mathbf{R}^D_{\sigma(v),  q_1} [T, +\infty))=  \mathbf{R}^{E'}_{\sigma(v),  r_{\sigma(v),q_1}(E')} [T, +\infty)$ 
satisfies 
$$r_{\sigma(v),q_1}(E')\in \mathbb D(q_1, \rho_{v, q}/2)\cap \partial \mathbb D_{\sigma(v)}.$$
Note  that $\big\{ \U_{\delta_E}(E); E\in \mathcal A \big\}$ is an open covering of the compact set $\mathcal A$.    
Let  $\epsilon_0\in (0, \delta]$  be its
 Lebesgue number.

 For  $v\in V$ and $q\in {\rm supp}(D_v^\partial)$, let $\gamma_{v,q}: (0,1)\rightarrow \mathbb D_v$ be a Jordan arc in  the
 $\epsilon_0$-neighborhood of $\overline{ \mathbb D(q, \delta)}\cap \partial \mathbb D_v$,  
with endpoints
$\gamma_{v,q}(0)=a(q), \gamma_{v,q}(1)=b(q)$. We further assume $\epsilon_0$ is small so that the arcs $\{\gamma_{v,q}\}_{q\in {\rm supp}(D_v^\partial)}$ are pairwise disjoint.
 Let $\mathcal D$ be the set of   $E=(B_u)_{u\in V} \in {\rm Div}(\mathbb D)^S$ of the form
$${\rm div}(B_u)={\rm div}(B_u^0)+\sum_{q\in {\rm supp}(D_u^\partial)}1\cdot \zeta_{u,q},  \ \zeta_{u,q} \in \gamma_{u,q}, \forall  q\in {\rm supp}(D_u^\partial),  \forall   u\in V.$$
Clearly $\mathcal D$ can be parameterized by $\prod_{u\in V}\prod_{q\in {\rm supp}(D_u^\partial)} \gamma_{u,q}$,   homeomorphic to $(0,1)^m$ with  $m=\sum_{u\in V}{\rm deg}(D_u^\partial)$. So  
$\overline{\mathcal D}$ is  homeomorphic to $[0,1]^m$.

Define a map 
$$F=\big(f_{v,q}\big)_{ v\in V,  q\in {\rm supp}(D_v^\partial) }: \mathcal D \rightarrow \mathbb R^m$$
so that $f_{v,q}(E)$ is the signed distance from  $B_v(c_{E,v}(q))$  to the  truncated ray:
$$f_{v,q}(E)={\rm sgn}(\zeta_{v,q})\cdot {\rm dist}(B_v(c_{E,v}(q)), H_{\sigma(v),  q_1}({E},  \mathbf{R}^{D}_{\sigma(v),  q_1}[T, +\infty)),$$
here   $ E=(B_u)_{u\in V}\in \mathcal D$, 
and $c_{E, v}(q)$ is the unique critical point of $B_{v}$  close to $q$; to define ${\rm sgn}(\zeta_{v,q})$,  note that for any $E\in \mathcal D$, 
we can extend   $\mathbf{R}^{E}_{\sigma(v),  r_{\sigma(v),q_1}(E)} [T, +\infty)$ to a longer arc $\alpha^{E}_{\sigma(v), q_1}\subset  \mathbb D_{\sigma(v)}$, whose two endpoints 	are on $\partial\mathbb D_{\sigma(v)}$, and $\alpha^{E}_{\sigma(v), q_1}$ is continuous with respect to $E\in \mathcal D$. The  two  components of  $\mathbb D_{\sigma(v)}\setminus \alpha^{E}_{\sigma(v), q_1}$ are denoted by 
 $W^{E,+}_{\sigma(v),  q_1}$ and $W^{E,-}_{\sigma(v),  q_1}$. Set
$${\rm sgn}(\zeta_{v,q}) =\begin{cases}  +1, & \text{ if }  B_v(c_{E, v}(q))\in W^{E,+}_{\sigma(v),  q_1},\\ 
		0,  &  \text{ if }  B_v(c_{E, v}(q))\in H_{\sigma(v),  q_1}({E},  \mathbf{R}^D_{\sigma(v),  q_1}[T, +\infty)), \\
	 -1, & \text{ if }  B_v(c_{E, v}(q))\in W^{E,-}_{\sigma(v),  q_1}.
\end{cases}$$
In this way, we get a continuous map $f_{v,q}: \mathcal D \rightarrow \mathbb R$.

\begin{figure}[h]  
	\begin{center}
		\includegraphics[height=4.8cm]{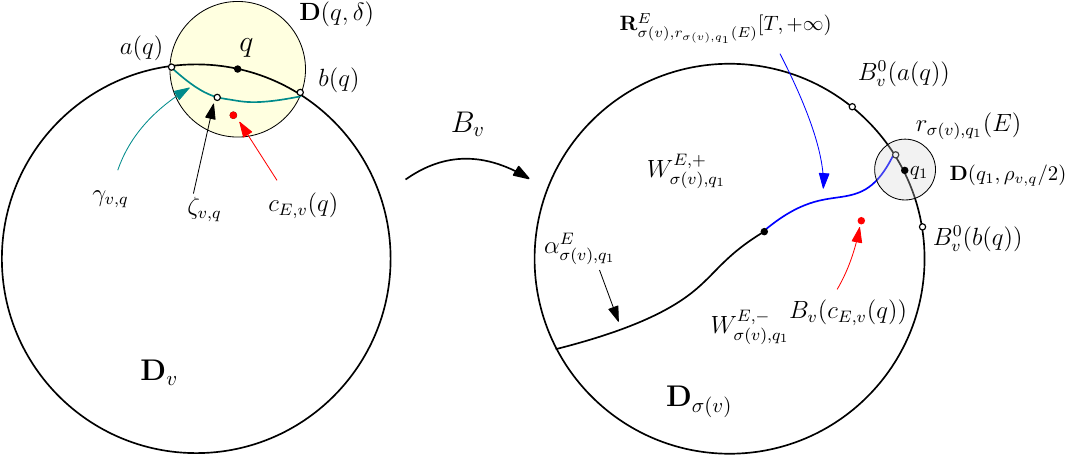}
	\end{center}
	\caption{The map $B_v$ when $E\in \mathcal D$.}
\end{figure}

It is important to note that when $\zeta_{v, q}$ moves along $\gamma_{v,q}$ to one endpoint, say $a(q)$ (resp. $b(q)$),  by Lemma \ref{cp-cv} and Remark \ref{cp-cv-rmk},  we have $B_v(c_{E,v}(q))\rightarrow B_v^0(a(q))$ (resp. $B_v^0(b(q))$). It follows that $F$ extends to a continuous map in $\overline{\mathcal D}$.
On the other hand, the landing point $r_{\sigma(v),q_1}({E})$ of the truncated ray 
$\mathbf{R}^{E}_{\sigma(v),  r_{\sigma(v),q_1}({E})}[T, +\infty)$ remains in $\mathbb D(q_1, \rho_{v, q}/2)\cap \partial \mathbb D_{\sigma(v)}$, which is bounded away from $B_v^0(a(q))$ and $B_v^0(b(q))$.  
By  this,  we find that 
$$f_{v,q}(E)|_{\zeta_{v,q}=a(q)}\cdot f_{v,q}(E)|_{\zeta_{v,q}=b(q)}<0, \ \forall q\in {\rm supp}(D_v^\partial),  \forall  v\in V.$$
In particular,  $\mathbf{0}\notin F(\partial \mathcal D)$.
By Theorem \ref{higher-ivt}, there is $E\in \mathcal D$ with $F(E)=\mathbf{0}$.

The above argument shows that for any given $\delta\in (0, \tau)$, we can find $E\in \mathcal D\subset \N_{\delta+\epsilon_0}(D)\subset \N_{2\delta}(D)$ with  required properties. Take a sequence $\delta_n\rightarrow 0$, we get a sequence of divisors $E_n$.
The proof is completed.
\end{proof}

\subsection{Polynomials with prescribed dynamics} \label{p-prescribed-dynamics}
 Let  $D=\big((B_u^0, D_u^{\partial})\big)_{u\in V}\in \mathbf{Div}_{\rm spp}$ be
  a generic  Misiurewicz  divisor. Suppose 
 	$({B}_u^0)'(0)\neq 0$ for all $u\in V$.

Let $f_n=\Phi(E_n)$, where $E_n=(B_{n,u})_{u\in V}$'s are given by Proposition \ref{c-prescribed-divisors}. 
Clearly,  $f_n'(v(f_n))\neq 0$ for all $v\in V$.
By Fact \ref{fact-def},  the accumulation set of the sequence $\{f_n\}_{n\geq 1}$ is contained in  $I_{\Phi}(D)$.  By choosing a subsequence, we assume $f_n\rightarrow f\in  I_{\Phi}(D)$ as $n\rightarrow \infty$.  For   $g=f_n$ or $f$,   by \cite[\S 8]{M},   there exist  Koenigs  linearization maps $\{\kappa_{g,v}: U_{g,v}\rightarrow \mathbb C\}_{v\in V}$ satisfying that 
$$\kappa_{g, \sigma(v)}(g(z))=g'(v(g)) \cdot \kappa_{g, v}(z), \ z\in U_{g,v}, \ v\in V,$$
with normalizations $\kappa_{g, v}(v(g))=0$ and $\kappa_{g, v}'(v(g))=1$. 


Define the potential function 
$$G_{g,v}(z)=\log |\kappa_{g, v}(z)|,  \  z\in U_{g,v}, \ v\in V.$$
It satisfies that  $G_{g, \sigma(v)}(g(z))=G_{g,v}(z)+{\log|g'(v(g))|}, v\in V$.

Recall that $\phi_{f_n,v}: \mathbb D_v\rightarrow U_{f_n,v}$(see Section \ref{imp-divisor}) and $\phi_{f,v}: \mathbb D_v\rightarrow U_{f,v}$ (given by Proposition \ref{combinatorial-property0}) are  conformal maps.
The linearization maps $\kappa_{g,v}: U_{g,v}\rightarrow \mathbb C$ and $\kappa_{E(g), v}:  \mathbb D_v\rightarrow\mathbb C$ satisfy that $\kappa_{g,v}\circ \phi_{g,v}=\phi_{g,v}'(0)\cdot \kappa_{E(g), v}$, where $E(f_n)=E_n$ and $E(f)=D$.
It follows that the potential functions $G_{g,v}: U_{g,v}\rightarrow\mathbb R\cup\{-\infty\}$ and $G_{E(g), v}: \mathbb D_v\rightarrow\mathbb R\cup\{-\infty\}$  satisfy
$$G_{g,v}(\phi_{g,v}(w))=G_{E(g),v}(w)+\log|\phi_{g,v}'(0)|, \  w\in \mathbb D_v.$$  

Let $v\in V$, $q\in {\rm supp}(D_v^\partial)$. Suppose $q$ is $(m, l)$-$D$-preperiodic, with orbit 
	$$q_{0}:=q\overset{B^0_{v}}{\longrightarrow}  q_{1}  \overset{B^0_{\sigma(v)}}{\longrightarrow} q_{2}   \overset{B^0_{\sigma^2(v)}}{\longrightarrow} \cdots.$$
For any $k\geq 1$,   by Lemma \ref{c-m-internal-ray-model}, there exist (possibly truncated) internal rays $\mathbf{R}^{E}_{\sigma^k(v),  r_{\sigma^k(v),q_k}(E)} [T, +\infty)$, continuous in $E\in \{D, E_n; n\geq 1\}$,  
 with landing points  $r_{\sigma^k(v),q_k}(E)$ (given by Proposition \ref{extension-repelling}) satisfying that $r_{\sigma^k(v),q_k}(D)=q_k$ .
The $\phi_{f_n,\sigma^k(v)}$-image of these rays are   dynamical internal rays in $U_{f_n, \sigma^k(v)}$: 
$$R^{f_n}_{\sigma^k(v), \xi_{n,k}}[T_{n,k}, +\infty)=\phi_{f_n,\sigma^k(v)}(\mathbf{R}^{E_n}_{\sigma^k(v),  r_{\sigma^k(v),q_k}(E_n)} [T, +\infty)),$$
parameterized in the way that $G_{f_n, \sigma^k(v)}(R_{\sigma^k(v), \xi_{n,k}}^{f_n}(t))=t$, 
where 
\begin{equation}\label{xi-nk}
	\xi_{n,k}=\phi_{f_n, \sigma^k(v)}(r_{\sigma^k(v),q_k}(E_n)), \  T_{n,k}=T+\log|\phi'_{f_n, \sigma^k(v)}(0)|. 
	\end{equation}

Recall from Proposition \ref{c-prescribed-divisors} that $c_{n, v}(q)\in \mathbb D_{v}$ is the critical point of $B_{n,v}$ such that $c_{n, v}(q)\rightarrow q$ as $n\rightarrow \infty$.
 Then $c^*_{n,v}(q):=\phi_{f_n, v}(c_{n, v}(q))\in U_{f_n, v}$ is $f_n$-critical, and $e^*_{n,v}(q):=f_n(c^*_{n,v}(q))\in U_{f_n, \sigma(v)}$ is an $f_n$-critical value.  By the property of $E_n$ in  Proposition \ref{c-prescribed-divisors}, we have   $e^*_{n,v}(q)\in R_{\sigma(v), \xi_{n,1}}^{f_n}[T_{n,1}, +\infty)$.
 Note that $T_{n,1}(q):=G_{f_n, \sigma(v)}(e^*_{n,v}(q))\rightarrow +\infty$ as $n\rightarrow +\infty$.

Consider the sequence of the truncated rays 
$$R_n:=R_{\sigma(v), \xi_{n,1}}^{f_n}[T_{n,1}, T_{n,1}(q)], \ n\geq 1.$$
By choosing a subsequence, we assume they have a limit $R$ in Hausdorff topology. 
Clearly, $R$ contains the $f$-critical value $e^*_{v}(q)=\lim_{n\rightarrow +\infty}e^*_{n,v}(q)$.

Before further discussion, we need the notion of the {\it band neighborhood} of an internal ray.  We first consider the periodic case.  Suppose $u\in V_{\rm p}$ and  $q'\in \partial \mathbb D_u$ is $D$-periodic. We recall that  the internal ray $\mathbf{R}^{E}_{u,  r_{u, q'}(E)}$ (see Lemmas \ref{internal-ray-model} and \ref{c-m-internal-ray-model})
 is defined so that its projection  $\Pi_{E, u}(\mathbf{R}^{E}_{u,  r_{u, q'}(E)})$ is a loop in  the torus $T_{E,v}$, without intersection with $C_{E, u}$.  Let $\mathbf{A}_{u,q'}^E$ be an annular neighborhood of $\Pi_{E, u}(\mathbf{R}^{E}_{u,  r_{u, q'}(E)})$  avoiding $C_{E,v}$.
 The  {\it band neighborhood}  $\mathbf{B}_{u,q'}^E$ of $\mathbf{R}^{E}_{u,  r_{u, q'}(E)}$ is the 
 component of $\Pi_{E,u}^{-1}(\mathbf{A}_{u,q'}^E)$ containing $\mathbf{R}^{E}_{u,  r_{u, q'}(E)}$.
Similarly, we can  define the band neighborhoods in the dynamical plane:  the set  $\mathcal B_{u,q'}^{g}=\phi_{g, u}(\mathbf{B}_{u,q'}^E)$ is called the {\it band neighborhood} of $R^g_{u, \phi_{g,u}(r_{u, q'}(E))}=\phi_{g, u}(\mathbf{R}^{E}_{u,  r_{u, q'}(E)})$, where $g=\Phi(E)$ (if $E\in  {\rm Div}{(\mathbb D)}^S\cap \U_{\tau}(D)$) or $g\in I_{\Phi}(E)$  (if $E\in  \partial {\rm Div}{(\mathbb D)}^S\cap \U_{\tau}(D)$) . 
 
 \begin{figure}[h]  
 	\begin{center}
 		\includegraphics[height=5.5cm]{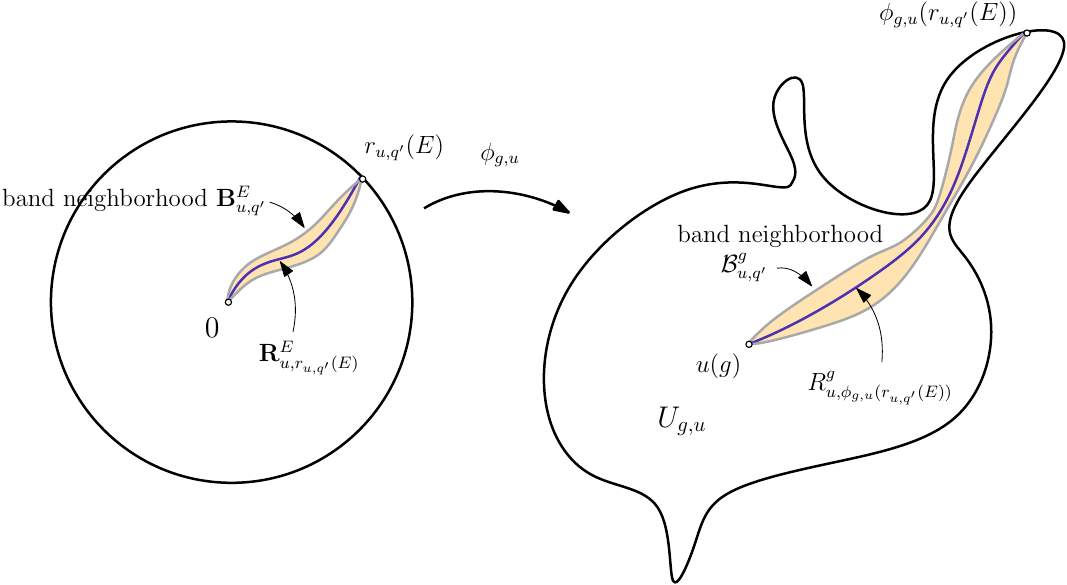}
 	\end{center}
 	\caption{The band neighborhoods}
 \end{figure}

To continue  our discussion, note that $q_m\in\partial\mathbb D_{\sigma^m(v)}$ is $D$-periodic.
Let $\mathbf{B}^{E_n}_{v,q, m}$ (resp. $\mathbf{B}^{D}_{v,q, m}$) be the band neighborhood of $\mathbf{R}^{E_n}_{\sigma^{m}(v),  r_{\sigma^{m}(v), q_{m}}(E_n)}$ (resp. $\mathbf{R}^{D}_{\sigma^{m}(v),   q_{m}}$). 
By iterations and taking preimages, for any $k\geq 1$,   the band neighborhood  $\mathbf{B}^{E_n}_{v,q, k}$ (resp. $\mathbf{B}^{D}_{v,q, k}$)  of  $\mathbf{R}^{E_n}_{\sigma^{k}(v),  r_{\sigma^{k}(v), q_{k}}(E_n)}$ (resp. $\mathbf{R}^{D}_{\sigma^{k}(v),   q_{k}}$) are defined. 
Then for $k\geq 1$,
$$\mathcal B_{v,q,k}^{f_n}=\phi_{f_n, \sigma^k(v)}(\mathbf{B}^{E_n}_{v,q,k})\  \text{ and } \ \mathcal B_{v,q,k}^{f}=\phi_{f, \sigma^k(v)}(\mathbf{B}^{D}_{v,q,k})$$
  are band neighborhoods of $R^{f_n}_{\sigma^{k}(v), \xi_{n,k}}$ and $R^f_{\sigma^{k}(v), \xi_{k}}=\phi_{f, \sigma^k(v)}(\mathbf{R}^{D}_{\sigma^{k}(v),   q_{k}})$, respectively,   where  $\xi_k=\phi_{f, \sigma^k(v)}(q_k)$.
 They satisfy $g(\mathcal B_{v,q,k}^{g})=\mathcal B_{v,q,k+1}^{g}$  and $\mathcal B_{v,q, m+l\ell_v}^{g}=\mathcal B_{v,q, m}^{g}$, for $g=f_n$ or $f$. 
  By passing to a subsequence, we assume $\mathcal B_{v,q, k}^{f_n}$ has a Hausdorff limit $\mathcal B_{v,q,k}$ as $n\rightarrow \infty$.  Clearly, $\mathcal B_{v,q,k}$ is  compact.

\begin{lem}  \label{hausdorff-limit} 1. The Hausdorff limit $R=\lim_{n\rightarrow +\infty}R_n$
satisfies that 
$$R\cap U_{f, \sigma(v)}=R_{\sigma(v), \xi_{1}}^{f}[T_1, +\infty), \ \ R\subset R_{\sigma(v), \xi_{1}}^{f}[T_1, +\infty)\cup L_{U_{f, \sigma(v)}, \xi_1},$$
where $T_1=T+\log|\phi'_{f, \sigma(v)}(0)|$.

2.  For any $k\geq 1$, the Hausdorff limit $\mathcal B_{v,q, k}$
 satisfies that 
 $$\mathcal B_{v,q, k}\cap U_{f, \sigma^k(v)}=\overline{\mathcal B_{v,q, k}^{f}}\setminus\{\xi_k\}, \  \mathcal B_{v,q, k} \subset  \overline{\mathcal B_{v,q,k}^{f}}\cup L_{U_{f, \sigma^k(v)}, \xi_k}.$$
\end{lem}

\begin{proof}  By Lemma  \ref{c-m-internal-ray-model} and the  local and uniform convergence  $\phi_{f_n, \sigma(v)}\rightarrow \phi_{f, \sigma(v)}$ (see Proposition \ref{combinatorial-property0}),  we have 
	$$T_{n,1}=T+\log|\phi'_{f_n, \sigma(v)}(0)|\rightarrow T+\log|\phi'_{f, \sigma(v)}(0)|=T_1, $$ 
 and	  for any $t\in (T_1, +\infty)$, 
	 $$R_{\sigma(v), \xi_{n,1}}^{f_n}[T_{n,1}, t]\rightarrow R_{\sigma(v), \xi_{1}}^{f}[T_{1}, t]   \ \text{ (in Hausdorff topology)},$$
	 as $n\rightarrow \infty$. Hence $R\cap U_{f, \sigma(v)}=R_{\sigma(v), \xi_{1}}^{f}[T_1, +\infty)$.
	
	To locate $R\setminus U_{f, \sigma(v)}$, for 
  any $\epsilon>0$,   take two $D$-pre-periodic points $a, b \in \mathbb D(q_1,  \epsilon)\cap \partial_{\sigma(v)}^D \mathbb D$  (defined by \eqref{notations})
   in two sides of $q_1$, so that orbit 
  $$s_{0}:=s\overset{B^0_{v}}{\longrightarrow}  s_{1}  \overset{B^0_{\sigma(v)}}{\longrightarrow} s_{2}   \overset{B^0_{\sigma^2(v)}}{\longrightarrow} \cdots, \ \ s\in\{a,b\}$$
  satisfies that  $s_k\notin {\rm supp}(D_{\sigma^k(u)}^\partial)$ for all $k\geq 0$.     By Proposition \ref{combinatorial-property0},   $\phi_{f, \sigma(v)}(a)$ and $ \phi_{f, \sigma(v)}(b)$ are $f$-pre-repelling. By Lemma  \ref{c-m-internal-ray-model}, there is an  (possibly truncated) internal ray $\mathbf{R}^{E}_{\sigma(v),  r_{\sigma(v), s}(E)}$ ($s=a,b; E=E_n, D$), landing at $r_{\sigma(v), s}(E)\in \partial \mathbb D_{\sigma(v)}$, continuous in $E\in \{D, E_n; n\geq 1\}$.  These rays together with  $ \partial \mathbb D_{\sigma(v)}$ and a suitable equipotential curve $G_{E,\sigma(v)}^{-1}(L)$ (where $L=L(\epsilon)$    satisfies $L(\epsilon)\rightarrow +\infty$ as $\epsilon\rightarrow 0$) bound a disk $V^\epsilon_{E, \sigma(v)}$ with ${\rm diam}({V^\epsilon_{E, \sigma(v)}})\leq 2\epsilon$.
 
  \begin{figure}[h]  
 	\begin{center}
 		\includegraphics[height=5.3cm]{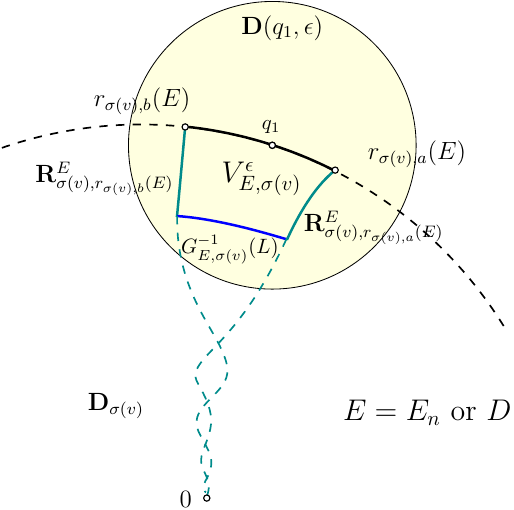}
 	\end{center}
 	\caption{The set $V_{E, \sigma(v)}^\epsilon$.}
 \end{figure}

By   the Hausdorff continuity:
$\partial V^\epsilon_{E_n, \sigma(v)}\rightarrow \partial V^\epsilon_{D, \sigma(v)}$ and $\phi_{f_n, \sigma(v)}(\partial V^\epsilon_{E_n, \sigma(v)}\cap  \mathbb D_{\sigma(v)})\rightarrow \phi_{f, \sigma(v)}(\partial V^\epsilon_{D, \sigma(v)}\cap  \mathbb D_{\sigma(v)})$\footnote{The reason of the Hausdorff continuity is: for $s\in\{a,b\}$, since  $\phi_{f, \sigma(v)}(s)$ is $f$-pre-repelling, 
	we have  $\lim_n \phi_{f_n, \sigma(v)}(r_{\sigma(v), s}(E_n))=\phi_{f, \sigma(v)}(s)$(by Proposition \ref{convergence-repelling}), this allows us to establish the Hausdorff continuity of truncated dynamical rays $\phi_{f_n, \sigma(v)}(\mathbf{R}^{E_n}_{\sigma(v),  r_{\sigma(v), s}(E_n)}[L, \infty))\rightarrow  \phi_{f, \sigma(v)}(\mathbf{R}^{D}_{\sigma(v), s}[L, \infty))$.
} as $n\rightarrow \infty$,   we get
$$R\subset R_{\sigma(v), \xi_{1}}^{f}[T_1, +\infty)\bigcup \phi_{f,\sigma(v)}(\overline{V^\epsilon_{D, \sigma(v)}}) \bigcup  
\bigsqcup_{\zeta\in   \alpha_\epsilon}L_{U_{f, \sigma(v)}, \zeta},$$
where $\alpha_\epsilon=\phi_{f,\sigma(v)}(\partial V^\epsilon_{D, \sigma(v)}\cap \partial \mathbb D_{\sigma(v)})$.
Clearly 
$\bigcap_{\epsilon>0}\phi_{f,\sigma(v)}(\overline{V^\epsilon_{D, \sigma(v)}})=\{\xi_1\}=\bigcap_{\epsilon>0} \alpha_\epsilon$. Since $\epsilon>0$ is arbitrary,    we have  
$R\subset R_{\sigma(v), \xi_{1}}^{f}[T_1, +\infty)\cup L_{U_{f, \sigma(v)}, \xi_1}$.

The second statement follows from the same argument. 
	\end{proof}

Recall that $q\in {\rm supp}(D_v^\partial)$ is $(m, l)$-$D$-preperiodic.  
Let 
$$\mu=\log |(\widehat{B}_{\sigma^m(v)}^0)'(0)|, \ \mu_n=\log |\widehat{B}_{n, \sigma^m(v)}'(0)|, n\geq 1.$$
Note that $\sigma^m(v)\in V_{\rm p}$, $\mu,\mu_n<0$  and $\lim_{n}\mu_n=\mu$.
By Lemma \ref{hausdorff-limit}, we can study the part $R\cap L_{U_{f, \sigma(v)}, \xi_1}$ from the dynamical viewpoint:
 \begin{pro} \label{inverseb} Let $T_*>0$ be a large number. 
For any $a\in R\cap L_{U_{f, \sigma(v)}, \xi_1}$, there exist a point $b=b(a)\in R^f_{\sigma^{m}(v), \xi_{m}}[T_*+l\mu, T_*]$, 
and a sequence of holomorphic maps $\big\{h_k^a:\mathcal B_{v,q,m}^f\rightarrow L_{U_{f, \sigma^{k+1}(v)}, \xi_{k+1}}\big\}_{k\geq 0}$, 
satisfying that:

(1). $h^a_k(b)=f^k(a)$ and $h^a_{k+1}=f\circ h^a_k$, for $k\geq 0$;

(2). $h^a_k\circ f^{l\ell_v}=f^{l\ell_v}\circ h^a_k$, for $k\geq m-1$;
 
(3). $\{h^a_k\}_{k\geq 0}$ are either all constant or all univalent. In the latter case,  for different $k$ and $s$, we have 
$$h^a_k(\mathcal B^f_{v,q,m})\cap h^a_s(\mathcal B^f_{v,q,m})\neq \emptyset \Longleftrightarrow k,s\geq m-1 \text{ and } k-s\in l\ell_v\mathbb Z.$$
 \end{pro}

 \begin{proof} Let $\{a_n\}_{n\geq 1}$ be a sequence of points with   $a_n\in R_n$ and $a_n\rightarrow a$.
Then we have $G_{f_n, \sigma(v)}(a_n)\rightarrow +\infty$ as $n\rightarrow +\infty$.

 Note that $f_n^{m-1}(a_n)\in R^{f_n}_{\sigma^{m}(v), \xi_{n,m}}$ 
 for each $n$.  For large $n$, there is an integer $t_n\geq 1$ so that $f_n^{t_nl\ell_v}(f_n^{m-1}(a_n))\in  R^{f_n}_{\sigma^{m}(v), \xi_{n,m}}[T_*+l\mu_n,T_*]$.
The fact  $G_{f_n, \sigma(v)}(a_n)\rightarrow +\infty$ implies that $t_n\rightarrow +\infty$.
 Let $s_n=t_nl\ell_v+m-1$ and $b_n:=f_{n}^{s_n}(a_n)$.
By the Hausdorff continuity of the truncated rays, we have  $R^{f_n}_{\sigma^{m}(v), \xi_{n,m}}[T_*+l\mu_n,T_*]\rightarrow R^{f}_{\sigma^{m}(v), \xi_{m}}[T_*+l\mu,T_*]$ as $n\rightarrow +\infty$. 
 By passing to a subsequence, we assume
 $b_n\rightarrow b\in  R^{f}_{\sigma^{m}(v), \xi_{m}}[T_*+l\mu,T_*]$.

\textbf{Claim}: We have the Carath\'eodory kernel convergence
$$(\mathcal B_{v,q,m}^{f_n}, b_n)\rightarrow (\mathcal B_{v,q,m}^{f}, b) \ \text{ as } n\rightarrow +\infty.$$

To see this,   for  any $L\in \mathbb R$, we have the Hausdorff convergence 
$$\partial(\mathcal B_{v,q,m}^{f_n}\cap G_{f_n, \sigma^{m}(v)}^{-1}(-\infty, L))\rightarrow 
\partial(\mathcal B_{v,q,m}^{f}\cap G_{f, \sigma^{m}(v)}^{-1}(-\infty, L)) \ \text{ as } n\rightarrow +\infty.$$
This verifies the first property of the kernel convergence (see Section \ref{s-ckc}).

Clearly $\sigma^m(v)(f_n)\rightarrow \sigma^m(v)(f)$.
Take $w\in \partial\mathcal B_{v,q,m}^{f}\setminus\{\sigma^m(v)(f)\}$.

If $w\in U_{f, \sigma^m(v)}$, let $t_w=G_{f, \sigma^m(v)}(w)$.
By the Hausdorff convergence $\partial\mathcal B_{v,q,m}^{f_n}\cap G_{f_n, \sigma^m(v)}^{-1}(t_w)\rightarrow \partial\mathcal B_{v,q,m}^{f}\cap G_{f, \sigma^m(v)}^{-1}(t_w)$, 
there exists $w_n\in  \partial\mathcal B_{v,q,m}^{f_n}\cap  G_{f_n, \sigma^m(v)}^{-1}(t_w)$ for each $n$, such that $\lim_{n\rightarrow +\infty}w_n= w$.

If $w=\xi_m\in \partial U_{f, \sigma^m(v)}$, for any integer $k>0$, there is  
$\zeta_k\in  \partial\mathcal B_{v,q,m}^{f}\cap U_{f, \sigma^m(v)}$ with $|\zeta_k-\xi_m|\leq 1/k$. For this $\zeta_k$, by the former case, there exist  an integer $n_k>0$ and
$w'_n\in  \partial\mathcal B_{v,q,m}^{f_n}\cap  G_{f_n, \sigma^m(v)}^{-1}(t_{k})$ (here  $t_k=G_{f, \sigma^m(v)}(\zeta_k)$) such that 
$|w'_n-\zeta_k|\leq 2/k$ for all $n\geq n_k$. We assume $n_k$ is increasing in $k$, and set 
$w_n=w'_n$ for $n_k\leq n< n_{k+1}$. Then we get a sequence $\{w_n\}_{n\geq1}$ with $w_n\in\partial\mathcal B_{v,q,m}^{f_n}$ and $\lim_{n\rightarrow +\infty}w_n=\xi_m$. This verifies the second property of the kernel convergence (see Section \ref{s-ckc}).  The proof of the claim is completed.

Now we trace the $f_n$-orbit  of $a_n$:
$$a_n\mapsto f_n(a_n)\mapsto \cdots \mapsto b_n=f^{s_n}(a_n)\in \mathcal B_{v,q,m}^{f_n}.$$
Note that  $a_n\in  \mathcal B_{v,q,1}^{f_n}$ and $f_n^{s_n}: \mathcal B_{v,q,1}^{f_n}\rightarrow \mathcal B_{v,q,m}^{f_n}$ is conformal.
 Let $h_{n,0}: \mathcal B_{v,q,m}^{f_n}\rightarrow \mathcal B_{v,q,1}^{f_n}$ be the inverse of $f_n^{s_n}|_{\mathcal B_{v,q,1}^{f_n}}$.
For any $L\in \mathbb R$, we have 
\begin{equation} \label{equi-potential-eq}
	h_{n,0}(\mathcal B_{v,q,m}^{f_n}\cap G_{f_n, \sigma^m(v)}^{-1}(L))=\mathcal B_{v,q,1}^{f_n} \cap G_{f_n, \sigma(v)}^{-1}(L_n),
\end{equation}
where 
$$L_n=L-t_n\mu_nl-\log|(f_n^{m-1})'(f_n(v(f_n)))|. $$
Note that 
$L_n\rightarrow +\infty $ as $n\rightarrow +\infty$. 

 For each $0\leq k\leq s_n$, define $h_{n,k}=f_{n}^k\circ h_{n,0}$. Clearly $h_{n,k}$ is univalent. 
 The maps $\{h_{n,k}\}_k$ satisfy the following properties:

(a). $h_{n,k}(b_n)=f_{n}^{k}(a_n)$ for $0\leq k\leq s_n$;
		
(b). $h_{n,k}\circ f_{n}^{l\ell_v}=f_{n}^{l\ell_v}\circ h_{n,k}$ for all $m-1\leq k\leq s_n-l\ell_v$;
		
(c).  Fix any $L\in \mathbb R$, there is an integer $n_L$ such that for $n\geq n_L$,   the sets 
$$h_{n,k}(\mathcal B_{v,q,m}^{f_n}\cap G_{f_n, \sigma^m(v)}^{-1}[L, +\infty)), \ 0\leq k\leq m+l \ell_v -2$$
are pairwise disjoint.


 Fix any integer $k\geq 0$, and consider the sequence of  univalent maps  $\{h_{n,k}:\mathcal B_{v,q, m}^{f_n}\rightarrow \mathbb C\}_{n\geq L_k}$, where $L_k>0$ is an integer depending on $k$ so that $h_{n,k}$ is defined for $n\geq L_k$. By Montel's Theorem  and the kernel convergence $(\mathcal B_{v,q,m}^{f_n}, b_n)\rightarrow (\mathcal B_{v,q,m}^{f}, b)$, and by choosing a subsequence if necessary, the sequence $\{h_{n,k}\}_{n\geq L_k}$ has a limit $h_k^a: \mathcal B_{v,q,m}^{f}\rightarrow \mathbb C$, with $h_k^a(b)=f^k(a)$. 
 By Hurwitz's theorem, $h_k^a$'s are either all 
 constant or all univalent.  By the fact that $f_{n}^{l\ell_v}\rightarrow f^{l\ell_v}$ as $n\rightarrow\infty$, we get  $h_k^a\circ f^{l\ell_v}=f^{l\ell_v}\circ h_k^a$ for $k\geq m-1$.
 
  By equality (\ref{equi-potential-eq}) and Lemma  \ref{hausdorff-limit},  
  $$h_k^a(\mathcal B_{v,q,m}^f)\subset \mathcal B_{v,q,k+1}\cap (\mathbb C\setminus U_{f, \sigma^{k+1}(v)})\subset  L_{U_{f, \sigma^{k+1}(v)}, \xi_{k+1}}$$


If $h_k^a$'s are all univalent, by  property (c) and Rouch\'e's Theorem, the sets 
$$h_{k}^a(\mathcal B_{v,q,m}^{f}), \ 0\leq k\leq m+l \ell_v -2$$
are pairwise disjoint. 
 This completes the proof of the Proposition.
\end{proof}

Note that $\xi_m\in \partial U_{f, \sigma^m(v)}$ is the landing point of the periodic ray  $R^f_{\sigma^{m}(v), \xi_{m}}=f^{l\ell_v}(R^f_{\sigma^{m}(v),\xi_{m}})$.
	By the Snail Lemma \cite[Lemma 16.2]{M}, $\xi_m$ is either repelling or parabolic. 

\begin{lem} \label{repelling-landing} If $\xi_m$ is repelling, then $\xi_0=\phi_{f,v}(q)$ is a critical point, 
 $\xi_1=f(\phi_{f,v}(q))$ is a critical value, and
$$R=R_{\sigma(v), \xi_{1}}^{f}[T_1, +\infty)\cup\{\xi_1\}.$$ 
\end{lem}
\begin{proof} It follows  from the continuity of  the internal rays (analogous to  Lemma \ref{stability-e-r}). 
\end{proof}


\begin{lem} \label{limit-Julia}   Suppose $\xi_m$ is parabolic. Let $a\in R\cap L_{U_{f, \sigma(v)}, \xi_1}$. Then 
$$a\in R\cap J(f) \Longleftrightarrow f^{m-1}(a)=f^{m-1+l\ell_v}(a) \Longleftrightarrow h_0^a \text{ is constant}.$$
\end{lem}
\begin{proof} 
	Let $\{h_k^a\}_{k\geq 0}$ be given 
by Proposition \ref{inverseb}.  Note that $h^a_{m-1}(b)=f^{m-1}(a)$ and $h^a_{m-1}(f^{l\ell_v}(b))=f^{m-1+l\ell_v}(a)$.
Since $b\neq f^{l\ell_v}(b)$, we see that $f^{m-1}(a)=f^{m-1+l\ell_v}(a)$ if and only if 
$h_{m-1}^a$ (hence $h_0^a$) is constant. 

Suppose $a\in R\cap J(f)$. If $h_0^a$ is not constant, then $h^a_0(\mathcal B^f_{v,q,m})$ is contained in Fatou set $F(f)$. This contradicts that $a\in h^a_0(\mathcal B^f_{v,q,m})\cap J(f)$. Hence  $h_0^a$ is constant, and consequently, $f^{m-1}(a)=f^{m-1+l\ell_v}(a)$.

Conversely, assume $f^{m-1}(a)=f^{m-1+l\ell_v}(a)$. If $a\in F(f)$, then $f^{m-1}(a)$ is either attracting or  Siegel. The former is impossible, because attracting cycle is stable under perturbation. Hence $f^{m-1}(a)$ is a Siegel point. Take a nearby point $a'\in R\setminus\{a\}$ in the Fatou component containing $a$. 
 Then we have  $f^{m-1}(a')\neq f^{m-1+l\ell_v}(a')$, implying that $h_0^{a'}$ is not constant.
Note that $f^{l\ell_v}\circ h_{m-1}^{a'}=h_{m-1}^{a'}\circ f^{l\ell_v}$. Hence the $f^{l\ell_v}$-orbit of $f^{m-1}(a')$ is not recurrent. This is again a contradiction. 
 \end{proof}

\begin{lem}  \label{limit-Fatou} Suppose $\xi_m$ is parabolic. 
 If $a\in R \cap L_{U_{f, \sigma(v)}, \xi_1}\cap F(f)$, then $f^{m-1}(a)$ is contained in an immediate parabolic basin of exact period $l\ell_v$.
\end{lem}
\begin{proof} By Lemma \ref{limit-Julia},   $h_0^a$ is not constant. So $h_{m-1}^a(\mathcal B_{v,q,m}^{f})$ is contained in a Fatou component $U$. Note that $h_{m-1}^a(\mathcal B_{v,q,m}^{f})=h_{m-1}^a\circ f^{l\ell_v}(\mathcal B_{v,q,m}^{f})=f^{l\ell_v}\circ h_{m-1}^a(\mathcal B_{v,q,m}^{f})$. We see that $f^{l\ell_v}(U)=U$. So the exact period $s$ of $U$ is a divisor of $l\ell_v$. Note that $U$ can not be attracting (since it is stable under perturbation) or Siegel (since any point on $h_{m-1}^a(R^f_{\sigma^{m}(v), \xi_{m}})$ is not recurrent). Hence $U$ is parabolic.

In the following, we show  $s=l\ell_v$.
 Let $P$ be an attracting petal in $U$, and $\alpha: P\rightarrow \mathbb H_r:=\{z; {\rm Re}(z)>0\}$ be the Fatou coordinate, conformal and  satisfying that $\alpha(f^s(z))=\alpha(z)+1$ for all $z\in P$.

 We assume by contradiction that $s<l\ell_v$. Let $\gamma: \mathbb R\rightarrow U$ be a curve defined by  $\gamma(t)=h_{m-1}^a(R^f_{\sigma^{m}(v), \xi_{m}}(t)), t\in \mathbb R$.  Clearly $f^{l\ell_v}(\gamma(t))=\gamma(t+l\mu)$. There is $t_0\in \mathbb R$ so that $\gamma(-\infty, t_0]\subset P$.  Let 
$$z_0=\gamma(t_0),\  w_0=f^{l\ell_v}(z_0), \ \zeta_0=f^{s}(z_0), \ \xi_0=f^{l\ell_v}(\zeta_0)=f^s(w_0).$$
It is clear that $z_0, w_0\in \gamma[t_0+l\mu, t_0]$, and $\zeta_0, \xi_0\in f^s(\gamma[t_0+l\mu, t_0])$. 
Note that the $\alpha$-images of $z_0, w_0, \zeta_0, \xi_0$ are 
$$\alpha(z_0), \alpha(w_0)=\alpha(z_0)+l\ell_v/s,  \alpha(\zeta_0)=\alpha(z_0)+1, \alpha(\xi_0)=\alpha(z_0)+l\ell_v/s+1.$$ 
They have the same imaginary parts,  with  ${\rm Re}(\alpha(z_0))< {\rm Re} (\alpha(\zeta_0))< {\rm Re} (\alpha(w_0))<{\rm Re}(\alpha(\xi_0))$.
Since $\alpha(\gamma[t_0+l\mu, t_0])$ connects two endpoints $\alpha(z_0)$ and $\alpha(w_0)$, 
$\alpha(f^s(\gamma[t_0+l\mu, t_0]))$ connects two endpoints $\alpha(\zeta_0)$ and $\alpha(\xi_0)$, also note that $ \alpha(f^s(\gamma[t_0+l\mu, t_0]))$ is the translation of $\alpha(\gamma[t_0+l\mu, t_0])$ by 1, 
we have 
$$\alpha(\gamma[t_0+l\mu, t_0])\cap \alpha(f^s(\gamma[t_0+l\mu, t_0]))\neq \emptyset.$$
It follows that $\gamma[t_0+l\mu, t_0]\cap f^s(\gamma[t_0+l\mu, t_0])\neq \emptyset$.
  However this contradicts the proven property (Proposition \ref{inverseb} (3)) that the curves  $f^{j}(\gamma), \ 0\leq j< l\ell_v$ are pairwise disjoint.
\end{proof}

\subsection{A criterion for $\mathcal H$-admissible divisors} \label{S-criterion-h-a}

In this part, we establish a criterion for $\mathcal H$-admissible divisors. 


Let  $D=\big((B_u^0, D_u^{\partial})\big)_{u\in V}\in \mathbf{Div}_{\rm spp}$  be  
a generic  Misiurewicz  divisor.  
 Let 
$$\mathbf{I}=\{(v,q); v\in V, q\in {\rm supp}(D_v^\partial)\}.$$

Recall that for each $v\in V_{\rm p}$, $\ell_v$ is the $\sigma$-period of $v$. Now for each 
$v\in V_{\rm np}$,   there is $m>0$ so that $\sigma^m(v)$ is $\sigma$-periodic, let $\ell_v$ be the  $\sigma$-period of $\sigma^m(v)$. In this way, for all $v\in V$, the number $\ell_v$ is defined. We further assume $V_{\rm p}$ can be decomposed into 
$s\geq 1$ $\sigma$-cycles, and the $k$-th one form the index set $V^k_{\rm p}, 1\leq k\leq s$.

 To emphasize the dependence of the  objects (periods and preperiods, sets, etc) with respect to the pair 
 $(v,q)\in \mathbf{I}$, we will add a subscript to them. For $v\in V, q\in {\rm supp}(D_v^\partial)$, suppose $q$ is $(m_{v,q}, l_{v,q})$-$D$-preperiodic. Let 
$$L_k (D)=\{l_{v,q};   \text{the $\sigma$-orbit of $v$ meets } V_{\rm p}^k,   q\in {\rm supp}(D_v^\partial)\}, \ 1\leq k\leq s.$$

  Recall that $f_0$ is the center of the hyperbolic component $\mathcal H$.
For each $k$,  $v\in V^k_{\rm p}$ and each $f$-periodic point $a\in \partial U_{f_0, v}$ so that the limb $L_{U_{f_0, v}, a}$ is non-trivial, let $m_{v,a}\geq 1$ be the minimal integer with $f_0^{m_{v,a}\ell_v}(a)=a$.  Note that the exact $f_0$-period  of $a$ is, in general, a divisor of $m_{v,a}\ell_v$. 

Let $Q_k(f_0)$ be the set of all such $m_{v,a}$ for $v$ ranging over $V^k_{\rm p}$, and $a\in \partial U_{f_0, v}$  ranging over  $f$-periodic points with non-trivial limbs $L_{U_{f_0, v}, a}$. Clearly, $Q_k(f_0)$  is a finite set.



\begin{pro}  \label{criterion-H-adm} Let  $D=\big((B_u^0, D_u^{\partial})\big)_{u\in V}\in \mathbf{Div}_{\rm spp}$ be  
	a generic  Misiurewicz  divisor. Suppose that
 	$({B}_u^0)'(0)\neq 0$ for all $u\in V$. Assume further

(a).  $L_k(D)\cap Q_k(f_0)=\emptyset$,  for all $1\leq k\leq s$;

(b).  for any different $(v_1, q_1), (v_2, q_2)\in \mathbf{I}$,  one has $l_{v_1,q_1}\ell_{v_1}\neq l_{v_2,q_2}\ell_{v_2}$.


Then $D$ is  $\mathcal H$-admissible.
\end{pro}

\begin{proof} Let $f\in  I_{\Phi}(D)$ be the limit of the maps $f_n=\Phi(E_n)$ given by  Section \ref{p-prescribed-dynamics}. For each $(v,q)\in \mathbf{I}$, recall that $\xi_{m_{v,q}}$ is the landing point of the dynamical internal ray $R^{f}_{\sigma^{m_{v,q}}(v), \xi_{{m_{v,q}}}}=\phi_{f,\sigma^{m_{v,q}}(v)}(\mathbf{R}^{D}_{\sigma^{m_{v,q}}(v),   q_{m_{v,q}}})$.

The index set $\mathbf{I}$ admits a  decomposition $\mathbf{I}_{r}(f)\sqcup \mathbf{I}_{p}(f)$, where
\bess 
&\mathbf{I}_{r}(f)=\{(v,q)\in \mathbf I; \ \xi_{m_{v,q}} \text{ is } f\text{-repelling}\},&\\
&\mathbf{I}_{p}(f)=\{(v,q)\in \mathbf I; \ \xi_{m_{v,q}} \text{ is } f\text{-parabolic}\}.&
\eess

{\it Claim 1: For any $(v, q)\in \mathbf{I}_r(f)$, the exact $f$-period of $\xi_{m_{v,q}}$ is $l_{v,q}\ell_v$. Moreover, the limb $L_{U_{f, \sigma^{m_{v,q}}(v)}, \xi_{m_{v,q}}}$ is trivial.}

Clearly,  $l_{v,q}$ is the minimal $l\geq 1$ with the property $f^{l\ell_v}(\xi_{m_{v,q}})=\xi_{m_{v,q}}$.  To show the exact $f$-period of $\xi_{m_{v,q}}$ is $l_{v,q}\ell_v$, it suffices to show that the limb $L_{U_{f, \sigma^{m_{v,q}}(v)}, \xi_{m_{v,q}}}$ is trivial. If not,  by Proposition \ref{holo-hyperbolic-set} and  Lemma \ref{stability-e-r},  there exist a neighborhood $\mathcal N$ of $f$, a continuous map $b: \mathcal N\cup \mathcal H\rightarrow \mathbb C$ so that
\begin{itemize}
	\item   $b(g)$ is $g$-repelling  for all $g\in \mathcal N\cup \mathcal H$,  with  $b(f)= \xi_{m_{v,q}}$;
	
	\item   $ b(g)\in \partial U_{g, \sigma^{m_{v,q}}(v)}$
	and   $L_{U_{g, \sigma^{m_{v,q}}(v)},  b(g)}$ is non-trivial for $g\in\mathcal H$.
	\end{itemize}
 Then $l_{v,q}$ is the minimal $s\geq 1$ with $g^{s\ell_v}(b(g))=b(g)$  for $g\in   \mathcal H$ (in particular, for $g=f_0$).   Assume  $\sigma^{m_{v,q}}(v)\in V_{\rm p}^k$ for some $k$, then the above argument implies that $l_{v,q}\in L_k(D)\cap Q_k(f_0)$, which is  a contradiction.

By Claim 1 and  assumption (b),  for different $(v_1, q_1), (v_2, q_2)\in \mathbf{I}_r(f)$,   we have $\xi_{m_{v_1,q_1}} \neq \xi_{m_{v_2,q_2}}$.
Hence  there are at most $\# \mathbf I-\# \mathbf{I}_{r}(f)=\# \mathbf{I}_{p}(f)$ critical points in all possible parabolic Fatou components.

Assume  $\mathbf{I}_{p}\neq \emptyset$, we will get a contradiction.
For each $(v,q)\in \mathbf{I}_{p}$,  by Lemma \ref{limit-Fatou}, the set $f^{m_{v,q}-1}(R\cap L_{U_{f, \sigma(v)}, \xi_1^{v,q}} \cap F(f))$ (here $\xi_1^{v,q}= f(\phi_{f,v}(q))$) is contained in finitely many parabolic Fatou components,  each has exact period $l_{v,q}\ell_v$. We denote the union of these parabolic Fatou components by $W_{v,q}$.  
 By   assumption (b),   also by counting the number of critical points in parabolic Fatou components,  we conclude that
\begin{itemize}
\item  for different $(v_1, q_1), (v_2, q_2)\in \mathbf{I}_p$,  one has $W_{v_1,q_1}\cap  W_{v_2,q_2}=\emptyset$;

\item  for each $(v, q)\in \mathbf{I}_p$, the set $W_{v,q}$ is contained in a cycle of Fatou components, whose union contains only one critical point. 
\end{itemize}


We can further explore the structure of $f^{m_{v,q}-1}(R\cap L_{U_{f, \sigma(v)}, \xi_1^{v,q}})$ and $W_{v,q}$:

{\it Claim 2:  $W_{v,q}$  consists of finitely many immediate parabolic basins of $\xi_{m_{v,q}}$ and  $f^{m_{v,q}-1}(R\cap L_{U_{f, \sigma(v)}, \xi_1^{v,q}})\subset \{\xi_{m_{v,q}}\}\cup W_{v,q}$,.}

To see this, note that $W_{v,q}$ contains at least one immediate parabolic basin $V$ of $\xi_{m_{v,q}}$. Then 
for any $a\in f^{m_{v,q}-1}(R)\cap V$, let $h_{m_{v,q}-1}^a$ be the 
univalent map constructed by Proposition \ref{inverseb}. The curve $\gamma: \mathbb R\rightarrow V$  defined by  $\gamma(t)=h_{m_{v,q}-1}^a(R^f_{\sigma^{m_{v,q}}(v), \xi_{m_{v,q}}}(t))$ satisfies  $f^{l_{v,q}\ell_v}(\gamma(t))=\gamma(t-l_{v,q})$, for all $t\in \mathbb R$. 
Since $V$ is the immediate parabolic basin of $\xi_{m_{v,q}}$ with $f^{l_{v,q}\ell_v}(V)=V$, we have 
$\lim_{t\rightarrow -\infty }\gamma(t)=\xi_{m_{v,q}}$.
The limit $\lim_{t\rightarrow +\infty }\gamma(t)$ also exists,  and is denoted by $b$.
If $b\neq \xi_{m_{v,q}}$, then $f^{l_{v,q}\ell_v}$ has two fixed points on $\partial V$. 
Since $\partial V$ is a Jordan curve (Proposition \ref{RY}),   the Blaschke model $B_V$ of $f^{l_{v,q}\ell_v}|_V$ has at least two fixed points on $\partial \mathbb D$, one of which is   parabolic  with two attracting petals, hence  having multiplicity (i.e. the order of the zero of $B_V-{\rm id}$) $3$.  So $B_V$ has at least $4$ fixed points (counting multiplicity).
By \cite[Lemma 12.1]{M}, we have 
${\rm deg}(B_V)={\rm deg}(f^{l_{v,q}\ell_v}|_V)\geq 3$. 
 Since the exact $f$-period of $V$ is $l_{v,q}\ell_v$,   
 the $f$-cycle of $V$ contains at least two $f$-critical points.
This contradicts the proven property of $W_{v,q}$.
Thus $b=\xi_{m_{v,q}}$.  The same reason applies to  other components of $W_{v,q}\setminus V$ (if any).   Combining  the connectivity of $ f^{m_{v,q}-1}(R\cap L_{U_{f, \sigma(v)}, \xi_1^{v,q}})$,  this implies  Claim 2. 

Note that the $f$-critical value $e^*_v(q)\in R\cap L_{U_{f, \sigma(v)}, \xi_1^{v,q}}\cap F(f)$.
 By  Claim 2 and taking preimages,  $R\cap L_{U_{f, \sigma(v)}, \xi_1^{v,q}}\cap F(f)$
  is  in a union of finitely many Fatou components attaching at $\xi_1^{v,q}$. 
 The one containing $e^*_v(q)$ is denoted by $W_1$.  
 There is a  component $W_0$ of $f^{-1}(W_1)$ containing  an $f$-critical point   and  satisfying  $W_0\subset L_{U_{f, v}, \xi_0^{v,q}}$ (here $\xi_0^{v,q}=\phi_{f,v}(q)$) and $\partial W_{0}\cap \partial U_{f,v}=\{\xi_0^{v,q}\}$. 
This $W_0$ is strictly $f$-pre-periodic(because  $\xi_0^{v,q}$ has this property). 
It follows that the grand orbit of $W_{v,q}$ contains at least two critical points.
 Hence the number of all critical points in grand orbits of all  parabolic Fatou components is at least  $2\# \mathbf{I}_{p}(f)$.  This   contradiction implies that $\mathbf{I}_{p}(f)=\emptyset$.

To finish, we  shall show that $D$ is $\mathcal H$-admissible.   
The proven fact $\mathbf{I}_{p}(f)=\emptyset$ implies that for all $(v, q)\in \mathbf{I}$, the point $\xi_0^{v,q}=\phi_{f,v}(q)$ is  $f$-critical
and pre-repelling. 
By Claim 1 and assumption (b),  the points  $\xi_0^{v,q}$, $(v,q)\in \mathbf{I}$ have disjoint  $f$-forward orbits.  This implies, in particular,  for all  $(v,q)\in \mathbf{I}$,
the orbit $f(\xi_0^{v,q})\mapsto f^2(\xi_0^{v,q})\mapsto \cdots$
meets no $f$-critical points. By Claim 1 and   taking preimages,  $L_{U_{f, \sigma(v)}, \xi_1^{v,q}}$ is trivial, implying that
  $D$ is $\mathcal H$-admissible.  
\end{proof}

\begin{rmk}    By Proposition \ref{criterion-H-adm},  one may verify that $\mathcal H$-admissible divisors are dense in $\partial_0^* {\rm Div}{(\mathbb D)}^S$.
\end{rmk}

\section{The impression for $\mathcal H$-admissible divisors} \label{impression-singleton}


In this section, we shall prove the following

\begin{pro} \label{divisor-singleton} 
Let $D=\big((B_v^0, D_v^{\partial})\big)_{v\in V}\in  \partial_0^* {\rm Div}{(\mathbb D)}^S$ be $\mathcal H$-admissible. Then  $I_{\Phi}(D)$ consists of a singleton.
\end{pro}

The proof needs the notion of lamination  and we recall some necessities.
Let $f$ be a polynomial with connected Julia set $J(f)$.
The \emph{rational lamination} $\lambda_\mathbb{Q}(f)\subset \mathbb Q^2$ of $f$ consists of the pairs $(\theta_1,\theta_2)\in \mathbb Q^2$ for which the external rays $R_f(\theta_1)$ and $R_f(\theta_2)$ land at the same point.  
If   $J(f)$ is also locally connected, then by Carath\'eodory's theorem, all external rays land. 
The \emph{real lamination} $\lambda_\mathbb{R}(f)\subset(\mathbb R / \mathbb Z)^2$ of $f$  consists of   $(\theta_1,\theta_2)\in (\mathbb R / \mathbb Z)^2$ for which the external rays $R_f(\theta_1)$ and $R_f(\theta_2)$ land at the same point. It's clear that  $\lambda_\mathbb{Q}(f)=\mathbb{Q}^2\cap\lambda_\mathbb{R}(f)$, and $\lambda_\mathbb{R}(f)$ is closed in $(\mathbb R / \mathbb Z)^2$.

Let $\lambda$ be a subset of $\mathbb Q^2$ (resp. $(\mathbb R / \mathbb Z)^2$).
The smallest closed equivalence relation in $\mathbb Q^2$ (resp.$(\mathbb R / \mathbb Z)^2$)
that contains $\lambda$ is denoted by $\langle\lambda\rangle_\mathbb{Q}$ (resp. $\langle\lambda\rangle_\mathbb{R}$). 


\begin{lem} [\cite{K}, Lemma 4.17]
	\label{connection-rational-real}
	Let $f$ be a polynomial with connected Julia set $J(f)$,  then  $\lambda_\mathbb{Q}(f)$ is  a closed  subset of  $\mathbb Q^2$. If   $J(f)$ is also locally connected,
then $\langle\lambda_\mathbb{Q}(f)\rangle_\mathbb{R}=\lambda_\mathbb{R}(f)$ if and only if
	every critical point on the boundary of any bounded Fatou component is preperiodic.
\end{lem}


  Recall that $f_0$ is the center of the hyperbolic component $\mathcal H$.

   \begin{lem} \label{lamination-subset} 
  	Let $D=\big((B_v^0, D_v^{\partial})\big)_{v\in V}\in  \partial_0^* {\rm Div}{(\mathbb D)}^S$ be $\mathcal H$-admissible.  For any $f\in  I_{\Phi}(D)$, we have $\lambda_{\mathbb R}(f_0)\subset \lambda_{\mathbb R}(f)$. 
  \end{lem}
  
  \begin{proof}
  	Since $f_0$ is hyperbolic, by Lemma \ref{connection-rational-real}, we have $\lambda_\mathbb{R}(f_0)=\langle \lambda_\mathbb{Q}(f_0)\rangle_\mathbb{R}$.
  	Note that for any $g\in \mathcal H$,  $\lambda_\mathbb{Q}(f_0)=\lambda_\mathbb{Q}(g)$.
  	By Proposition \ref{h-admissible-characterization},  $f$ is Misiurewicz, hence $J(f)$ is locally connected and $f$ has no parabolic  cycles.  Take any pair  $(\theta_1, \theta_2)\in \lambda_\mathbb{Q}(f_0)$ and  any sequence $\{f_n\}_{n\geq 1}\subset \mathcal H$ with $f_n\rightarrow f$. Note that $(\theta_1, \theta_2)\in \lambda_\mathbb{Q}(f_n)$ for all $n\geq 1$.    Lemma \ref{stability-e-r} implies
  the Hausdorff convergence  $\overline{R_{f_n}(\theta_k)}\rightarrow \overline{R_f(\theta_k)}$ for $k=1,2$.  Hence $(\theta_1, \theta_2)\in \lambda_\mathbb{R}(f)$. This shows that 
  	$\lambda_{\mathbb Q}(f_0)\subset \lambda_{\mathbb R}(f)$. 
  	Taking the closure, we get $\lambda_{\mathbb R}(f_0)\subset \lambda_{\mathbb R}(f)$. 
  \end{proof}

Let $V^{\infty}=\bigcup_{k\geq 0}f_0^{-k}(V)$. For $v\in V^{\infty}$, we still use  $U_{f_0, v}$ to denote the Fatou component of $f_0$  containing $v$. Let 
$$\Theta_v(f_0)=\{\theta\in   \mathbb R/\mathbb Z; \ R_{f_0}(\theta) \text{ lands at a point on } \partial U_{f_0,v}\}.$$
Extend $\sigma: V\rightarrow V$ to $\sigma: V^\infty\rightarrow V^\infty$ by defining $\sigma(v)=f_0(v)$; extend $\delta: V\rightarrow \mathbb N$ to $\delta: V^\infty\rightarrow  \mathbb N$ by defining $\delta(v)=1$ for all $v\in  V^\infty\setminus V$.
Recall that $\tau:  \mathbb R/\mathbb Z\rightarrow  \mathbb R/\mathbb Z$ is defined by $\tau(\theta)=d\theta  \ {\rm mod } \  \mathbb Z$.

\begin{ft} \label{property-theta}
We have the following facts:
	

1.  $\tau: \Theta_v(f_0)\rightarrow \Theta_{\sigma(v)}(f_0)$ is a $\delta(v)$-to-one covering map for $v\in V^{\infty}$.

2.  Given $v_1, v_2\in V^{\infty}$, with $v_1\neq v_2$ and $\sigma(v_1)=\sigma(v_2)$, we have $$\Theta_{v_1}(f_0)\cap \Theta_{v_2}(f_0)=\emptyset.$$
\end{ft}
\begin{proof} 
	  1. It is immediate.  For 2,  if $\theta\in \Theta_{v_1}(f_0)\cap \Theta_{v_2}(f_0)\neq \emptyset$,  
	then the landing point $q$ of $R_{f_0}(\theta)$ is $f_0$-critical, because 
	$f_0(U_{f_0, v_1})=f_0(U_{f_0, v_2})$ and $q\in \partial U_{f_0, v_1}\cap \partial U_{f_0, v_2}$. This is a contradiction because $f_0$ is hyperbolic. 
	\end{proof}

For $v\in  V^{\infty}$, 
there is  a decomposition $\partial U_{f_0,v}=\partial_0 U_{f_0,v}\sqcup \partial_* U_{f_0, v}$, where 
\bess &\partial_0 U_{f_0, v}=\{\zeta\in \partial U_{f_0, v}; L_{U_{f_0, v},\zeta}= \{\zeta\}\},&\\
&\partial_* U_{f_0, v}=\{\zeta\in \partial U_{f_0, v}; L_{U_{f_0, v},\zeta}\neq \{\zeta\}\}.&
\eess
Each $\zeta\in \partial_0 U_{f_0, v}$ is the landing point of a unique external ray, say $R_{f_0}(\theta_\zeta)$;  each  $\zeta\in \partial_* U_{f_0, v}$ is $f_0$-pre-periodic (see Remark \ref{limb-wandering}),  and there are at least two external rays landing at $\zeta$. In the latter case, let $\theta^+_v(\zeta), \theta^-_v(\zeta)\in \mathbb R/\mathbb Z$ be the angles so that  $R_{f_0}(\theta^-_v(\zeta)), L_{U_{f_0,v}, \zeta}, R_{f_0}(\theta^+_v(\zeta))$ attach at $\zeta$  in positive cyclic order.
This induces  two disjoint subsets $\Theta_v^0(f_0),  \Theta_v^*(f_0)$ of $\Theta_v(f_0)$: 
\begin{equation} \label{angles=pm}
	\Theta_v^0(f_0)=\{\theta_\zeta; \zeta\in \partial_0 U_{f_0, v}\}, \ \Theta_v^*(f_0)=\{\theta^+_v(\zeta), \theta^-_v(\zeta); \zeta\in \partial_* U_{f_0, v}\}.
	\end{equation}
Clearly $\tau(\Theta_v^0(f_0))=\Theta_{\sigma(v)}^0(f_0)$, $\tau(\Theta_v^*(f_0))=\Theta_{\sigma(v)}^*(f_0)$. 
Note that the union 
$$\Theta_v^0(f_0)\cup \Theta_v^*(f_0)=\mathbb R/\mathbb Z-\bigsqcup_{\zeta\in \partial_* U_{f_0, v}}(\theta^-_v(\zeta), \theta^+_v(\zeta)),$$
and it  might be strictly smaller than  $\Theta_v(f_0)$, because some $\zeta \in \partial_* U_{f_0, v}$ can be a landing point of at least three external rays.   Clearly $\Theta_v^0(f_0)\cup \Theta_v^*(f_0)$ is a compact subset of $\mathbb R/\mathbb Z$; all angles in $\Theta_v^*(f_0)$ are pre-periodic under   $\tau$.


For a polynomial $g$ with connected Julia set, if $R_g(t_1)$  and $R_g(t_2)$   land  at the same point, let $S_g(t_1, t_2)$ be the component of $\mathbb C\setminus(\overline{R_{g}(t_1)}\cup \overline{R_{g}(t_2)})$ containing the external rays $R_g(t)$ with $t_1<t<t_2$ (in positive cyclic order). 

Let $D=\big((B_v^0, D_v^{\partial})\big)_{v\in V}\in  \partial_0^* {\rm Div}{(\mathbb D)}^S$ be $\mathcal H$-admissible.  For each
$f\in I_{\Phi}(D)$ and each $v\in V^{\infty}$,  following \cite{IK}, we define
\begin{equation} \label{kfv}
	K_{f,v}=K(f)\setminus\bigsqcup_{\zeta\in \partial_* U_{f_0, v}} S_{f}(\theta_v^-(\zeta), \theta_v^+(\zeta)).
	\end{equation}
One may verify that $K_{f,v}$ is a compact subset of $K(f)$,   $f(K_{f,v})=K_{f,\sigma(v)}$.
When  $v\in V^{\infty}\setminus V$, let $l\geq 1$ be  minimal  so that $\sigma^l(v)\in  V$. By the location of critical points (Proposition \ref{h-admissible-characterization}(1)),  $f^l:  K_{f,v}\rightarrow K_{f,\sigma^l(v)}$ is one-to-one. Let 
\begin{equation}
	U_{f,v}=f^l|_{K_{f,v}}^{-1}(U_{f, \sigma^l(v)}).
	\end{equation}
In this way, for all $v\in V^{\infty}$, the marked Fatou component $U_{f,v}$ is well defined. 

For $g\in I_{\Phi}(D)$ or $g=f_0$, let $x_g(\theta)$ be the landing point of the  external ray
$R_g(\theta)$ with  $\theta\in \mathbb R/\mathbb Z$, let $\mathcal F_b(g)$ be the collection of all bounded Fatou components of $g$.
For $v\in V^\infty$,  define two restricted laminations
\bess
&\lambda_{v}(g)=\{(\alpha, \beta); \ \alpha,\beta \in \Theta_v^0(f_0)\cup \Theta_v^*(f_0)  \text{ and } x_g(\alpha)=x_g(\beta)\}, &\\
 &\lambda_{v}^F(g)=\{(\alpha, \beta)\in   \lambda_{v}(g);  \ x_g(\alpha)\in \partial U \text{ for some } U\in  \mathcal F_b(g)\}.&
 \eess 
 Clearly $\lambda_{v}^F(f_0)=\lambda_{v}(f_0)$.  By Lemma \ref{lamination-subset},  
\begin{equation}\label{lamination-restricted}
	\lambda_{v}(f_0)\cup \lambda_{v}^F(f)\subset \lambda_{v}(f)   \text{ for } f\in I_{\Phi}(D).
	\end{equation}

Recall the definitions of $\theta^-_{D,v},  \theta^+_{D,v}$ at the end of Section \ref{comb-pro-imp}.
Let
 $$\Theta_v(D)=\big\{\theta^-_{D,v}(q), \theta^+_{D,v}(q); q\in {\rm supp}(D_{v}^\partial) \big\}\subset \mathbb R/\mathbb Z, \  v\in V.$$

 \begin{lem} \label{disjoint angle} 
 	Let $D=\big((B_v^0, D_v^{\partial})\big)_{v\in V}\in  \partial_0^* {\rm Div}{(\mathbb D)}^S$ be  $\mathcal H$-admissible and $f\in I_{\Phi}(D)$.

 	
 	
 	1. For any   $v\in V$ and any $v'\in V^\infty\setminus\{v\}$, we have 
 \begin{equation}\label{angle-relation}\Theta_v(D)\subset \Theta_v^0(f_0),   \  \Theta_v(D)\cap \Theta_{v'}(f_0)=\emptyset.
 	\end{equation}

 	2. For any  $v\in V^\infty$ and  any
  $\theta\in \Theta^*_v(f_0)$, the pre-repelling  orbit $x_f(\theta)\mapsto f(x_f(\theta))\mapsto\cdots$ meets no critical point of $f$.

3. For $v_1, v_2\in V^{\infty}$, with $v_1\neq v_2$ and $\sigma(v_1)=\sigma(v_2)$, we have $$K_{f, v_1}\cap K_{f,v_2}=\emptyset.$$
\end{lem}
 \begin{proof}  
 1.  Since $D$ is $\mathcal H$-admissible, 
  by  Proposition \ref{h-admissible-characterization},  for any $q\in {\rm supp}(D_{v}^\partial)$,  $\phi_{f,v}(q)$ is $f$-critical  and the  limb $L_{U_{f, \sigma(v)}, f(\phi_{f, v}(q))}$ is trivial.  
  Hence  only one external ray lands at   $f(\phi_{f, v}(q))$.
  It follows that 
 	$d \theta^+_{D,v}(q)=d \theta^-_{D,v}(q)  \ {\rm mod \ }\mathbb Z$.
 	
 For any  $\zeta\in \partial_* U_{f_0, v}$,  since $f_0$ is locally conformal near $f_0^l(\zeta)$ for all $l\geq 0$,     we have $d^l \theta_{v}^-(\zeta)\neq  d^l \theta_{v}^+(\zeta)$ for all $l\geq 0$. 
 	This implies  $\Theta_v^*(f_0)\cap \Theta_v(D)=\emptyset$.

 By (\ref{lamination-restricted}), (\ref{angle-relation}) and the limb decomposition (Proposition \ref{RY}), we have
\begin{itemize}
\item $\overline{S_{f}(\theta_{v}^-(\zeta), \theta_{v}^+(\zeta))}\subset S_{f}(\theta^-_{D,v}(q), \theta^+_{D,v}(q))$, or
\item $\overline{S_{f}(\theta_{v}^-(\zeta), \theta_{v}^+(\zeta))}\cap \overline{S_{f}(\theta^-_{D,v}(q), \theta^+_{D,v}(q))}=\emptyset$.
\end{itemize}
The former corresponds to the situation $[\theta_{v}^-(\zeta), \theta_{v}^+(\zeta)]\subset (\theta^-_{D,v}(q), \theta^+_{D,v}(q))$, while the latter means that $[\theta_{v}^-(\zeta), \theta_{v}^+(\zeta)]\cap [ \theta^-_{D,v}(q), \theta^+_{D,v}(q)]=\emptyset$. Therefore
  $$\Theta_v(D)\subset  \mathbb R/\mathbb Z-\bigsqcup_{\zeta\in \partial_* U_{f_0, v}}(\theta^-_v(\zeta), \theta^+_v(\zeta))=\Theta_v^0(f_0)\cup \Theta_v^*(f_0).$$ 
  The disjointness $\Theta_v(D)\cap   \Theta_v^*(f_0)=\emptyset$ implies that $\Theta_v(D)\subset \Theta_v^0(f_0)$.

Note that for any $\theta\in  \Theta_v^0(f_0)$, we have
$\{u\in V^\infty; x_{f_0}(\theta)\in \partial U_{f_0, u}\}=\{v\}.$
This implies that $\Theta_v(D)\cap \Theta_{v'}(f_0)=\emptyset$ for any  $v'\in V^\infty\setminus\{v\}$.



2. Let  $v\in V^\infty$ and $\theta\in \Theta^*_v(f_0)$, suppose that the orbit $x_f(\theta)\mapsto f(x_f(\theta))\mapsto\cdots$ meets an $f$-critical point.
By Proposition \ref{h-admissible-characterization}(1),   we can find an integer  $m\geq 0$ and some $v'\in V$, so that  $f^m(x_f(\theta))=x_f(\tau^m(\theta))\in \phi_{f,v'}({\rm supp}(D_{v'}^\partial))$. 
It follows that $\tau^m(\theta)\in \Theta_{v'}(D)\cap \tau^m(\Theta_v^*(f_0))=\Theta_{v'}(D)\cap \Theta_{\sigma^m(v)}^*(f_0)$.  
However this contradicts (\ref{angle-relation}).


3. Given $v_1, v_2\in V^{\infty}$, with $v_1\neq v_2$ and $\sigma(v_1)=\sigma(v_2)$, suppose that $K_{f, v_1}\cap K_{f,v_2}\neq\emptyset$. By definition (\ref{kfv}), the intersection $K_{f, v_1}\cap K_{f,v_2}$ consists of a single point, of the form $x_f(\theta)$  for some  $\theta\in \Theta^*_{v_1}(f_0)\cap \Theta^*_{v_2}(f_0)$. Since $\sigma(v_1)=\sigma(v_2)$, the point $x_f(\theta)$ is   $f$-critical, which  contradicts statement 2.
 	\end{proof}

\begin{lem} \label{div-lam} 
	Let $D=\big((B_v^0, D_v^{\partial})\big)_{v\in V}\in  \partial_0^* {\rm Div}{(\mathbb D)}^S$ be $\mathcal H$-admissible. Let $v\in V^\infty$ and $f\in I_{\Phi}(D)$.
	
	1.   For any $\zeta\in \partial_*U_{f_0,v}$, we have 
	$$\big\{\alpha\in  \Theta_v^0(f_0)\cup \Theta_v^*(f_0);  \  x_f(\alpha)=x_f(\theta_v^+(\zeta))\big\}=\{\theta_v^+(\zeta), \theta_v^-(\zeta)\}.$$
	 
	2.  The set $\lambda_{v}^F(f)$  does not depend on the choice of 
	$f\in I_{\Phi}(D)$.
\end{lem}

\begin{proof} 1. By Lemma \ref{lamination-subset},   the  $\supset$ part is true.
	The converse follows from Lemma \ref{disjoint angle} (2) and Lemma \ref{stability-e-r}.

	2.  Let $(\alpha, \beta)\in  \lambda_{v}^F(f)\setminus \lambda_{v}(f_0)$. \vspace{2pt}
	
		{\it  Claim 1:  The $f$-orbit: $x_f(\alpha)\mapsto f(x_f(\alpha))\mapsto\cdots$ meets an $f$-critical point. }\vspace{2pt}
	 
	  If not, the local  conformality of $f$ at $f^{l}(x_f(\alpha))$ for all $l\geq 0$  implies that 
	$$\tau^l(\alpha)\neq \tau^l(\beta) \ {\rm mod } \ \mathbb Z, \ (\tau^l(\alpha),  \tau^l(\beta))\in  \lambda_{\sigma^l(v)}^F(f),  \  \forall \ l\geq 0.$$  It follows that there is an integer $n_0\geq 0$ such that $$\big\{y_k:=f^{n_0+\ell_{v_0} k}(x_f(\alpha)); k\geq 0\big\}\subset \partial U_{f, v_0}, \ v_0=\sigma^{n_0}(v)\in V_{\rm p},$$
	and the limbs $L_{U_{f, v_0}, y_k}$'s are not singletons. 
	
	If the $f$-orbit of $x_f(\alpha)$ is pre-periodic (hence pre-repelling), then the non-critical assumption and the stability of external rays imply that
	 $x_{f_0}(\alpha)=x_{f_0}(\beta)$, hence $\alpha, \beta\in \Theta_v^*(f_0)$ and 
	 $(\alpha, \beta)\in  \lambda_{v}(f_0)$.  A contradiction.
	
	If the $f$-orbit of $x_f(\alpha)$ is wandering, then by Remark \ref{limb-wandering},   the limb $L_{U_{f, v_0}, y_k}$ is a singleton for all large $k$. This again gives  a contradiction.

	This shows the claim,  which implies   $f^{m}(x_f(\alpha))$ is $f$-critical for some $m\geq 0$.
	In the following, we will prove the statement by induction on $m$.
	
 (1). 	If $m=0$, then $x_f(\alpha)$ is $f$-critical.  By Proposition \ref{h-admissible-characterization}(1), there is  $q\in {\rm supp}(D_v^\partial)$ with $(\alpha,\beta)=(\theta_{D, v}^-(q), \theta_{D, v}^+(q))  \text{ or }  (\theta_{D, v}^+(q), \theta_{D, v}^-(q))$.
 
 (2). 	If $m=1$, then $f(x_f(\alpha))$ is  $f$-critical.  Up to an order of $\alpha$ and $\beta$, we  assume $(\tau(\alpha),\tau(\beta))=(\theta_{D, \sigma(v)}^-(q), \theta_{D, \sigma(v)}^+(q))$ for some  $q\in {\rm supp}(D_{\sigma(v)}^\partial)$. 
 By Fact \ref{property-theta}(1), (\ref{angle-relation}) and the fact $\tau^{-1}(\Theta_{\sigma(v)}^0(f_0))\cap  \Theta_v(f_0) =\Theta_{v}^0(f_0)$,  we have
 \bess &\tau^{-1}(\tau(\alpha))\cap \Theta_v(f_0):=\{\alpha_1, \cdots, \alpha_{\delta(v)}\} \subset \Theta_v^0(f_0),&\\
 &\tau^{-1}(\tau(\beta))\cap  \Theta_v(f_0):=\{\beta_1, \cdots, \beta_{\delta(v)}\}\subset \Theta_v^0(f_0),&
 \eess
here, the angles $\alpha_k, \beta_j$ are numbered so that
 $$\alpha_1<\beta_1<\alpha_2<\beta_2<\cdots<\alpha_{\delta(v)}<  \beta_{\delta(v)}<\alpha_1+1$$
 in positive cyclic order of $\mathbb R/\mathbb Z$.
 
 In the following, we shall determine the pairs $(\alpha_k, \beta_j)\in \lambda_{v}^F(f)$.\vspace{2pt}
 
	{\it Claim 2: $(\alpha_k, \beta_j)\in \lambda_{v}^F(f) \Longleftrightarrow k=j$. }
 
 
 \begin{figure}[h]  
 	\begin{center}
 		\includegraphics[height=5cm]{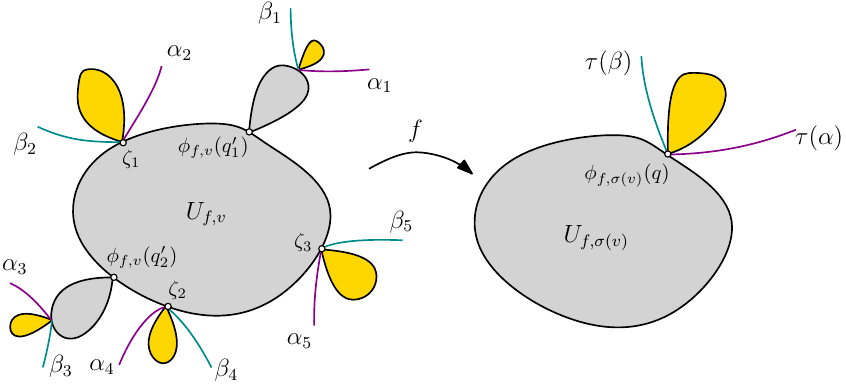}
 	\end{center}
 	\caption{The pairing $(\alpha_k, \beta_j)$ is uniquely determined, when $D$ is $\mathcal H$-admissible}
 	\label{fig: pairing-unique0}
 \end{figure}

 \begin{figure}[h]  
 	\begin{center}
 		\includegraphics[height=5cm]{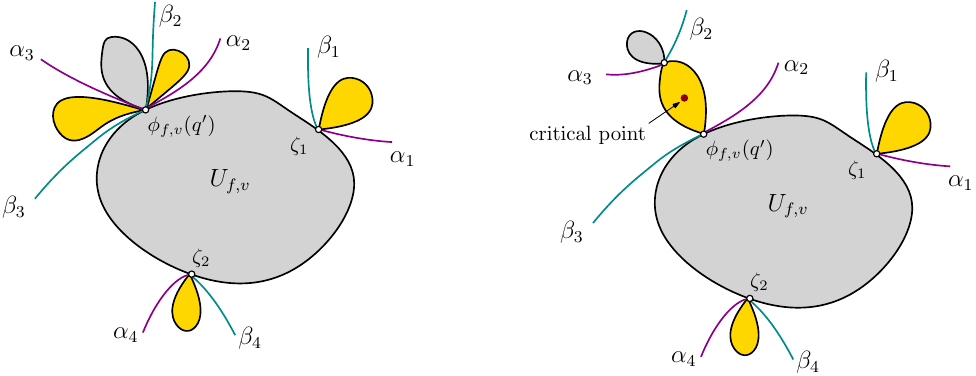}
 	\end{center}
 	\caption{Two possibe pairings of $(\alpha_k, \beta_j)$, when $D$ is not $\mathcal H$-admissible. Hence $\lambda_v^F(f)$  is not uniquely determined. (In this case, $B_v^0(q')=q\in  {\rm supp}(D_{\sigma(v)}^\partial)$ for some $q'\in  {\rm supp}(D_v^\partial)$.)}
 	\label{fig: not-unique}
 \end{figure}

 Note that $f(x_f(\alpha))=\phi_{f, \sigma(v)}(q)$.  By Lemma \ref{disjoint angle} (2), we have 
 $$f^{-1}(\phi_{f, \sigma(v)}(q))\cap K_{f,v}\subset K_{f,v}^{\rm res}:=K_{f,v}-\{x_f(\theta); \theta\in \Theta_v^*(f_0)\}.$$ 
 Moreover, $f^{-1}(\phi_{f, \sigma(v)}(q))\cap K_{f,v}$ consists of    $\delta(v)$ points
  in $ K_{f,v}^{\rm res}$, counting multiplicity.
  Claim 2 is obvious if $v\in V^{\infty}\setminus V$, in which case $\delta(v)=1$.
  
  Assume $v\in V$. 
Since $D$ is $\mathcal H$-admissible, we have two crucial properties:

\begin{itemize}
	\item  $f^{-1}(\phi_{f, \sigma(v)}(q))\cap \partial U_{f,v}$ contains exactly ${\rm deg}(f|_{U_{f,v}})={\rm deg}(B_v^0)$-points $\zeta_1$, $\cdots$, $\zeta_{{\rm deg}(B_v^0)}$, none  is $f$-critical;
	\item   $f^{-1}(\phi_{f, \sigma(v)}(q))\cap (K_{f,v}\setminus \partial U_{f,v})$ contains  exactly one point in each of the limbs $L_{U_{f,v}, \phi_{f, v}(q')}, \  q'\in {\rm supp}(D_v^\partial)$. 
\end{itemize}

These properties imply that  $f^{-1}(\phi_{f, \sigma(v)}(q))\cap K_{f,v}$
consists of exactly $\delta(v)$ points, each  is non-critical, distributed in $\delta(v)$  limbs (see Figure \ref{fig: pairing-unique0}):
$$L_{U_{f,v}, \zeta_k},   L_{U_{f,v}, \phi_{f, v}(q')};  \  1\leq k\leq {\rm deg}(B_v^0),\  q'\in {\rm supp}(D_v^\partial).$$

Since $f^{-1}(\phi_{f, \sigma(v)}(q))\cap K_{f,v}$ is exactly the set of all landing points of  $2\delta$ external rays   $R_f(\alpha_k), R_f(\beta_j)$,  
we conclude that $x_f(\alpha_k)=x_f(\beta_j)  \Longleftrightarrow k=j$.

 It's clear that these pairs  are independent of the choice of $f\in I_{\Phi}(D)$.
 
 (3). 	Assume $m\geq 2$. 
  Assume by induction that $(\alpha_0,\beta_0):=(\tau(\alpha), \tau(\beta))\in \lambda_{\sigma(v)}^F(f)\cap  (\Theta_{\sigma(v)}^0(f_0)\times  \Theta_{\sigma(v)}^0(f_0))$ is uniquely determined by $D$. 
   The angles  $\alpha_0, \beta_0$ are numbered so that $S_f(\alpha_0, \beta_0)$ is disjoint from $U_{f, \sigma(v)}$. Write 
 \bess \tau^{-1}(\alpha_0)\cap  \Theta_v(f_0)=\{\alpha_1, \cdots, \alpha_{\delta(v)}\}, \ \tau^{-1}(\beta_0)\cap  \Theta_v(f_0)=\{\beta_1, \cdots, \beta_{\delta(v)}\},
 \eess
so that $\alpha_1<\beta_1<\alpha_2<\beta_2<\cdots<\alpha_{\delta(v)}<  \beta_{\delta(v)}<\alpha_1+1$.
 Then there exist $r\in \partial U_{f, \sigma(v)}$ and $\widetilde{\alpha},\widetilde{\beta}\in \mathbb R/\mathbb Z$  so that $\partial S_f(\widetilde{\alpha}, \widetilde{\beta})\cap  \partial U_{f, \sigma(v)}=\{r\}$ and $S_f({\alpha}, {\beta})\subset S_f(\widetilde{\alpha}, \widetilde{\beta})$.
 By the same reason as (2),   
   the finite set $f^{-1}(r)\cap K_{f,v}$ is distributed in  $\delta(v)$ disjoint limbs (see Figure \ref{fig: pairing-m}):
 $$L_{U_{f,v}, \zeta},   L_{U_{f,v}, \phi_{f, v}(q')};  \ \zeta\in f^{-1}(r)\cap \partial U_{f, v},\  q'\in {\rm supp}(D_v^\partial).$$
  \begin{figure}[h]  
 	\begin{center}
 		\includegraphics[height=5cm]{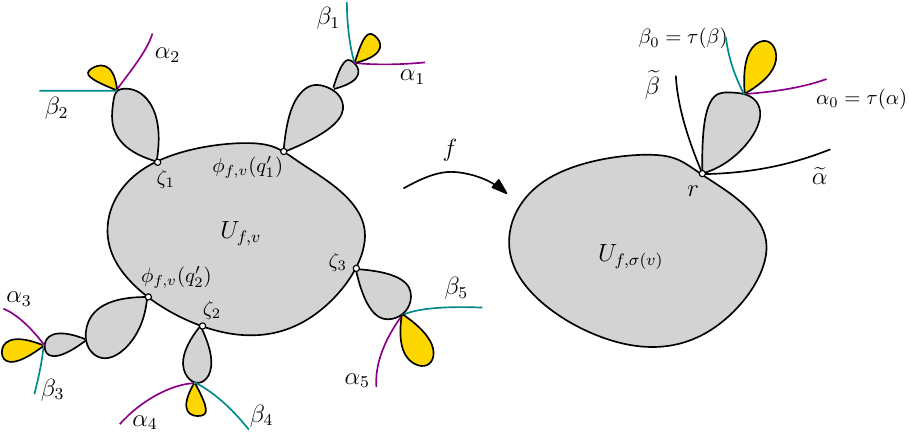}
 	\end{center}
 	\caption{The pairing $(\alpha_k, \beta_j)$ is uniquely determined, assuming $(\tau(\alpha), \tau(\beta))\in \lambda_{\sigma(v)}^F(f)$ is given,  when $m\geq 2$.}
 	\label{fig: pairing-m}
 \end{figure}

	Each   limb contains exactly one point in the set $f^{-1}(x_f(\alpha_0))\cap K_{f,v}$. 
	We conclude  again  that $(\alpha_k, \beta_j)\in \lambda_{v}^F(f)$ if and only if   $k=j$.   These pairs are independent of $f\in I_{\Phi}(D)$.  
 	\end{proof}
 
 \begin{rmk}    If $D$ is not $\mathcal H$-admissible, Lemma \ref{div-lam} is false, see Figure \ref{fig: not-unique}.
 \end{rmk}
 
 \begin{lem} \label{div-lam2} 
 	Let $D=\big((B_v^0, D_v^{\partial})\big)_{v\in V}\in  \partial_0^* {\rm Div}{(\mathbb D)}^S$ be $\mathcal H$-admissible. For any $v\in V^\infty$ and any $f\in I_{\Phi}(D)$,
 	$$\lambda_v(f)= \lambda_{v}(f_0)\cup \lambda_{v}^F(f).$$
 \end{lem}
\begin{proof}   
 For $u\in  V$,  set $K_{f,u}^0=\overline{U_{f,u}}$.
	Clearly $f(K_{f,u}^0)=K_{f,\sigma(u)}^0$.
 	Let $K_{f,u}^{1}$ be  the connected component of 
 	$f^{-1}(K_{f,\sigma(u)}^{0})$ containing $K_{f,u}^0$.
 	It's clear that $K_{f,u}^{0}\subsetneq  K_{f,u}^{1}$ if and only if ${\rm supp}(D_u^\partial)\neq \emptyset$, because    $\partial K_{f,u}^{0}\cap {\rm Crit}(f)=\phi_{f,u}({\rm supp}(D_u^\partial))$ (by Proposition \ref{h-admissible-characterization}(1)).
Let  
 	$K_{f,u}^{l}$  be  the connected component of $f^{-1}(K_{f,\sigma(u)}^{l-1})$ containing $K_{f,u}^{l-1}$ inductively for all $l\geq 1$, and let
 	$$K_{f,u}^{\infty}={\bigcup_{l\geq 0}K_{f,u}^{l}}, \  K_{f,u}^{*}=\overline{K_{f,u}^{\infty}}.$$
 	
 	 We first consider the case $v\in V_{\rm p}$. 	
 		If  ${\rm supp}(D_v^\partial), \cdots, {\rm supp}(D_{\sigma^{\ell_v-1}(v)}^\partial)$ are all empty, then $K_{f,v}^{l}=K_{f,v}^{0}$ for all $l\geq 1$. 
 		In this case,  by Mane's Theorem \cite{Mane}, $\partial U_{f,v}$ is an $f^{\ell_v}$-hyperbolic set.  By Proposition \ref{RY}, if the limb $L_{U_{f,v},a}$ is non-trivial for some $a\in \partial U_{f,v}$, then $a$ is $f$-pre-repelling.   By the implicit function theorem, there exist a  neighborhood $\mathcal N$ of $f$, a continuous  map $a: \mathcal N\cup \mathcal H\rightarrow \mathbb C$
 so that $a(g)$ is $g$-pre-repelling for $g\in \mathcal N\cup \mathcal H$, with $a(f)=a$.

 	By Proposition \ref{holo-hyperbolic-set} and  Lemma \ref{stability-e-r},  that $R_f(\theta)$ lands at $a$ implies that 
 $R_g(\theta)$ lands at $a(g)\in \partial U_{g, v}$ for $g\in  \mathcal H$.  Hence $\lambda_{v}(f)
\subset \lambda_{v}(f_0)$.  By (\ref{lamination-restricted}), we have  $\lambda_v(f)= \lambda_{v}(f_0)\cup\lambda_{v}^F(f)$.

 	
 	  Assume one of ${\rm supp}(D_v^\partial)$, $\cdots$, ${\rm supp}(D_{\sigma^{\ell_v-1}(v)}^\partial)$ is non-empty. For $u\in\{v, \cdots, \sigma^{\ell_v-1}(v)\}$,  we have
 	$K_{f,u}^{l}\supsetneq K_{f,u}^{0}$ for  $l\geq \ell_v$.  Hence $K_{f,u}^{\infty}\supsetneq K_{f,u}^{0}$.
 	 By Lemma \ref{disjoint angle},  $\Theta_{u}(D)\subset \Theta_{u}^0(f_0)$.	 By the definition of $K_{f,u}$  (see (\ref{kfv})) and the fact 
 	$\tau^{-1}(\Theta_{\sigma(u)}^0(f_0))\cap \Theta_u(f_0)=\Theta_u^0(f_0)$,   
 	all cut points \footnote{Here $b$ is a cut point of a connected and compact set $K$  if $K-\{b\}$ is disconnected.} of $K_{f,u}^1$ are contained in $K_{f,u}$, hence $K_{f,u}^1\subset K_{f,u}$.
 	By Lemma \ref{disjoint angle}(3), the set $K_{f,u}$ is a component of $f^{-1}(K_{f,\sigma(u)})$, we conclude by induction that $K_{f,u}^l\subset K_{f,u}$ for all $l\geq 1$. Hence $K_{f,u}^{*} \subset K_{f,u}$.

 	
 	\vspace{5pt }
 	{\it  Claim 1: Let $\zeta\in K_{f,u}^{\infty}$
 	be a  cut point.
 	
 	(a). $\zeta$ is pre-critical (i.e., there is $n\geq 0$ so that $f^n(\zeta)$ is $f$-critical);
 	
 	(b). Only two external rays $R_f(\theta_1), R_f(\theta_2)$ land at $\zeta$, and $\theta_1, \theta_2\in \Theta_u^0(f_0)$.}	
 

 Note that by the construction of $K_{f,u}^l$, a cut point appears if and only if it is an iterated preimage of a critical point on $\partial U_{f, v}\cup \cdots\cup \partial U_{f, \sigma^{\ell_v-1}(v)}$.  Hence there is an integer $n\geq 0$ so that $f^{n}(\zeta)$ is a critical point on 
 $\partial U_{f,\sigma^n(u)}$.
 
    Observe   that  for  $k\geq 0$,  $f^k(\zeta)\in K_{f, \sigma^k(u)}$ and  all critical points in $K_{f, \sigma^k(u)}$ are contained in $\overline{U_{f, \sigma^k(u)}}$ (by Proposition \ref{h-admissible-characterization} and Lemma \ref{disjoint angle}). 
 Since $D$ is $\mathcal H$-admissible,  the orbit $\zeta\mapsto  f(\zeta)\mapsto   \cdots$ meets  an $f$-critical point only at the moment $f^n(\zeta)$.  Since  only two  external rays  land at  $f^{n}(\zeta)$ (by Proposition \ref{h-admissible-characterization}), by pulling back via $f^n$,  there are exactly two external rays $R_f(\theta_1), R_f(\theta_2)$ landing at $\zeta$.  
 The fact $\tau^{-n}(\Theta_{\sigma^n(u)}^0(f_0))\cap \Theta_u(f_0)=\Theta_u^0(f_0)$ implies that $\theta_1, \theta_2\in \Theta_u^0(f_0)$.
 This proves Claim 1.
 
 \vspace{5pt}


	We assume by contradiction that
	$(\alpha,\beta)\in \lambda_v(f) \setminus (\lambda_{v}(f_0)\cup \lambda_{v}^F(f))\neq \emptyset$.
 	Consider the location of  $x_f(\alpha)\in K_{f,v}$.  The assumption $(\alpha,\beta)\notin \lambda_{v}(f_0)\cup \lambda_{v}^F(f)$ implies that $x_f(\alpha)\in K_{f,v}\setminus K_{f,v}^\infty$. Hence
there is a sequence of Fatou components $U_0=U_{f, v}, U_1, U_2, \cdots$ in $K_{f,v}^{*} $ with the following properties: 
	\begin{itemize}
		\item   for any $k\geq 0$,  $\partial U_{k}\cap \partial U_{k+1}$ is a singleton $q_k$, which is the landing point of exactly two external rays $R_f(\alpha_k), R_f(\beta_k)$ (by Claim 1(b)).

	 \item the limbs $L_{U_k, q_k}, k\geq 0$ satisfy that 
	 $$L_{U_0, q_0}\supset L_{U_1, q_1}\supset L_{U_2, q_2}\supset\cdots, \  x_f(\alpha)\in \bigcap_{k\geq 0}L_{U_k, q_k}.$$
	 
	 \item the angles $\alpha_k, \beta_k$ 
	 are numbered so that $\alpha_0<\alpha_1<\cdots <  \beta_1<\beta_0$.
	 \end{itemize}
 
  \begin{figure}[h]  
 	\begin{center}
 		\includegraphics[height=4cm]{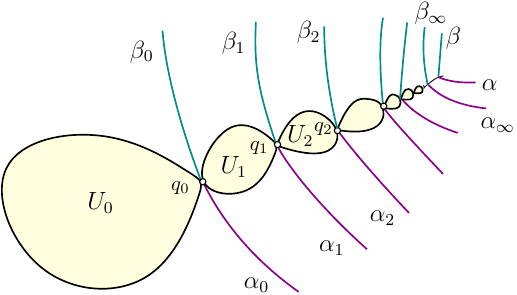}
 	\end{center}
 	\caption{The limbs and external rays.}
 \end{figure}

 
 Let $\alpha_\infty=\lim_{k\rightarrow \infty}\alpha_k$,  $\beta_\infty=\lim_{k\rightarrow \infty}\beta_k$. 
 By Claim 1(b),  $\alpha_k, \beta_k\in \Theta_v^0(f_0)$.  By the  compactness of $\Theta_v^0(f_0)\cup \Theta_v^*(f_0)$, we have $\alpha_\infty, \beta_\infty\in \Theta_v^0(f_0)\cup \Theta_v^*(f_0)$.
 Note that $\alpha\neq \beta$, by changing the order of $\alpha, \beta$ if necessary, we assume  
     $\alpha_\infty\leq \alpha< \beta\leq \beta_\infty$.  
  The  locally connectivity  of $J(f)$ implies  $x_{f}(\alpha_\infty)=x_f(\beta_\infty)$.
	 
	\vspace{5pt }
	{\it  Claim 2:  The orbit of $x_{f}(\alpha_\infty)$ meets a critical point. }
	\vspace{5pt }
	
	Suppose it is not true. We need to treat two cases.
	
	Case 1.  $x_{f}(\alpha_\infty)$ is pre-repelling. By the stability of external rays, we have  
	$(\alpha_\infty, \beta_\infty)\in \lambda_{v}(f_0)$. 	This situation happens only if $\alpha_\infty=\alpha\in  \Theta_v^*(f_0)$, $\beta_\infty=\beta\in  \Theta_v^*(f_0)$. 
	However,	this  contradicts the assumption $(\alpha,\beta)\notin \lambda_{v}(f_0)$.
	
	Case 2. $x_{f}(\alpha_\infty)$ is wandering. In this case, there is an integer $n_0\geq 0$ so that 
	$S_{f}(\tau^l(\alpha_\infty), \tau^l(\beta_\infty))$ contains no critical point of $f$ when $l\geq n_0$. So the angular difference grows exponentially as $k\rightarrow \infty$. A contradiction.
	
	\vspace{5pt}
	
	Claim 2 implies that there exist an integer $m\geq 0$  and $v'\in V_{\rm p}\setminus \{\sigma^m(v)\}$  so that $f^m(x_{f}(\alpha_\infty))=x_f(\tau^m(\alpha_\infty))$ is a critical point on $\partial U_{f, v'}$.  By Proposition \ref{h-admissible-characterization}(1), we have $ \tau^m(\alpha_\infty)\in \Theta_{v'}(D)\cap \tau^m(\Theta_v^0(f_0)\cup \Theta_v^*(f_0))\subset  \Theta_{v'}(D)\cap \Theta_{\sigma^m(v)}(f_0)$. However, this contradicts
	 Lemma \ref{disjoint angle} (1).
	 
	 This finishes the proof of $\lambda_v(f)= \lambda_{v}(f_0)\cup \lambda_{v}^F(f)$ when $v\in V_{\rm p}$.

	 When $v\in  V\setminus V_{\rm p}$,  assume   $\sigma(v)\in V_{\rm p}$. Take  $(\alpha,\beta)\in \lambda_v(f)$.  If $x_f(\alpha)$ is $f$-critical, then up to an order,  $(\alpha,\beta)=(\theta_{D, v}^-(q), \theta_{D, v}^+(q))\in \lambda_{v}^F(f)$; if  $x_f(\alpha)$ is not $f$-critical, then $(\tau(\alpha),\tau(\beta))\in \lambda_{\sigma(v)}(f)= \lambda_{\sigma(v)}(f_0)\cup \lambda_{\sigma(v)}^F(f)$,  implying that $(\alpha,\beta)\in \lambda_{v}(f_0)\cup \lambda_{v}^F(f)$.  Hence $\lambda_v(f)= \lambda_{v}(f_0)\cup \lambda_{v}^F(f)$.  By induction, this equality holds for all $v\in V$. 
	 
	 If $v\in  V^\infty\setminus V$, then  $\tau: \Theta_v(f_0)\rightarrow \Theta_{\sigma(v)}(f_0)$  is one-to-one. By the same reasoning and induction, one has 
	 $\lambda_v(f)= \lambda_{v}(f_0)\cup \lambda_{v}^F(f)$.
	\end{proof}

\begin{lem} \label{h-admissible-lamination} 
 	Let $D=\big((B_v^0, D_v^{\partial})\big)_{v\in V}\in  \partial_0^* {\rm Div}{(\mathbb D)}^S$ be $\mathcal H$-admissible.
 	Then the lamination $\lambda_\mathbb{R}(f)$ does not depend on the choice of 
$f\in I_{\Phi}(D)$.
 \end{lem}
\begin{proof} Let $f_1, f_2\in I_{\Phi}(D)$.  By symmetry, it suffices to show  $\lambda_\mathbb{R}(f_1)\subset\lambda_\mathbb{R}(f_2)$.
	
	 If it is not true, take $(\alpha, \beta)\in \lambda_\mathbb{R}(f_1)\setminus \lambda_\mathbb{R}(f_2)$. 
	 Without loss of generality, assume $\alpha, \beta$ are $f_1$-adjacent (i.e. for any $\theta$ with $\alpha<\theta<\beta$, $x_{f_1}(\theta)\neq x_{f_1}(\alpha)$).
	 By Lemma \ref{lamination-subset},  
	 $x_{f_0}(\alpha)\neq x_{f_0}(\beta)$.
	Then there is an allowable arc $[x_{f_0}(\alpha),  x_{f_0}(\beta)] \subset K(f_0)$ connecting $x_{f_0}(\alpha)$ and $ x_{f_0}(\beta)$ (see \cite[Chapter 4]{DH}). 
	 We claim that
  $(x_{f_0}(\alpha),  x_{f_0}(\beta))\cap J(f_0)=\emptyset$. 
  
   \begin{figure}[h]  
  	\begin{center}
  		\includegraphics[height=2.8cm]{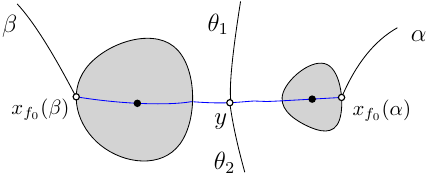}
  	\end{center}
  	\caption{External rays in different homotopy classes }
  \end{figure}
   If not, choose $y\in (x_{f_0}(\alpha),  x_{f_0}(\beta))\cap J(f_0)$. Then there are two  external rays  $R_{f_0}(\theta_1), R_{f_0}(\theta_2)$  corresponding to  two homotopy classes of paths to $y$ relative to  $[x_{f_0}(\alpha),  x_{f_0}(\beta)]$ so that $x_{f_0}(\theta_1)=x_{f_0}(\theta_2)$. 
  By Lemma \ref{lamination-subset},  
  $x_{f_1}(\alpha)=x_{f_1}(\beta)=x_{f_1}(\theta_1)=x_{f_1}(\theta_2)$. This  contradicts the assumption that
  $\alpha, \beta$ are $f_1$-adjacent.
  
The claim implies that $x_{f_0}(\alpha),  x_{f_0}(\beta)\in \partial U_{f_0, v}$ for some $v\in V^\infty$, hence $(\alpha, \beta)\in \lambda_{v}(f_1)$.  By Lemmas \ref{div-lam} and \ref{div-lam2}, we have $\lambda_{v}(f_1)=\lambda_{v}(f_2)$. Therefore $(\alpha, \beta)\in \lambda_\mathbb{R}(f_2)$.
 \end{proof}
 
 \begin{proof}[Proof of Proposition \ref{divisor-singleton}.]
 Let $f_1, f_2\in  I_{\Phi}(D)$. The idea is to construct a topological conjugacy $h$ between $f_1, f_2$.  To this end, we need to define $h$ piece by piece (in each Fatou component of $f_1$) and then glue them together.\vspace{3pt}
 
 {\textbf{Step 1.  Defining $h$ in the unbounded Fatou component.}}\vspace{3pt}

 For $j=1,2$, 
 let $\psi_{j, \infty}: \mathbb C\setminus K(f_j)\rightarrow \mathbb C\setminus \overline{\mathbb D}$ be the B\"ottcher map of $f_j$, normalized so that 
 $\psi_{j, \infty}(z)=z+O(1)$ near $\infty$. Then $h_\infty=\psi_{2, \infty}^{-1}\circ \psi_{1, \infty}:  \mathbb C\setminus K(f_1)\rightarrow  \mathbb C\setminus K(f_2)$ is a conformal conjugacy: $h_\infty\circ f_1=f_2\circ h_\infty$.
 
 By Lemma  \ref{h-admissible-lamination}, we have $ \lambda_{\mathbb R}(f_1)=\lambda_{\mathbb R}(f_2)$.  This means that for any $\alpha, \beta \in \mathbb R/\mathbb Z$,  $x_{f_1}(\alpha)=x_{f_1}(\beta)\Longleftrightarrow x_{f_2}(\alpha)=x_{f_2}(\beta)$.
   It follows that $h_\infty$ extends to a homeomorphism $h_\infty:  (\mathbb C\setminus K(f_1)) \cup J(f_1)\rightarrow  (\mathbb C\setminus K(f_2))\cup J(f_2)$ by defining $h_\infty(x_{f_1}(\theta))=x_{f_2}(\theta)$
 for all $\theta \in \mathbb R/\mathbb Z$.  It keeps the conjugacy $h_\infty\circ f_1|_{J(f_1)}=f_2\circ h_\infty|_{J(f_1)}$.\vspace{3pt}
 
 {\textbf{Step 2.  Defining $h$ in the Fatou components indexed by $V$.}} \vspace{3pt}
 
 First, the  boundary marking $\nu_{f_0}: V\rightarrow \mathbb C$ (defined in Section \ref{phc}) of $f_0$ extends to $\nu_{f_0}: V^\infty\rightarrow \mathbb C$ so that $\nu_{f_0}(v)\in \partial U_{f_0, v}$ and $f_0(\nu_{f_0}(v))=\nu_{f_0}(\sigma(v))$. 
Choose $(\theta_v)_{v\in V^\infty}$   so that $x_{f_0}(\theta_v)=\nu_{f_0}(v)$ and $\tau (\theta_v)=\theta_{\sigma(v)}$ for $v\in V^\infty$ (note that $\theta_v$ may be not unique for a  given $v\in V^\infty$, we take one of them). 
 The boundary marking $\nu_g: V^\infty\rightarrow \mathbb C$  with $g\circ \nu_g=\nu_g\circ \sigma$ can be defined in a continuous way  for $g\in \mathcal H$. It satisfies $\nu_g(v)=x_g(\theta_v)$ for  $v\in V^\infty$.

 Since each $f\in I_{\Phi}(D)$ is Misiurewicz (by Proposition \ref{h-admissible-characterization}), by Lemma \ref{stability-e-r},
     for  any sequence $\{g_n\}_{n\geq 1}\subset \mathcal H$ with $\lim_n g_n=f$, we have 
     $\lim_n x_{g_n}(\theta_v)= x_{f}(\theta_v)\in \partial U_{f,v}$ for $v\in V$. 
     If  we further require $\lim_n \Phi^{-1}(g_n)=D$, then
      Propositions \ref{convergence-repelling}, \ref{combinatorial-property0} and   Remark \ref{fixed-1} imply
     $\lim_n\phi_{g_n, v}(1)= \phi_{f, v}(1)$.  Since $\phi_{g_n, v}(1)=\nu_{g_n}(v)=x_{g_n}(\theta_v)$ for all $n$, we have 
     $$ x_{f}(\theta_v)=\phi_{f, v}(1), \ v\in V.$$

Define $h_v: U_{f_1, v}\rightarrow U_{f_2, v}$  by 
 $h_v=\psi_{f_2, v}^{-1}\circ \psi_{f_1, v}$.   It extends to a homeomorphism $h_v:  \overline{U_{f_1, v}}\rightarrow \overline{U_{f_2, v}}$ (by Proposition \ref{RY} and Caratheodory Theorem).
 satisfying that  $h_{\sigma(v)}\circ f_1=f_2\circ h_v$ and $h_v(x_{f_1}(\theta_v))=x_{f_2}(\theta_v)$. 
 
 In the following, we show that $h_v|_{\partial U_{f_1,v}}=h_\infty|_{\partial U_{f_1,v}}$ for $v\in V$. By taking preimages, it suffices to consider the case $v\in V_{\rm p}$. In this case,
 \begin{equation} \label{conj-normaliztion} \phi\circ f_1^{\ell_v}|_{\partial U_{f_1,v}}=f_2^{\ell_v}\circ \phi,  \  \phi(x_{f_1}(\theta_v))=x_{f_2}(\theta_v),  
 	\end{equation}
 for $\phi\in \{h_v|_{\partial U_{f_1,v}}, h_\infty|_{\partial U_{f_1,v}}\}$.
 
By  (\ref{conj-normaliztion}),
 $\phi$ sends $X_1:=f_1^{\ell_v}|^{-1}_{\partial U_{f_1,v}}(x_{f_1}(\theta_v))$  to $Y_1:=f_2^{\ell_v}|^{-1}_{\partial U_{f_2,v}}(x_{f_2}(\theta_v))$.  Since $\phi$ is orientation preserving,  the restriction $\phi|_{X_1}: X_1\rightarrow Y_1$ is uniquely determined.  By induction,  $\phi$ sends $X_k:=f_1^{\ell_v}|^{-k}_{\partial U_{f_1,v}}(x_{f_1}(\theta_v))$  to $Y_k:=f_2^{\ell_v}|^{-k}_{\partial U_{f_2,v}}(x_{f_2}(\theta_v))$, and $\phi|_{X_k}$ is   uniquely  determined,  for all $k\geq 1$.  By   density  
 $\overline{\cup_{k\geq 1} X_k}=\partial U_{f_1,v}$, we get a unique $\phi$. Hence   $h_v|_{\partial U_{f_1,v}}=h_\infty|_{\partial U_{f_1,v}}$.\vspace{3pt}
 

{\textbf{Step 3.  Defining $h$ in other pre-periodic Fatou components.}}\vspace{3pt}

 For $f\in \{f_1, f_2\}$, recall that $\mathcal F_b(f)$ is the set of all bounded Fatou components of $f$.  Let
 $\mathcal F_b^*(f)=\mathcal F_b(f)\setminus\{U_{f,v};v\in V\}$. For each $v\in V$, let
  \bess 
  &\Theta_v(f)=\{\theta\in   \mathbb R/\mathbb Z; \ x_f(\theta)\in  \partial U_{f,v}\},&\\
   &\Theta_v^0(f)=\{\theta\in   \Theta_v(f);    L_{U_{f,v}, x_f(\theta)}=\{ x_f(\theta)\}\}. &
 \eess 
Note that $\tau(\Theta_v(f))=\Theta_{\sigma(v)}(f), \tau(\Theta_v^0(f))=\Theta_{\sigma(v)}^0(f)$.
 By Proposition \ref{combinatorial-property-2}, $\Theta_v^0(f_1)=\Theta_v^0(f_2)$ for all $v\in V$.

 
 Suppose $V_{\rm p}$ consists of $m$ $\sigma$-cycles. We choose a representative in each cycle and denote them by $v_1, \cdots, v_m$. For   $1\leq j\leq m$, choose $\theta_j\in \Theta_{v_j}^0(f)$ so that $x_{f}(\theta_j)$ avoids the grand orbit of all $f$-critical points.
 Set $\theta(v_j)=\theta_j$ and $\theta(\sigma^k(v_j))=\tau^k(\theta_j)$ for   $1\leq k<\ell_{v_j}$.  If $v\in V_{\rm np}$, we define $\theta(v)$ inductively so that $\tau(\theta(v))=\theta(\sigma(v))$.
   In this way,  each $v\in V$ is assigned an angle $\theta(v)\in \Theta_v^0(f)$,  so that $x_{f}(\theta(v))$ avoids all $f$-critical orbits.
   Let 
 $$A=\tau^{-1}\big\{\theta(v); v\in V\big\}  \setminus\bigcup_{u\in V} \Theta_u(f), \ \ T=\bigcup_{k\geq 0}\tau^{-k}(A).$$
 
 {\it Claim:  There is a bijection $\theta_f: \mathcal F_b^*(f)\rightarrow T$.}
 
 To prove the claim, we fix $v\in V$.  Since  $\theta(v)\in \Theta_v^0(f)$  and $x_{f}(\theta(v))$ avoids all $f$-critical orbits, for any $\alpha\in \tau^{-1}(\theta(v)) \setminus\bigcup_{u\in V} \Theta_u(f)$, there is a unique 
 connected component $V$ of $f^{-1}(U_{f, v})\setminus \bigcup_{u\in V} U_{f, u}$ so that 
  $x_f(\alpha)\in \partial V$; conversely, each  component $V$ of $f^{-1}(U_{f, v})\setminus \bigcup_{u\in V} U_{f, u}$ maps to $U_{f, v}$ conformally, hence there is a unique point $y\in \partial V\cap f^{-1}(x_f(\theta(v)))$. 
  By the choice of $\theta(v)$, we see that $y$ is not $f$-critical,  hence  $y$  is the landing point of a unique external ray $R_f(\alpha)$ with $\alpha\in \tau^{-1}(\theta(v))\setminus \bigcup_{u\in V} \Theta_u(f)$.  This gives a bijection between the components of $f^{-1}(U_{f, v})\setminus \bigcup_{u\in V} U_{f, u}$  and the set $\tau^{-1}(\theta(v))\setminus \bigcup_{u\in V} \Theta_u(f)$. The claim  follows by induction.

 
Let  $U\in \mathcal F_b^*(f_1)$, and let $s(U)\geq 1$  be the minimal integer so that $f_1^{s(U)}(U)\in \{U_{f,v};v\in V\}$.  Write  $f_1^{s(U)}(U)=U_{f_1, v(U)}$ for some $v(U)\in V$. 
 Then   $\tau^{s(U)-1}(\theta_{f_1}(U))\in A$ and $\tau^{s(U)}(\theta_{f_1}(U))=\theta(v(U))$.
 Set $U'=\theta_{f_2}^{-1}\circ\theta_{f_1}(U)\in  \mathcal F_b^*(f_2)$, then $f_2^{s(U)}: U'\rightarrow U_{f_2, v(U)}$ is a conformal map.  Define  a conformal map $h_U: U\rightarrow U'$  so that the
  diagram commutes
  $$\xymatrix{ 
 	U \ar[r]^{f_1^{s(U)}} \ar[d]^{h_{U}}& U_{f_1, v(U)} \ar[d]^{h_{v(U)}}\\
 	U'\ar[r]^{f_2^{s(U)}} & U_{f_2, v(U)}
 }$$
The map $h_U$ extends to a homeomorphism $h_U: \overline{U}\rightarrow \overline{U'}$, satisfying that $h_U(x_{f_1}(\theta_{f_1}(U)))=x_{f_2}(\theta_{f_1}(U))$.  
 By Step 3,  $h_{v(U)}|_{\partial U_{f_1, v(U)}}=h_\infty|_{\partial U_{f_1, v(U)}}$.  Hence
  $$h_U|_{\partial U}=f_2^{s(U)}|_{U'} ^{-1}\circ h_\infty|_{\partial U_{f_1, v(U)}} \circ f_1^{s(U)}|_{\partial U}=h_\infty|_{\partial U}.$$
 
 {\textbf{Step 4.   Showing that $h={\rm id}$.}}\vspace{3pt}
 
By gluing the maps  in $\{h_{\infty}, h_{v}, h_U;  v\in V, \text{ and } U\in  \mathcal F_b^*(f_1)\}$, 
 we get a homeomorphism $h: \mathbb C\rightarrow \mathbb C$.
It is a topological conjugacy between $f_1$ and $f_2$, conformal in  Fatou set $F(f_1)$.  Since the Julia set $J(f_1)$ is conformally removable (see \cite[Corollary 1]{JS}), $h$ is  conformal.   The normalization  $h'(\infty)=1$ implies $h(z)=z+b$. Since  $f_1$ and $f_2$ are centered, we have $b=0$. Thus $f_1=f_2$, and $I_{\Phi}(D)$ is a singleton.
\end{proof}

\begin{rmk} By contrast, if $D\in  \partial_0^* {\rm Div}{(\mathbb D)}^S$ is not $\mathcal H$-admissible, it can happen that $I_{\Phi}(D)$ is not a singleton.  The reason is that
	  Lemma \ref{div-lam}(2)  does not hold without $\mathcal H$-admissibility.
It is a fundamental problem to give a description of  $I_{\Phi}(D)$ in general cases. 
\end{rmk}

	

\section{Boundary Extension Theorem} \label{bet}

In this section, we shall prove the boundary extension theorem: Theorem \ref{thm:parameterization}.   Recall that $\mathcal B^S$ is identified as ${\rm Div}(\mathbb D)^S$, and $\mathcal A$ is the set of all $\mathcal H$-admissible divisors.

By Propostion \ref{divisor-singleton},  
the map $\Phi: {\rm Div}(\mathbb D)^S\rightarrow \mathcal H$ has a continuous extension
 $$\overline{\Phi}: {\rm Div}(\mathbb D)^S\cup \mathcal A\to \mathcal H\cup \partial_\mathcal{A}\mathcal H,$$
  where  $\overline{\Phi}(D)$ is defined as the single map in $I_{\Phi}(D)$ for $D\in \mathcal A$, and  $\partial_\mathcal{A}\mathcal H:=\overline{\Phi}(\mathcal A)$.  To prove Theorem \ref{thm:parameterization}, it is equivalent  to show that $\overline{\Phi}$ is  injective and $\overline{\Phi}^{-1}$ is continuous.   This relies on the   non self-intersection property of $\partial \mathcal H$ along $\partial_\mathcal{A}\mathcal H$ (Proposition \ref{map-imp}) and  Lemma \ref{neigh-continuous}.

For any $g\in \partial \mathcal H$, it is reasonable to define the following set
$${\rm Div}[g]:=\Big\{E\in \partial {\rm Div}{(\mathbb D)}^S; g\in I_{\Phi}(E) \Big\}.$$

From the geometric viewpoint, the set ${\rm Div}[g]$ gives a characterization of  self-intersection  of $\partial \mathcal H$ at $g$. 
If ${\rm Div}[g]$ is a singleton, then $\partial \mathcal H$ has no  self-intersection at $g$; otherwise,  $\partial \mathcal H$ has   self-intersection at $g$. See Figure \ref{fig:toy-model-si}.

    \begin{figure}[h]   
	\begin{center}
		\includegraphics[height=4cm]{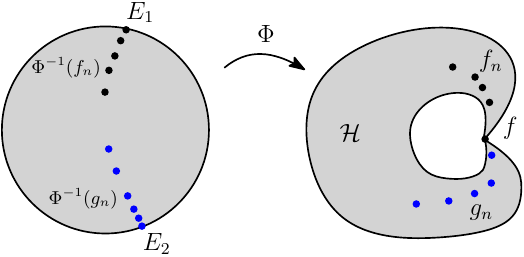}
	\end{center}
	\caption{In this toy model, $\partial \mathcal H$ has   self-intersection at $f$.  There are two  ways approaching $f$, and ${\rm Div}[f]=\{E_1, E_2\}$.}
	\label{fig:toy-model-si}
\end{figure}




\begin{pro}\label{map-imp}
	Let $D=\big((B_v^0, D_v^{\partial})\big)_{v\in V}\in  \partial_0^* {\rm Div}{(\mathbb D)}^S$ be $\mathcal H$-admissible, with $I_\Phi(D)=\{f\}$.  
	There is a neighborhood $\mathcal N$ of $f$ such that  
	for any
	$g\in \mathcal N \cap \partial\mathcal H$,  the  set 	${\rm Div}[g]$
consists of a singleton $E_g$, which is  in $\partial^*_0 {\rm Div}{(\mathbb D)}^S$.
\end{pro}

\begin{proof}
	
		\vspace{5pt}
	{\textbf{Step 0.  Some   notations and properties for the pair $(f, D)$.}}
	\vspace{5pt}
	
	Recall that each map $g\in \mathcal H$ is assigned an internal marking 	${\boldsymbol v}(g)=(v(g))_{v\in V}$ and  a boundary marking $\nu_g=(\nu_g(v))_{v\in V}$.

  Take any sequence $\{f_n\}_{n\geq 1}\subset \mathcal H$ with $f_n\rightarrow f$ and $\Phi^{-1}(f_n)\rightarrow D$ as $n\rightarrow \infty$. 
 Since $f\in \partial_{\rm reg}\mathcal H$, the point $v_D(f):=\lim_{n} v(f_n)$ is  $f$-pre-attracting  for each $v\in V$.  By Lemma \ref{carathodory-convergence}, we have the kernel convergence 
 $(U_{f_n, v}, v(f_n))\rightarrow (U_{f, v_D}, v_D(f)),  \  \forall v\in V,$
 where $U_{f, v_D}:=U_f(v_D(f))$ is the Fatou component of $f$ containing  $v_D(f)$.

	Here, for the objects (points, sets, maps, etc.) associated with $f$,  we emphasize their dependence   
	 on $D$, 
	 and mark a 
	  subscript $D$ to them.
	  The reason is that if ${\rm Div}[f]$ is not a singleton, then there are at least two ways approaching $f$ (i.e. there is  $E\in {\rm Div}[f]\setminus\{D\}$ and a sequence $\{g_n\}_{n\geq 1}\subset \mathcal H$ with $g_n\rightarrow f$  and $\Phi^{-1}(g_n)\rightarrow E$ as $n\rightarrow \infty$, see Figure \ref{fig:toy-model-si}). 
	  This means that the notations  actually depend on the pair $(f,D)$.

	  We summarize what have been proven. 
		By Proposition \ref{combinatorial-property0}: 
	the  conformal maps $\phi_{f_n, v}: \mathbb{D}\rightarrow U_{f_n, v}$  converges locally and uniformly to the conformal map $\phi_{f, v_D}: \mathbb{D}\rightarrow U_{f, v_D}$.
	By Caratheodory Theorem and Proposition \ref{RY},  these maps $\phi_{f_n, v},\phi_{f, v_D}$  extend to homeomorphisms between the closures of their domains and ranges.
	 By Proposition \ref{h-admissible-characterization},  $f$ is Misiurewicz, 
	  implying that
	
	\begin{itemize}
		\item the convergence of the boundary marking (by Proposition \ref{convergence-repelling} and Remark  \ref{fixed-1}): for any $v\in V$,
		\begin{equation}
			\label{convergence-landing-point}
		 \nu_{f_n}(v)=\phi_{f_n, v}(1)\rightarrow \phi_{f, v_{D}}(1), \  \text{ as } n\rightarrow \infty. 
		\end{equation}
		
		\item  	the   Hausdorff  convergence of external rays:
		 \begin{equation}
				\label{convergence-external-ray}
			\overline{R_{f_n}(\theta)} \rightarrow   \overline{R_{f}(\theta)}, \  \text{ as } n\rightarrow \infty,
			\end{equation}
		for any rational  angle $\theta\in \mathbb Q/\mathbb Z$ (see Lemma \ref{stability-e-r}).

	\item $ \phi_{f, v_{D}}(1)=x_f(\theta_v)$, where $\theta_v\in \mathbb Q/\mathbb Z$ is chosen  so that $R_{f_0}(\theta_v)$ lands at $\nu_{f_0}(v)$ (by   (\ref{convergence-landing-point}) and (\ref{convergence-external-ray}), since $x_{f_n}(\theta_v)=\nu_{f_n}(v)$ for all $n$).
		\end{itemize}

	
	
	
	Let $v\in V$.  Choose  $\alpha, \beta \in \mathbb Q/\mathbb Z$  so that  $x_f(\alpha), x_f(\theta_{v}),x_f(\beta)$ are different $f$-pre-repelling   points  on $\partial U_{f, v_D}$,  in positive cyclic order.   
	We may further require that $x_f(\alpha),  x_f(\beta)\in \partial_0 U_{f, v_D}$ are sufficiently close to $x_f(\theta_v)$ so that ${\rm Crit}(f)\cap \partial U_{f, v_D}$ and  $x_f(\theta_{v})$
	are in different components of $\partial U_{f, v_D}\setminus \{x_f(\alpha),  x_f(\beta)\}$, and the $f$-orbits of $x_f(\alpha),  x_f(\beta)$ meet no critical points.
	
	In the following, we first give a neighborhood $\mathcal N$ of $f$ with some required properties (Step 1),
	 then show  that ${\rm Div}[f]\subset \partial^*_0 {\rm Div}{(\mathbb D)}^S$ in Step 2 (the key step); in  Step  3, we show ${\rm Div}[f]=\{D\}$;  in Step 4, we complete the proof by showing that  for each map  $g\in \mathcal N\cap \partial{\mathcal H}$,    ${\rm Div}[g]$ is a singleton in $\partial^*_0 {\rm Div}{(\mathbb D)}^S$.
	
	\vspace{5pt}
	{\textbf{Step 1.  A neighborhood $\mathcal N$ of $f$ with prescribed properties.}}
	\vspace{5pt}
	
There exist a neighborhood $\mathcal N$ of $f$, 
	satisfying that for any $v\in V$,
	
	
\textbf{(a)}.   The landing points $ x_{g}(\alpha), x_{g}(\theta_v), x_{g}(\beta)$ are continuous with respect to $g\in \mathcal  H\cup (\mathcal N \cap \overline{\mathcal H})$.
	In particular, $\{x_{f_n}(\alpha),  x_{f_n}(\theta_v), x_{f_n}(\beta)\}\subset \partial U_{f_n,v}$  for large $n$.
	Since $\{f_n\}_n\subset \mathcal H$, by J-stability (see \cite{MSS}), 
  $\{x_{g}(\alpha),  x_{g}(\theta_v), x_{g}(\beta)\}\subset \partial U_{g,v}$ for all $g\in \mathcal H$.
	
\textbf{(b)}.  There is a Jordan disk $W$ so that  $\overline{W} \subset U_{g,v}$ for all $g\in \mathcal N \cap  {\mathcal H}$ (by  Lemma \ref{carathodory-convergence}), and 
$W$ contains all critical points of $f|_{U_{f, v_D}}$.

\textbf{(c)}.  Let $g\in \mathcal N\cap {\mathcal H}$.   For  $t\in \{\alpha, \beta\}$,  there is an arc $\gamma_g(t)$ starting at $a(t)\in \partial W$ and converging to $x_g(t)$;   there is also an arc $\gamma_g(\theta_v)$ starting at $v(g)$ and converging to $x_g(\theta_v)$ (see Proposition \ref{holo-hyperbolic-set}). See Figure \ref{fig: notations-nis}.  These arcs are  required to contain their endpoints.  Further, for any $t\in \{\alpha, \beta,\theta_v\}$,  $g\mapsto \gamma_g(t)$ is  Hausdorff    continuous in $g\in (\mathcal N\cap  {\mathcal H})\cup \{f\}$.

 \begin{figure}[h]   
 	\begin{center}
 		\includegraphics[height=5cm]{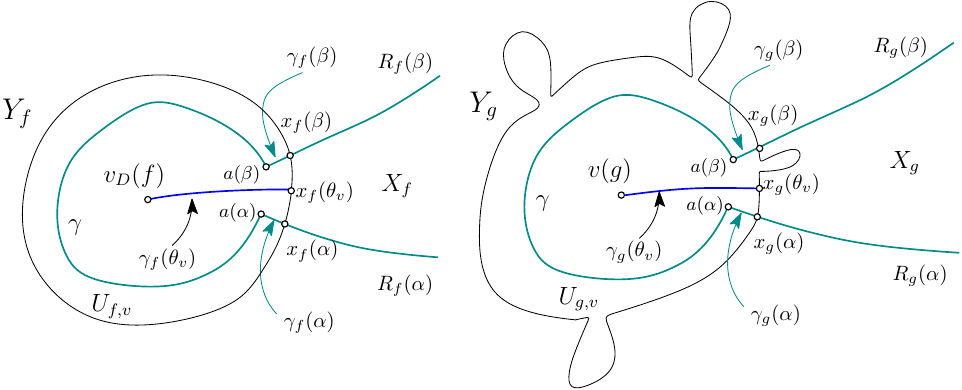}
 	\end{center}
 	\caption{Curves and sets in the dynamical plane of $f$ and a nearby map $g\in \mathcal H$.} 	
 	\label{fig: notations-nis}
 \end{figure}

Let $\gamma$ be the component of $\partial W\setminus\{a(\alpha), a(\beta)\}$, disjoint from $\gamma_g(\theta_v)$.  Let 
$$C_g=\gamma\cup \gamma_g(\alpha)\cup \gamma_g(\beta)\cup R_g(\alpha)\cup R_g(\beta), \ g\in  (\mathcal N\cap  {\mathcal H})\cup \{f\}.$$
Then $\mathbb C\setminus C_g$ has  two components: $X_g$ and $Y_g$, where $X_g$ is  the one containing $x_g(\theta_v)$,  and $Y_g$ is   the other one.

By the choice of $\alpha, \beta$ in Step 0, we have ${\rm Crit}(f)\cap (\overline{U_{f,v_D}}\setminus W)\subset Y_f$.
By  the Hausdorff continuity of $g\mapsto C_g$,  also by shrinking $\mathcal N$ if necessary,    for all $g\in  (\mathcal N\cap  {\mathcal H})\cup \{f\}$,  we have  $\#({\rm Crit}(g)\cap W)=\#({\rm Crit}(f)\cap W)$ and 
\begin{equation} \label{critical-position-g} 
  {\rm Crit}(g)\cap ( {U_{g,v}}\setminus W)\subset Y_g.
		 \end{equation}

	

 \begin{figure}[h]   
 	\begin{center}
 		\includegraphics[height=11cm]{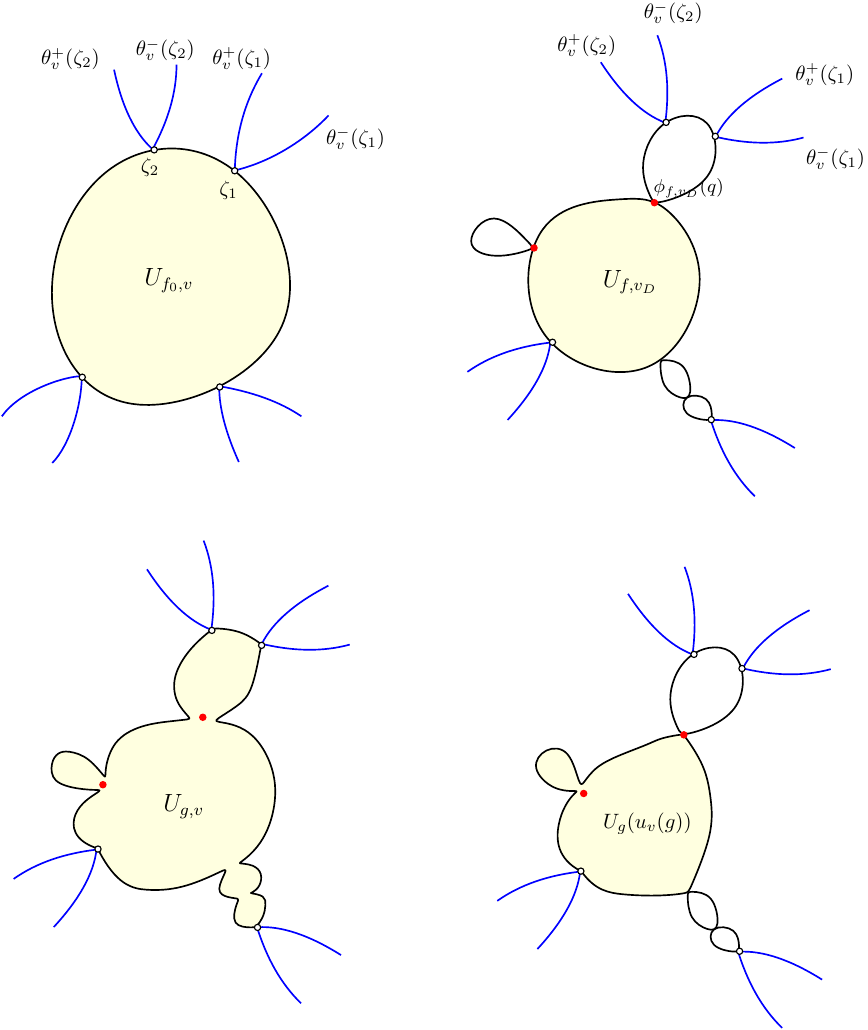}
 	\end{center}
 	\caption{Upper left: the dynamical plane of $f_0$ (the unique PCF in $\mathcal H$); Upper right: the dynamical plane of $f$; 
 		Lower left: the dynamical plane for some $g\in   \mathcal H\cap \mathcal N$; Lower right: the dynamical plane for some $g\in \partial\mathcal H\cap \mathcal N$. The red points are critical points on or close to the boundary of the marked Fatou component.}	
 	\label{fig: non-intersection-graphs}
 \end{figure}
 
\textbf{(d)}. Let 
\begin{equation} \label{partial-v}
	\partial_{v}:=\{\xi\in \partial_*U_{f_0,v}; L_{U_{f_0, v},\xi}\cap {\rm Crit}(f_0)\neq \emptyset\}.
	\end{equation}
For any $\zeta\in \partial_{v}$,  the set $ \overline{R_g(\theta_v^-(\zeta) )\cup R_g(\theta_v^+(\zeta))}$  (here $\theta_v^{\pm}(\zeta)$ are defined in Section \ref{impression-singleton}) is Hausdorff continuous in  $g\in  \mathcal H\cup   \{f\}$, see Figure \ref{fig: non-intersection-graphs}. 
Hence, the number of the $g$-critical points (counting multiplicity) in  $S_g(\theta_{v}^-(\zeta), \theta_{v}^+(\zeta))$, is independent of   $g\in \mathcal H\cup  \{f\}$. 
Denote this number by $n(\theta_{v}^-(\zeta), \theta_{v}^+(\zeta))$.


	
	

		\vspace{5pt}
	{\textbf{Step 2.   ${\rm Div}[f]\subset \partial_0^*{\rm Div}{(\mathbb D)}^S.$}}
	\vspace{5pt}
	
Let $E=\big((B_v, E_v^{\partial})\big)_{v\in V}\in {\rm Div}[f]$,
and $\{g_{n}\}_{n\geq 1}\subset \mathcal H$ be  a sequence with
$$g_{n}\rightarrow f, \ \Phi^{-1}(g_{n}):=(B_{n,v})_{v\in V}\rightarrow E \ \text{ as } n\rightarrow \infty.$$
To show $E\in  \partial_0^*{\rm Div}{(\mathbb D)}^S$,   it is equivalent  to show the near   degenerate critical points of each factor $B_{n,v}$ of  $\Phi^{-1}(g_n)$ has  uniformly  bounded distance (for all $n$) from  $1\in \partial \mathbb D_v$.    To this end, we will use the extremal length method.

	
	Without the assumption $E\in  \partial_0^*{\rm Div}{(\mathbb D)}^S$, by Lemma \ref{carathodory-convergence} and Montel's Theorem,    we have the following:
	
  {\it Fact:   There exist a subsequence $\{g_{n_k}\}_{k\geq 1}$ of  $\{g_n\}_n$, and  unimodular  numbers
  	$(e^{i\theta_v})_{v\in V}$, such that for each $v\in V$,  the conformal maps 
  	$\psi_{g_{n_k}, v}$ converge in Carath\'eodory topology on functions \cite[Chapter 5, pp. 67]{Mc} to $e^{i\theta_v}\psi_{f, v_D}$, where $\psi_{f, v_D}$ is the inverse of the conformal map $\phi_{f,v_D}:\mathbb D\rightarrow U_{f, v_{D}}$.}

For $g\in \mathcal N\cap \mathcal H$, 
let $Q_g=(X_g\cap U_{g,v})\setminus  {\gamma_{g}(\theta_v)}$.
Let $\Gamma_g$ be the collection of rectifiable curves   in $Q_g$,
with one endpoint in $\partial X_g\cap U_{g,v}$ and the other  in  $ {\gamma_{g}(\theta_v)}$, see Figure \ref{fig: step2-modulus}. 
By the definition of the extremal length $\lambda(\Gamma_g)$,
$$\lambda(\Gamma_g)\geq \frac{L(Q_g)^2}{{\rm area}(Q_g)}, \ \  \forall g\in \mathcal N\cap \mathcal H,$$
where $L(Q_g)$ is the infimum of the Euclidean length of the curves in $\Gamma_g$.  By the Hausdorff continuity of $g\mapsto \gamma_g(\alpha)\cup \gamma_g(\beta)\cup \gamma_g(\theta_v)$, 
 we have $L(Q_g)\geq L$ for all $g\in \mathcal N\cap \mathcal H$, here $L>0$ is a  constant.
By area theorem (see \cite[\S 5.1]{A}),  ${\rm area}(Q_g)\leq {\rm area} (K(g))\leq \pi$. 
Hence $\lambda(\Gamma_g)\geq L^2/\pi $ for all $g\in \mathcal N\cap \mathcal H$.

\begin{figure}[h]   
	\begin{center}
		\includegraphics[height=5cm]{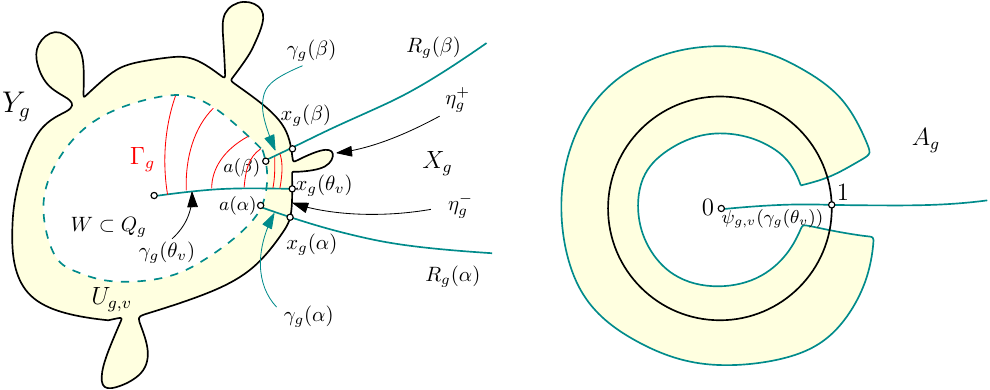}
	\end{center}
	\caption{Bounding the moduli in the dynamical plane (left) and  the model plane (right)} 	
	\label{fig: step2-modulus}
\end{figure}

 The set $(\partial Q_g\cap \partial U_{g,v})\setminus\{x_g(\theta_v)\}$ consists of two arcs: $\eta_g^-$ and $\eta_g^+$.    Let $R(z)=1/\overline{z}$ be the reflection with respect to $\partial \mathbb D$, then 
  $$A_g:=\psi_{g,v}(Q_g\cup \eta_g^{-}  \cup \eta_g^+)\cup R\circ \psi_{g,v}(Q_g)$$   is an annulus separating   
 $\psi_{g,v}(\overline{Y_g\cap U_{g,v}})$ and $\psi_{g,v}({\gamma_g(\theta_v)})\cup R\circ \psi_{g,v}(\gamma_g(\theta_v))$.
 By the symmetry principle (see \cite[Theorem 5, pp 13]{A2}), 
 $${\rm mod}(A_g)=\frac{1}{2}\lambda(\Gamma_g)\geq \frac{L^2}{2\pi}.$$
 
 Clearly,  ${\rm diam}\psi_{g,v}(\overline{Y_g\cap U_{g,v}}) \geq {\rm diam}\psi_{g,v}(\gamma)$.  For the subsequence $\{g_{n_k}\}_k$ given by   the Fact,  we have 
 $$\lim_{k\rightarrow +\infty}{\rm diam} \psi_{g_{n_k}, v}(\gamma)= {\rm diam} \psi_{f, v_D}(\gamma)>0.$$
Hence for all large $k$, $g_{n_k}\in \mathcal H\cap \mathcal N$ and ${\rm diam} \psi_{g_{n_k}, v}(\gamma)\geq c_v:=  {\rm diam} \psi_{f, v_D}(\gamma)/2$.
 
 Let ${\rm dist}_k$ be the Euclidean distance between the two boundary components of $A_{g_{n_k}}$.  
 Let $Z_k$ be the bounded component of $\mathbb C\setminus {A_{g_{n_k}}}$, with Euclidean diameter ${\rm diam}_k$. 
 Take  $\zeta_1\in Z_k$,   $\zeta_2\in \mathbb C\setminus ({A_{g_{n_k}}}\cup Z_k)$ so that $|\zeta_1-\zeta_2|={\rm dist}_k$. Then $\partial{\mathbb D(\zeta_1,  {\rm diam}_k/2)}\cap Z_k\neq \emptyset$.  Take $\zeta_3\in \partial{\mathbb D(\zeta_1, {\rm diam}_k/2)}\cap Z_k$,  then the annulus  $A_{g_{n_k}}$ separates $\{\zeta_1, \zeta_3\}$ and $\{\zeta_2, \infty\}$.
  By   Teichm\"uller's theorem  \cite[Theorem 4.7]{A},
 $$ {\rm mod}(A_{g_{n_k}})\leq h\bigg(\frac{|\zeta_1-\zeta_2|}{|\zeta_1-\zeta_3|}\bigg) = h\bigg(2\frac{{\rm dist}_k}{{\rm diam}_k}\bigg),$$
 where $h(r):={\rm mod}(\mathbb C\setminus ([-1,0]\cup [r, +\infty)))$  is the Teichm\"uller modulus function.
 It follows that  for  large $k$,
 $${\rm dist}({\rm Crit}({B}_{n_k,v})\cap \psi_{g_{n_k},v}(Y_{g_{n_k}}\cap U_{g_{n_k}, v}) , 1)\geq  {\rm dist}_k   \geq \frac{{\rm diam}_k}{2}   h^{-1}\bigg(\frac{L^2}{2\pi}\bigg)\geq   \frac{c_v}{2} h^{-1}\bigg(\frac{L^2}{2\pi}\bigg).$$
 
 Combine this lower bound and the divisor convergence ${B}_{n_k,v}\rightarrow (B_v, E^\partial_v)$,  we have $1\notin {\rm supp}( E^\partial_v)$,  for all $v\in V$.  Hence $E\in  \partial_0^*{\rm Div}{(\mathbb D)}^S$.
 
  

	\vspace{5pt}
	{\textbf{Step 3.    	    ${\rm Div}[f]=\{D\}$}.}
	\vspace{5pt}
	
	By Step 2, ${\rm Div}[f]\subset \partial_0^*{\rm Div}{(\mathbb D)}^S$.  Let $E=((B_v, E_v^{\partial}))_{v\in V}\in {\rm Div}[f]$.  In this step, we  show $E=D$.
Let $\{g_{n}\}_{n\geq 1}\subset \mathcal H$ be  a sequence with
	$$g_{n}\rightarrow f, \ \Phi^{-1}(g_{n})\rightarrow E \ \text{ as } n\rightarrow \infty.$$
	
	Note that the convergence  $g_{n}\rightarrow f$ actually means   $(g_n, \boldsymbol{v}(g_n)) \to (f, \boldsymbol{v}_D(f))$ as $n\rightarrow \infty$. 
	Hence $v_E(f)=v_D(f)$ for all $v\in V$.
	 Similar as we did for $\phi_{f, v_{D}}$ in Step 0,  
	 the  conformal maps $\phi_{g_n, v}: \mathbb{D}\rightarrow U_{g_n, v}$  converge locally and uniformly to the conformal map $\phi_{f, v_E}: \mathbb{D}\rightarrow U_{f, v_E}$, which satisfies   $\phi_{f, v_{E}}(1)=x_f(\theta_v)$.
	 

	 	Since $\phi_{f, v_{D}}$ and $\phi_{f, v_{E}}$ are two conformal mappings from  $\mathbb D$ to $U_{f, v_D}$, with same normalizations
	 	$\phi_{f, v_{D}}(0)=\phi_{f, v_{E}}(0)= v_D(f)$, $\phi_{f, v_{D}}(1)=\phi_{f, v_{E}}(1)=x_f(\theta_v)$,
	 	 we have $\phi_{f, v_{D}}=\phi_{f, v_{E}}$  for all $v\in V$.  By Lemma \ref{model-map}, 
	 		 $${B}^0_{v}=\phi^{-1}_{f, \sigma(v)_{D}}\circ f\circ \phi_{f_{v_{D}}}=\phi^{-1}_{f, \sigma(v)_{E}}\circ f\circ \phi_{f,v_{E}}={B}_{v}, \ \ \forall v\in V.$$
	 		 By definitions of $*_{v}^{\partial}(\mathcal H), \theta_{*, v}^\pm (q)$ for $*\in \{D, E\}$ (see Definition \ref{an induced divisor}    and  Section  \ref{cp2}),  for $v\in V$,
	 $$D_{v}^{\partial}(\mathcal H)=E_{v}^{\partial}(\mathcal H);   \ \theta_{D, v}^+(q)=\theta_{E, v}^+(q),  \ \theta_{D, v}^-(q)=\theta_{E, v}^-(q), \  \forall q\in {\rm supp}(D_{v}^{\partial}(\mathcal H)).$$
	 By Proposition \ref{divisor-correspondence} and Lemma \ref{disjoint angle}, we may write   
	 $$*_{v}^{\partial}=*_{v}^{\partial}(\mathcal H)-\sum_{q\in  {\rm supp} (*_{v}^{\partial}(\mathcal H))}\Big(\sum_{\zeta\in I^*_v(q)}
	n(\theta_v^-(\zeta), \theta_v^+(\zeta))\Big)\cdot q, \ \ *\in \{D, E\},$$
	 where  $I^*_v(q)=\{\zeta\in \partial_v; (\theta_v^-(\zeta), \theta_v^+(\zeta))\subset (\theta_{*, v}^-(q),  \theta_{*, v}^+(q)) \text{ as intervals in } \mathbb R/\mathbb Z\}$ (here $ \partial_v$ is defined in Step 1).  It follows that $I^D_v(q)=I^E_v(q)$ and  $D_{v}^{\partial}=E_{v}^{\partial}$.  Hence $D=E$.

		\vspace{5pt}
	{\textbf{Step 4.   For any $g\in \mathcal N \cap \partial\mathcal H$,  the  set 	${\rm Div}[g]$  consists of a singleton $E_g$, which is in  $\partial^*_0 {\rm Div}{(\mathbb D)}^S$}.}
	\vspace{5pt}
	
	Because of Step 3, we may write $\phi_{f, v_D}, v_D(f), U_{f, v_D}$  as $\phi_{f, v},  v(f), U_{f, v}$.
	
	By  Step 3, for any sequence   $\{g_{n}\}_{n\geq 1}\subset \mathcal H$  approaching $f$, we have  $\Phi^{-1}(g_{n}):=(B_{n,v})_{v\in V}\rightarrow D \ \text{ as } n\rightarrow \infty$.
	By Proposition \ref{combinatorial-property0}, 
		the sequence  of conformal maps
		$\{\phi_{g_n, v}\}_{n\geq 1}$ converges locally and uniformly in $\mathbb D$ to $\phi_{f,v}$.
		It follows  that the conformal maps $\{\psi_{g_n, v}\}_n$   converge in Carath\'eodory topology on functions \cite[Chapter 5, pp. 67]{Mc} to $\psi_{f, v}$.
		As we did in Step 2, there is a  constant $c_v>0$ such that for all large $n$, 
		  $${\rm diam}(\psi_{g_n,v}(\overline{Y_{g_n}\cap U_{g_n,v}}))\geq c_v.$$
		
		
Repeating the  extremal length argument in Step 2, there is $b_v>0$ so that    
$${\rm dist}\big({\rm Crit}({B}_{n,v})\cap \psi_{g_n,v}(Y_{g_n}\cap U_{g_n, v}) , 1\big)\geq b_v, \ \forall \text{ large } n.$$

Since the above behavior holds for any sequence $\{g_n\}_n\subset \mathcal H$  approaching $f$, by shrinking $\mathcal N$ if necessary, we get  a uniform lower bound $b'_v>0$ so that 
\begin{equation} \label{uniform-bound} {\rm dist}\big({\rm Crit}({B}_{g,v})\cap \psi_{g,v}(Y_{g}\cap U_{g, v}) , 1\big)\geq b'_v, \ \ 
\forall g\in \mathcal H\cap \mathcal N, 
\end{equation}
where $B_{g,v}$ is a factor of $\Phi^{-1}(g)=(B_{g,v})_{v\in V}$.

	For any $g\in \partial \mathcal H\cap \mathcal N$ and any sequence   $\{\widehat{g}_{n}\}_{n\geq 1}\subset \mathcal H$ approaching $g$,   it is clear that $\widehat{g}_{n}\in  \mathcal H\cap \mathcal N$ for all large $n$. 
	By (\ref{uniform-bound}),  
	 the accumulation set of $\{\Phi^{-1}(\widehat{g}_{n})\}_{n\geq 1}$ is  contained in $\partial^*_0 {\rm Div}{(\mathbb D)}^S$.
This shows   ${\rm Div}[g]\subset \partial^*_0 {\rm Div}{(\mathbb D)}^S$.

	Applying the same argument  in Step 0 and  Step  3 to $g$ (note that the properties of the maps near  $f$ used in these steps  also hold for the maps near $g$, hence $\mathcal H$-admissibility is not necessary), we conclude that  	${\rm Div}[g]$  consists of a singleton $E_g$. 
	This completes the proof.
\end{proof}

\begin{rmk} \label{divisor-def-g} The divisor $E_g:=((B_{g,v}, E_{g, v}^\partial))_{v\in V}$ for $g\in  \partial \mathcal H\cap \mathcal N$ in Step 4   can be written  explicitly as follows:

	 For each $v\in V$, there is a unique Riemann mapping $\psi_{g,v}:  U_{g}(u_v(g))\rightarrow \mathbb D$ with   $\psi_{g, v}(u_v(g))=0$ and  $\psi_{g, v}(x_g(\theta_v))=1$. 
	 For $v\in V$, let $\phi_{g,v}=\psi_{g,v}^{-1}$ and 
	 $$B_{g, v}=\phi_{g, \sigma(v)}^{-1}\circ g \circ \phi_{g,  v}, \  E_{v}^\partial(g)=\sum_{q\in \partial\mathbb D}\Big(\sum_{c\in C_{q,v}}({\rm deg}(g,c)-1)\Big)\cdot q, $$
	 where   $C_{q,v}={\rm Crit}(g)\cap L_{U_{g}(u_v(g)), \phi_{g,v}(q)}$. 
	 For $q\in {\rm supp}(E_{v}^\partial(g))$, let $\theta_{g, v}^{\pm}(q)$ be given in the paragraph before Corollary \ref{limb-angles}.
	 Then $E_{g, v}^\partial$ is given by 
	 $$E_{g,v}^\partial =E_{v}^\partial(g)-\sum_{q\in  {\rm supp} (E_{v}^{\partial}(g))}\Big(\sum_{ \zeta\in I_v(q)}
	 n(\theta_v^-(\zeta), \theta_v^+(\zeta))\Big)\cdot q,$$
	 where  
	 $I_v(q)$    consists of $\zeta\in \partial_v$  such that 
	 $(\theta_v^-(\zeta), \theta_v^+(\zeta))\subset (\theta_{g, v}^-(q),  \theta_{g, v}^+(q))$  as intervals in $\mathbb R/\mathbb Z$.

	 Observe that for $g\in  \partial \mathcal H\cap \mathcal N$, the divisor  $E_g$ may be not $\mathcal H$-admissible,   and the impression $I_{\Phi}(E_g)\supset\{g\}$ may be not a  singleton. 
	\end{rmk}
	
	

	 

Let  $\mathcal N$ be the neighborhood  of $f$
given by Proposition \ref{map-imp}. For   $g\in \mathcal N\cap \partial \mathcal H$,  let $E_g$ be the divisor given by Proposition \ref{map-imp}; for $g\in  \mathcal N\cap  \mathcal H$, let $E_g=\Phi^{-1}(g)$.  In this way, for each $g\in   \overline{\mathcal H}\cap \mathcal N$, the divisor $E_g$ is well-defined. 
Moreover, Proposition \ref{map-imp} implies that for each $g\in \mathcal N \cap \partial\mathcal H$ and each $v\in V$, the conformal maps $\phi_{g, v}, \psi_{g, v}$ are well-defined (as given by Remark \ref{divisor-def-g}).

Recall that $\mathcal N_\epsilon(g)$ is the $\epsilon$-neighborhood of $g$ in the polymial space $\mathcal P_d$.  
\begin{lem} \label{neigh-continuous}
	Let $D=\big((B_v^0, D_v^{\partial})\big)_{v\in V}\in  \partial_0^* {\rm Div}{(\mathbb D)}^S$ be $\mathcal H$-admissible, with $I_\Phi(D)=\{f\}$.  Then for any 
	$g\in \mathcal N \cap \partial\mathcal H$, any sequence  $\{g_n\}_{n\geq 1}\subset  \overline{\mathcal H}$ satisfying that 
	$g_n\rightarrow g$, we have  $E_{g_n}\rightarrow E_g$, and $\phi_{g_n,v}$ converges locally and uniformly on $\mathbb D$ to $\phi_{g,v}$.
\end{lem}
\begin{proof}  To show $E_{g_n}\rightarrow E_g$, 
	it suffices to consider two cases: 
	
	Case 1: $\{g_n\}_{n\geq 1}\subset \mathcal H$;
	
	Case 2: $\{g_n\}_{n\geq 1}\subset \partial\mathcal H\cap \mathcal N$.
	
	For Case 1,  by the compactness of $\partial  {\rm Div}{(\mathbb D)}^S$ and  choosing a  subsequence if necessary, we assume $E_{g_n}\rightarrow E\in \partial  {\rm Div}{(\mathbb D)}^S$.
	By Fact \ref{fact-def},   $g\in I_{\Phi}(E)$.
	It follows from Proposition \ref{map-imp} that $E=E_g$. Therefore $ E_{g_n} \rightarrow E_g$.
	
	For Case 2, by Fact \ref{fact-def}, for each $n$, there is $f_n\in \mathcal{H}\cap \mathcal N_{1/n}(g_n)$ such that $E_{f_n}=\Phi^{-1}(f_n)\in \mathbf{N}_{1/n}(E_{g_n})$. The convergence $g_n\rightarrow g$ implies that $f_n\rightarrow g$. 
	By  Case 1, we have  $ E_{f_n}\rightarrow E_g$.
	It turns out that $E_{g_n}\rightarrow  E_g$.

	To show that $\phi_{g_n,v}$ converges locally and uniformly on $\mathbb D$ to $\phi_{g,v}$,	
	it is equivalent to show that for any $\epsilon>0$ and any $\rho\in(0,1)$, 
	there is an integer $M>0$ such that $\|\phi_{g_n,v}-\phi_{g,v}\|_{\mathbb D(0,\rho)}\leq \epsilon$ for $n\geq M$, where  $\|h\|_K:=\sup_{z\in K}|h(z)|$.
	For Case 1, it follows immediately from Proposition \ref{combinatorial-property0}.
	We only need to deal with Case 2.
	In fact, for the given $\epsilon,\rho$, for each $n$,  by Proposition \ref{combinatorial-property0},  we may further require $f_n\in \mathcal N_{1/n}(g_n)\cap  \Phi(\mathbf{N}_{1/n}(E_{g_n}))\subset \mathcal H$ to satisfy that 
	$\|\phi_{g_n,v}-\phi_{f_n,v}\|_{\mathbb D(0,\rho)}\leq \epsilon/2$. The assumption $g_n\rightarrow g$ implies that $f_n\rightarrow g$.
	By  Proposition \ref{combinatorial-property0},  $\phi_{f_n,v}$ converges locally and uniformly on $\mathbb D$ to $\phi_{g,v}$. Hence there is $M=M(\epsilon, \rho)$ such that for all $n\geq M$, 
	$\|\phi_{f_n,v}-\phi_{g,v}\|_{\mathbb D(0,\rho)}\leq \epsilon/2$. It turns out that
	$$\|\phi_{g_n,v}-\phi_{g,v}\|_{\mathbb D(0,\rho)}\leq \|\phi_{g_n,v}-\phi_{f_n,v}\|_{\mathbb D(0,\rho)}+\|\phi_{f_n,v}-\phi_{g,v}\|_{\mathbb D(0,\rho)}
	\leq \epsilon.$$
	This completes the proof of the local and uniform convergence.
\end{proof}

\begin{proof}[Proof of Theorem \ref{thm:parameterization}. ]  By Proposition \ref{map-imp}, the continuous  extension  $\overline{\Phi}:\mathcal B^S\cup \mathcal A\to \mathcal H\cup \partial_\mathcal{A}\mathcal H$ is injective.  By  Lemma \ref{neigh-continuous}, the inverse $\overline{\Phi}^{-1}$ is continuous. Hence $\overline{\Phi}$ is a homeomorphism.
\end{proof}


\section{Local connectivity} \label{lc}
The aim of this part is to prove Theorem \ref{boundary-lc}:  $\partial_\mathcal{A} \mathcal H\subset \partial_{\rm LC} \mathcal H$.

Let $D\in \mathcal A$ and $f=\overline{\Phi}(D)$.  Let  $\mathcal N$ be the neighborhood  of $f$ given by  Proposition \ref{map-imp}.	
 By  Lemma \ref{neigh-continuous},  the following map
\begin{equation}  \label{pi-d}
	\pi_{D}: 
\begin{cases}  \overline{\mathcal H}\cap \mathcal N\rightarrow \partial_0^* {\rm Div}{(\mathbb D)}^S\cup    {\rm Div}{(\mathbb D)}^S,\\
	g\mapsto E_g.
\end{cases}
\end{equation}
  is continuous.  For any set $\mathbf{W}\subset \pi_D(\partial{\mathcal H}\cap \mathcal N)\subset \partial_0^* {\rm Div}{(\mathbb D)}^S$, define
 $$\pi_D^{-1}(\mathbf{W}):=\bigcup_{E\in \mathbf{W}}I_{\Phi}(E).$$

\begin{lem}\label{open-fibre}  Let $D, f, \mathcal N$ be given as above.
	
1. There is  $\delta_0>0$ such that  $\mathbf{U}_0:=\U_{\delta_0}(D)\cap \partial {\rm Div}{(\mathbb D)}^S$  is an open subset of $\partial_0^* {\rm Div}{(\mathbb D)}^S$, and $\mathbf{U}_0\subset \pi_D(\partial\mathcal H\cap \mathcal N).$

2. For any open subset $\mathbf{U}$ of $\partial_0^* {\rm Div}{(\mathbb D)}^S$ contained in  $\mathbf{U}_0$, the set
$\pi_D^{-1}(\mathbf{U})$ is an open subset of $\partial \mathcal H$. 
\end{lem}
\begin{proof}
	1.   By Proposition \ref{divisor-singleton},   $I_{\Phi}(D)=\bigcap_{\delta>0} \overline{\Phi(\N_{\delta}(D))}=\{f\}$, so we have
	$${\rm diam}(\overline{\Phi(\N_{\delta}(D))})\rightarrow 0  \text{ as } \delta\rightarrow 0.$$
Since $\mathcal N$ is a neighborhood of $f$, there is  a small number $\delta_0>0$ such that $\overline{\Phi(\N_{\delta_0}(D))}\subset \mathcal N$ and 
$\U_0:= \U_{\delta_0}(D)\cap \partial {\rm Div}{(\mathbb D)}^S\subset \partial_0^* {\rm Div}{(\mathbb D)}^S$.

For any  $E\in \U_0$,
there is  $\epsilon>0$ so that $\N_{\epsilon}(E)\subset \N_{\delta_0}(D)$. Therefore
\begin{equation}\label{inclusion0} I_{\Phi}(E)\subset \overline{\Phi(\N_{\epsilon}(E))}\subset \overline{\Phi(\N_{\delta_0}(D))}\subset \mathcal N.
	\end{equation}
It follows that $\bigcup_{E\in \U_0 }I_\Phi(E)\subset \mathcal N$.  Therefore $\mathbf{U}_0\subset \pi_D(\mathcal N \cap \partial\mathcal H)$.

2. 	 
  For any $g\in \pi_D^{-1}(\mathbf{U})$, we have $\pi_D(g)\in \U$. Since $\mathbf{U}$ is an open subset of $\partial_0^* {\rm Div}{(\mathbb D)}^S$,  there is an $\epsilon_g>0$ so that $\U_{\epsilon_g}(\pi_D(g)) \cap \partial_0^* {\rm Div}{(\mathbb D)}^S \subset \mathbf{U}$.

By the continuity of $\pi_D$, there is a neighborhood $\mathcal N_{g}$ of $g$ such that
$$\pi_D(\mathcal N_{g}\cap \partial \mathcal H)\subset \U_{\epsilon_g}(\pi_D(g))\cap \partial_0^* {\rm Div}{(\mathbb D)}^S\subset  \U.$$
It follows that $\mathcal N_{g} \cap \partial \mathcal H\subset \pi_D^{-1}(\U)$. Hence  $\pi_D^{-1}(\U)$ is open in $\partial \mathcal H$.
\end{proof}

\begin{lem} \label{compact-fibre}  Let $\U_0$ be given by  Lemma \ref{open-fibre},  and let $\mathbf{K}$ be  a compact subset of $\U_0$.  
	We have the following properties for  $\pi_D^{-1}(\mathbf{K})$:
	
	1. the set $\pi_D^{-1}(\mathbf{K})$ is a compact subset of $\partial \mathcal H$.
	
	2. $\mathbf{K}$ is connected if and only if 
$\pi_D^{-1}(\mathbf{K})$ is connected.
\end{lem}
\begin{proof} 
		Let $\epsilon_0>0$ so that the $\epsilon_0$-neighborhood of $\mathbf{K}$ is contained in
	$\U_0$. We claim 
$$\pi_D^{-1}(\mathbf{K})=\bigcap_{0<\epsilon<\epsilon_0} \overline{\Phi\big(\bigcup_{E\in \mathbf{K}} \N_{\epsilon}(E)\big)},$$
which implies that $\pi_D^{-1}(\mathbf{K})$ is compact.

In fact `$\subset$' follows   from the definition of $I_{\Phi}(E)$. Now take
$g\in \bigcap_{0<\epsilon<\epsilon_0} \overline{\Phi(\cup_{E\in \mathbf{K}} \N_{\epsilon}(E))}$. For any $\epsilon\in (0, \epsilon_0)$, we have $g\in \overline{\Phi(\cup_{E\in \mathbf{K}} \N_{\epsilon}(E))}$.
It follows that there exist $E_\epsilon \in \mathbf{K}$ and $g_\epsilon \in \Phi(\N_\epsilon(E_\epsilon))\cap \mathcal N_\epsilon(g)$.
 Then there exist a sequence $\{\epsilon_n\}_{n\geq 1}$ with $\epsilon_n\rightarrow 0^+$ and
 a divisor $E\in \partial_0 {\rm Div}{(\mathbb D)}^S$, such that
 $E_{\epsilon_n}\rightarrow E\in \partial_0 {\rm Div}{(\mathbb D)}^S$ and $g_{\epsilon_n}\rightarrow g$. By the compactness of $\mathbf{K}$, we have $E\in \mathbf{K}$. By Fact \ref{fact-def},  we get $g\in I_{\Phi}(E)\subset  \pi_D^{-1}(\mathbf{K})$.
 This show the  `$\supset$' part, hence the equality holds.


If $\mathbf{K}$ is connected, then for any $\epsilon>0$, the set $\overline{\Phi(\cup_{E\in \mathbf{K}} \N_{\epsilon}(E))}$ is a connected and compact subset of $\overline{\mathcal H}$. Since  $\{\overline{\Phi(\cup_{E\in \mathbf{K}} \N_{\epsilon}(E))}\}_{0<\epsilon<\epsilon_0}$ is  a family of compact continua,  monotonically decreasing as  $\epsilon\rightarrow 0$, their intersection which equals to $\pi_D^{-1}(\mathbf{K})$ is compact and connected. 

If $\mathbf{K}$ is disconnected, then there are two  open sets $\N_1, \N_2\subset \U_0$ such that 
$$\N_1\cap\N_2=\emptyset, \  \N_1\cap \mathbf{K}\neq \emptyset, \ \N_2\cap \mathbf{K}\neq \emptyset, \  \mathbf{K}\subset \N_1\cup \N_2.$$ 
By Lemma \ref{open-fibre},   $\pi_D^{-1}(\N_1)$ and $\pi_D^{-1}(\N_2)$ are two  open subsets of $\partial \mathcal H$, separating the compact set
$\pi_D^{-1}(\mathbf{K})$.   Hence $\pi_D^{-1}(\mathbf{K})$ is disconnected.
  \end{proof}

\begin{proof}[Proof of Theorem \ref{boundary-lc}. ]
To show the local connectivity of $\partial \mathcal H$ at $f$, we need to construct a  shrinking sequence of connected open neighborhoods of $f$.
By Lemma \ref{open-fibre}, there is $\delta_0>0$ so that $\U_0=\U_{\delta_0}(D)\cap \partial {\rm Div}{(\mathbb D)}^S\subset \pi_D(\mathcal N \cap \partial\mathcal H)\subset \partial_0^* {\rm Div}{(\mathbb D)}^S$.
 For $\rho\in(0, \delta_0)$, define
$$\mathcal V_\rho(f):=\pi_D^{-1}(\U_{\rho}(D)\cap \partial_0^* {\rm Div}{(\mathbb D)}^S).$$

Clearly $\mathcal V_\rho(f)\subset \mathcal N \cap \partial\mathcal H$.
We will show that $\{\mathcal V_\rho(f); \rho\in(0,\delta_0)\}$  is a family of open and connected neighborhoods of $f$
 on $\partial \mathcal H$, with the property that $\lim_{ \rho \rightarrow 0^+} {\rm diam}(\mathcal V_{ \rho }(f))=0$.

In fact, the openness of  $\mathcal V_\rho (f)$ follows from Lemma \ref{open-fibre}. 
The proof of the connectivity of $\mathcal V_\rho  (f)$ goes as follows. Take any two maps $g_1, g_2\in \mathcal V_\rho(f)$, it is clear that $\pi_D(g_1), \pi_D(g_2)\in  \U_{\rho}(D)\cap \partial_0^* {\rm Div}{(\mathbb D)}^S$.
Since $\U_{\rho}(D)\cap \partial_0^* {\rm Div}{(\mathbb D)}^S$ is path connected, there is a path $\gamma\subset  \U_{\rho}(D)\cap \partial_0^* {\rm Div}{(\mathbb D)}^S$, such that $\gamma(0)=\pi_D(g_1)$ and  $\gamma(1)=\pi_D(g_2)$. By  Lemma \ref{compact-fibre}, $\pi_D^{-1}(\gamma)$ is a connected and compact subset of $\mathcal V_\rho(f)$, contained in $\mathcal V_\rho(f)$ and containing $I_{\Phi}(\pi_D(g_1))\owns g_1$ as well as $I_{\Phi}(\pi_D(g_2))\owns g_2$. 
Hence $\mathcal V_\rho(f)$ is a connected subset of $\partial \mathcal H$.


Finally, we  shall prove   ${\rm diam}(\mathcal  V_\rho(f))\rightarrow 0$ as  $\rho\rightarrow 0$.  To this end, it suffices to show that $\mathcal  V_\rho(f)\subset \overline{\Phi(\N_{\rho}(D))}$.
In fact, for any $g\in \mathcal  V_\rho(f)$, we have $\pi_D(g)\in \U_{\rho}(D)  \cap \partial_0^* {\rm Div}{(\mathbb D)}^S$,
so there is a $\epsilon_g>0$ with $\N_{\epsilon_g}(\pi_D(g))\subset \N_{\rho}(D)$. Therefore
$$g\in I_{\Phi}(\pi_D(g))\subset \overline{\Phi(\N_{\epsilon_g}(\pi_D(g)))}\subset \overline{\Phi(\N_{\rho}(D))}.$$
This shows  that $\mathcal  V_\rho(f)\subset \overline{\Phi(\N_{\rho}(D))}$.  Proposition \ref{divisor-singleton} gives that
  $I_{\Phi}(D)=\bigcap_{\rho>0} \overline{\Phi(\N_{\rho}(D))}=\{f\}$,   implying that
${\rm diam}(\overline{\Phi(\N_{\rho}(D))})\rightarrow 0  \text{ as } \rho\rightarrow 0.$
Hence  $\lim_{\rho\rightarrow 0^+} {\rm diam}(\mathcal V_{\rho}(f))=0$. 
\end{proof}


To conclude this section, we discuss the continuity of a kind of {\it induced  divisors}. 

Let $\mathcal N$ be 
given by Proposition \ref{map-imp}.   For   $g\in \mathcal N\cap \overline{\mathcal H}$, recall the definition of the divisor $E_g$ before Lemma \ref{neigh-continuous}. Write   $E_g=\big((B_{g,v}, E_{g,v}^{\partial})\big)_{v\in V}$ if $g\in \mathcal N\cap \partial{\mathcal H}$; $E_g=(B_{g,v})_{v\in V}$ if $g\in \mathcal N\cap {\mathcal H}$. For each $v\in V$, recall the definition of the divisor $E_{v}^\partial (g)$   in Definition \ref{an induced divisor} (for the case $g\in  \mathcal N\cap {\mathcal H}$) and
 Remark \ref{divisor-def-g} (for the case $g\in  \mathcal N\cap \partial{\mathcal H}$).
It's clear that 
for each $v\in V$, the induced divisor $(B_{g, v}, E_{v}^\partial (g))\in 
{\rm Div}_{d-1}(\overline{\mathbb D})$.

\begin{lem} \label{induced-continuous} Given a sequence of maps $\{g_n\}_{n\geq1}\subset \mathcal N\cap \overline{\mathcal H}$,  satisfying that 
	$g_n\rightarrow g\in \mathcal N\cap \overline{\mathcal H}$, we have the convergence of divisors
	$$(B_{g_n, v}, E_{v}^\partial (g_n))\rightarrow (B_{g, v}, E_{v}^\partial (g)), \ \forall v\in V.$$
\end{lem}
\begin{proof} The case $g\in \mathcal N\cap  {\mathcal H}$ is immediate, so we may  assume $g\in \mathcal N\cap \partial{\mathcal H}$.  It suffices to treat two cases: $\{g_n\}_{n\geq1}\subset \mathcal N\cap  {\mathcal H}$ or  $\{g_n\}_{n\geq1}\subset \mathcal N\cap  \partial{\mathcal H}$.
	
	If  $\{g_n\}_{n\geq1}\subset \mathcal N\cap  {\mathcal H}$, by Lemma \ref{neigh-continuous}, we have $E_{g_n}:=\big(B_{g_n,v}\big)_{v\in V}\rightarrow E_g$, implying that $B_{g_n,v}\rightarrow (B_{g,v}, E_{g,v}^{\partial})$  for each $v\in V$. By Proposition \ref{divisor-correspondence}, $E_{v}^\partial (g_n)+E_{g, v}^\partial\rightarrow E_{v}^\partial (g)$. It follows that 
	$(B_{g_n, v}, E_{v}^\partial (g_n))\rightarrow (B_{g, v}, E_{v}^\partial (g)).$
	
	If $\{g_n\}_{n\geq1}\subset \mathcal N\cap  \partial{\mathcal H}$, the proof is essentially same as the previous case. Here is a sketch. By Lemma \ref{neigh-continuous}, we have 
$$E_{g_n}:=\big((B_{g_n,v}, E_{g_n ,v}^\partial)\big)_{v\in V}\rightarrow E_g=\big((B_{g,v}, E_{g,v}^{\partial})\big)_{v\in V},$$ implying that $(B_{g_n,v}, E_{g_n ,v}^\partial)\rightarrow (B_{g,v}, E_{g,v}^{\partial})$  for each $v\in V$. By choosing subsequences, we assume ${\rm deg}(E_{g_n ,v}^\partial)$ is independent of $n$, and   $\lim_{n\rightarrow\infty }E_{g_n ,v}^\partial=E_{\infty, v}^\partial$ (it is clear that $E_{\infty, v}^\partial\leq E_{g,v}^{\partial}$).  Then we have 
	$$B_{g_n,v}\rightarrow (B_{g,v}, E_{g,v}^{\partial}-E_{\infty, v}^\partial).$$
	
	Lemma \ref{neigh-continuous} asserts that
	$\phi_{g_n,v}$ converges locally and uniformly  to $\phi_{g,v}$.  This allows us to 
	repeat  the argument  of Proposition \ref{divisor-correspondence}, and to deduce that 
	$$E_{v}^\partial (g_n)+E_{g ,v}^\partial-E_{\infty ,v}^\partial\rightarrow E_{v}^\partial (g).$$
	Combining above convergences, we have  $(B_{g_n, v}, E_{v}^\partial (g_n))\rightarrow (B_{g, v}, E_{v}^\partial (g))$ for all $v\in V$. The proof is completed.
\end{proof}

By Proposition \ref{map-imp}, 
$\pi_D(\mathcal N\cap \overline{\mathcal H})\subset \partial_0^* {\rm Div}{(\mathbb D)}^S\cup   {\rm Div}{(\mathbb D)}^S$. Each $E=((B_{u}, E_{u}^{\partial}))_{u\in V}\in \pi_D(\mathcal N\cap \overline{\mathcal H})$  can be written as $E=E_g$ for some $g\in \mathcal N\cap \overline{\mathcal H}$. For $v\in V$, let $E_v^{\partial}(\mathcal H)=E_v^{\partial}(g)$.  By Proposition \ref{combinatorial-property-2},   $E_v^{\partial}(\mathcal H)$ does not depend on the choice of $g\in I_{\Phi}(E)$, when $E\in  \partial_0^* {\rm Div}{(\mathbb D)}^S\cap \pi_D(\mathcal N\cap \overline{\mathcal H})$.



\begin{pro} \label{divisors-continuous}
	Take $r\in (0, \delta_0)$, 
where $\delta_0$ is given by  Lemma \ref{open-fibre}. Then for any $v\in V$, the map   
	$$\eta_{\mathcal H, v}: \begin{cases}   \overline{\U_r(D)} \rightarrow  {\rm Div}_{d-1}{(\overline{\mathbb D})},\\
		E=((B_u, E_u^{\partial}))_{u\in V}\mapsto (B_v, E_v^{\partial}(\mathcal H))
	\end{cases}$$
	is continuous.
\end{pro}

\begin{proof} 
We first claim that 
	$\pi_D^{-1}(\overline{\U_r(D)} )\subset \overline{\Phi(\N_{\delta_0}(D))}$. 
 First note that 
 \begin{equation} \label{inclusion1}
 	\overline{\U_r(D)} \cap {\rm Div}{(\mathbb D)}^S\subset \N_{\delta_0}(D).
 	\end{equation}
	For any $E\in \overline{\U_r(D)} \cap \partial {\rm Div}{(\mathbb D)}^S$, there is $\epsilon>0$ with $\N_\epsilon (E)\subset \N_{\delta_0}(D)$, it follows that $I_\Phi(E)\subset  \overline{\Phi(\N_\epsilon (E))}\subset  \overline{\Phi(\N_{\delta_0}(D))}$.  Hence
	 \begin{equation} \label{inclusion2}
	 \pi_D^{-1}(\overline{\U_r(D)} \cap \partial{\rm Div}{(\mathbb D)}^S)\subset \overline{\Phi(\N_{\delta_0}(D))}.
	\end{equation}
 Combining (\ref{inclusion1}) and  (\ref{inclusion2}), we get the claim.  

	
	Take a sequence  of divisors $E_n=((B_{n,u}, E_{n,u}^{\partial}))_{u\in V} \in \overline{\U_r(D)} $  converging to $E=((B_u, E_u^{\partial}))_{u\in V} \in  \overline{\U_r(D)}$.  Choose a representative sequence $\{g_n\}_{n\geq 1}$ so that $\pi_D(g_n)=E_n$. Then $g_n\in  \pi_D^{-1}(\overline{\U_r(D)} )\subset \overline{\Phi(\N_{\delta_0}(D))}$.   Since $ \overline{\Phi(\N_{\delta_0}(D))}$ is a compact subset of $\mathcal N$
	(see  (\ref{inclusion0})),  by choosing a  subsequence, we assume $g_n\rightarrow g\in \overline{\Phi(\N_{\delta_0}(D))}\subset  \mathcal N \cap \overline{\mathcal H}$.
	
	If $g\in  \mathcal N \cap {\mathcal H}$,  then $\Phi(E)=g$. In this case, the Julia set moves holomorphically near $g$, and  it's clear that $(B_{n, v}, E_{n,v}^\partial (\mathcal H))\rightarrow (B_{v}, E_{v}^\partial (\mathcal H))$.  
	If $g\in  \mathcal N \cap \partial{\mathcal H}$,  by  Lemma \ref{neigh-continuous},  $\pi_D(g)=E$. 
	By Lemma \ref{induced-continuous},  $(B_{n, v}, E_{n,v}^\partial (\mathcal H))\rightarrow (B_{v}, E_{v}^\partial (\mathcal H))$ as $n\rightarrow \infty$. 
\end{proof}
Proposition \ref{divisors-continuous} 
 is useful in the study of the  perturbation on  $\partial\mathcal H$ in next section (see the proof of Proposition \ref{pertubation-boundary}).

\section{Perturbation on the boundary  $\partial\mathcal H$}\label{pb-boundary}

In this section, we shall establish the local  perturbation theory near an $\mathcal H$-admissible map $f\in \partial\mathcal H$. We shall show that in any arbitrarily small neighborhood of $f$,  there are  bunch of $\mathcal H$-admissible maps on each local boundary strata of the local connectedness locus $\partial_{LC} \mathcal H$.
This allows us to understand the local structure of the boundary  $\partial\mathcal H$, and to estimate the local Hausdorff dimensions (in the next section). 



\subsection{Correspondence via continuous motions}

This part establishes  a  correspondence of the hyperbolic sets between the polynomial space and  its  model space.

\begin{pro} \label{perturbation-hyp-set-general} Let $D=\big((B_u^0, D_u^{\partial})\big)_{u\in V}\in  \partial_0^* {\rm Div}{(\mathbb D)}^S$ be  $\mathcal H$-admissible,  with $I_\Phi(D)=\{f\}$.
	Let $v\in V_{\rm p}$ and $\ell\geq 1$ be the $\sigma$-period of $v$. Let $\widehat{B}^0_{v}$  be defined in Proposition \ref{convergence}. 
	Let  $\mathcal N$ be the neighborhood  of $f$
	given by  Proposition \ref{map-imp}, and let $\pi_D$ be defined in (\ref{pi-d}).
	
 Given an   $f^\ell$-hyperbolic set  $X$  on $\partial U_{f,v}$, write $\Lambda=\psi_{f,v}(X)\subset \partial \mathbb D$. Then there exist
	
	\begin{itemize}
		\item a number $\tau_0>0$, 
		a continuous motion $H: \U_{\tau_0}(D)\times \Lambda\rightarrow \partial \mathbb{D}$, with base point $D$, compatible with dynamics:
		$$\widehat{B}_v \circ H(E, \lambda)= H(E, \widehat{B}_v^0(\lambda)), \ \forall (E, \lambda)\in \U_{\tau_0}(D)\times \Lambda,$$
		where 	$E=\big((B_u, E_u^{\partial})\big)_{u\in V}\in  \U_{\tau_0}(D)$,   $\widehat{B}_v:=B_{\sigma^{\ell-1}(v)}\circ \cdots \circ B_v$;
		
		\item a neighborhood $\mathcal{N}_0\subset \mathcal N$ of $f$, and a holomorphic motion
		 $h: \mathcal{N}_0\times X\rightarrow  \mathbb{C}$, with base point $f$, compatible
		 with dynamics:
		 $$g^\ell \circ h(g,  x)= h(g, f^\ell(x)), \ \forall (g, x)\in \mathcal{N}_0\times X,$$
		\end{itemize} 
	such that $\pi_D(\mathcal N_0\cap \overline{\mathcal H})\subset {\U_{\tau_0}(D)}$, and  the following equality holds
		\begin{equation}\label{two-motions}
	H(\pi_D(g), \psi_{f,v} (x))=\psi_{g,v}\circ h(g, x), \ \forall (g, x)\in (\mathcal N_0\cap \overline{\mathcal H})\times X.
	\end{equation}
\end{pro}

The relations of the maps in Proposition \ref{perturbation-hyp-set-general} are  illustrated by 
the following  commutative diagram (here $X_g=h(g, X), \  \Lambda_{\pi_D(g)}=H(\pi_D(g), \Lambda)$):
$$
\xymatrix{
	X\ar[rd]^{h(g, \cdot)}\ar[rr]^{f^\ell}\ar[dd]_{\psi_{f,v}}&&X\ar[rd]^{h(g, \cdot)}\ar[dd]^{\psi_{f,v}}|\hole\\
	&X_g\ar[rr]^{g^\ell \ \ \ \ \ \ \ }\ar[dd]^{\psi_{g,v}}&&X_g\ar[dd]^{\psi_{g,v}}\\
	\Lambda\ar[rd]_{H(\pi_D(g), \cdot)}\ar[rr]^{ \widehat{B}_v^0\ \ \ \ \ \ }|\hole&&\Lambda\ar[rd]\ar[rd]_{H(\pi_D(g), \cdot)}\\
	&\Lambda_{\pi_D(g)}\ar[rr]_{ \widehat{B}_v}&&\Lambda_{\pi_D(g)}\\
}
$$

\begin{proof}
	 Since $X$ is $f^\ell$-hyperbolic, we have $f^k(X)\cap {\rm Crit}(f)=\emptyset $ for all $k\geq 0$. 
Since $D$ is $\mathcal H$-admissible, by Proposition \ref{h-admissible-characterization},  
$f^k(X)\cap \phi_{f, \sigma^{k}(v)} ({\rm supp}(D_{\sigma^{k}(v)}^\partial))=\emptyset$ for all $k\geq 0$.
It follows that $\Lambda\cap Z_v=\emptyset$, where $Z_v$ is given by Proposition  \ref{extension-repelling}. 
Since  $\widehat{B}_{v}^0$ is a hyperbolic map with Julia set $\partial \mathbb D$, the compact set $\Lambda$  satisfying that  $\widehat{B}_{v}^0(\Lambda)\subset \Lambda$ is  a hyperbolic set of $\widehat{B}_{v}^0$.

Let  $\rho=\rho(z)|dz|$ be an expanding metric of $f^\ell$ defined in a neighborhood $U_X$ of $X$, let  $\sigma=\sigma(\zeta)|d\zeta|$ be an expanding metric of $\widehat{B}_{v}^0$ defined  in a neighborhood $U_\Lambda$ of $\Lambda$ (see the proof of Proposition \ref{hyp-set} for the construction of these metrics).  
Note that these metrics are comparable with the Euclidean metric.
By the uniform continuity of $\psi_{f, v}: \overline{U_{f,v}}\rightarrow \overline{\mathbb D}$,  for any small $\delta>0$, there is  $r(\delta)>0$ such that for any 
$z\in \overline{U_{f,v}}\cap U_X$ and $x\in X$, 
$$ d_{\rho}(z, x)<r(\delta) \Longrightarrow d_\sigma(\psi_{f, v}(z), \psi_{f,v}(x)) <\delta/2.$$

By Proposition \ref{convergence} (to apply  it, note that $Z_v\subset W_{v}^\ell$), there exist $\tau_0>0$ and $\epsilon_0>0$ such that  
  $\widehat{B}_{v}: \C\setminus \bigcup_{q\in Z_v}\overline{ \mathbb{D}(q, \epsilon_0)}\rightarrow  \C $ defines a continuous  family of holomorphic  maps, parameterized by 
$E=((B_u, E_u^{\partial}))_{u\in V}\in\U_{\tau_0}(D)$.
By the proven fact $\Lambda\cap Z_v=\emptyset$, we may assume $\epsilon_0>0$ is small  so that 
  $\Lambda \subset \C\setminus \bigcup_{q\in Z_v} \overline{ \mathbb{D}(q, \epsilon_0)}$. By Proposition \ref{hyp-set}, and  shrinking $\tau_0$ if necessary,  there is a continuous motion $H:
	\U_{\tau_0}(D)\times \Lambda\rightarrow  \partial\mathbb{D}$, compatible with dynamics, as required in Proposition \ref{perturbation-hyp-set-general}. Moreover, as indicated by the proof of Proposition \ref{hyp-set},    there is $\delta>0$, such that 
	\begin{equation} \label{expanding-model}
		 \begin{cases} 
		|\widehat{B}'_{v}(\lambda)|_\sigma>1, \ \forall  (E, \lambda)\in \U_{\tau_0}(D)\times \Lambda^\delta, \ 
		\Lambda^\delta=\bigcup_{\lambda\in \Lambda} D_\sigma(\lambda, \delta),\\
		D_\sigma(\widehat{B}^0_{v}(\lambda), \delta)\Subset \widehat{B}_{v}(D_\sigma(\lambda, \delta)),  \ \forall  (E, \lambda)\in \U_{\tau_0}(D)\times \Lambda.
		 \end{cases} 
	\end{equation}
	

For the $f^\ell$-hyperbolic set $X$,  again by Proposition \ref{hyp-set},  there exist a neighborhood $\mathcal{N}_0\subset \mathcal N$ of $f$, and a holomorphic motion
$h: \mathcal{N}_0\times X\rightarrow  \mathbb{C}$, compatible with dynamics,
 and $r\in (0, r(\delta)]$ such that 
 \begin{equation} \label{expanding-polynomial}
 	\begin{cases} 
 		|(g^\ell)'(z)|_\rho>1, \ \forall  (g, z)\in \mathcal{N}_0\times X^r, \ 
 		X^r=\bigcup_{x\in X} D_\rho(x, r),\\
 		D_\rho(f^\ell(x), r)\Subset g^\ell(D_\rho(x, r)),  \ \forall  (g, x)\in \mathcal{N}_0\times X.
 	\end{cases} 
 \end{equation}

	By the continuity of $\pi_D$, we may shrink $\mathcal N_0$  so that $\pi_D(\mathcal N_0\cap \overline{\mathcal H})\subset {\U_{\tau_0}(D)}$.
	
	In the following, we shall show  (\ref{two-motions}).
	By Proposition \ref{holo-hyperbolic-set} and shrink $\mathcal N_0$ if necessary,  the holomorphic motion $h: \mathcal{N}_0\times X\rightarrow  \mathbb{C}$  extends to a continuous motion  $h: \mathcal{N}_0\times \bigsqcup_{x\in X} l_x\rightarrow  \mathbb{C}$ (still denoted by $h$), where $l_x$ is an arc in $\overline{U_{f,v}}$  with one endpoint $x\in X\subset \partial U_{f,v}$ and the other endpoint $a(x)\in U_{f,v}$, satisfying that $l_x\setminus \{x\}\subset U_{f,v}$ and $l_{f(x)}\subsetneqq f^\ell(l_{x})$.  
 	Note that the component $l_x(k)$ of $f^{-\ell k}(l_{f^{\ell k}(x)})$ containing $x$ shrinks down to   $x$ as $k\rightarrow \infty$. Replacing $l_x$ by $l_x(m)$ for some large $m$ (independent of $x\in X$) if necessary, we   assume
	$l_x\subset D_{\rho}(x, r/2), \ \forall  x\in X$.
Then  $\psi_{f,v}(l_x)\subset D_{\sigma}(\psi_{f,v}(x), \delta/2)$ for all $x\in X$.

	Now for each $x\in X$, the arc $l_x$ contains a  subarc $l_x^0$, with endpoints $a(x)$ and $b(x)$, where $b(x)\in l_x$ satisfies $f^\ell(b(x))=a(f^\ell (x))$. 
	The set $K=\bigsqcup_{x\in X} l_x^0$ is  a compact subset    of $U_{f, v}$.
	
	By Lemma \ref{neigh-continuous},  the map $(\mathcal N_0\cap \overline{\mathcal H})\times \mathbb D\rightarrow \mathbb C$ defined by $(g, z)\mapsto \phi_{g,v}(z)$ is continuous. 
	 By Lemma \ref{carathodory-convergence}, we may further shrink $\mathcal{N}_0$ so that $K\subset U_{g,v}$ for all $g\in  \mathcal N_0\cap \overline{\mathcal H}$. By the uniform continuities of $h|_{\mathcal N_0\times K}$ and 
 the map $(\mathcal N_0\cap \overline{\mathcal H})\times K\ni (g, z)\mapsto \psi_{g,v}(z)$,  we have 
 $$h(g, l_x^0)\subset D_{\rho}(x, r), \ \forall (g, x)  \in   (\mathcal N_0\cap \overline{\mathcal H}) \times X,$$ 
	$$d_\sigma(\psi_{g,v}(z), \psi_{f,v}(z)) <\delta/2, \ \forall (g, z)  \in   (\mathcal N_0\cap \overline{\mathcal H}) \times K.$$
	It follows that 
	$\psi_{g,v}(l^0_x)\subset D_{\sigma}(\psi_{f,v}(x), \delta)$ for all $(g, x)\in   (\mathcal N_0\cap \overline{\mathcal H}) \times X$.
	
For  $(g, x)\in  (\mathcal N_0\cap \overline{\mathcal H})\times X$ and   $n\geq 0$,   we define  two sequences of arcs
\bess 
&l_{g,n,x} =g^\ell|_{D_{\rho}(x, r)}^{-1}\circ \cdots \circ g^\ell|_{D_{\rho}(f^{\ell n}(x), r)}^{-1}(h(g, l_{f^{\ell(n+1)}(x)}^0)), &\\
&L_{g,n,x}=\widehat{B}_{v}|_{D_{\sigma}(\lambda_x, \delta)}^{-1}\circ \cdots \circ
\widehat{B}_{v}|_{D_{\sigma}((\widehat{B}_{v}^0)^n(\lambda_x), \delta)}^{-1}(\psi_{g, v}(h(g, l_{f^{\ell(n+1)}(x)}^0))),&
\eess
where $\lambda_x=\psi_{f,v}(x)$, $\pi_D(g)=((B_u, E_u^{\partial}))_{u\in V}$.

 By the relation 
$\widehat{B}_{v}\circ \psi_{g,v}=\psi_{g,v} \circ g^{\ell}$ in $U_{g,v}$,  the arc
$L_{g, 0, x}$ is a $\widehat{B}_{v}$-preimage of $\psi_{g,v}(h(g, l_{f^\ell (x)}^0))$ in $D_{\sigma}(\lambda_x, \delta)$. 
Note that  $\psi_{g,v}(l_{g,0,x})$ is also a $\widehat{B}_{v}$-preimage of $\psi_{g,v}(h(g, l_{f^\ell(x)}^0))$. 
Since 
$L_{g, 0, x}$ and $\psi_{g,v}(l_{g,0,x})$ contain  a common   point 
$$\widehat{B}_{v}|_{D_{\sigma}(\lambda_x, \delta)}^{-1}(\psi_{g, v}(h(g, a(f^\ell(x)))))=\psi_{g,v} \circ g^\ell |_{D_{\rho}(x, r)}^{-1}(h(g, a(f^\ell (x)))),$$
 which is  a $\widehat{B}_{v}$-preimage of $\psi_{g, v}(h(g, a(f^\ell(x))))$, 
we have $L_{g, 0, x}=\psi_{g,v}(l_{g,0,x}).$

Similarly, by induction, for all $n\geq 1$,
\begin{equation} \label{relation-m-d}
	L_{g, n, x}=\psi_{g,v}(l_{g,n,x}).
	\end{equation}

The expanding properties  (\ref{expanding-model}) and (\ref{expanding-polynomial}) yield the Hausdorff convergence
 $$L_{g, n, x}\rightarrow H(\pi_D(g), \lambda_x), \  l_{g,n,x}\rightarrow h(g, x), \  \text{ as } n\rightarrow \infty.$$
Letting $n\rightarrow \infty$ in (\ref{relation-m-d})
gives  $H(\pi_D(g), \psi_{f,v} (x))=\psi_{g,v}\circ h(g, x)$.
	\end{proof}

\subsection{Local perturbation  on the boundary $\partial \mathcal H$}

Let $E=((B_u, E_u^{\partial}))_{u\in V}\in  \partial_0^* {\rm Div}{(\mathbb D)}^S$, and write
 $$E=E^*+E^\partial, \ \text{ where } E^*=(B_{u})_{u\in V}, \ E^{\partial}=(E_{u}^\partial)_{u\in V}.$$
We call $E^*$ the {\it dynamical part} of $E$, $E^\partial$  the {\it boundary part} of $E$.

Recall that for each $v\in V$,  we view  $B_v$ as a map $B_v:\C_v\rightarrow \C_{\sigma(v)}$, and denote a set $X\subset \C_v$ by $X_v$.

\begin{defi}  [Forward and  inverse orbits]
 Let $\lambda\in \partial \mathbb D_v$ for some $v\in V$.  Let $(\lambda_k)_{0\leq k\leq m}$ be a finite  (i.e. $m$ is a positive integer) or infinite sequence (i.e. $m=\infty$, in this case, $(\lambda_k)_{0\leq k\leq \infty}$ means $(\lambda_k)_{k\geq 0})$. We call $(\lambda_k)_{0\leq k\leq m}$ 

\begin{itemize}
		\item
	an  {\it $E^*$-orbit (or $E^*$-forward orbit)} of $\lambda$, if $\lambda_0=\lambda, v_0=v$, and
	$$\lambda_k  \in \partial \mathbb D_{v_k},  \sigma(v_{k})=v_{k+1},   B_{v_k}(\lambda_{k})=\lambda_{k+1},  \ \forall \ 0\leq k\leq m-1.$$
	
	\item
an  {\it $E^*$-inverse orbit} of $\lambda$, if $\lambda_0=\lambda, v_0=v$,
$$v_k\in V,  \lambda_k  \in \partial \mathbb D_{v_k},  \sigma(v_k)=v_{k-1},   B_{v_k}(\lambda_k)=\lambda_{k-1},  \ \forall \ 1\leq k\leq m.$$
\end{itemize}
\end{defi}

Given $(\lambda_k)_{0\leq k\leq m}$, an  $E^*$-orbit or $E^*$-inverse orbit of $\lambda$, 
it is called {\it non-repeating} if for any $0\leq k_1<k_2\leq m$ with $v_{k_1}=v_{k_2}$, we have $\lambda_{k_1}\neq \lambda_{k_2}$;   {\it  repeating}, otherwise.  Clearly, if an inverse orbit is repeating, then $v\in V_{\rm p}$ and $\lambda$ is $\widehat{B}_v$-periodic.
A point $\lambda'\in \partial \mathbb D_u$ for some $u\in V$ is in (or meets) the $E^*$-forward/inverse orbit  $(\lambda_k)_{0\leq k\leq m}$, if $v_k=u$ and  $\lambda'=\lambda_k$ for some $0\leq k\leq m$.  Moreover,  if $\lambda'$ is in the inverse orbit, 
the {\it depth} ${\rm dep}(\lambda')$ of $\lambda'$ is the minimal $k\geq 0$ so that $v_k=u$ and $\lambda'=\lambda_k$.   

 It is clear that the collection of all non-repeating $E^*$-inverse orbits  of $\lambda$ form a tree structure, so that $\lambda$ is the root (a vertex with depth $0$), and there is an arrowed edge from  $\lambda'$ to $\lambda''$ if $B_u(\lambda')=\lambda''$ for some $u\in V$ and  two points $\lambda',\lambda''$ with adjacent depths in the  same inverse orbit.

Let $D=\big((B_v^0, D_v^{\partial})\big)_{v\in V}\in  \partial_0^* {\rm Div}{(\mathbb D)}^S$ be $\mathcal H$-admissible. 
By  Propositions \ref {h-admissible-characterization} and   \ref{divisor-singleton},  $I_{\Phi}(D)=\{f\}$ and $f$ is  Misiurewicz. Moreover, for any $u\in V$, $f$ has exactly ${\rm deg}(D_{u}^\partial)$ distinct critical points on $\partial U_{f,u}$, and we mark them by 
$c_{u,k}(f)$, $1\leq k\leq {\rm deg}(D_{u}^\partial)$.  
Let
\bess
\mathcal{I}=\{(u, k); u\in V, 1\leq k\leq  {\rm deg}(D_{u}^\partial)\}.
 \eess
For any $(u,k)\in \mathcal{I}$, the marked critical point $c_{u,k}(f)$ extends to a unique continuous map $c_{u,k}: \mathcal N_{\varepsilon_0}(f)\rightarrow \mathbb C$  
in a $\varepsilon_0$-neighborhood $\mathcal N_{\varepsilon_0}(f)$ of $f$,  so that  $c_{u,k}(g)$ is a $g$-critical point for each $g\in \mathcal N_{\varepsilon_0}(f)$.

We may further assume that $\varepsilon_0$ is small so that the graph 
\begin{equation} \label{graph-gamma-g}
\Gamma(g)=\bigcup_{u\in V}\bigcup_{\zeta\in  \partial_{u}} \overline{R_g(\theta_u^-(\zeta) )\cup R_g(\theta_u^+(\zeta))}
\end{equation}
is Hausdorff continuous in $g\in  \mathcal N_{\varepsilon_0}(f)\cap \overline{\mathcal H}$, and avoids all $g$-critical points in $\mathbb C$, here recall that $\partial_{u}$ is given by \eqref{partial-v}, and $\theta_u^{\pm}(\zeta)$ are defined above \eqref{angles=pm}.

Let $\mathcal{C}$ be a non-empty subset of $V_{\rm p}$, consisting of indices in different $\sigma$-cycles. Let 
$$ \mathcal{I}(\mathcal C)=\big\{(u,k)\in \mathcal{I}; 
\text{ the } \sigma\text{-orbit of } u \text{ meets } \mathcal C\big\}.$$

\begin{pro} [Perturbation in the same boundary strata] \label{pertubation-boundary} 
	Let  $D=\big((B_u^0, D_u^{\partial})\big)_{u\in V}\in  \partial_0^* {\rm Div}{(\mathbb D)}^S$ be $\mathcal H$-admissible,  with $I_{\Phi}(D)=\{f\}$. 
	Let $\mathcal{C}$ be a non-empty subset of $V_{\rm p}$, consisting of points in different $\sigma$-cycles.
	Each  $v\in \mathcal{C}$ is associated with a triple $(\ell_v, X_v, h_v)$, where  
	\begin{itemize}
		
		\item  $\ell_v$ is the $\sigma$-period of $v$;
		
		\item  $X_v$ is an $f^{\ell_v}$-hyperbolic set on $\partial_0 U_{f,v}$;
		
		\item  $h_v: \mathcal N_0\times X_v\rightarrow \mathbb C$ is the holomorphic motion  of the hyperbolic set $X_v$, given by  Proposition \ref {perturbation-hyp-set-general}.
		\footnote{By shrinking $\mathcal N_0$ if necessary, we assume $\mathcal N_0$ is independent of $v\in \mathcal C$.} 
		
	\end{itemize}
	
	For any $\varepsilon\in (0, \varepsilon_0)$ with $\mathcal N_{\varepsilon}(f)\subset \mathcal N_0$,   any  non-empty index set $\mathcal J\subset \mathcal{I}(\mathcal C)$ and any 
	multipoint $ \mathbf x=(x_{u,k})_{(u,k)\in \mathcal{J}}\in  \prod_{(u,k)\in   \mathcal{J}} X_{v(u)}$\footnote{The notation means that $x_{u,k}\in X_{v(u)}$, for all $(u,k)\in  \mathcal{J}$. It is possible that $x_{u,k}=x_{u',k'}$ for different pairs $(u,k), (u',k')\in  \mathcal{J}$.}, where  $v(u)\in \mathcal{C}$ is the unique index in the $\sigma$-orbit of $u$,   there exist 
	a map 
	$g\in 
	\partial \mathcal H\cap  \mathcal N_{\varepsilon}(f)$, and positive integers $(m_{u,k})_{(u,k)\in  \mathcal{J}}$,
	such that the following statements hold: 
	\begin{itemize}
		\item for any $(u,k)\in  \mathcal{J}$,  the critical point $c_{u,k}(g)\in \partial U_{g, u}$, and  $$g^{m_{u,k}}(c_{u,k}(g))=h_{v(u)}(g,  x_{u,k});$$ 
		
		
		\item for any  $(u,k)\in \mathcal{I}\setminus \mathcal{J}$,   the critical point  $c_{u,k}(g)\in \partial U_{g, u}$;
		
		\item $\pi_D(g)$ is $\mathcal H$-admissible.
	\end{itemize}
\end{pro}

\begin{proof}  Without loss of generality,  we assume $\mathcal C$ is maximal in the sense that 
	 $\mathcal{I}(\mathcal C)= \mathcal{I}$
 (in fact, if $V_{\rm p}\setminus \mathcal{I}(\mathcal C)\neq \emptyset$,  we may choose a representative $v$ in  each  $\sigma$-cycle in  $V_{\rm p}\setminus \mathcal{I}(\mathcal C)$,    and let $X_v$  consist of an $f^{\ell_v}$-repelling point on $\partial_0 U_{f,v}$), and  it suffices to consider the extremal case 
	$\mathcal J=\mathcal{I}(\mathcal C)=\mathcal{I}$.
	We  may further assume	$\mathbf x=(x_{u,k})_{(u,k)\in \mathcal{J}}$ satisfies the following property:  if $x_{u,k}$ and $x_{u',k'}$ (here $(u,k)\neq (u',k')$) are in the same $f$-periodic orbit, then $x_{u,k}=x_{u',k'}$.
	
	  For each $v\in\mathcal C$, let $\Lambda_v=\psi_{f,v}(X_v)$, and let	$H_v: \U_{\tau_0}(D)\times \Lambda_v\rightarrow \partial \mathbb{D}$ be the continuous motion given by Proposition \ref{perturbation-hyp-set-general}.  For each $u=\sigma^s(v)$ with $1\leq s<\ell_v$, let $\Lambda_u=B^0_{\sigma^{s-1}(u)}\circ \cdots\circ B^0_{u}(\Lambda_{v})$, and let 
	  $H_u: \U_{\tau_0}(D)\times \Lambda_u\rightarrow \partial \mathbb{D}$ be the continuous motion satisfying that  for any 
	  $E=((B_u, E^\partial_u))_{u\in V}\in {\U_{\tau_0}(D)}$,
	   $$H_u(E, B^0_{\sigma^{s-1}(v)}\circ \cdots\circ B^0_{v}(x))=
	  B_{\sigma^{s-1}(v)}\circ \cdots\circ B_{v}(H_v(E, x)), \ x\in \Lambda_v.$$
	  
	The assumption $X_v\subset \partial_0 U_{f,v}$ implies  that
	$\Lambda_u\cap {\rm supp}(D_u^\partial(\mathcal H))=\emptyset$ for all $u=\{v, \cdots, \sigma^{\ell_v-1}(v)\}$.
	By the Hausdorff continuity of $\U_{\tau_0}(D)\ni E\mapsto H_u(E, \Lambda_u)$ and  	Proposition \ref{divisors-continuous},  
we may choose $r_0\in (0,\min\{\tau_0,\delta_0\})$ (here   $\delta_0$ is given by  Lemma \ref{open-fibre}) so that 
	\begin{equation} \label{disjoint-o1} H_u(E, \Lambda_u)\cap {\rm supp}(E_u^\partial(\mathcal H))=\emptyset, \ \forall E\in \U_{r_0}(D),  \forall u\in V_{\rm p}. 
	\end{equation}

	The proof consists of three steps.
	
	\vspace{5pt}
	{\textbf{Step 1. Construction of a small neighborhood in model space.} }
	\vspace{5pt}
	
	By Proposition \ref{divisor-singleton},   $I_{\Phi}(D)=\bigcap_{r>0}\overline{\Phi(\N_r(D))}=\{f\}$.
	This implies that 
	$${\rm diam}(\overline{\Phi(\N_r(D))})\rightarrow 0 \text{ as  } r\rightarrow 0.$$
Hence for  any $\varepsilon>0$, there is $r=r(\varepsilon)\in (0, r_0)$  so that
	 $${\rm diam}(\overline{\Phi(\N_r(D))})< \varepsilon\Longrightarrow \overline{\Phi(\N_r(D))}\subset  \mathcal N_\varepsilon(f).$$
	We   assume $\varepsilon$ is small so that $\overline{\mathcal N_\varepsilon(f)}\subset \mathcal N_0\subset \mathcal N$ (recall that $\mathcal N_0, \mathcal N$ are given by   Propositions \ref {perturbation-hyp-set-general} and \ref{map-imp},  respectively).

	\vspace{5pt}
	{\textbf{Step 2. Construction of a suitable divisor $D_0\in  \U_{r}(D)$.}}
	\vspace{5pt}

	 Write $q_{u,k}=\psi_{f, u}(c_{u,k}(f)), (u,k)\in \mathcal{I}$; 
	$\lambda_{u,k}=\psi_{f, v(u)}(x_{u,k})\in \Lambda_{v(u)}$, $(u,k)\in  \mathcal{J}$.
	Since $D$ is $\mathcal H$-admissible, by Proposition   \ref{h-admissible-characterization},  for each $u\in V$, 
	$$\partial U_{f,u}\cap {\rm Crit}(f)=\phi_{f,u}({\rm supp}(D_u^\partial)),  \ D_u^\partial=\sum_{1\leq k\leq  {\rm deg}(D_{u}^\partial)}1\cdot q_{u,k}.$$
	
	For any  $v\in \mathcal{C}$,  since $X_v$
	is an   $f^{\ell_v}$-hyperbolic set, we get $X_v\cap   {\rm Crit}(f)=\emptyset$, implying that $\Lambda_{v}\cap {\rm supp}(D_v^\partial)=\emptyset$.
 We   assume  $r$ in Step 1 is small  so that 
	\begin{itemize}
		\item
		for all $v\in \mathcal{C}$,  $\Lambda_{v}\cap \big(\bigcup_{q\in {\rm supp}(D_v^\partial)} \overline{\mathbb D(q, r)}\big)=\emptyset$;
		
		\item for any $u\in V$,  the disks $\{\mathbb D(q, r)\}_{q\in {\rm supp}(D_u^\partial)}$ have disjoint closures.
	\end{itemize}

	Let $u\in V$.  Suppose $ \sigma^s(u)=v\in V_{\rm p}$ for some  minimal integer $s\geq 0$. 
	For  any $\lambda\in \partial \mathbb D_v$, we define
	$$\Lambda^\infty_{u}(\lambda)= \begin{cases}    \bigcup_{l\geq 0}(\widehat{B}_{v}^0)^{-l}(\lambda),  &\text{ if }s=0,\\
		\big(B^0_{\sigma^{s-1}(u)}\circ \cdots\circ B^0_{u}\big)^{-1}\big(\Lambda^\infty_{v}(\lambda)\big), &\text{ if }  s>0. \end{cases}$$

		
		The set $\Lambda^\infty_{u}(\lambda)$ 
		is dense in $\partial \mathbb D_u$, because
		the mapping degree  ${\rm deg}(\widehat{B}_{v}^0)\geq 2$.

		\begin{figure}[h]   
			\begin{center}
				\includegraphics[height=4.5cm]{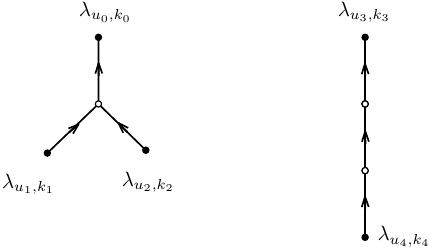}
			\end{center}
			\caption{Tree structure for the multipoint $(\lambda_{u_j, k_j})_{0\leq j\leq 4} $. In this case,   $B^0_{\sigma(u_1)}\circ B^0_{u_1}(\lambda_{u_1,k_1})=\lambda_{u_0,k_0}$, $B^0_{u_1}(\lambda_{u_1,k_1})=B^0_{u_2}(\lambda_{u_2,k_2})$, $B^0_{\sigma^2(u_4)}\circ B^0_{\sigma(u_4)}\circ B^0_{u_4}(\lambda_{u_4,k_4})=\lambda_{u_3,k_3}$. 
		Note that $(\widehat{u}_0,\widehat{k}_0)=(u_1,k_1)$ or $(u_2, k_2)$, 	$(\widehat{u}_1,\widehat{k}_1)=(u_1,k_1)$, $(\widehat{u}_2,\widehat{k}_2)=(u_2,k_2)$,
		$(\widehat{u}_3,\widehat{k}_3)=(\widehat{u}_4,\widehat{k}_4)=(u_4,k_4)$.}	
			\label{fig: multipoint-tree}
		\end{figure}

		The multipoint $(\lambda_{u,k})_{(u,k)\in  \mathcal{J}}\in \prod_{(u,k)\in   \mathcal{J}} \partial \mathbb D_{v(u)}$   can be viewed as a disjoint union of finitely many finite trees: if $\lambda_{u',k'}$ is in a $D^*$-inverse orbit of $\lambda_{u,k}$, with depth $m\geq 0$, then there is a branch of length $m$ from $\lambda_{u,k}$ to $\lambda_{u',k'}$. 
		 In this way, we get a tree structure, see Figure \ref{fig: multipoint-tree}. 
		For each $(u,k)\in  \mathcal{J}$, let $(\widehat{u},\widehat{k})\in  \mathcal{J}$ be  such that $\lambda_{\widehat{u},\widehat{k}}$ is in the  $D^*$-inverse  orbit of $\lambda_{u,k}$ with maximal depth. Note that this $(\widehat{u},\widehat{k})$ may be not unique, and  we  choose one of them;  we set $(\widehat{u},\widehat{k})=({u},{k})$ if there is no  $\lambda_{u',k'}$ with $(u',k')\neq (u,k)$ in the $D^*$-inverse orbits of $\lambda_{u,k}$. 
		
		Take $(u_0, k_0)\in   \mathcal{J}$, by the density of  $\Lambda^\infty_{u_0}(\lambda_{\widehat{u}_0,\widehat{k}_0})$   in $\partial \mathbb D_{u_0}$, there is  $\zeta_{u_0,k_0}\in  \mathbb D(q_{u_0,k_0}, r) \cap \Lambda^\infty_{u_0}(\lambda_{\widehat{u}_0,\widehat{k}_0})$
		so that $\zeta_{u_0,k_0}$ is in a non-repeating  $D^*$-inverse orbit of $\lambda_{\widehat{u}_0,\widehat{k}_0}$ with depth $\geq L_0$, where
		$$L_0=\max_{v\in \mathcal C}\big\{\ell_{v}\big\} \bigg(\Big[\frac{\log\# \mathcal J}{\log 2}\Big]+2\bigg), \ [x] \text{ is the integral part of } x\in \mathbb R.$$
		
		Initially, set $ \mathcal{J}'=\{(u_0, k_0)\}$.  
		Assume by induction that $ \mathcal{J}'$ is given so that for each $(u', k')\in \mathcal J'$,  $\zeta_{u', k'}$ is   in a non-repeating  $D^*$-inverse orbit of $\lambda_{\widehat{u'},\widehat{k'}}$ with depth $\geq L_0$.
		
		For  $(u,k)\in  \mathcal{J}\setminus  \mathcal{J}'$, we choose $\zeta_{u,k}\in  \mathbb D(q_{u,k}, r) \cap \Lambda^\infty_{u}(\lambda_{u,k})$ as follows: since  ${\rm deg}(\widehat{B}_{\widehat{u}}^0)\geq 2$, there is $y_{u, k}\in (\widehat{B}_{\widehat{u}}^0)^{-1}(\lambda_{\widehat{u},\widehat{k}})$
		which is not $\widehat{B}_{\widehat{u}}^0$-periodic.  	Let $m\geq 1$ be the minimal integer with $2^m\geq \# \mathcal J$, then
		$$m=\Big[\frac{\log\# \mathcal J}{\log 2}\Big]+1, \    \ \#(\widehat{B}_{\widehat{u}}^0)^{-m}(y_{u, k})={\rm deg}(\widehat{B}_{\widehat{u}}^0)^m\geq 2^m\geq   \# \mathcal J.$$
		
	Note that each point $a\in (\widehat{B}_{\widehat{u}}^0)^{-m}(y_{u, k})$ is in a $D^*$-inverse orbit of $\lambda_{\widehat{u},\widehat{k}}$,  with depth $(m+1)\ell_{\widehat{u}}\leq L_0$. 
			If $\# \mathcal J'<\# \mathcal J$, by the induction hypothesis,  there is  $a_{u,k}\in (\widehat{B}_{\widehat{u}}^0)^{-m}(y_{u, k})$ whose inverse orbits avoid all $\zeta_{u', k'}, (u', k')\in \mathcal{J}'$. Choose 
		$$\zeta_{u,k}\in  \mathbb D(q_{u,k}, r) \cap \Lambda^\infty_{u}(a_{u,k})\subset   \mathbb D(q_{u,k}, r) \cap \Lambda^\infty_{u}(\lambda_{\widehat{u},\widehat{k}}) \subset  \mathbb D(q_{u,k}, r) \cap \Lambda^\infty_{u}(\lambda_{u,k})$$
		whose depth in the $D^*$-inverse orbit of $\lambda_{\widehat{u},\widehat{k}}$  is  $\geq L_0$.   
		We may verify that 
		
			(a). for any  $(u',k')\in \mathcal{J}'$,   $\zeta_{u,k}$ does not meet the finite
		orbit
		$$\zeta_{u',k'}\overset{B^0_{u'}}{\longrightarrow}  \zeta_{u', k'}^1  \overset{B^0_{\sigma(u')}}{\longrightarrow} \zeta_{u', k'}^2  \overset{B^0_{\sigma^2(u')}}{\longrightarrow} \cdots
		{\longrightarrow}   \zeta_{u', k}^{{\rm dep}(\zeta_{u',k'})}=\lambda_{u',k'}.$$

		(b).  for any $(u',k')\in \mathcal{J}'$,  $\zeta_{u',k'}$ does not meet the 
		orbit
		$$\zeta_{u,k}\overset{B^0_{u}}{\longrightarrow}  \zeta_{u, k}^1  \overset{B^0_{\sigma(u)}}{\longrightarrow} \zeta_{u, k}^2  \overset{B^0_{\sigma^2(u)}}{\longrightarrow} \cdots
		{\longrightarrow}   \zeta_{u, k}^{{\rm dep}(\zeta_{u,k})}=\lambda_{u,k}.$$

		We add $(u,k)$ to $\mathcal{J}'$.   The inductive procedure terminates when $\mathcal{J}'=\mathcal{J}$. 
		
		Let
		\bess
			&D_{u,0}^\partial= \sum_{1\leq k\leq {\rm deg}(D_{u}^\partial)}1\cdot \zeta_{u,k},  \ u\in  V;&\\
			&D_0=D^*+ (D_{u,0}^\partial)_{u\in V}=\big((B_u^0, D_{u,0}^{\partial})\big)_{u\in V}.&
			\eess
		For each $(u,k)\in 
	\mathcal J$, note that $\zeta_{u,k}$ is in the $D^*$-inverse orbit of $\lambda_{u,k}$, let $m_{u,k}$ be the depth of $\zeta_{u,k}$. Thus we get the positive integers $(m_{u,k})_{(u,k)\in \mathcal J}$.

		\vspace{5pt}
		{\textbf{Step 3.   $D_0$  is $\mathcal H$-admissible.}}
		\vspace{5pt}


		Note that $D_0\in \U_{r}(D)\subset \U_{r_0}(D)$.
		By  equation  (\ref{disjoint-o1}), 
			for any $(u,k)\in \mathcal J$, and any $l\geq m_{u,k}$,
		\begin{equation}  \label{adm-h-adim}
			B_{\sigma^{l-1}(u)}^0\circ \cdots \circ B_{\sigma(u)}^0\circ B_{u}^0(\zeta_{u,k})\notin 
			{\rm supp}((D_0)_{\sigma^l(u)}^\partial(\mathcal H)).
		\end{equation}

	
	
	
		Let $l_{u,k}\in [1, m_{u,k}]$ be minimal  so that (\ref{adm-h-adim}) holds for $l\geq l_{u,k}$. 
		To show   $D_0$ is  $\mathcal H$-admissible, it is equivalent to show  $l_{u,k}=1$ for all $(u,k)\in \mathcal J$.
		
		Recall the graph $\Gamma(g)$ in (\ref{graph-gamma-g}).  Let $W_{g,u}$ be the component of $\mathbb C\setminus \Gamma(g)$ containing $U_{g,u}$. By the assumption that $\Gamma(g)\cap {\rm Crit}(g)=\emptyset$ for  $g\in  \mathcal N_{\varepsilon_0}(f)\cap \overline{\mathcal H}$, 
	  for each $(u,k)\in \mathcal J$ and each $g\in I_{\Phi}(D_0)\subset \mathcal N_{\varepsilon_0}(f)\cap \overline{\mathcal H}$,  there is exactly one critical point in $L_{U_{g, u}, \phi_{g,u}(\zeta_{u,k})}\cap W_{g,u}.$ Hence there are exactly 
		$\sum_{u\in V}{\rm deg}(D_{u}^\partial)$ critical points in $\bigcup_{(u,k)\in \mathcal J}(L_{U_{g, u}, \phi_{g,u}(\zeta_{u,k})}\cap W_{g,u})$.

	 If $l_{u,k}>1$ for some $(u,k)\in \mathcal J$, by Proposition \ref{para-dym}, for any $g\in I_{\Phi}(D_0)$, 
	 $\phi_{g,\sigma^{l_{u,k}-1}(u)}(B_{\sigma^{l_{u,k}-2}(u)}^0\circ \cdots \circ B_{\sigma(u)}^0\circ B_{u}^0(\zeta_{u,k}))$ is a $g$-critical point. This critical point is not in $\bigcup_{(u,k)\in \mathcal J}(L_{U_{g, u}, \phi_{g,u}(\zeta_{u,k})}\cap W_{g,u})$ because $D_0$ satisfies the properties (a)(b)(in Step 2). Hence it is an additional critical point, which is  a contradiction.  This implies  $l_{u,k}=1$ for all $(u,k)\in \mathcal J$, and $D_0$ is  $\mathcal H$-admissible.  	By Proposition \ref{divisor-singleton}, $ I_{\Phi}(D_0)$ consists of a singleton $g\in 
	 \partial \mathcal H\cap  \mathcal N_{\varepsilon}(f)$.  By the choice of $D_0$ and Proposition \ref{h-admissible-characterization}, the map $g$ has required properties. 
		\end{proof}
 
 \begin{rmk} \label{key-property} In the proof of Proposition \ref{pertubation-boundary}, the choice of $\zeta_{u,k}$ (with a lower bound $L_0$ of depth)  in  Step 2 is a little bit technical, but  it is crucial to guarantee the properties (a) and (b).   Here is an example to illustrate why a  lower bound on the depth is necessary.  Assume
 	$$v(u)\equiv v, \lambda_{u,k}\equiv \lambda,  \   \forall (u,k)\in \mathcal J.$$
Let $(u_0,k_0)\in \mathcal J$.  Assume each  $D^*$-inverse orbit of $\lambda$ meets a point in $\{q_{u,k}; (u,k)\in \mathcal J\setminus \{(u_0,k_0)\}\}$.  For $(u,k)\in \mathcal J\setminus \{(u_0,k_0)\}$,
we  take $\zeta_{u,k}=q_{u, k}$ (without a lower bound on depth).    No matter how to choose $\zeta_{u_0, k_0}$, the   properties (a) or (b) can not be satisfied. See Figure \ref{fig: tree-no-depth}

	\begin{figure}[h]   
	\begin{center}
		\includegraphics[height=4.5cm]{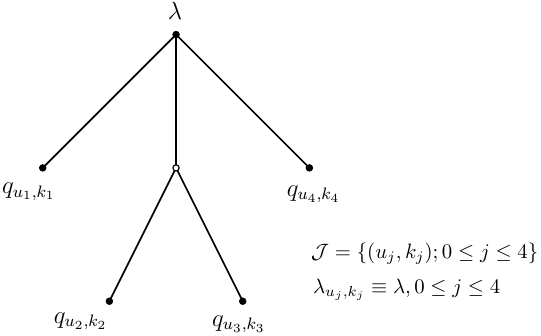}
	\end{center}
	\caption{In this case, each $D^*$-inverse orbit of  $\lambda$ meets a point in $\{q_{u_j, k_j}; 1\leq j\leq 4\}$. No matter how to choose $\zeta_{u_0, k_0}$, either the $D^*$-orbit of $\zeta_{u_0, k_0}$ meets some $q_{u_j, k_j}$, or   the $D^*$-orbit of some $q_{u_j, k_j}$ meets $\zeta_{u_0, k_0}$.}	
	\label{fig: tree-no-depth}
\end{figure}
 \end{rmk}

\begin{pro} [Perturbation to arbitrary boundary strata] \label{pertubation-boundary-strata} 
	Let  $D=\big((B_u^0, D_u^{\partial})\big)_{u\in V}\in  \partial_0^* {\rm Div}{(\mathbb D)}^S$ be $\mathcal H$-admissible,  with $I_{\Phi}(D)=\{f\}$. 
	Let $\mathcal{C}$ be a non-empty subset of $V_{\rm p}$, consisting of points in different $\sigma$-cycles.
	Each  $v\in \mathcal{C}$ is associated with a triple $(\ell_v, X_v, h_v)$, where  
	\begin{itemize}
		
		\item  $\ell_v$ is the $\sigma$-period of $v$;
		
		\item  $X_v$ is an $f^{\ell_v}$-hyperbolic set on $\partial_0 U_{f,v}$;
		
		\item  $h_v: \mathcal N_0\times X_v\rightarrow \mathbb C$ is the holomorphic motion  of the hyperbolic set $X_v$, given by  Proposition \ref {perturbation-hyp-set-general}.
		\footnote{By shrinking $\mathcal N_0$ if necessary, we assume $\mathcal N_0$ is independent of $v\in \mathcal C$.} 
		
	\end{itemize}
	
	Given any $\varepsilon\in (0, \varepsilon_0)$ with $\mathcal N_{\varepsilon}(f)\subset \mathcal N_0$,   any  index set $\emptyset\neq \mathcal J\subset \mathcal{I}(\mathcal C)$, and any 
	multipoint $ \mathbf x=(x_{u,k})_{(u,k)\in \mathcal{J}}\in  \prod_{(u,k)\in   \mathcal{J}} X_{v(u)}$, where  $v(u)\in \mathcal{C}$ is the unique index in the $\sigma$-orbit of $u$,
	there exist 
		$g_0\in 
	\partial \mathcal H\cap  \mathcal N_{\varepsilon}(f)$ and  positive integers $(m_{u,k})_{(u,k)\in  \mathcal{J}}$,
so that 
	\begin{itemize}
		\item   $g_0^{m_{u,k}}(c_{u,k}(g_0))=x_{u,k}(g_0):=h_{v(u)}(g_0,  x_{u,k})$ for all $(u,k)\in \mathcal J$;
	
		\item  $c_{u,k}(g_0)\in U_{g_0, u}$ for all  $(u,k)\in \mathcal{I}\setminus \mathcal{J}$;
		
		\item $\pi_D(g_0)$ is $\mathcal H$-admissible.
	\end{itemize}
\end{pro}


	\begin{figure}[h]   
	\begin{center}
		\includegraphics[height=4cm]{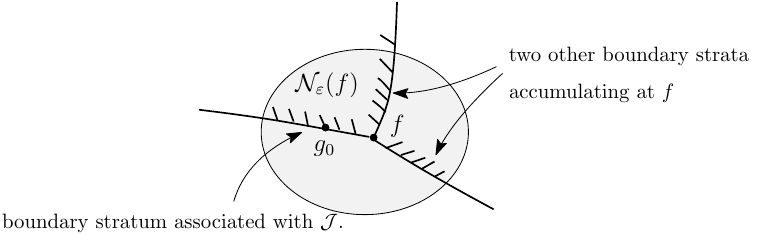}
	\end{center}
	\caption{Some boundary strata accumulating at $f$. }
	\label{fig: boundary-strata}
\end{figure}

\begin{proof}
	 For  any given  $\varepsilon\in (0, \varepsilon_0)$,  by Proposition \ref{divisor-singleton}, there is $r=r(\varepsilon)$  so that
	$${\rm diam}(\overline{\Phi(\N_r(D))})< \varepsilon\Longrightarrow \overline{\Phi(\N_r(D))}\subset  \mathcal N_\varepsilon(f).$$

	We may assume $\mathcal C$ is maximal in the sense that $\mathcal{I}(\mathcal C)=\mathcal{I}$, and we extend the multipoint $\mathbf x$ to $\widehat{\mathbf x}=(x_{u,k})_{(u,k)\in \mathcal{I}}\in  \prod_{(u,k)\in   \mathcal{I}} X_{v(u)}$ (as we did in the proof of Proposition \ref{pertubation-boundary}).
	By Proposition \ref{pertubation-boundary}, there exist  $f_1\in 
	\partial \mathcal H\cap  \mathcal N_{\varepsilon}(f)$,   positive integers $(m_{u,k})_{(u,k)\in  \mathcal{I}}$,
	such that  
	\begin{itemize}
		\item for any $(u,k)\in  \mathcal{I}$,  the critical point $c_{u,k}(f_1)\in \partial U_{f_1, u}$, and  $$f_1^{m_{u,k}}(c_{u,k}(f_1))=h_{v(u)}(f_1,  x_{u,k});$$ 
		

		\item $D_1:=\pi_D(f_1)$ is $\mathcal H$-admissible.
	\end{itemize}
Replacing $(f,D)$ by $(f_1,D_1)$, we may assume $f=f_1, D=D_1$.

For each $v\in\mathcal C$, let $\Lambda_v=\psi_{f,v}(X_v)$.
Write $q_{u,k}=\psi_{f, u}(c_{u,k}(f))$, 
$\lambda_{u,k}=\psi_{f, v(u)}(x_{u,k})\in \Lambda_{v(u)}$ for $(u,k)\in  \mathcal{I}$.

For any index set   $\mathcal J\subset \mathcal{I}$ with $\mathcal J\neq \emptyset$,   any $u\in V$, let  
$$\mathcal J_u=\{(u',k')\in \mathcal J; u'=u\}, \ \mathcal I_u=\{(u',k')\in \mathcal I; u'=u\}.$$
Write the $D_u^\partial$-factor of $D=\big((B_u^0, D^\partial_u)\big)_{u\in V}$ as
$D_u^\partial=D_{u, \mathcal J}^{\partial}+F_{u, \mathcal J}^{\partial}$,
where 
$$D_{u, \mathcal J}^\partial=\sum_{(u',k')\in \mathcal J_u} 1\cdot q_{u',k'}, \ \ F_{u, \mathcal J}^\partial=\sum_{(u',k')\in  \mathcal I_u\setminus \mathcal J_u} 1\cdot  q_{u',k'}.$$ 
It's clear that if $\mathcal J=\mathcal I$, then $D_{u, \mathcal J}^\partial=D_u^{\partial}$, $F_{u, \mathcal J}^\partial=0$.
 
For any  $\tau\in (0,r]$ so that $\U_\tau(D)\cap \partial {\rm Div}{(\mathbb D)}^S\subset  \partial_0^* {\rm Div}{(\mathbb D)}^S$,  define 
 $$\U_\tau^{\mathcal J}(D)=\left\{\big((B_u, E^\partial_u)\big)_{u\in V}\in {\U_{\tau}(D)};  \ E^\partial_u\in \U^0_\tau(D_{u, \mathcal J}^{\partial}), \
 \forall u\in V
 \right\}.$$
 Note that for $((B_u, E^\partial_u)\big)_{u\in V}\in \U_\tau^{\mathcal J}(D)$,  we have $\ {\rm deg}(E^\partial_u)=\#\mathcal J_u, \ \forall u\in V$.
View $\U_\tau^{\mathcal J}(D)$  as a boundary stratum in the model space.
For any sequence $E_n=E_n^*+E_n^\partial=((B_{n,u}, E_{n,u}^\partial))_{u\in V}\in \U_{\tau}^{\mathcal J}(D)$ with 
$E_n\rightarrow D$, 
we have 
\bess
& E_n^\partial=(E_{n,u}^\partial)_{u\in V}\rightarrow (D^\partial_{u, \mathcal J})_{u\in V}, & \\
&E_n^*=(B_{n,u})_{u\in V}\rightarrow \big((B_u^0, F_{u, \mathcal J}^\partial)\big)_{u\in V}.&
\eess

By Lemma \ref{composition-law} and Proposition \ref{convergence},  for any $u\in V$, any integer $m\geq 1$,  $$B_{n,\sigma^{m-1}(u)}\circ \cdots\circ B_{n,u}\rightarrow 
B^0_{\sigma^{m-1}(u)}\circ \cdots\circ B^0_{u},  \ \text{ as } n\rightarrow \infty$$
locally and uniformly in $\C\setminus Z_{u,\mathcal J}^m$, where  $Z_{u,\mathcal J}^1={\rm supp}(F_{u, \mathcal J}^\partial)$ and
$$Z_{u,\mathcal J}^m={\rm supp}(F_{u, \mathcal J}^\partial)\bigcup 
\bigcup_{l=1}^{m-1}\big(B^0_{\sigma^{l-1}(u)}\circ \cdots \circ B^0_u\big)^{-1}
\big({\rm supp}(F_{\sigma^{l}(u), \mathcal J}^{\partial})\big), \ m\geq 2.$$
Note that $Z_{u,\mathcal J}^{m}\subset Z_{u,\mathcal J}^{m+1}$.
 Since $D$ is $\mathcal H$-admissible,     $Z_{u,\mathcal J}^m\cap {\rm supp}(D_{u, \mathcal J}^\partial)=\emptyset$.



Let $Z_{u,\mathcal J}^m(t)=\bigcup_{\zeta\in  Z_{u,\mathcal J}^m}\mathbb D(\zeta, t)$ for 
$t>0$.  Let $m_u=\max_{(u,k)\in \mathcal J_u} m_{u,k}$.  Choose small $t_0>0$ so that 
$$ \mathbb D(q_{u,k}, t_0)\subset \C\setminus Z_{u,\mathcal J}^{m_u}(t_0),  \  \forall (u,k)\in \mathcal J.$$

Let $\tau$ be small. 
Then for any  $(u,k)\in  \mathcal{J}$,  the map $R_{u,k}^{E}: \C\setminus \overline{Z_{u,\mathcal J}^{m_u}(t_0)}\rightarrow \C$ given by  
$$R_{u,k}^{E}= B_{\sigma^{m_{u,k}-1}(u)}\circ \cdots\circ B_{u}, \ \forall \ E=((B_{u}, E_{u}^\partial))_{u\in V}\in   \U_{\tau}^{\mathcal J}(D)\cup\{D\},$$
 defines
	a continuous family of holomorphic maps, parameterized by $E\in    \U_{\tau}^{\mathcal J}(D)\cup\{D\}$.
	Note that $R_{u,k}^{E}$  depends only on the $E^*$-part of $E$, and is independent of $E^{\partial}$. Because of this,  we write $R_{u,k}^{E}$ as $R_{u,k}^{E^*}$.

	For each $v\in\mathcal C$,  let	$H_v: \U_{\tau_0}(D)\times \Lambda_v\rightarrow \partial \mathbb{D}$ be the continuous motion given by Proposition \ref{perturbation-hyp-set-general}.  The value $H_v(E, \lambda)$ depends only on the $E^*$-part of $E$, hence is also written as  $H_v(E^*, \lambda)$.
	
	Note that $R_{u,k}^{D^*}(q_{u,k})=\lambda_{u,k}$.
	By the continuity of   $R_{u,k}^{E^*}|_{\mathbb D(q_{u,k}, t_0)}: \mathbb D(q_{u,k}, t_0)\rightarrow \C$ in  $E\in \U_{\tau}^{\mathcal J}(D)\cup\{D\}$, there exist $s>0$ and  $\tau_1\in (0, \min\{\tau_0,\tau\}]$  so that
	 for any $(u,k)\in  \mathcal{J}$ and any $E\in  \U_{\tau_1}^{\mathcal J}(D)\cup\{D\}$, 
	 \begin{itemize}
	 	\item $R_{u,k}^{E^*}(q_{u,k})\in \mathbb D(H_{v(u)}(E^*,\lambda_{u,k}), s)$;
	 	
	 	\item	the  component $V_{u,k}^s(E^*)$ of 
	$(R_{u,k}^{E^*})^{-1}(\mathbb D(H_{v(u)}(E^*,\lambda_{u,k}), s))$ containing $q_{u, k}$  is a Jordan disk in $\mathbb D(q_{u,k}, t_0)$; 
	\item  $R_{u,k}^{E^*}: V_{u,k}^s(E^*)\rightarrow  \mathbb D(H_{v(u)}(E^*,\lambda_{u,k}),s)$ is conformal.  
	\end{itemize}
	
	By shrinking $s, \tau_1$ if necessary, we further assume
	\begin{itemize}
		\item
	 for any $u\in V$, the  disks $\{V_{u,k}^s(E^*)\}_{(u,k)\in \mathcal J_u}$ have disjoint closures.
	
	\item  for any $(u,k)\in  \mathcal{J}$ and any integer $1\leq j\leq m_{u,k}$,  the   orbit
	$$B_{u}(V_{u,k}^s(E^*))  \overset{B_{\sigma(u)}}{\longrightarrow}  B_{\sigma(u)}\circ B_{u}(V_{u,k}^s(E^*))   \overset{B_{\sigma^2(u)}}{\longrightarrow} \cdots
	{\longrightarrow}   \mathbb D(H_{v(u)}(E^*,\lambda_{u,k}), s),$$
	when restricted to $\C_{\sigma^j(u)}$,  avoids the disks   $\{V_{u',k'}^s(E^*)\}_{(u',k')\in \mathcal J_{\sigma^j(u)}}$. 
	\end{itemize}

 
 	 
 

Choose $E_0=E_0^*+E_0^\partial=(B_{u,0})_{u\in V}+(E^\partial_{u,0})_{u\in V}\in  \U_{\tau_1}^{\mathcal J}(D)$ 
so that
$$E_{u,0}^\partial=\sum_{(u,k)\in   \mathcal J_u} 1\cdot   R_{u,k}^{E_0^*}|^{-1}_{V_{u,k}^s(E_0^*)}(H_{v(u)}(E_0^*,\lambda_{u,k})), \  u\in V.$$

By the same argument as  Step 3 in the proof of Proposition \ref{pertubation-boundary},   $E_0$ is $\mathcal H$-admissible. By Proposition \ref{divisor-singleton},   $I_{\Phi}(E_0)$ consists of one map $g_0$.  By (\ref{two-motions}),  $g_0$ enjoys the  required properties.
  \end{proof}

	
	
	


 \section{Local Hausdorff dimension via perturbation}\label{lchd-via-p}

In this section, we shall establish the local Hausdorff dimension estimates for some typical boundary strata of $\partial_{\mathcal A}\mathcal H$.
 
 
 
Recall that an $\mathcal H$-admissible map is the one in the impression of an $\mathcal H$-admissible divisor.  However, this definition is conceptual and it relies on the definition of $\mathcal H$-admissible divisors. In the following, we shall give a  more intuitive description of such maps.
 
 Let $D \in  \partial_0^* {\rm Div}{(\mathbb D)}^S$ be $\mathcal H$-admissible,  with $I_{\Phi}(D)=\{f\}$. By Proposition \ref{h-admissible-characterization},  this $f$ has the following  properties:
 \begin{itemize}
 \item [(a)] $f\in \partial_{\rm reg}\mathcal H$ (see \eqref{reg-boundary}) and $f$ is Misiurewicz;
 \item [(b)] Each critical point of $f$ on $J(f)$ is simple and disconnects $J(f)$ into two parts;
 \item [(c)] There is  an internal marking ${\boldsymbol v}(f)=(v(f))_{v\in V}$ of $f$, as the limit of the internal markings ${\boldsymbol v}(f_n)=(v(f_n))_{v\in V}$ of some sequence $\{f_n\}_{n\geq 1}\subset \mathcal H$ approaching $f$, such that ${\rm Crit}(f)\subset \bigcup_{v\in V}\overline{U_f(v(f))}$;
 \item [(d)]  There is a boundary marking $\nu_f=(\nu_f(v))_{v\in V}$ of $f$, so that  $\nu_f(v)\in \partial U_f(v(f))\setminus {\rm Crit}(f)$ for each $v\in V$.
   \end{itemize}

 We remark that the boundary marking $\nu_f$  is defined in the same way as Section \ref{phc}: it
 means a function $\nu_f: V\rightarrow \mathbb C$  which assigns to each  $v\in V$ a  point $\nu_f(v)\in\partial U_{f}(v(f))$, satisfying that 
 $ f(\nu_f(v))=\nu_f(\sigma(v)), \  \forall\ v\in V$.
 
 
 
 
 \begin{figure}[h]  
 	\begin{center} 
 		\includegraphics[height=5cm]{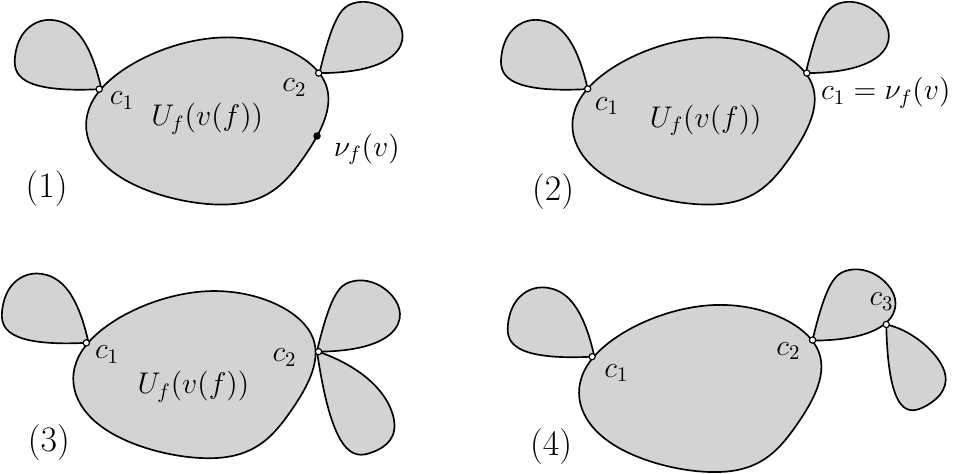}
 	\end{center}
 	\caption{The $v$-part of (a part of) the collapsed Fatou component, $c_k$'s are critical points. Only (1) is $\mathcal H$-admissible; (2) does not satisfy (d); (3) does not satisfy (b), because $c_2$ disconnects $J(f)$ into at least three parts; (4) does not satisfy (c), because no internal marking exists.}
 	\label{fig: pairing-unique}
 \end{figure}
 
 The following fact shows that these properties characterize $\mathcal H$-admissible maps.
 
 
 \begin{ft}\label{equiv-def}  The following statements are equivalent:
 	
 	1. $f$  satisfies above properties (a)(b)(c)(d).
 	
 	2. $f\in I_{\Phi}(D)$  for some parameterization $\Phi: {\rm Div}(\mathbb D)^S\rightarrow   {\mathcal H}$ and some $\mathcal H$-admissible divisor $D\in  \partial_0^*{\rm Div}{(\mathbb D)}^S$.
 \end{ft}
 \begin{proof} It suffices to show 1 implies 2.  Let ${\boldsymbol v}(f)=(v(f))_{v\in V}$, $\nu_f=(\nu_f(v))_{v\in V}$  be the  internal marking and  the boundary marking defining $f$, respectively.
 	
 	Since $f$ is Misiurewicz,     for each $v\in V$,  the marked  point $\nu_f(v)\in \partial U_f(v(f))$ is $f$-pre-repelling, whose orbit avoids critical points (by (d)). 
 	By perturbation, both  ${\boldsymbol v}(f)$ and $\nu_f$  can extend to nearby maps, hence to all maps in $\mathcal H$.   
 	By Proposition \ref{holo-hyperbolic-set}, we have $\nu_g(v)\in \partial U_{g, v}$ for all $v\in V$ and all $g\in \mathcal H$.
 	
 	The pair $(g, {\boldsymbol v}(g))$ together with  $\nu_g$ for $g\in \mathcal H$  induces a natural parameterization  $\Phi: {\rm Div}(\mathbb D)^S\rightarrow   {\mathcal H}_*$ (here, $ {\mathcal H}_*$ is the internally marked hyperbolic component. As before, if no confusion arises, we simply write it as   $\mathcal H$). 
 	
 	Let $\{f_n\}_{n\geq 1}$ be a sequence in $\mathcal H$ approaching $f$. Assume  $B_{n}:=\Phi^{-1}(f_n)$ has a limit $D:=((B^0_v,  D_v^\partial))_{v\in V}\in \partial {\rm Div}(\mathbb D)^S$.  By Fact \ref{fact-def}, $f\in I_{\Phi}(D)$.
 	By the same argument as Step 2 in the proof Proposition \ref{map-imp}, we have $1\notin {\rm supp}(D_v^\partial)$ for all $v\in V$.  Hence $D\in  \partial_0^*{\rm Div}{(\mathbb D)}^S$. 
 	
 	By Proposition \ref{combinatorial-property0}, for any $v\in V$,  the conformal maps $\phi_{f_n, v}:  \mathbb D\rightarrow U_{f_n, v}$ have a limit $\phi_{f,v}:  \mathbb D\rightarrow U_{f}(v(f))$. 
 	By Lemma \ref{model-map},  $${B}^0_v=\phi^{-1}_{f,\sigma(v)}\circ f\circ \phi_{f,v}, \ \forall v\in V.$$
 	This implies that $\phi_{f,v}(1)$ is $f$-pre-repelling. By Proposition \ref{convergence-repelling}, we have $\lim_n\phi_{f_n, v}(1)=\phi_{f, v}(1)$. Since $\phi_{f_n, v}(1)=\nu_{f_n}(v)$, and $\lim_n\nu_{f_n}(v)=\nu_{f}(v)$, we get $\phi_{f, v}(1)=\nu_{f}(v)$.  Since $f$ is Misiurewicz, by the same reasoning as Remark \ref{divisor-def-g} and also by 
 	Proposition \ref{divisor-correspondence}, we get
 	
 	$$D_v^\partial=\sum_{c\in {\rm Crit}(f)\cap \partial U_{f}(v(f))} 1\cdot \phi^{-1}_{f,v}(c), \ \forall v\in V.$$
 	
 	
 	
 	Note that $f$ satisfies property (a),   for each $v\in V$ and each $c\in {\rm Crit}(f)\cap \partial U_{f}(v(f))$, the orbit of $c$ does not meet other critical points,  hence   the limbs $L_{ U_{f}(f^k(v(f))), f^k(c)}, k\geq 1$  are either all trivial  or all non-trivial. The latter is impossible because the non-triviality of $L_{ U_{f}(f(v(f)), f(c)}$ implies that $c$ disconnects $J(f)$  into at least four parts, contradicting (b).
 	Now by the triviality of $L_{ U_{f}(f^k(v(f))), f^k(c)}, k\geq 1$  
 	and the Misiurewicz property of $f$, we conclude that $D$ is $\mathcal H$-admissible.
 \end{proof}

 Let $f$ be an $\mathcal H$-admissible map with internal marking ${\boldsymbol v}(f)=(v(f))_{v\in V}$.
 Write $U_{f,v}=U_f(v(f))$ for $v\in V$.
 By the Implicit Function Theorem, there exist 
 a neighborhood $\mathcal U$ (small enough to satisfy Proposition \ref{map-imp}) of $f$, and a continuous map
 ${\boldsymbol v}: \mathcal U\rightarrow \mathbb C^{\# V}$, defined as ${\boldsymbol v}(g)=(v(g))_{v\in V}$,  so that for each $v\in V$,  $v(g)$ is $g$-pre-attracting for  
 $g\in \mathcal U$\footnote{By  convention in this  paper, both the map $f$ and the neighborhood $\mathcal U$ should be understood as the objects in the internally marked space. Hence, even if internal marking ${\boldsymbol v}(f)=(v(f))_{v\in V}$ contains an $f$-critical point $v(f)$ for some $v\in V_{\rm np}$, the map  ${\boldsymbol v}: \mathcal U\rightarrow \mathbb C^{\# V}$ is always well-defined. }.   Write $U_{g,v}=U_g(v(g))$ for $v\in V$ and $g\in \mathcal U$.
 
 
 We mark the critical points of $f$ on Julia set  by  $c_1, \cdots, c_{n}$.  
  By Proposition   \ref{h-admissible-characterization} and Fact \ref{equiv-def},   for each $1\leq j\leq n$, there is a   unique $v_j\in V$ so that $c_j\in \partial U_{f, v_j}$.  Shrink $\mathcal U$ if necessary, there is a continuous map $c_j:   \mathcal U\rightarrow \mathbb C$ so that $c_j(g)$ is a $g$-critical point for  $g\in  \mathcal U$, with $c_j(f)=c_j$,  for each $1\leq j\leq n$,  and all other (non marked) critical points of $g$ are contained in the Fatou set.


 
 
Recall that $\partial_\mathcal{A} \mathcal H \subset \partial \mathcal H$ consists of all $\mathcal H$-admissible maps.
 For    any  non-empty index set $\mathcal J\subset \{1, \cdots, n\}$, define the  
   boundary strata of $\partial_\mathcal{A} \mathcal H $ associated with $\mathcal J$:
   $$(\partial_\mathcal{A} \mathcal H)^{\mathcal J}_{\mathcal U}:=\Bigg\{g\in \mathcal U\cap \partial_\mathcal{A} \mathcal H;    \begin{cases} c_j(g)\in \partial U_{g,v_j}, \forall  j\in \mathcal J \\
   	c_{j}(g)\in  U_{g, v_j},  \forall  j\in \{1, \cdots, n\}\setminus \mathcal J \end{cases}
   \Bigg\}.$$
   

    For each $g\in \mathcal U$,  let $\mathbf n(g)=(n_v(g))_{v\in V}$, where $n_v(g)=\#\{1\leq j\leq n  ; c_j(g)\in \partial U_{g,v} \}$ for $v\in V$.  For any $\mathbf n=(n_v)_{v\in V}\neq \mathbf 0$ with   $\mathbf 0\leq \mathbf n\leq \mathbf n(f)$ (meaning that $0\leq n_v\leq n_v(f)$ for all $v\in V$), recall that the $\mathbf n$-strata of $\partial_\mathcal{A} \mathcal H$ is
     $$(\partial_\mathcal{A} \mathcal H)^{\mathbf n}_{\mathcal U}:=\big\{g\in \mathcal U\cap\partial_\mathcal{A} \mathcal H;   \mathbf n(g)=\mathbf n
    \big\}.$$
    
    The following result estimates the local Hausdorff dimensions for the two kinds of boundary strata.  The first statement is a slightly more precise form of the second one, which is the restatement of Theorem \ref{local-hausdorff-dim}.
    
 \begin{thm} \label{local-hausdorff-dim-J} 
	Let $f\in \partial\mathcal H$ be an $\mathcal H$-admissible map.   
	
	1.  For  any  non-empty index set $\mathcal J\subset \{1, \cdots, n\}$, we have 
	$${\rm H.dim}((\partial_\mathcal{A} \mathcal H )^{\mathcal J}_{\mathcal U}, f)\geq \sum_{v\in V} \#\mathcal J_v \cdot {\rm H.dim}(\partial U_{f,v})+2(d-1- \#\mathcal J),$$
	where $\mathcal J_v=\{j\in \mathcal J; v_j=v\}$. 

2. 	For any $\mathbf n=(n_v)_{v\in V}\neq \mathbf 0$ with  $\mathbf 0\leq \mathbf n\leq \mathbf n(f)$, we have
$${\rm H.dim}( (\partial_\mathcal{A} \mathcal H)^{\mathbf n}_{\mathcal U}, f)\geq \sum_{v\in V} n_v\cdot {\rm H.dim}(\partial U_{f,v})+2(d-1-|\mathbf n|).$$
\end{thm}

The proof of Theorem \ref{local-hausdorff-dim-J}  is based on Theorem \ref{local-perturbation0}, restated as follows





\begin{thm}  [Perturbation on  $\partial \mathcal H$] \label{local-perturbation1} 
	Let $f\in \partial\mathcal H$ be an $\mathcal H$-admissible map, whose marked critical points on Julia set are  $c_1, \cdots, c_n$.  Suppose  the $f$-orbit of $c_j$ meets $\partial U_{f, u_j}$ for some  $u_j\in V_{\rm p}$.   Given  an $f$-hyperbolic set $X\subset \bigcup_{v\in V_{\rm p}}\partial_0 U_{f, v}$,  a holomorphic motion $h: \mathcal U\times X\rightarrow \mathbb C$  of $X$  based at $f$ (shrink $\mathcal U$ if necessary), an  index set $\emptyset\neq\mathcal J\subset \{1, \cdots, n\}$.

	Given  any  $\varepsilon>0$ with $\mathcal N_\varepsilon(f)\subset \mathcal U$, and any 
	multipoint $ \mathbf x=(x_{j})_{j\in \mathcal{J}}\in  \prod_{j\in \mathcal J}(X\cap \partial U_{f, u_j})$.
	
	(1).  There exist  
	$g_0\in    (\partial_\mathcal{A} \mathcal H )^{\mathcal J}_{\mathcal U}\cap  \mathcal N_{\varepsilon}(f)$,  positive integers $(m_{j})_{j\in  \mathcal{J}}$, 
	so that 
	$$g_0^{m_{j}}(c_{j}(g_0))=h(g_0,  x_{j}) , \forall j\in \mathcal J; \ \ c_{j}(g_0)\in U_{g_0, v_j},  \forall j\in\{1, \cdots, n\}\setminus \mathcal{J}.$$

	(2). Moreover, there is $\rho_0\in (0, \varepsilon)$ so that 
	$$\mathcal E(g_0, \rho_0):=\big\{g\in \mathcal N_{\rho_0}(g_0); \ g^{m_j}(c_j(g))\in h(g, X),  \  \forall j\in \mathcal J\big\}\subset  (\partial_\mathcal{A} \mathcal H )^{\mathcal J}_{\mathcal U}\cap \mathcal  N_{\varepsilon}(f). $$
\end{thm}

In the following, we first introduce a transfer inequality in Section \ref{transfer-inequality}, and then prove 
Theorem \ref{local-hausdorff-dim-J}  assuming Theorem \ref{local-perturbation1} in Section \ref{proof-loc-hd-assumption}. The proof of  Theorem \ref{local-perturbation1}  is given in Section \ref{loc-perturbation}.


 \subsection{Transfer inequality} \label{transfer-inequality}
 
 For $\boldsymbol{z}=(z_1, \cdots, z_n)\in \mathbb C^n$, let $\|\boldsymbol{z}\|=\max_k|z_k|$.
 Let  $\mathbb D^n(\boldsymbol{z}_0, \delta)=\{\boldsymbol{z}\in \mathbb C^n; \|\boldsymbol{z}-\boldsymbol{z}_0\|<\delta\}$ be the polydisk centered at $\boldsymbol{z}_0\in \mathbb C^n$ with radius $\delta$.  For simplicity, write $\mathbb D^n=\mathbb D^n(\boldsymbol{0},  1)$. 
 
 Let $n\geq m\geq 1$ be two positive integers. 
 Let $X_1,  \cdots, X_m$ be subsets of $\C$. For each $1\leq k\leq m$, let 
 $H_k: \mathbb D^n\times X_k\rightarrow \C$ be a holomorphic motion of $X_k$, with base point $\boldsymbol{0}=(0, \cdots, 0)\in \mathbb D^n$. 
 Write $\mathbf{X}=X_1\times \cdots\times X_m\subset \C^m$ and $\boldsymbol{z}=(z_1, \cdots, z_m)\in \C^m$.
 Note that 
 \begin{equation} \label{h-m-higher}
 	H: 
 \begin{cases} \mathbb D^n\times \mathbf{X}\rightarrow \C^m,\\
 	(\lambda, \boldsymbol{z})\mapsto (H_1(\lambda, z_1), \cdots, H_m(\lambda, z_m))
 \end{cases}
 \end{equation}
 is a holomorphic motion of $\mathbf{X}$ with base point $\boldsymbol{0}$.

 \begin{pro} 
 	[Transfer inequality]
 	\label{HD-trans}  Let $n\geq m\geq 1$ be two  integers.
 	Let $X_1,  \cdots, X_m$ be subsets of $\C$, and $H: \mathbb D^n\times \mathbf{X}\rightarrow \C^m$ be the holomorphic motion given by (\ref{h-m-higher}).
 	Let $V=(V_1, \cdots, V_m): \mathbb D^n\rightarrow \mathbb C^m$ be a holomorphic map, with 
 	$V(\boldsymbol{0})=H(\boldsymbol{0}, \boldsymbol{a})$ for some $\boldsymbol{a}=(a_1, \cdots, a_m)\in \mathbf{X}$. Let
 	$$\mathcal X=\big\{\lambda\in \mathbb D^n; V(\lambda)\in H(\lambda, \mathbf{X})\big\}. $$
 	
 Assume the `weak transversality':
 	 ${\rm codim}\big((V-H(\cdot, \boldsymbol{a}))^{-1}(\boldsymbol{0})\big)=m$, 
 then  
 	$$  {\rm H.dim}(\mathcal X, \boldsymbol{0})\geq {\rm H.dim}(\mathbf X, \boldsymbol{a})+2(n-m)=\sum_{k=1}^m {\rm H.dim}(X_k,  a_k)+2(n-m).$$

 \end{pro}

 


\begin{proof}
 We adapt here the argument of the proof of  Proposition 4.5 in  \cite{G}.	 By translation, assume $\boldsymbol{a}=\boldsymbol{0}$. By the harmonic-$\lambda$-lemma (named by McMullen-Sullivan \cite{McS},    see Bers-Royden \cite{BR} for $n=1$,  Mitra \cite{Mitra} for $n\geq 1$), and rescaling $\lambda$ if necessary, $H_k$ extends to a holomorphic motion 
	$H_k: \mathbb D^n \times \C\rightarrow \C$.  Hence $H$ extends to  a holomorphic motion 	$H: \mathbb D^n \times \C^m\rightarrow \C^m$.
	
	We define a continuous map $\varphi=(\varphi_k)_{1\leq k\leq m}:  \mathbb D^n\rightarrow \mathbb C^m$ by 
	$$\varphi_k(\lambda)=H_k(\lambda, \cdot)^{-1}(V_k(\lambda)), 1\leq k\leq m.$$
	For any $\boldsymbol{z}\in \mathbb C^m$, let $\zeta_{\boldsymbol{z}}(\lambda)=V(\lambda)-H(\lambda, \boldsymbol{z})$.
	Clearly  $\zeta_{\boldsymbol{z}}(\lambda)=\boldsymbol{0}$	 if and only if $\boldsymbol{z}=\varphi(\lambda)$. So $\mathcal X$ has the equivalent definition:
	$$\mathcal X=\big\{\lambda\in \mathbb D^n; \varphi(\lambda)\in \mathbf{X}\big\}=\varphi^{-1}(\mathbf{X}).$$
	
	First assume $n=m$. 
	
 {\it Claim: there is $\delta_0>0, \eta_0>0$ so that $\varphi: \mathcal{X}\cap \mathbb D^n(\boldsymbol{0}, \delta_0)\cap \varphi^{-1}(\mathbb D^n(\boldsymbol{0}, \eta_0))\rightarrow \mathbb D^n(\boldsymbol{0}, \eta_0)\cap \mathbf{X}$ is  {surjective} and finite-to-one.}
 
  By the weak transversality, $\lambda=\boldsymbol{0}$ is  an  isolated zero of    $\zeta_{\boldsymbol{0}}(\lambda)=\boldsymbol{0}$,  and let $l\geq 1$ be its multiplicity.  
  Choose $\delta_0>0$ so that $\zeta_{\boldsymbol{0}}^{-1}(\boldsymbol{0})\cap  \overline{\mathbb D^n(\boldsymbol{0}, \delta_0)}=\{\boldsymbol{0}\}$.
 By the continuity of $(\boldsymbol{z}, \lambda)\mapsto \zeta_{\boldsymbol{z}}(\lambda)$, there is $ \eta_0>0$ so that for all $\boldsymbol{z}\in \mathbb D^n(\boldsymbol{0}, \eta_0)$,  
 $$\|\zeta_{\boldsymbol{z}}(\lambda)-\zeta_{\boldsymbol{0}}(\lambda)\|<\min_{\|\alpha\|=\delta_0}\|\zeta_{\boldsymbol{0}}(\alpha)\|\leq \|\zeta_{\boldsymbol{0}}(\lambda)\|, \ \forall \lambda\in \partial\mathbb D^n(\boldsymbol{0}, \delta_0).$$
 By the generalized Rouch\'e Theorem \cite{L},  $\zeta_{\boldsymbol{z}}(\lambda)=0$  has exactly $l$ zeros  in $\mathbb D^n(\boldsymbol{0}, \delta_0)$.  In particular, for $\boldsymbol{z}\in \mathbb D^n(\boldsymbol{0}, \eta_0)\cap \mathbf{X}$, there are exactly $l$ points 
 $\lambda^{(1)}, \cdots, \lambda^{(l)}\in  \mathbb D^n(\boldsymbol{0}, \delta_0)$ so that $\zeta_{\boldsymbol{z}}(\lambda^{(k)})=\boldsymbol{0}$, or equivalently $\boldsymbol{z}=\varphi(\lambda^{(k)})$ (this means $\lambda^{(k)}\in \mathcal X$).  Hence $\varphi: \mathcal{X}\cap \mathbb D^n(\boldsymbol{0}, \delta_0)\cap \varphi^{-1}(\mathbb D^n(\boldsymbol{0}, \eta_0))\rightarrow \mathbb D^n(\boldsymbol{0}, \eta_0)\cap \mathbf{X}$ is   {surjective} and $l$-to-$1$, proving the claim.

 Note that  $\zeta_{\varphi(\lambda)}(\lambda')-\zeta_{\varphi(\lambda)}(\lambda')=
 \zeta_{\varphi(\lambda')}(\lambda')-\zeta_{\varphi(\lambda)}(\lambda')$ for 
      $\lambda, \lambda'\in \mathbb D^n(\boldsymbol{0}, \delta_0)$. By the H\"older estimate for holomorphic motions (see \cite[Lemma 1.1]{D}), for any  $\varepsilon\in (0,1)$,   $\lambda, \lambda'\in \mathbb D^n(\boldsymbol{0},\varepsilon \delta_0)$, there is a constant $C_1(\varepsilon)$  so that
     \bess \| \zeta_{\varphi(\lambda')}(\lambda')-\zeta_{\varphi(\lambda)}(\lambda')\|\geq C_1(\varepsilon)\|\varphi(\lambda')-\varphi(\lambda)\|^{\frac{1+\varepsilon}{1-\varepsilon}}.\eess
     Combining with $\|\zeta_{\varphi(\lambda)}(\lambda)-\zeta_{\varphi(\lambda)}(\lambda')\|\leq C_2(\varepsilon)\|\lambda-\lambda'\|$ yields  $$\|\varphi(\lambda')-\varphi(\lambda)\|\leq C(\varepsilon)\|\lambda-\lambda'\|^{\frac{1-\varepsilon}{1+\varepsilon}}.$$
By above claim, we get 
$${\rm H.dim}(\mathbf X, \boldsymbol{0})\leq  \frac{1+\varepsilon}{1-\varepsilon} {\rm H.dim}(\mathcal X, \boldsymbol{0}).$$
Letting $\varepsilon\rightarrow 0$ gives the required estimate.

Assume $n>m$.  Since ${\rm codim}\big((V-H(\cdot, \boldsymbol{0}))^{-1}(\boldsymbol{0})\big)=m$,   up to a linear change of coordinate, we may assume $\boldsymbol{0}$ is isolated in the analytic set
$$\zeta_{\boldsymbol{0}}^{-1}(\boldsymbol{0})\cap \{\lambda=(\lambda_k)_{1\leq k\leq n}\in  \mathbb D^n; \lambda_{m+1}=\cdots=\lambda_n=0\}.$$
For $m<k\leq n$, let $X_k=\C$, $H_{k}(\lambda, z_k)={z_k}$, $V_k(\lambda)=\lambda_k$. Write $\widehat{\mathbf{X}}=\prod_{1\leq k\leq n} X_{k}$,  $\widehat{H}=(H_k)_{1\leq k\leq n}$, $\widehat{V}=(V_k)_{1\leq k\leq n}$. The   `weak transversality' becomes ${\rm codim}\big((\widehat{V}-\widehat{H}(\cdot, \boldsymbol{0}))^{-1}(\boldsymbol{0})\big)=0$, and we are in the former case.
The proof is completed by noting that ${\rm H.dim}(\widehat{\mathbf{X}}, \boldsymbol{0})={\rm H.dim}(\mathbf X, \boldsymbol{0})+2(n-m)$.
	\end{proof}

\begin{rmk} In Proposition \ref{HD-trans},  if  the strong transversality holds:  
	$${\rm rank}(D\zeta_{\boldsymbol{a}})|_{\lambda=\boldsymbol{0}}=m, $$
	where  $D\zeta_{\boldsymbol{a}}$ is the Jacobian matrix of $\zeta_{\boldsymbol{a}}(\lambda)$ with respect to $\lambda$, then 
	\bess {\rm H.dim}(\mathcal X, 0)={\rm H.dim}(\mathbf X,  \boldsymbol{a})+2(n-m)
	\eess    
	The case $m=1$ is obtained in \cite[Section 7]{CWY}.
\end{rmk}

%

 
 \subsection{Proof of  Theorem \ref{local-hausdorff-dim-J}  assuming Theorem \ref{local-perturbation1}} \label{proof-loc-hd-assumption}
 
A rational map $f$ is  {\it semi-hyperbolic}, if $f$ has no parabolic cycle and  $c\notin \overline{\{f^n(c); n\geq 1\}}$ for any critical point $c\in J(f)$.

\begin{lem}  \label{hyp-dim-b} Let $f$ be a semi-hyperbolic polynomial of degree $d\geq 2$, and suppose $U$ is a bounded periodic Fatou component of $f$. Then
	$${\rm H.dim}(\partial U)={\rm dim}_{\rm hyp}(\partial U). \footnote{The hyperbolic Hausdorff dimension ${\rm dim}_{\rm hyp}(\partial U) $ of $\partial U$ is defined to be
		$${\rm dim}_{\rm hyp}(\partial U)= \sup\{ {\rm H.dim}(X); X  \text{ is a hyperbolic set of } \partial U\}.$$} $$
	Moreover, for any $\varepsilon>0$, there is a hyperbolic set $X_\varepsilon\subset \partial U$, such that
	\begin{itemize}
		\item ${\rm H.dim}(X_\varepsilon)\geq {\rm H.dim}(\partial U)-\varepsilon$;
		\item $X_\varepsilon\subset \partial_0 U:=\{q\in \partial U; L_{U, q}=\{q\}\}$;
		\item  $X_\varepsilon$ is an invariant set of a repelling system, hence is homogeneous: $${\rm H.dim}(X_\varepsilon,x)={\rm H.dim}(X_\varepsilon), \ \forall x\in X_\varepsilon.$$
\end{itemize} 
\end{lem}

Lemma \ref{hyp-dim-b} is probably known to experts for a long time, and a detailed proof is recently  given in the appendix of \cite{CWY}. The property of $X_\varepsilon$ is not explicitly stated therein, but can be  extracted from its proof \footnote{In the proof of \cite[Lemma A.5]{CWY}, take $t_0=\sup_{n\geq 1} {\rm H.dim}(X_n)$, where $X_n\subset \partial U$ is the invariant set of some repelling system,  the proof still works and implies that $t_0= {\rm H.dim}(\partial U)$.
Also in the proof, take  the repelling system avoiding the non-trivial limbs $L_{U, q}$ whose root point $q$ is periodic, then the invariant set is a subset of $\partial_0 U$.}.

Another ingredient in the proof of Theorem   \ref{local-hausdorff-dim-J}  is  the following weak transversality result by Gauthier \cite[Theorem 1.2 and Lemma 4.2]{G}:

\begin{lem}   [Weak transversality] \label{w-transversality} Let $f$ be a $k$-Misiurewicz polynomial of degree $d\geq 2$, and suppose $X$ is a  hyperbolic set  capturing the orbits of $k$ marked critical points $c_1(f), \cdots, c_k(f)$ on $J(f)$.  Let $h: \mathcal U\times X\rightarrow \mathbb C$ be a holomorphic motion of the 
hyperbolic	set $X$ in a neighborhood $\mathcal U\subset \mathcal P_d$ of $f$.   
Assume $f^{n_j}(c_j(f))\in X$ for $1\leq j\leq k$, and let 
$$\chi:   \mathcal U\rightarrow \mathbb C^k,  \ g\mapsto \Big  (g^{n_j}(c_j(g))-h(g, f^{n_j}(c_j(f)))\Big)_{1\leq j\leq k}.$$
Then the codimension of $\chi^{-1}(0)$ is exactly $k$.
\end{lem}
The original statement in \cite{G} concerns rational maps, the argument also works  for polynomials, by marking the critical points $c_{d}=\cdots=c_{2d-2}\equiv\infty$.





\begin{proof}
[Proof of  Theorem \ref{local-hausdorff-dim-J}  assuming Theorem \ref{local-perturbation1}. ]
   Fix  small  $\varepsilon>0$ with $\mathcal N_\varepsilon(f)\subset \mathcal U$.  Let $\mathcal C=\{u_j; j\in \mathcal J\}\subset V_{\rm p}$.  By Lemma \ref{hyp-dim-b}, for each $v\in \mathcal C$, there is an $f^{\ell_v}$-hyperbolic set $X_{v}\subset \partial_0 U_{f,v}$ with
	 $${\rm H.dim}(X_{v},x)={\rm H.dim}(X_{v})\geq {\rm H.dim}(\partial U_{f,v})-\varepsilon, \ \forall x\in X_{v}.$$
	 By Proposition \ref{hyp-set}, 
	 there is $\varepsilon_0\in (0,\varepsilon]$,
	 and a holomorphic motion 
	 $h_v:   \mathcal{N}_{\varepsilon_0}(f)\times X_{v}\rightarrow \mathbb C$ of the hyperbolic set $X_{v}$, with base point $f$. Write $X_{v}(g)=h_v(g, X_{v})$   for $g\in  \mathcal{N}_{\varepsilon_0}(f)$.

	  Note that   there is a smaller neighborhood $\mathcal{N}_{\varepsilon_1}(f)\subset  \mathcal{N}_{\varepsilon_0}(f)$, so that     
	\begin{equation}\label{hyp-0-u} X_{v}(g)\subset \partial_0 U_{g}(v(g)), \ \forall g\in  \mathcal{N}_{\varepsilon_1}(f). \footnote{ To see this, one can adapt the argument of Proposition \ref{divisor-correspondence}:  assume $f=I_{\Phi}(D)$ for some $\mathcal H$-admissible divisor $D=((B^0_u,  D_u^\partial))_{u\in V}$. By assumption $X_v\subset  \partial_0 U_{f,v}$,  one can  choose the graph in  Proposition \ref{divisor-correspondence} so that  $X_v\cap \big(\bigcup_{q\in {\rm supp}(D_{v}^\partial(f))} X_{v,q}(f)\big)=\emptyset$. 
	  	By the Hausdorff continuity of the graphs and hyperbolic sets, there is   $\varepsilon_1$ so that for all $g\in \mathcal N_{\varepsilon_1}(f)$, all critical points of $g$ outside $U_{g}(v(g))$ are contained in $A_v(g):=\bigcup_{q\in {\rm supp}(D_{v}^\partial(f))} X_{v,q}(g)$ and $X_v(g)\cap  A_v(g)=\emptyset$. This implies that $X_{v}(g)\subset \partial_0U_{g}(v(g))$. }
  \end{equation}
  	   By the H\"older estimate for holomorphic motions (see \cite[Lemma 1.1]{D}), and also by shrinking $\varepsilon_1$,  we get
	 $${\rm H.dim}(X_{v}(g), x(g))\geq  \frac{\varepsilon_0-\varepsilon_1}{\varepsilon_0+\varepsilon_1}\cdot {\rm H.dim}(X_{v}, x) \geq {\rm H.dim}(\partial U_{f,v})-2\varepsilon,$$
	 where $x(g)=h_v(g,x)$ for any $x\in X_{v}$ and $g\in  \mathcal{N}_{\varepsilon_1}(f)$.
	 
	 Let 
	 $\mathbf{x}=(x_j)_{j\in \mathcal{J}}\in \prod_{j\in \mathcal{J}} X_{u_j}$
	 be a multipoint.  By Theorem \ref{local-perturbation1},  there exist 
	 $g_0\in ( \partial_\mathcal{A} \mathcal H)^{\mathcal J}_{\mathcal U}\cap  \mathcal N_{\varepsilon_1}(f)$, positive integers $(m_{j})_{j\in  \mathcal{J}}$,    and positive  number $\rho_0>0$   so that  
	 \bess g_0^{m_{j}}(c_{j}(g_0))=x_{j}(g_0)=h_{u_j}(g_0,  x_{j}),  \ \forall j\in \mathcal J; \  \mathcal E(g_0, \rho_0)\subset ( \partial_\mathcal{A} \mathcal H)^{\mathcal J}_{\mathcal U}\cap  \mathcal N_{\varepsilon_1}(f).
	 \eess

	For each $j\in  \mathcal{J}$, define 
	$$V_{j}: \begin{cases} \mathcal{N}_{\rho_0}(g_0)\rightarrow\mathbb C\\
	g\mapsto g^{m_{j}}(c_{j}(g)) 
	\end{cases} \text{ and } \ H_{j}: 
	\begin{cases} \mathcal{N}_{\rho_0}(g_0)\times X_{u_j}(g_0)\rightarrow \mathbb C\\
		(g,x)\mapsto h_{u_j}(g, h_{u_j}(g_0, \cdot)^{-1}(x))
	\end{cases}.$$
Clearly  $V=(V_{j})_{j\in  \mathcal{J}}: \mathcal{N}_{\rho_0}(g_0)\rightarrow\mathbb C^{\#\mathcal J}$ is a holomorphic map, and
$H=(H_{j})_{j\in  \mathcal{J}}: \mathcal{N}_{\rho_0}(g_0)\times \mathbf{X}(g_0)\rightarrow \mathbb C^{\#\mathcal J}$ is a holomorphic motion of $\mathbf{X}(g_0)=\prod_{j\in  \mathcal{J}}  X_{u_j}(g_0)$. They satisfy the relation
$$V(g_0)=H(g_0, \boldsymbol{x}(g_0))=(g_0^{m_{j}}(c_{j}(g_0)))_{j\in  \mathcal{J}},  \ \boldsymbol{x}(g_0)=(x_{j}(g_0))_{j\in  \mathcal{J}}.$$

	By the transversality (see Lemma \ref{w-transversality}),   
	$${\rm codim}\Big((V-H(\cdot,  \boldsymbol{x}(g_0)))^{-1}(\mathbf 0)\Big)= \#\mathcal J, $$
	
	Note that $\mathcal E(g_0, \rho_0)=\big\{g\in  \mathcal{N}_{\rho_0}(g_0);  V(g)\in H(g, \mathbf{X}(g_0))\big\}$.
		By Proposition \ref{HD-trans},   
	\bess 
 	{\rm H.dim}( (\partial_\mathcal{A} \mathcal H)^{\mathcal J}_{\mathcal U}\cap  \mathcal N_{\varepsilon}(f))
 	&\geq& 	{\rm H.dim}(\mathcal E(g_0, \rho_0))
 	\geq{\rm H.dim}(\mathcal E(g_0, \rho_0), g_0)\\
	&\geq& {\rm H.dim}(\mathbf X(g_0),  \boldsymbol{x}(g_0))+2(d-1- \#\mathcal J)\\
	&=&\sum_{j\in \mathcal J} 
	{\rm H.dim}( X_{u_j}(g_0),  {x}_{j}(g_0))+2(d-1- \#\mathcal J)\\
&\geq &  \sum_{j\in \mathcal J}  ({\rm H.dim}(\partial U_{f,u_j})-2\varepsilon)+2(d-1- \#\mathcal J)\\
&=&  \sum_{u\in V}  \#\mathcal J_u \cdot {\rm H.dim}(\partial U_{f,u})-2 \varepsilon \#\mathcal J+2(d-1- \#\mathcal J).
\eess
Since $\varepsilon>0$ is arbitrary, by the definition of local Hausdorff dimension, we get the required inquality.   This proves the first statement.

Given $\mathbf n=(n_v)_{v\in V}$,  choose an index set $\mathcal J$ so that $\#\mathcal J_v=n_v$ for all $v\in V$.  Note that $(\partial_\mathcal{A} \mathcal H)^{\mathcal J}_{\mathcal U}\subset (\partial_\mathcal{A} \mathcal H)^{\mathbf n}_{\mathcal U}$.  The proof is completed  by applying the first statement.
 \end{proof}

\subsection{Proof of Theorem \ref{local-perturbation1}} \label{loc-perturbation}

 We first establish a criterion to determine which hyperbolic polynomials near an $\mathcal H$-admissble map are contained in $\mathcal H$.
 
 \begin{pro}\label{pro:criterion}
 	Let $f$ be an $\mathcal H$-admissible map, whose marked critical points on Julia set are  $c_1(f), \cdots, c_n(f)$, with $c_j(f)\in \partial U_{f, v_j}$.  Let $c_j:   \mathcal U\rightarrow \mathbb C$ be the continuous marking of critical point, defined in a neighborhood $\mathcal U$ of $f$.  
 	 Then there is $\epsilon>0$   with $ \mathcal{N}_{\epsilon}(f)\subset \mathcal U$,   such that a map $g\in\mathcal{N}_{\epsilon}(f)$ is contained in $\mathcal H$ provided that  $c_j(g) \in U_{g,v_j}$ for all $1\leq j\leq n$.
 \end{pro}

 \begin{proof}
 	The proof needs the notion of  \emph{critical portrait} and we recall some necessities. One can refer to \cite{Poi} for more details. 
 	
 	Let $g\in \mathcal C_d$ be a hyperbolic polynomial, with  the associated mapping scheme $S=(V,\sigma,\delta)$ and a boundary marking $\nu_{g}$. Set $V'=\{v\in V; \delta(v)>1\}$. The \emph{critical portrait} of $(g,\nu_{g})$ is a collection
 	$\{\mathcal{F}_v; v\in V'\}$
 	of finite subsets of $\mathbb{Q}/\mathbb{Z}$ constructed below: \vspace{3pt}

 	For each $v\in V$, let $\alpha_v$ be the angle so that $R_{g}(\alpha_v)$ \emph{left supports} $U_{g,v}$ at $\nu_{g}(v)\in\partial U_{g,v}$, i.e., $R_{g}(\alpha_v)$ lands at $\nu_{g}(v)$, and $R_{g}(\alpha_v),R_{g}(\alpha'),U_{g,v}$ are in the counterclockwise order near $\nu_{g}(v)$ for any other external ray $R_{g}(\alpha')$ (if any) landing at $\nu_{g}(v)$. Clearly, $\alpha_{\sigma(v)}=d\alpha_{v} \ {\rm mod }\ \mathbb Z$.
 	
 	If $\delta(v)>1$, we define
 	$\mathcal{F}_v$ to be the set of  arguments of the external rays in $g^{-1}(R_{g}(\alpha_{\sigma(v)}))$ that land at $\partial U_{g,v}$.  It follows immediately that $\alpha_v\in\mathcal{F}_v$ and $\# {\mathcal F}_v= {\rm deg}(g|_{U_{g,v}})=\delta(v)$.  

 	Now we start to prove the proposition.	According to condition (d) of $\mathcal{H}$-admissible maps, there exists a  boundary marking $\nu_{f}$ of $f$ such that $\nu_{f}(v)\in\partial U_{f,v}\setminus{\rm Crit}(f)$ for each $v\in V$. This implies that $\{\nu_{f}(v); v\in V\}$ is an $f$-hyperbolic set. By Proposition \ref{holo-hyperbolic-set}, the boundary marking $\nu_{f}$ extends to a continuous map $\nu:\mathcal{N}_{\epsilon_0}(f)\to \mathbb{C}^{\# V},\ g\mapsto \nu_g$, so that $\nu_g$ is a boundary marking of $g$ for an $\epsilon_0>0$. Furthermore, this marking map $\nu$  extends continuously to $\mathcal{N}_{\epsilon_0}(f)\cup\mathcal{H}$. 
 	It follows that all the marked maps $(g,\nu_g),g\in\mathcal{H},$ have the same critical portrait, denoted by $\{\mathcal{F}_v; v\in V'\}$.
 	
By \cite[Theorem 1.1]{Poi}, to prove that a hyperbolic polynomial $g$ near $f$ belongs to $\mathcal H$, it suffices to show the marked polynomial $(g,\nu_{g})$ has the critical portrait 	$\{\mathcal{F}_v; v\in V'\}$.
 	
 	For any $g\in \mathcal{H}$ and each $v\in V'$, denote by  $U_{g,v,1},\ldots U_{g,v,q_v}$ the components of $g^{-1}(U_{g,\sigma(v)})\setminus U_{g,v}$. Then $\overline{U_{g,v,i}}\cap \overline{U_{g,v}}=\emptyset$ for all $i$, since no critical points of $g$ lie in the Julia set.  For each $i\in\{1,\ldots,q_v\}$, there exist two angles $\theta_i^{\pm}$ such that $R_g(\theta_i^\pm)$ land at a common point in $\partial U_{g,v,i}$, and $\overline{R_g(\theta_i^+)}\cup\overline{R_g(\theta_i^-)}$ separates $U_{g,v}$ from $U_{g,v,i}$ \footnote{ These angles $\theta_i^\pm,i=1,\ldots,q_v$ are independent of $g\in\mathcal{H}$.}.
 	Let $Y_{g,v}$ be the   component of $\mathbb C\setminus \bigcup_{i=1}^{q_v}(\overline{R_g(\theta_i^+)}\cup\overline{R_g(\theta_i^-)})$  containing $U_{g,v}$. The above discussion implies that 
 	\begin{equation}\label{eq:111}
 		\Big(\bigcup_{\theta\in\mathcal{F}_v}\overline{R_g(\theta)}\Big)\bigcup U_{g,v}\subset Y_{g,v}\text{ and }g^{-1}(R_g(\alpha_{\sigma(v)}))\bigcap Y_{g,v}=\bigcup_{\theta\in\mathcal{F}_v}R_g(\theta).
 	\end{equation}
 	
 	\begin{figure}[h]  
 		\begin{center} 
 			\includegraphics[height=5cm]{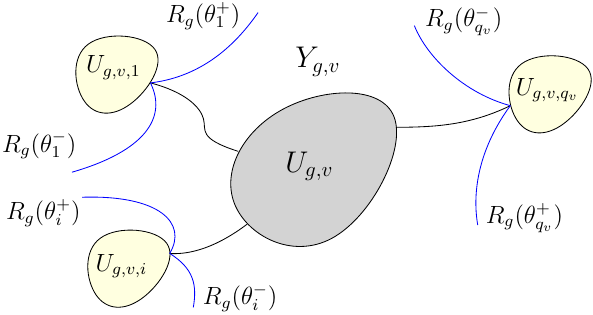}
 		\end{center}
 		\caption{The relative positions of $U_{g,v}$ and $U_{g,v,i}$'s.}
 		\label{fig: criterion-in-H}
 	\end{figure}
 	
 	Note that $\theta_i^\pm$ are preperiodic and $R_{g}(d^k\theta_i^\pm)$ are two distinct external rays landing at a common point for every integer $k\geq 1$ and $g\in\mathcal H$.
 	Since $f$ has no parabolic points,  it follows from Lemma \ref{stability-e-r} (1) that the external rays $R_{f}(d^k\theta_i^\pm)$ land at a common pre-repelling point for every $k\geq0$. As $f$ is $\mathcal{H}$-admissible, each of its post-critical point in the Julia set receives only one external ray. Hence, the 
 	orbit of the common landing point of $R_{f}(\theta_i^\pm)$ avoid  the critical points of $f$ for every $i\in\{1,\ldots,q_v \}$.  
 	
 	By the above discussion together with Lemma \ref{stability-e-r} (2) and Lemma \ref{rays-number}, there exists $\epsilon\in(0,\epsilon_0)$ such that 
 	\[\Gamma_v(g):=\bigcup_{1\leq i\leq q_v}\left(\overline{R_g(\theta_i^-)}\cup\overline{R_g(\theta_i^+)}\right)\]
 	varies continuously for $g\in \mathcal{N}_{\epsilon}(f)$. Thus, for any $g\in\mathcal{N}_\epsilon(f)$ and $v\in V'$, there is a component $Y_{g,v}$ of $\mathbb C\setminus \Gamma_v(g)$ that contains $U_{g,v}$, and 
 	$\overline{Y_{g,v}}$ varies continuously in $g$.
 	
 	On the other hand,  {by Lemma \ref{rays-number}}, the external ray $R_{g}(\alpha_v)$ left supports $U_{g,v}$ at $\nu_{g}(v)$ for every $v\in V$ and $g\in\mathcal{N}_{\epsilon}(f)$ (decreasing $\epsilon$ if necessary). Using Lemma \ref{stability-e-r} again, $R_{g}(\alpha_v)$ varies continuously with respect to $g$. Thus,  we can decrease $\epsilon$ so that
 	relation \eqref{eq:111} holds for all $g\in\mathcal{N}_{\epsilon}(f)$ and $v\in V'$.

 	Now suppose $g \in \mathcal{N}_{\epsilon}(f)$ satisfies that $c_j(g) \in U_{g,v_j}$ for all $1\leq j\leq n$. It follows that ${\rm deg}(g|_{U_{g,v}})=\delta(v)$ for all $v\in V$.  Since relation \eqref{eq:111} holds for all $g\in\mathcal{N}_{\epsilon}(f)$,  an external ray $R_{g}(\theta)\subset g^{-1}(R_{g}(\alpha_{\sigma(v)}))$ lands on  $\partial U_{g,v}$ if and only if $\theta\in\mathcal{F}_v$, for each $v\in V'$. It follows that  $(g,\nu_{g})$ has the critical portrait $\{\mathcal{F}_v; v\in V'\}$. 
 \end{proof}


\begin{proof}[Proof of Theorem \ref{local-perturbation1}] Without loss of generality, assume  for different indices $j_1, j_2\in \mathcal J$, either $u_{j_1}=u_{j_2}$, or 
	$u_{j_1}$ and $u_{j_2}$ are in different $\sigma$-cycles. 
	
	(1). It follows from Proposition \ref{pertubation-boundary-strata} (1).
	
  (2). The key point is to check that $\mathcal{E}(g_0,\rho_0)\subset (\partial_\mathcal{A} \mathcal H )^{\mathcal J}_{\mathcal U}$ for some $\rho_0>0$. 
	
	\vspace{5pt}

	{\bf  Claim.}
	{\it There exists $\rho_0>0$ with $\mathcal{N}_{\rho_0}(g_0)\subset \mathcal U$, such that  
		if $g\in\mathcal{N}_{\rho_0}(g_0)$ satisfies $g^{m_j}(c_j(g))\subset U_{g,\sigma^{m_j}(v_j)}$ for all $j \in \mathcal{J}$, then each  $c_j(g)$ belongs to $U_{g,v_j}$,  where $(m_j)_{j\in \mathcal J}$ is given in (1). } 
	\begin{proof}[Proof of the Claim]
		The argument goes by contradiction. Suppose that the claim is false. Then  there exist an index $j_0\in\mathcal{J}$ and a sequence $\{g_k\}_{k\ge1}\subset \mathcal{P}_d\cap \mathcal U$ such that $g_k\to g_0$ as $k\to\infty$, and that for every $k\geq1$, 
		\[ \text{$g_k^{m_j}(c_j(g_k))\subset U_{g_k,\sigma^{m_j}(v_j)}$,  $\forall\,j\in\mathcal{J}$; but $c_{j_0}(g_k)\not\in U_{g_k,v_{j_0}}$. }\]
		
		Note that  each $g_k$ is quasi-conformally conjugate to a postcritically finite polynomial $\tilde{g}_k$ near the Julia set. Since the lengths of the postcritical orbits for all $\tilde{g}_k$ are uniformly bounded above, there are finitely many choices of $\tilde{g}_k,k\geq1$. As a consequence, by taking a subsequence, we can assume that all $g_k,k\geq1$ are contained in the same bounded hyperbolic component $\mathcal{H}'$.  
		
		 
		 	 {	Since  $c_{j_0}(g_k)\not\in U_{g_k,v_{j_0}}$, we get $\partial{U_{g_k,v_{j_0}}}\cap \partial{U_{g_k}(c_{j_0}(g_k))}=\emptyset$ (if not, since  $g_k^{m_j}(U_{g_k,v_{j_0}})=g_k^{m_j}(U_{g_k}(c_{j_0}(g_k)))= U_{g_k, \sigma^{m_j}(v_j)}$,  the common boundary point is in an inverse orbit of a critical point,  which contradicts the hyperbolicity of $g_k$).  There exist four different angles $\theta_\pm, \alpha_{\pm} \in \mathbb{Q}/\mathbb{Z}$ so that $R_{g_k}(\theta_\pm)$ land at a common point on $\partial U_{g_k}(c_{j_0}(g_k))$, and $\overline{R_{g_k}(\theta_+)}\cup\overline{R_{g_k}(\theta_-)}$ separates $\overline{U_{g_k,v_{j_0}}}$ from $ U_{g_k}(c_{j_0}(g_k))$; $R_{g_k}(\alpha_\pm)$ land at a common point on $\partial U_{g_k,v_{j_0}}$, and $\overline{R_{g_k}(\alpha_+)}\cup\overline{R_{g_k}(\alpha_-)}$ separates $\overline{U_{g_k}(c_{j_0}(g_k))}$ from $U_{g_k,v_{j_0}}$, for every $k\geq1$. 
		 	By Lemma \ref{stability-e-r} and the $\mathcal H$-admissibility of $g_0$, 
		 	we see that 	$R_{g_0}(\theta_\pm)$ land at the same point $a$, and $R_{g_0}(\alpha_\pm)$ land at the same point $b$.  If $a=b$, then $a$ is a landing point of at least four external rays, and the forward orbit of $a$ meets a critical point,   contradicting the defining property (b) of  $\mathcal H$-admissible maps.
		 	Hence $a\neq b$.
		 	It follows that
		 	$\overline{R_{g_0}(\theta_+)}\cup\overline{R_{g_0}(\theta_-)}$ separates $\overline{U_{g_0,v_{j_0}}}$ from $c_{j_0}(g_0)$. 
		 	However, this contradicts  that $c_{j_0}(g_0)\in \partial U_{g_0,v_{j_0}}$. }
	\end{proof}

	By the claim and Proposition \ref{pro:criterion}, to prove that $\mathcal{E}(g_0,\rho_0)\subset  \partial  \mathcal H $, it is enough to show  that for any  $g_* \in \mathcal{E}(g_0,\rho_0)$, there exists a sequence  $\{g_k\}_{k\geq1}\subset \mathcal{P}_d$ such that \[\lim_{k\to\infty}g_k= g_*\quad\text{and}\quad g_k^{m_j}(c_j(g_k))\subset U_{g_k,\sigma^{m_j}(v_j)},\  \forall\,j\in\mathcal{J},\ k\geq1.\]


Add the critical markings $c_j: \mathcal U\rightarrow \mathbb C$  for $n<j\leq d-1$.  
Let $\mathcal{S}$  be the component of  
$$\Big\{g \in \mathcal{N}_{\rho_0}(g_0) ;  \ c_j(g)=c_j(g_*), ~\forall j \in \{1,\cdots, d-1\}\setminus \mathcal{J} \Big\}$$
that contains $g_*$. Then $\mathcal{S}$ is a $\#\mathcal{J}$-dimensional complex {analytic set}.



Let $\mathbf{y}:=( y_j)_{j\in \mathcal{J}}\in \prod_{j\in \mathcal J}(X\cap \partial U_{f, u_j})$ be a multipoint  such that 
$g_*^{m_j}(c_j(g_*))=h(g_*,y_j)$ for every $j\in\mathcal{J}$. 
Then the map $\xi_{\mathbf{y}}: \mathcal{S} \to \mathbb{C}^{\# \mathcal{J}}$ given by
$$g \mapsto (g^{m_j}(c_j(g))-h(g, y_j ))_{j\in \mathcal{J}}$$
is  holomorphic and $\xi_{\mathbf{y}}(g_*)=\mathbf{0}$.
By Lemma \ref{w-transversality}, the polynomial $g_*$ is an isolated zero of $\xi_{\mathbf{y}}$ (see \cite[Corollary, pp 34]{C}).
Thus, for any $\rho\in(0,\rho_0)$ with  $\overline{\mathcal{N}_{\rho}(g_*)}\subset  \mathcal{N}_{\rho_0}(g_0)$, 
$$\min_{g\in\partial (\mathcal{S}\cap \mathcal{N}_{\rho}(g_*))} \| \xi_{\mathbf{x}} (g) \|=\eta_\rho>0.$$

According to Remark \ref{rem:2.4} (3), by decreasing $\rho_0$, we obtain a continuous map 
$\gamma_j : \mathcal{S} \times [0,1] \to \mathbb{C}$ for each $j\in\mathcal{J}$, such that
\begin{itemize}
	\item [(1)]  $\gamma_j(g,(0,1]) \subset U_{g,\sigma^{m_j}(v_j)}$ and $\gamma_j(g,0)=h(g, y_j)$, for any $g\in\mathcal{S}$;
	\item [(2)] $\gamma_j(\cdot,s):\mathcal{S}\to \mathbb{C}$ is continuous, for any $s\in[0,1]$.
\end{itemize} 
As a  result, there exists a continuous map $\zeta_k:\mathcal{S}\to \mathbb{C}^{\mathcal{J}}$ for each $k\geq1$: 
$$\zeta_k(g):= (g^{m_j}(c_j(g)) - \gamma_j(g,1/k))_{j \in \mathcal{J}}. $$

Let $\{\rho_k\}_{k\geq1}\subset (0,\rho_0)$ be a sequence of numbers decreasing to $0$. For each $k$, one can find an integer $n_k>0$ such that  
$$
\| \xi_{\mathbf{y}}(g) - \zeta_{n_k}(g) \| = \max_{j\in \mathcal{J}} |\gamma_j(g,0)-\gamma_j(g,1/n_k)| < \eta_{\rho_k} \leq \| \xi_{\mathbf{y}}(g) \|
$$
for every $g\in \partial (\mathcal{S}\cap \mathcal{N}_{\rho_k}(g_*))$.
Then $h_t(g)=\xi_{\mathbf{y}}(g)-t(\xi_{\mathbf{y}}(g) - \zeta_{n_k}(g))$ is continuous in $(t, g)\in [0,1]\times \overline{\mathcal{S}\cap \mathcal{N}_{\rho_k}(g_*)}$, so that 
$\mathbf{0}\notin h_t(\partial(\mathcal{S}\cap \mathcal{N}_{\rho_k}(g_*)) )$ for all $t\in [0,1]$.
Now apply the degree theory  (see \cite[Section 3]{L} for the definition and basic properties of the degree ${\rm Deg}$),
we have 
$${\rm Deg}(\zeta_{n_k}, \mathcal{S}\cap \mathcal{N}_{\rho_k}(g_*), \mathbf{0})={\rm Deg}(\xi_{\mathbf{y}}, \mathcal{S}\cap \mathcal{N}_{\rho_k}(g_*), \mathbf{0})>0.$$
Therefore $\zeta_{n_k}$ has a zero in $\mathcal{S}\cap \mathcal{N}_{\rho_k}(g_*)$  (see  \cite[Section 3, (4) in pp. 262]{L}). That is, there exists  $g_k \in \mathcal{S}\cap \mathcal{N}_{\rho_k}(g_*)$ with $\zeta_{n_k}(g_k)=\mathbf{0}$. It follows that  
$g_k^{m_j}(c_j(g_k))=\gamma_j(g_k,1/n_k) \in U_{g_k,\sigma^{m_j}(v_j)}$ for every $j \in \mathcal{J}$.

This finishes the proof of the fact $\mathcal{E}(g_0,\rho_0)\subset \partial\mathcal{H}$.  
By the Hausdorff continuity of $g^{-m_j}(\gamma_j(g, [0,1]))$ in $g\in \mathcal S$, and by taking a limit,   the argument above also  yields that $c_j(g)\in \partial U_{g, v_j}$ for all $j\in \mathcal J$ and $g\in \mathcal{E}(g_0,\rho_0)$.
It remains to show $\mathcal{E}(g_0,\rho_0)\subset  (\partial_\mathcal{A} \mathcal H )^{\mathcal J}_{\mathcal U}$.  In fact this follows from the same argument as   Step 3 in the proof of Proposition \ref{pertubation-boundary}, by shrinking $\rho_0$ if necessary.
\end{proof}

\section{Bifurcation of parabolic fixed points}\label{sec:parabolic-implosion}
In Sections \ref{sec:parabolic} and \ref{sec:perturbation}, we review the parabolic bifurcation theory due to Douady-Hubbard \cite{DH}, Lavaurs \cite{La}, and Shishikura   \cite{S,Shi-book}. Our presentation follows from Shishikura. Subsequently, we  use this theory to establish a lemma in Section \ref{sec:control} that is useful for our further analysis.

In the following, a function $f$ is always associated with its domain of definition ${\rm Dom}(f)$. We say a sequence of maps ${f_n}$ \emph{converges} to $f$ if, given any compact set $K\subseteq {\rm Dom}(f)$,
\begin{itemize}
	\item $K\subseteq {\rm Dom}(f_n)$ for all large $n$;
	\item $f_n|_K$ uniformly converges to $f|_K$ as $n\to\infty$.
\end{itemize}
The topology under such convergence is usually called the \emph{compact-open} topology.
We set
$$\FFF=\Big\{f:{\rm Dom}(f)\to\widehat{\mathbb C} \text{ is holomorphic}; \ 0\in{\rm Dom}(f)\subseteq\widehat{\mathbb C}\text{ and }f(0)=0\Big\}.$$

For any point $a\in\mathbb C$,  let $T_a(z):=z+a$ denote the translation by $a$.

\subsection{Parabolic fixed points }\label{sec:parabolic}

\subsubsection{Parabolic petals, fundamental domains and parabolic basins}
\

Let $p<q$ be a pair of coprime positive integers. We denote by $\FFF_0$ the collection of maps $f\in\FFF$
such that  $f'(0)=e^{2\pi ip/q}$ and
\[f^q(z)=z+z^{q+1}+O(z^{q+2})\]
in a neighborhood of $0$.

Fix a map $f_0\in\FFF_0$.
Then $f_0$ has $q$ repelling directions with arguments $2k\pi/q,k=0,\ldots,q-1$; and $q$ attracting directions with arguments $-\pi/q+2k\pi/q,k=0,\ldots,q-1$.

For each $k\in \Z_q=\Z/q\Z$, there exist two  disks $Q_{+}^{(k)}, Q_{-}^{(k)}\subseteq {\rm Dom}(f_0)$, called $k$-th \emph{repelling/attracting petals},\footnote{Here we use the subscripts ``+'' to represent ``repelling'' objects and ``-'' to represent ``attracting'' objects. This notations follows from \cite{Shi-book}.  However in a closely related paper \cite{S}, the subscripts ``+'' represents ``attracting'' objects and ``-'' represents ''repelling'' objects.} such that
\begin{enumerate}
	\item  $0\in\partial  Q_{\pm}^{(k)}$, and the closures of ${ Q}_{-}^{(k)},\ { Q}_{+}^{(k)},k\in\Z_q$, intersect only at $\{0\}$;\vspace{2pt}
	\item $ Q_{+}^{(k)}$ and $ Q_{-}^{(k)}$ contain the $k$-th repelling and attracting directions, respectively;\vspace{2pt}
	\item  $f_0({ Q}_{-}^{(k)})\subset  Q_{-}^{(k+p)}$ and $f_0( Q_{+}^{(k)})\supset { Q}_{+}^{(k+p)}$; \vspace{2pt}
	\item the union of $Q_\pm^{(k)},k\in\Z_q,$ together with the point $0$ is a neighborhood of $0$, on which $f_0$ is defined and  injective; and \vspace{2pt}
	\item $f_0^{nq}\to 0$ as $n\to\infty$ uniformly in $ Q_{-}^{(k)}$.
\end{enumerate}

\begin{figure}[h]   
	\begin{center}
		\includegraphics[height=8.8cm]{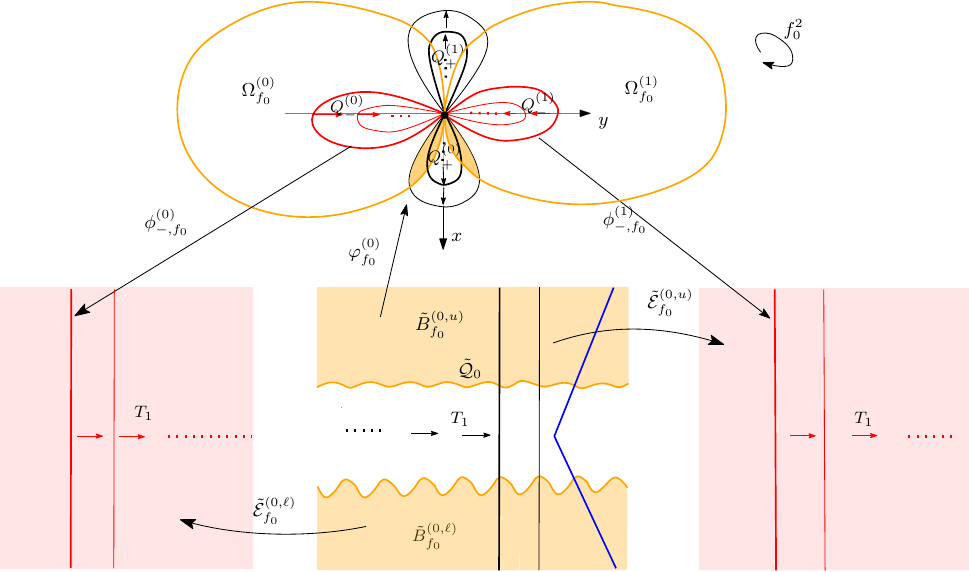}
	\end{center}
	\caption{The parabolic petals, Fatou coordinates and horn maps for $f_0$}\label{fig:f0}
\end{figure}

Set $\ell_{\pm,f_0}^{(k)}:=\partial Q_{\pm}^{(k)}$. We denote by $S_{\pm,f_0}^{(k)}$ the closure of the simply-connected domain bounded by $\ell_{\pm,f_0}^{(k)}$  and $f^q(\ell_{\pm,f_0}^{(k)})$. Then  ${S}_{\pm,f_0}^{(k)}\setminus\{0\},k\in\Z_q,$ are pairwise disjoint. The sets $S_{+,f_0}^{(k)}$ and $S_{-,f_0}^{(k)}$ are called the $k$-th repelling and attracting {\it fundamental domains} of $(f_0^q,0)$, respectively.

The $k$-th \emph{parabolic basin} of  $(f_0^q,0)$ is defined to be $\bigcup_{l\geq0}f^{-lq}_0( Q_{-}^{(k)})$. A connected open set $\Omega^{(k)}_{f_0}\subset {\rm Dom}(f_0^q)$ is called the $k$-th \emph{immediate parabolic basin} of  $(f_0^q,0)$ if $\Omega^{(k)}_{f_0}$ is a connected component of the $k$-th parabolic basin  such that $f^q_0(\Omega^{(k)}_{f_0})=\Omega^{(k)}_{f_0}$ and $f^q_0:\Omega^{(k)}_{f_0}\to \Omega^{(k)}_{f_0}$ is a branched covering of finite degree. The following result is due to \cite[Lemma 4.5.2]{Shi-book}

\begin{pro}\label{pro:immediate}
	The map $f_0^q$ has at most $q$ immediate parabolic basins of $0$, each of which contains at least one critical point of $f^q_0$ except for a parabolic M\"{o}bius transformation. Moreover, if an immediate parabolic basin contains only one critical point, then it is simply connected.
\end{pro}

\subsubsection{Fatou coordinates}\label{sec:Fatou-coordinate}
\

For each $k\in\Z_q$, there are two conformal injections $\phi_{\pm,f_0}^{(k)}: Q_{\pm}^{(k)}\to \C$ such that
\begin{enumerate}
	\item  $\phi_{\pm,f_0}^{(k)}\circ f_0^q(z)=\phi_{\pm,f_0}^{(k)}(z)+1$ if $z,f_0^q(z)\in Q_{\pm}^{(k)}$;\vspace{2pt}
	\item  $\phi_{\pm,f_0}^{(k)}=\phi_{\pm,f_0}^{(0)}\circ f_0^{r_k}$, where $0\leq r_k<q$ such that $k+r_kp\equiv 0({\rm mod}\ q)$.
\end{enumerate}

These collection of conformal maps $\phi_{+,f_0}^{(k)},\phi_{-,f_0}^{(k)},k\in\Z_q$ are called \emph{repelling and attracting Fatou coordinates} of $f_0$, respectively. Each Fatou coordinate is unique up to translations.

We denote $\varphi_{f_0}^{(k)}$ the maximal analytic extension of $(\phi_{+,f_0}^{(k)})^{-1}$ such that $\varphi_{f_0}^{(k)}(\t{w}+1)=f_0^q\circ \varphi_{f_0}^{(k)}(\t{w})$ if one side is defined. Then
\begin{itemize}
	\item [(3)] $\varphi_{f_0}^{(k)}({\rm Dom}(\varphi_{f_0}^{(k)}))\subset {\rm Dom}(f_0)$ and  ${\rm Dom}(\varphi_{f_0}^{(k)})$ contains the region
	$$\t{\QQQ}_{0}=\t{\QQQ}_{0}(\t w_*):=\{\tilde{w}\in\mathbb C:\pi/3<{\rm arg}(\tilde{w}+\t{w}_*)<5\pi/3\}$$
	for  a large number $\t w_*$.
\end{itemize}

The attracting coordinate $\phi_{-,f_0}^{(k)}$ also has an analytic extension $\psi_{f_0}^{(k)}$ to the $k$-th parabolic basin of $(f_0^q,0)$  by defining
\begin{equation*}\label{eq:extension}
	\psi_{f_0}^{(k)}(z)=\phi_{-,f_0}^{(k)}(f_0^{mq}(z))-m,\ \text{ if } f_0^{mq}(z)\in Q_{-}^{(k)}.
\end{equation*}

\subsubsection{Horn maps}\

In this subsection, we will consistently assume that:

\noindent{\bf Assumption 1}.
\emph{The map $f_0^q$ has exactly $q$ immediate parabolic basins at $0$, and their union $\Omega_{f_0}^*$, called the \emph{immediate parabolic basin of $(f_0,0)$}, contains only one critical point of $f_0$}.

Recall that $\Omega^{(k)}_{f_0}$ denotes the immediate parabolic basin of $(f_0^q,0)$ which contains $ Q_{-}^{(k)}$. Then $ Q_{+}^{(k)}\cap \,\Omega^{(k+1)}_{f_0}\not=\emptyset$. Note that the half plane $\{\t{w}\,|\ {\rm Im}\,\t{w}>\eta_0\}$  is contained in $(\varphi_{f_0}^{(k)})^{-1}(\Omega^{(k+1)}_{f_0})$ for some $\eta_0>0$. Let $\t{B}^{(k,u)}_{f_0}$ denote the component of $(\varphi_{f_0}^{(k)})^{-1}(\Omega^{(k+1)}_{f_0})$ that contains $\{\t{w} |\ {\rm Im}\,\t{w}>\eta_0\}$. It follows that  $T_1(\t{B}^{(k,u)}_{f_0})=\t{B}^{(k,u)}_{f_0}$.

We define a map $\t{\EEE}^{(k,u)}_{f_0}:\t{B}^{(k,u)}_{f_0}\to\mathbb C$ by
$$\t{\EEE}^{(k,u)}_{f_0}=\psi_{f_0}^{(k+1)}\circ \varphi_{f_0}^{(k)}.$$
It holds that $\t{\EEE}^{(k,u)}_{f_0}(\t{w}+1)=\t{\EEE}_{f_0}(\t{w})+1$ for $\t{w}\in\t{B}^{(k,u)}_{f_0}$. Hence one obtains a well-defined map
\[\EEE_{f_0}^{(k,u)}=\pi\circ\t{\EEE}_{f_0}^{(k,u)}\circ \pi^{-1}:B^{(k,u)}_{f_0}\to \mathbb C^*\]
via $\pi(\t{w}):=e^{2\pi i\t{w}}$ with $B^{(k,u)}_{f_0}:=\pi(\t{B}^{(k,u)}_{f_0})$. Moreover this map extends to $0$ holomorphically by defining $\EEE_{f_0}^{(k,u)}(0)=0$.

Since $ Q_{+}^{(k)}\cap \Omega^{(k)}_{f_0}$ is also non-empty, we can similarly define
\begin{itemize}
	\item $\t{B}^{(k,\ell)}_{f_0}$, the component of $(\varphi_{f_0}^{(k)})^{-1}(\Omega^{(k)}_{f_0})$   containing $\{\t{w} |\ {\rm Im}\,\t{w}<-\eta_0\}$;\vspace{2pt}
	\item a  holomorphic map $\t{\EEE}_{f_0}^{(k,\ell)}=\psi_{f_0}^{(k)}\circ\varphi_{f_0}^{(k)}:\t{B}^{(k,\ell)}_{f_0}\to\mathbb C$;\vspace{2pt}
	\item  a  holomorphic map $\EEE_{f_0}^{(k,\ell)}:B^{(k,\ell)}_{f_0}\to \mathbb C^*$, which extends  holomorphically  to $\infty$  by  $\EEE_{f_0}^{(k,\ell)}(\infty)=\infty$.
\end{itemize}

The Fatou coordinates $\phi_{\pm,f_0}^{(k)}$ can be normalized such that
\[ \t{\EEE}_{f_0}^{(k,u)}(\t{w})=\t{w}+o(1), \,\text{ as } {\rm Im}\,\t{w}\to +\infty.\]
Then we have
\begin{equation}\label{eq:normalize}
	\big(\EEE_{f_0}^{(k,u)}\big)'(0)=1\quad\text{ and }\quad \big(\EEE_{f_0}^{(k,\ell)}\big)'(\infty)\not=0.
\end{equation}

The maps $\EEE_{f_0}^{(k,u)},\EEE_{f_0}^{(k.\ell)}$ are called  the $k$-th upper and lower \emph{horn maps} associated to $f_0$.  They have the following properties.

\begin{pro}\label{pro:covering}
	For each $k\in\Z_q$, we have that
	\begin{enumerate}
		\item the sets $B^{(k,u)}_{f_0}\cup\{0\},B^{(k,\ell)}_{f_0}\cup\{\infty\}$ are simply connected and disjoint; 
		\item $\EEE_{f_0}^{(k,u)}:B^{(k,u)}_{f_0}\to\mathbb C^*$ \emph{(resp. $\EEE_{f_0}^{(k+1,\ell)}:B^{(k+1,\ell)}_{f_0}\to\mathbb C^*$)} is a branched covering of infinite degree, ramified only over one point $\pi\circ\psi_{f_0}^{(k+1)}(c_{k+1})$, where $c_{k+1}\in \Omega^{(k+1)}_{f_0}$ is a point in the orbit of the unique critical point of $f_0$ within  $\Omega_{f_0}^*:=\bigcup_{j=0}^{q-1}\Omega_{f_0}^{(j)}$;
		\item $(\EEE_{f_0}^{(k,u)})''(0)\not=0$ and $\EEE_{f_0}^{(k,u)}$ has a simply-connected immediate parabolic basin, which contains only one critical point.
	\end{enumerate}
\end{pro}
If $q=1$, this proposition is the combination of Proposition 4.5.4 and Lemma 4.5.5 in \cite{Shi-book}. The argument in the case of $q>1$ is similar, and we omit the proof.

\subsection{Perturbation of parabolic fixed points}\label{sec:perturbation}

Let
\begin{equation}\label{eq:F1}
	\FFF_1=\bigg\{f\in\FFF\mid f'(0)=e^{2\pi i \frac{p+\alpha(f)}{q}}\text{ with $\alpha(f)\not=0$ and $|{\rm arg}\,\alpha(f)|<\pi/4$}\bigg\}.
\end{equation}
In the following, we always assume  $f\in\FFF_1$ with $\alpha(f)$ sufficiently small.

By the parabolic bifurcation theory, such maps $f$ still have  Fatou coordinates, fundamental domains and horn maps, and these quantities depends continuously on $f$.
We will review some necessities  in this topic.

\begin{figure}[h]   
	\begin{center}
		\includegraphics[height=8.8cm]{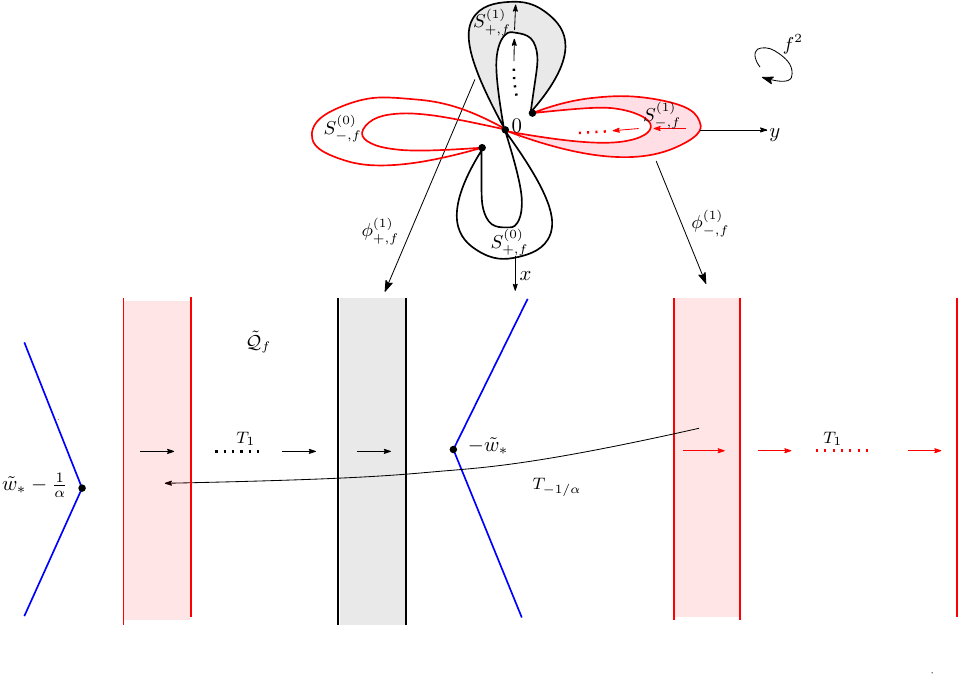}
	\end{center}
	\caption{The fundamental domains, Fatou coordinates and return maps for $f$,  in this case $q=2$.}\label{fig:f}
\end{figure}

The materials used in this subsection are sourced from \cite{S,Shi-book}. When $q$ equals to $1$, the detailed discussions and proofs appear in \cite{Shi-book}. The arguments can be similarly extended to the case of $q>1$ following the way explained in \cite[Appendix A.5]{S}.


As $ f\to f_0$, the fixed point $0$ of $f_0^q$ with multiplicity $q+1$ splits into a fixed point $0$ of $f$ and a cycle of period $q$ of $f$. These $q$-periodic points of $f$ near $0$ can be labelled so that
\[\sigma^{(k)}(f)=(-2\pi i\alpha(f))^{1/q} e^{2\pi i k/q}(1+o(1)),\ k\in\Z_q,\]
where $|{\rm arg}(-2\pi i\alpha(f))^{1/q}|<\pi/q$. Then $f(\sigma^{(k)}(f))=\sigma^{(k+p)}(f)$.

\subsubsection{Fatou coordinates and fundamental domains}\label{sec:Fatou-coordinates}
\

For each $k\in\Z_q$, there exist closed Jordan domains $S_{+,f}^{(k)},S_{-,f}^{(k)}$ and holomorphic maps $\phi_{+,f}^{(k)},\phi_{-,f}^{(k)}$,  satisfying the following properties (see Figure \ref{fig:f}):
\begin{itemize}
	\item [(1)] $S_{\pm,f}^{(k)}$ is bounded by an arc $\ell_{\pm,f}^{(k)}$ and its image $f^q(\ell_{\pm,f}^{(k)})$ such that
	\vspace{3pt}
	\begin{itemize}
		\item [(a)] $\ell_{\pm,f}^{(k)}$ joins the points $\{0,\sigma^{(k)}(f)\}$ and $\ell_{\pm,f}^{(k)}\bigcap f^q(\ell_{\pm,f}^{(k)})=\{0,\sigma^{(k)}(f)\}$,\vspace{1pt}
		\item [(b)] ${S}_{+,f}^{(k)}\bigcap {S}_{-,f}^{(k)}=\{0,\sigma^{(k)}(f)\}$, and ${S}_{\pm,f}^{(k)}\setminus\{0,\sigma^{(k)}(f)\},k\in \Z_q,$ are pairwise disjoint;
	\end{itemize}
	\vspace{3pt}
	\item [(2)] $\phi_{\pm,f}^{(k)}$ is defined, holomorphic and injective in a neighborhood of $S_{\pm,f}^{(k)}\setminus\{0,\sigma^{(k)}(f)\}$, such that \begin{itemize}
		\item [(a)] $\phi_{\pm,f}^{(k)}\circ f^q(z)=\phi_{\pm,f}^{(k)}(z)+1$ if $z,f^q(z)\in {\rm Dom}(\phi_{\pm,f}^{(k)})$, and\vspace{1pt}
		\item  [(b)] $\phi_{\pm,f}^{(k)}=\phi_{\pm,f}^{(0)}\circ f^{r_k}$, where $0\leq r_k<q$ satisfies $k+r_kp\equiv 0\,({\rm mod}\ q)$;
	\end{itemize}\vspace{3pt}
	\item [(3)] if $z\in S_{-,f}^{(k)}$, then there is an $n\geq1$ such that $f^n(z)\in S_{+,f}^{(k)}$ and for the smallest such $n$,
	\[\phi_{+,f}^{(k)}(f^{qn}(z))=\phi_{-,f}^{(k)}(z)-\frac{1}{\alpha(f)}+n;\]
	\item [(4)] when $f\in\FFF_1$ tends to $f_0$, we have ${S}_{\pm,f}^{(k)}\to {S}^{(k)}_{\pm,f_0}$ in the Hausdorff metric.
\end{itemize}

The above results are quoted from \cite[Proposition 3.2.2]{Shi-book}. We call $S_{\pm,f}^{(k)}$ and $\phi_{\pm,f}^{(k)}$ the $k$-th repelling/attracting {\it fundamental domains} and {\it Fatou coordinates} of $f$, respectively.

These facts show that, when perturbing $f_0$ in the family $\FFF_1$, the gate  between the $k$-th attracting and repelling petals of $f_0$ is open.

\subsubsection{Extension of Fatou coordinates}\label{sec:repelling-coordinate}
\

The inverse $(\phi_{+,f}^{(k)})^{-1}$ of the $k$-th repelling Fatou coordinate of $f$ can be extended to
an analytic map $\varphi_f^{(k)}:{\rm Dom}(\varphi_f^{(k)})\to\mathbb C$  with the following properties.
\begin{itemize}
	\item [(1)] If $\t{w},\t{w}+1\in{\rm Dom}(\varphi_f^{(k)})$, then $\varphi_f^{(k)}(\t{w})\in{\rm Dom}(f^q)$ and $\varphi_f^{(k)}(\t{w}+1)=f^q\circ\varphi_f^{(k)}(\t{w})$.
	\item [(2)] ${\rm Dom}(\varphi_f^{(k)})$ contains the region
	\begin{equation}\label{eq:domain}
		\t{\QQQ}_f=\t\QQQ_f(\t w_*):=\left\{\t{w}\in\mathbb C\mid \pi/3<{\rm arg}(\t{w}+\t{w}_*)<5\pi/3,\Big|{\rm arg}\left(\t{w}+\frac{1}{\alpha}-\t{w}_*\right)\Big|<2\pi/3\right\}.
	\end{equation}
	
	\item [(3)] If $\t{w}\in\t{\QQQ}_{f}$ and ${\rm Im}\,\t{w}\to+\infty$, then $\varphi_f^{(k)}(\t{w})\to 0$; and  when $\t w\in\t{\QQQ}_{f}$ and ${\rm Im}\,\t{w}\to-\infty$, we have $\varphi_f^{(k)}(\t{w})\to \sigma^{(k)}(f)$.
	\item [(4)] When $f\in\FFF_1$ tends to $f_0$, we have $\varphi_f^{(k)}\to \varphi_{f_0}^{(k)}$ in the compact-open topology.
	\item [(5)]  The $k$-th repelling and attracting coordinates of $f$ are related by the formula
	\[\varphi_f^{(k)}\big(\phi_{-,f}^{(k)}(z)-1/\alpha\big)=z,\ \forall\,z\in S_{-,f}^{(k)}.\]
	
\end{itemize}

Statements (1)-(4) are quoted directly from \cite[Proposition 3.2.3 (i),(iv)]{Shi-book}. Statement (5) is due to a special normalization of $\phi_{\pm,f}^{(k)}$ proved in \cite[Lemma 3.4.2]{Shi-book}.

Given any $\epsilon>0$, the restriction of $f\in\FFF_1$ on $\D(0,\epsilon)$ still belong to $\FFF_1$ and $f_0|_{\D(0,\epsilon)}\in\FFF_0$. By Section \ref{sec:Fatou-coordinate} (3), there exists $\t w_*>0$ such that $\varphi_{f_0|_{\D(0,\epsilon)}}^{(k)}(\t\QQQ_0)=\varphi_{f_0}^{(k)}(\t\QQQ_0)\subset {\rm Dom}(f_0|_{\D(0,\epsilon)})=\D(0,\epsilon)$. Together with statement (2) above, we obtain that

\begin{itemize}
	\item [(6)] For any $\epsilon>0$, there is $\t w_*>0$ such that $\varphi_f^{(k)}(\t{\QQQ}_f)\subset \D(0,\epsilon)$ for all $f\in\FFF_1$ close  to $f_0$.
\end{itemize}

The $k$-th attracting Fatou coordinate $\phi_{-,f}^{(k)}$ also has an analytic extension to a suitable region by the relation $\phi_{-,f}^{(k)}\circ f^q(z)=\phi_{-,f}^{(k)}(z)+1$. The extended map is denoted by $\psi_f^{(k)}$.


\subsubsection{Return maps associated to $f$}\label{sec:return-map}
\

{ In suitable regions, we can define the $k$-th upper and lower \emph{return maps} of $f$ by
	\[\tilde{\RRR}_f^{(k,u)}:=\psi_f^{(k+1)}\circ\varphi_f^{(k)}-\frac{1}{\alpha}\ \text{ and }\ \tilde{\RRR}_f^{(k,\ell)}:=\psi_f^{(k)}\circ\varphi_f^{(k)}-\frac{1}{\alpha},\]
	where\, $\tilde{\EEE}_f^{(k,u)}=\psi_f^{(k+1)}\circ\varphi_f^{(k)}$ and\, $\tilde{\EEE}_f^{(k,\ell)}=\psi_f^{(k)}\circ\varphi_f^{(k)}$ serve as the $k$-th upper and lower \emph{horn maps} of $f$, respectively. By analytic extension,}
we obtain the global return maps $\t{\RRR}_f^{(k,u)}:{\rm Dom}(\t{\RRR}_f^{(k,u)})\to \mathbb C$ and $\t{\RRR}_f^{(k,\ell)}:{\rm Dom}(\t{\RRR}_f^{(k,\ell)})\to\mathbb C$ that satisfy the following properties.

For simplicity of the statements, we use the symbol ``$\star$'' to represent either ``$u$'' or ``$\ell$'', and set $\H^u(\eta):=\{z:{\rm Im}\,z\geq \eta\}$ and $\H^\ell(\eta):=\{z:{\rm Im}\,z\leq -\eta\}$ for any $\eta>0$.
\begin{enumerate}
	\item The set ${\rm Dom}(\t{\RRR}_f^{(k,\star)})$ contains $\H^\star(\eta_0)$ for a constant $\eta_0>0$ independent of $f$, and is invariant under the translation $T_1$.\vspace{2pt}
	\item  If $\t w\in {\rm Dom}(\t{\RRR}_f^{(k,\star)})$, then $\t{\RRR}_f^{(k,\star)}(\t w+1)= \t{\RRR}_f^{(k,\star)}(\t w)+1$.\vspace{2pt}
	\item $\t{\RRR}_f^{(k,u)}(\t w)-\t w$ tends to $-1/\alpha(f)$ when ${\rm Im}\,\t w\to+\infty$, and $\t{\RRR}_f^{(k,\ell)}(\t w)-\t w$ tends to a constant when ${\rm Im}\,\t w\to-\infty$.\vspace{2pt}
	\item If $\t w\in {\rm Dom}(\t{\RRR}_f^{(k,\star)})\cap \t\QQQ_f$ and $\t w'=\t{\RRR}_f^{(k,
		\star)}(\t w)+n\in \t\QQQ_f$ for an integer $n\geq 0$, then
	$$f^{qn}(\varphi_f^{(k)}(\t w))=\varphi_f^{(k+1)}(\t w')\text { if $\star=u$, and }f^{qn}(\varphi_f^{(k)}(\t w))=\varphi_f^{(k)}(\t w')\text{ if }\star=\ell.$$
	\item If $f\in\FFF_1$ and $f\to f_0$, then $\t{\RRR}_f^{(k,\star)}+1/\alpha(f)\to\t{\EEE}_{f_0}^{(k,\star)}$ in the compact-open topology, and the convergence is uniform on $\H^\star(\eta_0)$.
\end{enumerate}
The above results are quoted directly from Propositions 3.2.3 and 4.4.1 in \cite{Shi-book}.

Since $\t{\RRR}_f^{(k,u)}$ commutes with $T_1$, we obtain a
well-defined analytic mapping
\[\RRR_f^{(k,u)}=\pi \circ \t{\RRR}_f^{(k,u)}\circ \pi^{-1}: {\rm Dom}({\RRR}_f^{(k,u)})\to \mathbb C^*.\]
The domain ${\rm Dom}(\RRR_f^{(k,u)}):= \pi({\rm Dom}(\t{\RRR}_f^{(k,u)}))$ contains $\{w:0<|w|<e^{-2\pi\eta_0}\}$, and $\RRR_f^{(k,u)}$ extends to $0$ analytically by defining  $\RRR_f^{(k,u)}(0)=0$.

Similarly, the function $\t{\RRR}_f^{(k,\ell)}$ induced a analytic map
\[\RRR_f^{(k,\ell)}=\pi \circ \t{\RRR}_f^{(u,\ell)}\circ \pi^{-1}: {\rm Dom}({\RRR}_f^{(k,\ell)})\to \mathbb C^*\]
such that ${\rm Dom}(\RRR_f^{(k,\ell)})$ contains $\{w:e^{2\pi\eta_0}<|w|<\infty\}$ and $\RRR_f^{(k,\ell)}$ extends to $\infty$ analytically by defining  $\RRR_f^{(k,\ell)}(\infty)=\infty$.
By statement (3) above, it follows that
\begin{equation}\label{eq:derivative}
	(\RRR_f^{(k,u)})'(0)=e^{-2\pi i\frac{1}{\alpha(f)}}.
\end{equation}

As an immediate consequence of statement (5), it follows that
\begin{itemize}
	\item [(6)]$e^{2\pi i/\alpha(f)}\RRR^{(k,u)}_f\to \EEE_{f_0}^{(k,u)}$ and $e^{2\pi i/\alpha(f)}\RRR^{(k,\ell)}_f\to \EEE_{f_0}^{(k,\ell)}$ as $f\in\FFF_1$ and $f\to f_0$.
\end{itemize}
The following fact relates  the iteration of $\RRR_f^{(k,u)}$ and that of $f$ \footnote{ We will not directly use this fact in our argument below. It only appears in Step 6 when sketching Shishikura's proof of Theorem \ref{thm:shishikura}. In the following Lemma \ref{lem:key1}, we shall give a more precise description of the relation between the iteration of $\RRR_f^{(k,u)}$ and that of $f$.}.
\begin{itemize}
	\item [(7)] If $U,U'$ are domains contained in of $\t{\QQQ}_f$ such that $({\RRR}_f^{(k,u)})^m$  is defined on $\pi(U)$ and $({\RRR}_f^{(k,u)})^m(\pi(U))\subset\pi(U')$ for some $m\geq1$, $\varphi_{f}^{(k)}|_U,\pi|_{U'}$ are injective, and $|{\rm arg}(\t{w}'+\frac{1}{2\alpha(f)}-\t{w})|<2\pi/3$ for $\t{w}\in U,\t{w}'\in U'$. Then there exists $n>m$ and $0\leq s<n$  such that
	\begin{equation}\label{eq:return1}
		f^{qn+s}=\varphi_f^{(k)}\circ(\pi|_{U'})^{-1}\circ (\RRR_f^{(k,u)})^m\circ\pi\circ(\varphi_f^{(k)}|_U)^{-1}\text{ on }\varphi_f^{(k)}(U). \footnote{ {If $q=1$, then $s=0$, because $f^n(\pi(U))$ can only enter the unique repelling fundamental domain of $f$. If $q>1$, then $s$ may be non-zero, since the iteration $f^{qn}(\pi(U))$ possibly enters the $k'$-th repelling fundamental domain  with $k'\not=k$. So we need an iteration $f^s$ sending $f^{qn}(\pi(U))$ to the $k$-th repelling fundamental domain.}}
	\end{equation}
\end{itemize}
This result is a consequence of statement (4) above. It appears in \cite{Shi-book} as Proposition 3.2.3 (iii') (Page 342). One can also refers to \cite[(4.2.3)\ and\ Section 7]{S}.

\subsection{Controlling the orbit}\label{sec:control}

To prove  Theorem \ref{thm:boundary-dimension0}, we need to document the location of the $f$-orbit of $\varphi_f^{(k)}(U)$  given in \eqref{eq:return1}.

We follow the notations and assumptions  in Sections \ref{sec:parabolic} and \ref{sec:perturbation}. Let $\Omega_{f_0}^{(0)},\ldots,\Omega_{f_0}^{(q-1)}$ denote the immediate parabolic basins of $(f_0^q,0)$.
Recall that $\t{B}_{f_0}^{(k,u)}=(\varphi_{f_0}^{(k)})^{-1}(\Omega_{f_0}^{(k+1)})$ and $\t{B}_{f_0}^{(k,\ell)}=(\varphi_{f_0}^{(k)})^{-1}(\Omega_{f_0}^{(k)})$. Moreover $B_{f_0}^{(k,u)}=\pi(\t{B}_{f_0}^{(k,u)})$ and $B_{f_0}^{(k,\ell)}=\pi(\t{B}_{f_0}^{(k,\ell)})$.

Let $\rho$ be any positive number. By Section \ref{sec:repelling-coordinate} (6), there exist a number $\t w_*>0$ and a neighborhood $\NNN_0$ of $f_0$ such that
\begin{equation}\label{eq:177}
	\bigcup_{k=0}^{q-1}f^k\left(\varphi_f^{(1)}(\t{\QQQ}_f(\t w_*))\right)\subset \D(0,\rho),\ \forall\, f\in\NNN_0\cap \FFF_1.
\end{equation}

\begin{lem}\label{lem:key1}
	Under the condition \eqref{eq:177}, given any  numbers $x_1<x_2$ and compact set $E\subset B_{f_0}^{(0,u)}$, there exist a compact set $K\subset \bigcup_{k=0}^{q-1} \Omega_{f_0}^{(k)}$  and a neighborhood $\NNN$ of $f_0$ with the following property:
	
	For any $f\in\FFF_1\cap \NNN$ and any domains $U,U'\subset\t\QQQ_f=\t{\QQQ}_{f}(\t w_*)$ such that
	\begin{itemize}
		\item $U,U'$ are contained in $\{\t w:  x_1< {\rm Re} (\t w)<x_2\}$, 
		\item $\pi(U)\subset E$, ${\RRR}_f^{(0,u)}(\pi(U))\subset\pi(U')$ and $\varphi_{f}^{(0)}|_{U},\pi|_{U'}$ are injective,
	\end{itemize}
	we have a minimal integer $n>0$ such that $f^{qn}(\varphi_f^{(0)}(U))\subset \varphi_f^{(1)}(U')$, and that the orbit
	\begin{equation}\label{eq:contral}
		\varphi_f^{(0)}(U)\overset{f^q}{\longrightarrow}\cdots\overset{f^q}{\longrightarrow}f^{qn}(\varphi_f^{(0)}(U))\overset{f^r}{\longrightarrow} f^{qn+r}(\varphi_f^{(0)}(U))\subset \varphi_f^{(0)}(U')
	\end{equation}
	is located in $ \D(0,\rho)\cup K$, where $0\leq r<q$ satisfies $rp\equiv-1\,({\rm mod}\,q)$.
	
\end{lem}
\begin{proof}
	According to Section \ref{sec:return-map} (6), we have a  neighborhood $\NNN\subset \NNN_0$ of $f_0$ such that $E\subset {\rm Dom}(\RRR_f^{(0,u)})$
	for all $f\in\NNN\cap\FFF_1$.
	
	Since $\t{\EEE}_{f_0}^{(0,u)}(\t w)=\t w+o(1)$ as $\t w\to+\infty$, by shrinking $\NNN$, it follows from Section \ref{sec:return-map} (5) that there exists a large number $\eta_0>0$ such that
	\begin{equation*}\label{eq:contral}
		|\t{\RRR}_f^{(0,u)}(\t{w})-(\t{w}-1/\alpha(f))|\leq 1
	\end{equation*}
	for any $f\in \NNN\cap\FFF_1$ and all $\t w\in\mathbb C$ with ${\rm Im}\,\t{w}>\eta_0$. By this and the definition of $\t\QQQ_f$, we can find a number $y_0>\eta_0$ such that for any $f\in \NNN\cap\FFF_1$,
	\begin{equation}\label{eq:estimate}
		\t\RRR_f^{(0,u)}(\t w)=\t w-1/\alpha(f)+O(1)\in \t\QQQ_f,\text{ when ${\rm Re}\,\t w\in[x_1,x_2],  {\rm Im}\,\t w\geq y_0$.}
	\end{equation}

	Since $\RRR_f^{(0,u)}(\pi(U))\subset \pi(U')$ and $\pi$ is injective on $U'$, there exists an integer $n$ with   minimal  $|n|$ such that \begin{equation}\label{eq:translation}
		T_{n}\circ \t{\RRR}_f^{(0,u)}(U)\subset U'\subset \t{\QQQ}_f.
	\end{equation}

	 \emph{Claim. There is a compact set $K_1\subset \Omega_{f_0}^{(1)}$ depending only on $f_0, x_1,x_2$ and $E$ such that, if $f$ is close to $f_0$, then  $f^{qn}(\varphi_f^{(0)}(U))\subset \varphi_f^{(1)}(U')$ with $n>0$ and $\varphi_f^{(0)}(U),\ldots, f^{qn}(\varphi_f^{(0)}(U))$ are contained in $K_1\cup \varphi^{(1)}_f(\t\QQQ_f)$.}
	
	One can quickly prove the lemma under this claim. By the claim and \eqref{eq:177}, the domains $\varphi_f^{(0)}(U),\ldots, f^{qn}(\varphi_f^{(0)}(U)), f^{qn+r}(\varphi_f^{(0)}(U))$ are contained in $K\cup\, \D(0,\rho)$ with $$K:=\bigcup_{k=0}^{q-1}f^k(K_1),$$ and $f^{qn}(\varphi_f^{(0)}(U))\subset \varphi_f^{(1)}(U')$.  Moreover, Section \ref{sec:Fatou-coordinates} (2).(b) implies that $f^r\circ\varphi_f^{(1)}(U')=\varphi_f^{(0)}(U')$.  Then the lemma is proved.
	It remains to prove the above claim in the following.

	First assume that $U$ lies above the horizontal line $\{\t w:{\rm Im}\,\t w=y_0\}$. Since $|x_2-x_1|\ll {\rm Re}(1/\alpha(f))$,  we conclude from \eqref{eq:estimate}  that the integer $n$ in \eqref{eq:translation} is positive and that $T_j\circ \t{\RRR}_f^{(0,u)}(U)\subset \t{\QQQ}_f$ for $j=0,\ldots,n$. Therefore, it follows from Section \ref{sec:return-map} (4) that the orbit
	\begin{equation}\label{eq:44}
		\varphi_f^{(0)}(U)=	\varphi_f^{(1)}(\t{\RRR}_f^{(0,u)}(U)) \overset{f^q}{\longrightarrow}\,\ldots\,\overset{f^q}{\longrightarrow} f^{qn}(\varphi_f^{(0)}(U))\subset \varphi_f^{(1)}(U')
	\end{equation}
	is contained in $\varphi_f^{(1)}(\t\QQQ_f)$. So the claim holds in this special case.

	To prove the claim in the general case, we  define a compact set
	\[L=\{\t w\in \pi^{-1}(E): x_1\leq {\rm Re}\,\t w \leq x_2,\ {\rm Im}\,\t w\leq y_0\}\subset {\rm Dom}(\t\EEE_{f_0}^{(0,u)}).\]
	There exists a minimal integer $m_0\geq0$ such that
	$P:=T_{m_0}(\t\EEE_{f_0}^{(0,u)} (L))=\t\EEE_{f_0}^{(0,u)}(T_{m_0}(L))$
	is a compact subset of
	$$-\t\QQQ_{0}=\{\tilde{w}\in\mathbb C:|{\rm arg}(\tilde{w}-\t{w}_*)|<2\pi/3\}.$$
	By shrinking $\NNN$, we have $P\subset T_{1/\alpha(f)}(\t\QQQ_f)$ for all $f\in\NNN\cap \FFF_1$, and it follows from Section \ref{sec:return-map} (2),(5)  that
	\begin{equation}\label{eq:33}
		T_{m_0}(\t\RRR_f^{(0,u)}(L))=\t\RRR_f^{(0,u)}\left(T_{m_0}(L)\right)
		\approx T_{-1/\alpha(f)} (P)\subset \t{\QQQ}_f.
	\end{equation}

	On the other hand,  $\varphi_{f_0}^{(0)}(L)$ is compact in $\Omega_{f_0}^{(1)}$. So there is a compact set $K_1\subset \Omega_{f_0}^{(1)}$ whose interior contains $\bigcup_{j=0}^{m_0}f_0^{qj}(\varphi_{f_0}^{(0)}(L))$.  By shrinking $\NNN$,  
	\begin{equation}\label{eq:compact}
		\text{$\bigcup_{j=0}^{m_0}f^{qj}(\varphi_{f}^{(0)}(L))\subset K_1$,\  $\forall\, f\in\NNN\cap\FFF_1$}
	\end{equation}
	
	We divide $U$ into $U_1=U\setminus L$ and $U_2=U\cap L$. By applying \eqref{eq:estimate} and \eqref{eq:33} to $U_1$ and $U_2$ respectively, we get
	\[\text{$T_{m_0}(\t{\RRR}_f^{(0,u)}(U))\subset \t\QQQ_f$, $\forall\,f\in\NNN\cap\FFF_1$}.\]
	
	Note that ${\rm Re}(1/\alpha(f))\gg\max\{\max\{|{\rm Re}\,\t w|; \t w\in P\}, |x_1|,|x_2|, m_0\}$. By  \eqref{eq:33}, we have $n\gg m_0$, and $T_j( \t{\RRR}_f^{(0,u)}(U))\subset \t{\QQQ}_f$ for $j=m_0,\ldots,n$. Therefore, as previously shown, the domains
	$$ f^{qm_0}\big(\varphi_f^{(0)}(U)\big),\ldots, f^{qn}\big(\varphi_f^{(0)}(U)\big)\subset \varphi_f^{(1)}(\t\QQQ_f).$$

	Moreover, by applying  \eqref{eq:44} to $U_1$, we have
	$
	\varphi_f^{(0)}(U_1),\ldots, f^{qm_0}\big(\varphi_f^{(0)}(U_1)\big)\subset \varphi_f^{(1)}(\t\QQQ_f).
	$
	Combining this fact together with  \eqref{eq:compact}, it holds that
	$$\varphi_f^{(0)}(U),\ldots, f^{qm_0}(\varphi_f^{(0)}(U))\subset K_1\cup \varphi_f^{(1)}(\t\QQQ_f).$$
	Then the claim is proved, and the lemma follows.
\end{proof}

\section{Hausdorff dimension for the boundary of attracting domain}  \label{hd-attracting}

\subsection{Hyperbolic sets with  large Hausdorff dimensions}\label{sec:construction}
We say that a compact set $X\subset{\C}$ is \emph{homogeneous} if $X$ has
no isolated points, and ${\rm H.dim}(U\cap X)={\rm H.dim}(X)$ for any open set $U$ intersecting $X$.

Let $f$ be a rational map. 
Let $U$ be a simply-connected open set in $\C$ with $\#(\partial U)\geq2$. Suppose that $U_1,\ldots,U_{\sp N}$ are pairwise disjoint simply-connected domains compactly contained in $U$, and $l_1,\ldots,l_{\sp N}$ are positive integers such that $f^{l_k}:U_k\rightarrow U$ is bijective. Then we call $\big(f,U,\{U_k\}_{k=1}^{\sp N}\big)$ a \emph{repelling system}. Define $g:\bigcup_{k}U_k\to U$ by $g|_{U_k}=f^{l_k}|_{U_k}$, and let $X_0$ be the set of \emph{non-escaping points} under the iteration of $g$, i.e., $x\in X_0$ precisely if $g^n(x)\in \bigcup_{k}U_k$ for every $n\geq0$.  Then  $X:=\bigcup_{j=0}^l f^j(X_0)$ is a homogenous hyperbolic set of $f$, where $l:=\max_{1\leq j\leq  N} l_j$ (see  \cite[Lemma 2.1]{S}).  


\begin{thm}[Shishikura]\label{thm:shishikura}
	Let $f_0\in {\rm Rat}_d$ have a parabolic fixed point $\zeta$ with multiplier $e^{2\pi ip/q}$ \emph{($p,q\in\Z,(p,q)=1$)} and that the immediate parabolic basin of $\zeta$ contains only one critical point of $f_0$. Suppose that $\FFF_2\subset {\rm Rat}_d$ is a family of maps tending to $f_0$ that satisfy the following  property:
	there is a fixed point $\zeta_f$ of $f$ close to $\zeta$, with  multiplier $f'(\zeta_f)={\rm exp}(2\pi i\frac{p+ \alpha(f)}{q})$, where $\alpha(f)$ has the form
	\begin{equation}\label{eq:good}
		\alpha(f)=\frac{1}{a_1(f)-\dfrac{1}{a_2(f)+\nu(f)}}
	\end{equation}
	such that $\mathbb N\ni a_1(f),a_2(f)\to\infty$ and $\nu(f)\to \nu_0\in\mathbb C$ as $f\to f_0$.
	
	Then given any $\epsilon>0$, there exists a neighborhood $\NNN$ of $f_0$ in ${\rm Rat}_d$ such that if $f\in\NNN\cap\FFF_2$,
	its Julia set $J(f)$ admits a homogenous hyperbolic set $X_f$ with
	${\rm H.dim}(X_f)>2-\epsilon.$
\end{thm}

This theorem is  \cite[Theorem 2]{S}.  For  our discussion below, we will   sketch   the construction of  $X_f$ under the following \vspace{3pt}

\textbf{Assumption: }{\it No critical points of $f_0$ lie on the boundary of the immediate parabolic basin $\Omega_{f_0}^*=\bigcup_{k=0}^{q-1} \Omega_{f_0}^{(k)}$ of $\zeta$.}

By a continuous family of translations,  we may assume $\zeta=0$ and $\zeta_f=0$ for $f\in\FFF_2$.  Then  $f_0$ satisfies all assumptions in Section \ref{sec:parabolic}. 
Note that 
\begin{equation}\label{attracting-condition}
	\zeta_f \text { is attracting } \Longleftrightarrow {\rm Im}\,\alpha(f)>0  \Longleftrightarrow  {\rm Im}\,\nu(f)<0.
\end{equation}
We can choose a neighborhood $\NNN$ of $f_0$ in ${\rm Rat}_d$ such that for any $f\in\NNN\cap\FFF_2$, the argument of $\alpha(f)$ is contained in $(-\pi/4,\pi/4)$.
Thus $\NNN\cap \FFF_2$
is  a subset of $\FFF_1$ that is defined
in \eqref{eq:F1}.\vspace{3pt}

\noindent\emph{Step 1. An inverse orbit $\{z_j\}$ of $f_0$.}\vspace{3pt}

A sequence $\{x_n\}_{n\geq0}$ is called an \emph{inverse orbit} of a map $F$ if $F(x_j)=x_{j-1}$ for all $j\geq1$.
Recall that $\Omega_{f_0}^{(k)},k=0,\ldots,q-1$ are the components of immediate parabolic basins of $(f_0,0)$.
Then there exists an inverse orbit $\{z_j\}_{j\geq0}$ of $f_0^q$ such that $z_0=0,0\not=z_j\in\partial \Omega^{(1)}_{f_0}$ for  $j>0$, and $z_j\to 0$ as $j\to\infty$ within the $0$-th repelling petal of $f_0^q$.\vspace{3pt}

\noindent\emph{Step 2. The maps $g_0,h_0$ and the point $\xi_0$}.\vspace{3pt}

In the following, we always assume $k=0$, and hence omit the superscript $(k)$ when referring to the  quantities in Section \ref{sec:parabolic-implosion}. By Section \ref{sec:parabolic}, we obtain $\varphi_{f_0}$ and $g_0:=\EEE_{f_0}^{u}$.
Then there exists a point $\tilde{w}_0\in {\rm Dom}(\varphi_{f_0})$ corresponding to the inverse orbit $\{z_j\}$ such that $\varphi_{f_0}(\tilde{w}_0-j)=z_j$.

As shown in \cite[Lemma 6.1]{S}, one can find two inverse orbit $\{w_j\}_{j\geq0},\{w_j'\}_{j\ge0}$ of $w_0:=\pi(\tilde{w}_0)$ for the map $g_0$ such that: $w_0=w_0'$; $w_j,w_j'\to 0\, (j\to\infty)$; $\{w_j\}_{j\geq1}\cap\{w_j'\}_{j\geq1}=\emptyset$; and $w_j,w_j'\,(j\geq1)$ are not critical points of $g_0$.

By Proposition \ref{pro:covering}, $g_0(z)=z+b_0z^2+O(z^3)$ near $0$ with $b_0\not=0$. So we also obtain $\varphi_{g_0}$ and $h_0:=\EEE_{g_0}^u\cup\EEE_{g_0}^\ell$. Similarly as above, there exist different $\t{\xi}_0,\t{\xi}_0'\in {\rm Dom}(\varphi_{g_0})$ corresponding to the orbits $\{w_j\},\{w_j'\}$ respectively, such that $w_0=\varphi_{g_0}(\t{\xi}_0)$ and $w_0'=\varphi_{g_0}(\t{\xi}_0')$. Set $\xi_0:=\pi(\t{\xi}_0)$ and $\xi_0':=\pi(\t{\xi}_0')$. It follows $\xi_0\not=\xi_0'$ since $\{w_j\}_{j\geq1}\cap \{w_j'\}_{j\geq 1}=\emptyset$.   Notice that $h_0$ has only one critical value by Proposition \ref{pro:covering}. As a consequence, one of the points $\xi_0$ and $\xi_0'$, say $\xi_0$, is not the critical value of $e^{-2\pi i\nu_0}h_0$. See Figure \ref{fig:step1}.
\begin{figure}
	\centering
	\includegraphics[width=15cm]{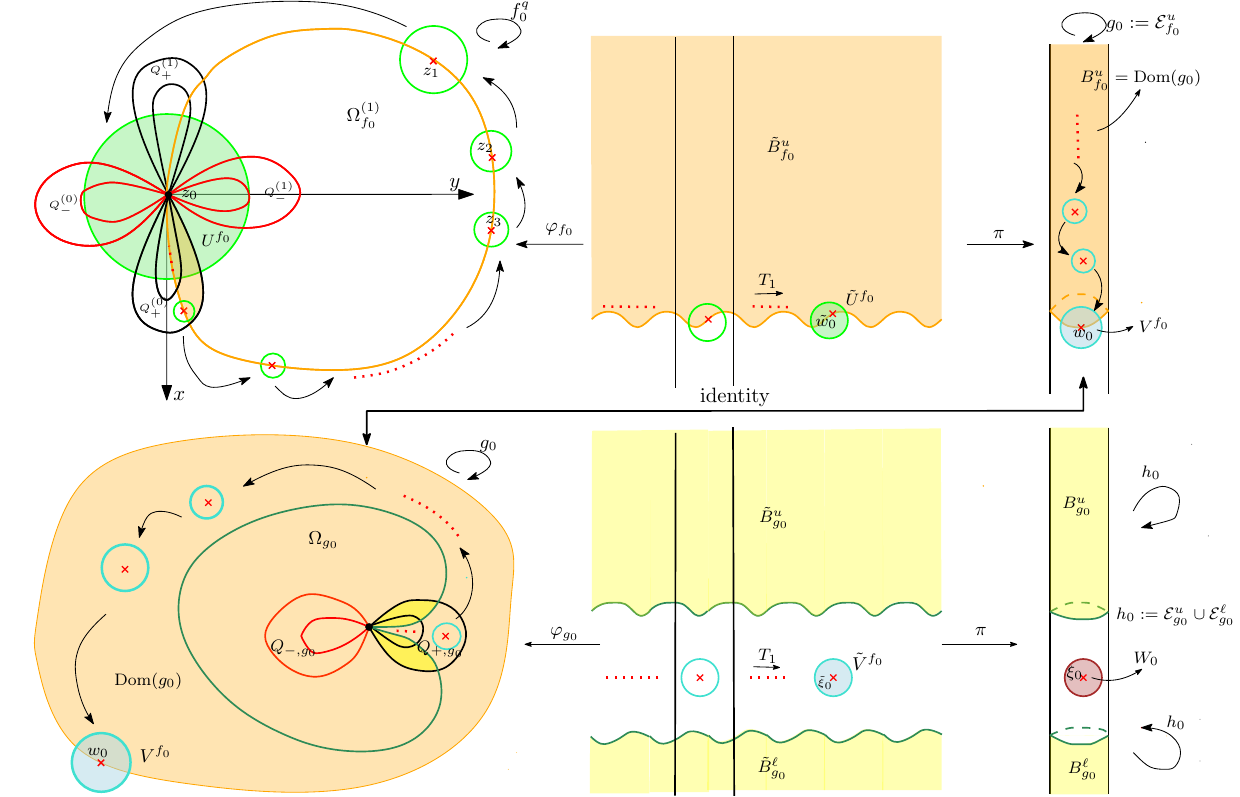}
	\caption{The maps $g_0,h_0$ and the points $w_0,\xi_0$}\label{fig:step1}
\end{figure}
\vspace{3pt}

\noindent\emph{Step 3. The maps $g,h$ and disks $\{W_k^f\}$.}\vspace{3pt}

When $f\in\NNN\cap \FFF_2$, we can apply the results  in Section \ref{sec:perturbation} to $f$. By formula \eqref{eq:derivative} and Section \ref{sec:return-map} (6),
the map $g:=\RRR_f^u$ is well-defined such that $g'(0)={\rm exp}(2\pi i\alpha(g))$ with  $\alpha(g)=1/(a_2(f)+\nu(f))$, and that $g\to g_0$ as $\FFF_2\ni f\to f_0$. With a similar reason, by shrinking $\NNN$ if necessary, the map $h:=\RRR^u_g\cup \RRR^\ell_g$ still exists and it tends to $e^{-2\pi i\nu_0}h_0$ as $\FFF_2\ni f\to f_0$.

Since the map  $e^{-2\pi i\nu_0}h_0$ has only one critical value (by Proposition \ref{pro:covering}), one of the two connected components $B^u_{h_0},B^\ell_{h_0}$ of ${\rm Dom}(h_0)$, say $B^u_{h_0}$, is disjoint from this critical value. Hence the local inverse of $e^{-2\pi i \nu_0}h_0$ near $0$ can be extended to $B^{u}_{h_0}\cup\{0\}$. So we have an analytic injection $H_0:B^u_{h_0}\cup\{0\}\to B^u_{h_0}\cup\{0\}$ such that $e^{-2\pi i\nu_0}h_0\circ H_0={\rm id},H_0(0)=0$ and $|H_0'(0)|<1$.

There exist a disk $ D(\ni0)\Subset B_{h_0}^{u}$ and a disk $W_0\ni \xi_0$ such that:
$\ov{W_0}$ is disjoint with the critical value of ${e^{-2\pi i\nu_0}}h_0$ (by the choice of $\xi_0$) and
$e^{-2\pi i \nu_0}h_0(D)$ covers $\ov{W_0}$ and $B^{u}_{h_0}$.
Therefore, by shrinking $\NNN$, we obtain
\begin{itemize}
	\item a disk $W_1^f\subseteq D$ such that $h:W_1^f\to W_0$ is a conformal homeomorphism, and
	\item an analytic injection $H:D\to D$ with $h\circ H={\rm id}$
\end{itemize}
for any $f\in\FFF_2\cap\NNN$. Here we use the subscript $f$ in $W_1^f$ since $h$ stems from $f$. Then for each $k\geq2$, we define $W_k^f:=H^{k-1}(W_1^f)$.\vspace{3pt}

\noindent \emph{Step 4. From $W_0$ to $U^f$.}\vspace{3pt}

The point $z_0=0$ in the dynamical plane of $f_0,f$ and the point $\xi_0$ in the dynamical plane of $h_0,h$ are related as follows:
\[z_0\overset{\varphi_{f_0}}{\longleftarrow} \tilde{w}_0\overset{\pi}{\longrightarrow}w_0\overset{\varphi_{g_0}}{\longleftarrow}\tilde{\xi}_0\overset{\pi}{\longrightarrow}\xi_0. \]
Since the inverse orbits $\{z_j\}_{j\geq1}$ of $f_0$ and $\{w_j\}_{j\geq1}$ of $g_0$ avoid the critical points of $f_0$ and $g_0$ respectively,
the map $\varphi_{f_0}\circ\pi^{-1}\circ \varphi_{g_0}\circ\pi^{-1}$ is well-defined and injective on $W_0$
by shrinking $W_0$ if necessary. For $f\in\FFF_2$ close to $f_0$, the map $\varphi_{f}\circ\pi^{-1}\circ \varphi_{g}\circ\pi^{-1}$ is still well-defined and injective on $W_0$, as $\varphi_f,\varphi_g$ are close to $\varphi_{f_0},\varphi_{g_0}$ respectively. Let us denote \[\text{$V^f:=\varphi_g\circ\pi^{-1}(W_0)$\ \ and\ \ $U^f:=\varphi_f\circ\pi^{-1}(V^f)$ (see Figure \ref{fig:proof1}).}\]
\begin{figure}
	\centering
	\includegraphics[width=15.5cm]{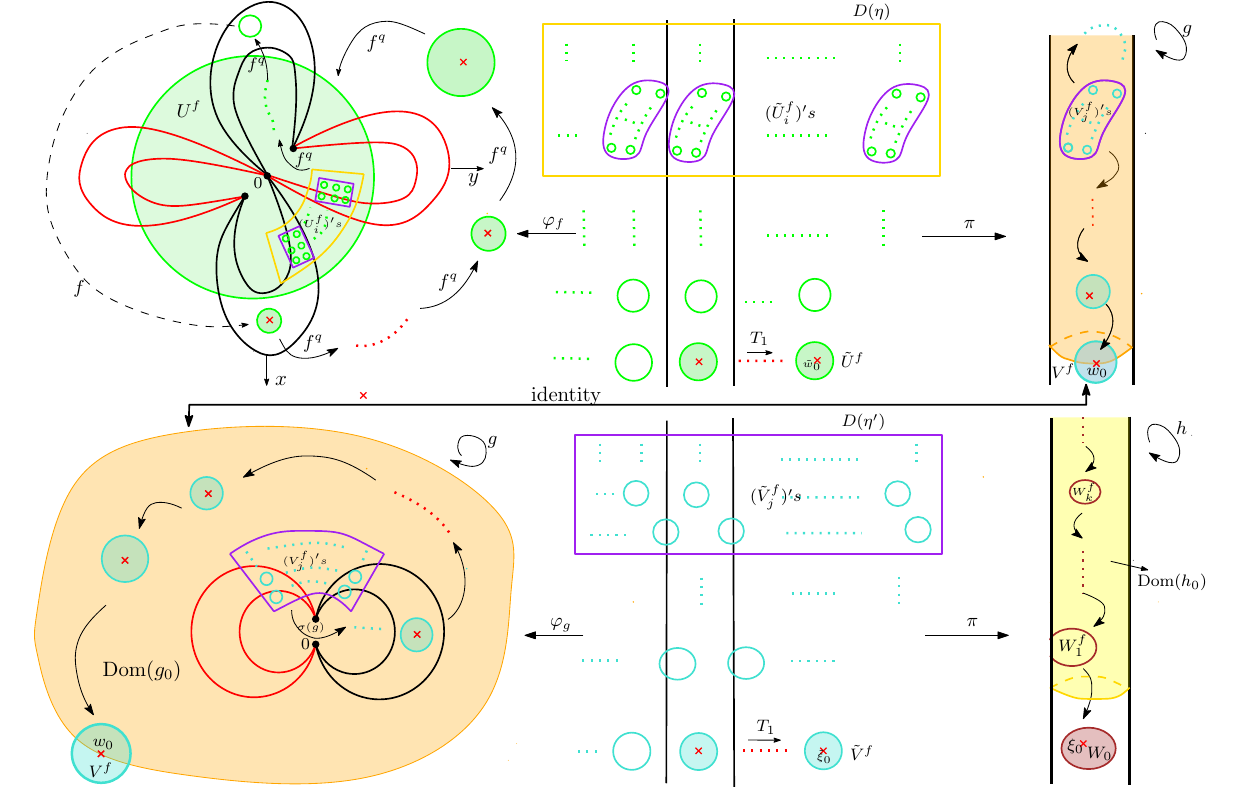}\\
	\caption{The construction of $\{W_k^f\},\{V_j^f\}$ and $\{U_i^f\}$.}\label{fig:proof1}
\end{figure}

\noindent\emph{Step 5. From $\{W_k^f\}$ to $\{U_i^f\}$.}

By a similar process as shown in step 4, we can pullback the collection of disks $\{W_k^f\}$ in the dynamical plane of $h$ to the collection of disks $\{U_i^f\}$ in the dynamical plane of $f$.

For any real number $t$, denote $D(t)$ the square with the center $it$ and the edge length $|t|/2$ .
Let $\eta>0$ be a large number and set $\eta'=e^{2\pi\eta}$. By shrinking $\NNN$, we have that $\varphi_{f}$ is injective on $D(\eta)$ and $\varphi_g$ is injective on $D(\eta')$ for any $f\in \FFF_2\cap \NNN$. Let $\{\t{V}_j^f\}$ denote the collection of components of $\pi^{-1}(W_k^f),k\geq1,$ contained in $D(\eta')$. Set $V_j^f:=\varphi_g(\t{V}_j^f)$ for all $j\geq 0$. Then these $V_j^f$ are located in the dynamical plane of $g$. Also, we denote by $\{\t{U}_i^f\}$ the collection of components of $\pi^{-1}(V_j^f),j\geq0$ that are contained in $D(\eta)$, and set $U_i^f:=\varphi_f(\t{U}_i^f)$. The collection $\{U_i^f\}$ of disks are in the dynamical plane of $f$, see Figure \ref{fig:proof1}.
\vspace{3pt}

\noindent\emph{Step 6. A hyperbolic set $X_f$ and the dimension estimate.}\vspace{3pt}

We apply Section \ref{sec:return-map} (7) successively to $(h,\{W_k^f\}),(g,\{V_j^f\})$ and $(f,\{U_i^f\})$. It then follows that, for  any $f\in\NNN\cap\FFF_2$ and every $k$, there exist an integer $l_k>0$ such that $f^{l_k}:U_k^f\to U^f$ is a homeomorphism. By choosing $\eta$ large enough, the collection of disks $\{U_k^f\}$ are contained in $U^f$. Hence
$(f,U^f,\{U^f_k\})$ is a repelling system described in the beginning of Section \ref{sec:construction}. Let $X_f\subset J(f)$ denote the hyperbolic set generated from this repelling system. According to  \cite[Section 6]{S},  its Hausdorff dimension  
\begin{equation}\label{eq:dimension}
	{\rm H.dim}(X_f)\geq\frac{C_1+\log \eta+4\pi \eta}{\log C_2+(q+1)\log \eta+2\pi\eta}.
\end{equation}
The right hand side tends to $2$ as $\eta\to\infty$, which implies Theorem \ref{thm:shishikura}.

\subsection{Hyperbolic sets on the boundary of attracting  domains}
In general, one just knows that the hyperbolic set $X_f$ in Theorem \ref{thm:shishikura} belongs to the Julia set. We will show how this hyperbolic set lies on the boundaries of attracting domains under specific properties. 

We first prove a distortion result about the variation of arguments under  conformal maps.  Let $\Omega\subset \mathbb C$ be a simply connected domain with $0\notin \Omega$.
The argument function $Arg$ admits a     continuous branch  ${\rm arg}_{\Omega}$   in  $\Omega$. For any set $K\subset \Omega$,  define the {\it angular width} ${\rm aw}_{\Omega}(K)$ of $K$ by  
$${\rm aw}_{\Omega}(K)=\sup_{w\in K} {\rm arg}_{\Omega}(w)-\inf_{w\in K} {\rm arg}_{\Omega}(w).$$
Clearly, this quantity does not depend on the choice of the branch ${\rm arg}_{\Omega}$.

\begin{lem}\label{lem:variation}  There is a  positive function $c: (0,1)\rightarrow (0, +\infty)$
  with the following property:  for any univalent map $f:\D\to \mathbb C$ with $0\not\in \Omega:=f(\D)$, and any $r\in (0, 1)$,  we have ${\rm aw}_{\Omega}(K_r)\leq c(r)$, where $K_r=f(\overline{\mathbb D(0,r)})$.
\end{lem}
\begin{proof}
	Choose  $a,b\in \overline{\mathbb D(0,r)}$ so that $\min_{w\in K_r} {\rm arg}_{\Omega}(w)={\rm arg}_{\Omega}(f(a))$, $\max_{w\in K_r} {\rm arg}_{\Omega}(w)={\rm arg}_{\Omega}(f(b))$. Let $\log\circ f$ be a holomorphic  branch of ${\rm Log} \circ f$. Then 
	$${\rm aw}_{\Omega}(K_r)\leq |\log\circ f(b)-\log\circ f(a)|=\Big|\int_{[a,b]}(\log\circ f)' dz\Big|\leq \int_{[a,b]}\frac{|f'(z)|}{|f(z)|}|dz|, $$
	where $[a,b]$ is the geodesic segment in $\mathbb D$.  
	By the properties of the hyperbolic metric $\rho_\Omega(w)|dw|$ of $\Omega$,
	$$\rho_\Omega(f(z))|f'(z)|=\frac{2}{1-|z|^2}, \ \rho_\Omega(w)\geq \frac{1}{2\cdot {\rm dist}(w, \partial \Omega)}\geq \frac{1}{2|w|}.$$
	It follows that 
	$|f'(z)|/|f(z)|\leq 4/{(1-|z|^2)}$. Hence 
	$${\rm aw}_{\Omega}(K_r)\leq \int_{[a,b]}\frac{4 }{1-|z|^2}|dz|\leq 8 \int_{[0,r]}\frac{|dz| }{1-|z|^2}=4\log\frac{1+r}{1-r}:=c(r).$$
	The proof is completed.
		\end{proof}

Now we follow the notations in Section \ref{sec:construction}. We still assume   that the boundary  $\partial \Omega_{f_0}^*$ of the immediate parabolic basin  avoids the $f_0$-critical points.

\begin{lem}\label{lem:orbit}
	Let $f_0$ and $\FFF_2$ satisfy the  conditions of Theorem \ref{thm:shishikura}.
	For any $M,\epsilon>0$ and any domain $\De$ containing   $\ov{\Omega_{f_0}^*}$,
	we can choose a neighborhood $\NNN\subset {\rm Rat}_d$ of $f_0$ such that, if $f\in\NNN\cap\FFF_2$,  there exists a repelling system $(f, U^f,\{U_i^f\})$ that satisfies:
	\begin{enumerate}
		\item each orbit \, $U_i^f\overset{f}{\to}\cdots\overset{f}{\to} f^{l_i}(U_i^f)=U^{f}$
		is entirely contained in $\De$,\vspace{1pt}
		\item ${\rm H.dim}(X_f)>2-\epsilon$, where $X_f$ is the hyperbolic set of $(f, U^f,\{U_i^f\})$,  \vspace{2pt}
		\item any periodic point in $X_f$ has period larger than $M$.
	\end{enumerate}
\end{lem}

\begin{proof}
	Without loss of generality, we may assume that $\zeta=0$ and $\zeta_f=0$ for all $f\in\FFF_2$.
	Usually, we omit the superscript $(k)$ when referring to the quantities in Section \ref{sec:parabolic-implosion}, since $k$ is always taken to be $0$.
	
	Let $\{z_n\}_{n\ge 0}\subset \partial \Omega_{f_0}^{(1)}$ be an inverse orbit of $f_0^q$ with $z_0=0$ and $z_n\to 0$ from the $0$-th repelling petal (on the up-left of   Figure \ref{fig:step1}). Fix a number $\rho>0$ such that $\bigcup_{k=0}^{q-1}f^k\left( \D(0,\rho)\right)\subset \De$. For any $x>0$, set
	\[S_x:=\{z\in\mathbb C:-x-1\leq{\rm Re} (z)<-x \} \text{\ \ and \ } \ell_x:=\{z\in\mathbb C:{\rm Re}(z)=-x\}.\]
	
	Let $\rho_1>0$ be a number with $\D(0,\rho_1)\Subset {\rm Dom}(g_0)=B^u_{f_0}$.
	By Section \ref{sec:repelling-coordinate} (6), we can choose two numbers $\t w_*,\t \xi_*>0$ and a neighborhood $\NNN$ of $f_0$ such that the domains $\t\QQQ_f:=\t\QQQ_f(\t w_*)$ and $\t\QQQ_g:=\t\QQQ_g(\t \xi_*)$ satisfy
	\begin{equation}\label{eq:277}
		\bigcup_{k=0}^{q-1}f^k\left(\varphi_f^{(1)}(\t{\QQQ}_f)\right)\subset \D(0,\rho)\ \text{ and }\ \varphi_g(\t\QQQ_g)\subset \D(0,\rho_1),\ \forall\, f\in\NNN\cap \FFF_2.
	\end{equation}
	We also denote $\t\QQQ_{f_0}:=\t\QQQ_0(\t w_*)$ and $\t\QQQ_{g_0}:=\t\QQQ_0(\t \xi_*)$.

	Recall  that $\t{w}_0\in {\rm Dom}(\varphi_{f_0})$ is the point corresponding to the inverse orbit $\{z_n\}$ of $f_0^q$ with $z_n=\varphi_{f_0}(\t{w}_0-n)$.
	We record the number $n_0$ such that $\t{w}_0-n_0\in S_{\t{w}_*}$.    By enlarging $\t w_*$ and shrinking $W_0$, we can make the following assumption:
	
	{\bf Assumption 1}. \emph{The disks in the inverse orbit of $U^{f_0}$ under $f^q_0$ along $z_0,z_1,\ldots,z_{n_0}$ are compactly contained in $\De$ with pairwise disjoint closures, and $n_0\geq 2M$}.
	
	
	Recall that $\t{\xi}_0\in {\rm Dom}(\varphi_{g_0})$ corresponds to the inverse orbit $\{w_n\}$ of $g_0$ with $w_0=\pi(\t w_0)$.
	Let $m_0\in \mathbb Z$ be such that $\t{\xi}_0-m_0\in S_{\t{\xi}_*}$.

	Recall the construction of $\{W_k^f\},\{V_j^f\}$ and $\{U_i^f\}$ in the dynamical plane of $h,g$ and $f$ respectively from the proof of Theorem \ref{thm:shishikura}.
	
	Let $\eta$ be a large number that satisfies: $D(e^{2\pi \eta})\subset \t\QQQ_{g_0}, D(\eta)\subset\t\QQQ_{f_0}$, $\varphi_{f_0}(D(\eta))$ is compactly contained in $U^{f_0}$ and the right side of equation \eqref{eq:dimension} is larger than $2-\epsilon$. As a consequence, we obtain a neighborhood $\NNN$ of $f_0$ such that, for any $f\in\NNN\cap \FFF_2$,
	\begin{itemize}
		\item $D(e^{2\pi \eta})\subset \t\QQQ_{g}$ and $D(\eta)\subset\t\QQQ_{f}$,
		\item $(f,U^f,\{U_i^f\})$ forms a repelling system, and
		\item  ${\rm H.dim}(X_f)>2-\epsilon$, where $X_f$ is the hyperbolic set  of $(f,U^f,\{U_i^f\})$.
	\end{itemize}
	Then statement (2) of the lemma follows.

	Moreover, all disks $\{W_k^f\}_{k\geq1}$ are contained in a domain  $D_{h_0}\Subset B^{u}_{h_0}$ for any $f\in\FFF_2\cap \NNN$ (shrinking $\NNN$ if necessary).
	For such a map $f$ and every $k\geq0$, we denote $\hat {V}_k^f$ the component of $\pi^{-1}(W_k^f)$ which intersects $S_{\t{\xi}_*}$ but avoids $\ell_{\t{\xi}_*}$, see  Figure \ref{fig:W-V}. Then $\t\xi_0-m_0\in \hat{V}_0^f$.
	\begin{figure}[http]
		\centering
		\includegraphics[width=13cm]{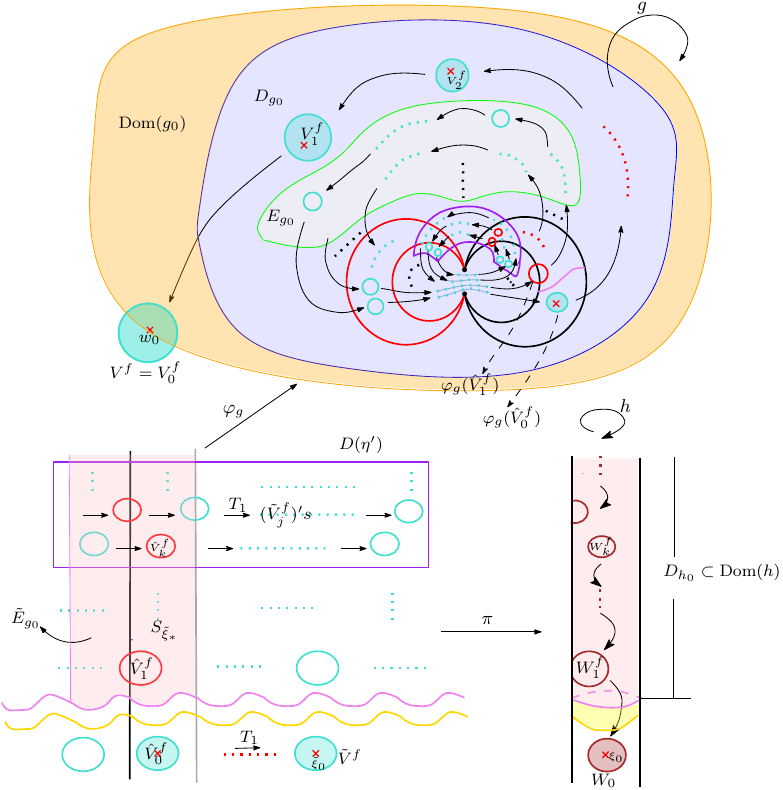}\\
		\caption{The orbits of $(V^f_j)'s$}\label{fig:W-V}
	\end{figure}

	{ \emph{Claim: The widthes of all disks $\hat{V}_k^{f},k\geq1$ are bounded above by a universal constant $\kappa$  for any $f\in\FFF_2\cap\NNN$. }

		\emph{Proof of Claim.}
		Each $W_k^f$ is the image of $W_0$ under a conformal map, denoted by $\chi_k^f$. By shrinking $W_0$, we can assume that $W_0=\D(\xi_0,\de)$, and that $\chi_k^f:\D(\xi_0,2\de)\to\mathbb C\setminus\{0\}$  is univalent.
		The width of  $\hat{V}_k^f$, which equals $\frac{1}{2\pi}{\rm aw}_{\chi_k^f(\D(\xi_0,2\de))}(\chi_k^f(W_0))$, is bounded above by $2\log(3)/\pi:=\kappa$, by  Lemma \ref{lem:variation}.
		\hfill $\Box$
	}\vspace{3pt}
	
	By  this claim  and condition \eqref{eq:277},  we can successively apply Lemma \ref{lem:key1} to $g$ with $E=\ov{D_{h_0}},\, x_1=-\t \xi_*-{ \kappa},\,x_2=-\t \xi_* $  and $(U,U')=(\hat{V}_k^f,\hat{V}_{k-1}^f)$ for  every $k\geq1$.
	Therefore, there exists a compact set $E_{g_0}\subset \Omega_{g_0}$ such that the orbit
	$$\varphi_g(\hat{V}_k^f)\overset{g}{\longrightarrow} \cdots\overset{g}{\longrightarrow} \varphi_g(\hat{V}_0^f)$$
	is contained in $E_{g_0}\cup \D(0,\rho_1)$ for any $f\in\FFF_2\cap\NNN$ and each $k\geq1$.
	
	Recall that the collection of disks $\{\t{V}_j^f\}$ consists of the components  of $\pi^{-1}(W_k^f),k\geq1$ contained in $D(e^{2\pi \eta})$.
	Then these $\t{V}_j^f$'s are the elements of  $\{T_m(\hat{V}_k^f):m\in\Z,k\geq 1\}$ which lie in $D(e^{2\pi \eta})$. Thus, for each $\t{V}_j^f$, there exist  unique $k\in\N$ and $m_j\in\Z$ such that $T_{m_j}(\hat{V}_k^f)=\t{V}_j^f$. 
	It follows from Section \ref{sec:repelling-coordinate} (1),(2) that
	$$\begin{cases}g(\varphi_g(\hat{V}_k^f)),\ldots,g^{m_j}(\varphi_g(\hat{V}_k^f))=V_j^f, & \text{ if } m_j>0,\\[3pt]
		V_j^f,	g(V_j^f),\ldots,g^{-m_j}(V_j^f)=\varphi_g(\hat{V}_k^f),  & \text{ if } m_j\leq 0, 
	\end{cases}$$
	are contained in $\varphi_g(D(2\pi\eta))\subset\varphi_g(\t\QQQ_g)\subset \D(0,\rho_1)$. Combining this point and the discussion in the last paragraph, we have that all $V_j^f$'s  are in the same orbit, and  for each $V_j^f$,
	there is an integer $n_j^f>0$ such that the orbit
	$$V_j^f\overset{g}{\longrightarrow} g(V_j^f)\overset{g}{\longrightarrow}	\cdots\overset{g}{\longrightarrow}g^{n_j^f}(V_j^f)=\varphi_g(\hat{V}_0^f)$$
	is entirely contained in $E_{g_0}\cup \D(0,\rho_1)$.
	
	Recall that  $V^f=\varphi_g(\t V^f)$, where $\t V^f$ is the component of $\pi^{-1}(W_0)$ containing $\t\xi_0$, see Figure \ref{fig:W-V}. Since $\t\xi_0-m_0\in \hat{V}_0^f$,  for each $V_j^f$,
	\[\text{$g^{n_j^f+m_0}(V_j^f)=g^{m_0}(\varphi_g(\hat{V}_0^f))=\varphi_g(\t V^f)=V^f,\  \forall\,f\in\NNN\cap \FFF_2$}.\]
	
	We can find a compact set $D_{g_0}\subset {\rm Dom}(g_0)=B_{f_0}^u$  such that its interior contains $E_{g_0}\cup \D(0,\rho_1)$ and the inverse orbit of $V^{f_0}$ under $g_0$ along the points $w_1,\ldots, w_{m_0}$. The previous discussion shows that
	the orbit
	\[ V_j^f\overset{g}{\longrightarrow} g(V_j^f)\overset{g}{\longrightarrow}\cdots\overset{g}{\longrightarrow}g^{n_j^f+m_0-1}(V_j^f)=g^{m_0-1}(\varphi_g(\hat{V}_0^f))\]
	is located in $D_{g_0}$ for any $f\in\NNN\cap\FFF_2$ and each $V_j^f$ (by shrinking $\NNN$ if necessary), see  Figure \ref{fig:W-V}.
	
	Enlarge the family of $\{V_j^f\}$ so that it contains all elements in the orbit from $V_j^f$ to $V^f$ for each $j$. Thus there exists a source element $V_*^f\in\{V_j^f\}$ in the sense that, each $V_j^f$ is an iterated image of $V_*^f$ under $g$. Remunerate the subscript of $\{V_j^f\}_{j=0}^{m^f}$ such that $V^f_0=V^f$ and $g(V_j^f)=V_{j-1}^f$.
	
	Similarly as above, for any $f\in\FFF_2\cap \NNN$ and  $j\geq0$, let $\hat{U}_j^f$ be the component of $\pi^{-1}(V_j^f)$ which intersects the strip $S_{\t{w}_*}$ but avoids the vertical line $\ell_{\t{w}_*}$. { By a similar argument as in the proof of the claim above, we  obtain that
		the widths  of all disks $\hat{U}_j^{f},0\leq j\leq m^f$ are no more than $\kappa$ for any $f\in\FFF_2\cap\NNN$.}

	According to this fact and condition \eqref{eq:277},  we can successively apply Lemma \ref{lem:key1} to $f$ with $E=\ov{D_{g_0}},\, x_1=-\t w_*-{\kappa},\,x_2=-\t w_*$ and $(U,U')=(\hat{U}_j^f,\hat{U}_{j-1}^f)$ for any $f\in\FFF_2\cap\NNN$ and every $1\leq j\leq m^f$. Note that $g=\mathcal R_f$ and $\mathcal R_f(\pi (\hat{U}_j^f))=\pi (\hat{U}_{j-1}^f)$. Let $m_j^f>0$ be the minimal integer so that $f^{qm_j^f}(\varphi_f^{(0)}(\hat{U}_j^f))=\varphi_f^{(1)}(\hat{U}_{j-1}^f)$.
	It follows that the orbit
	\[\varphi_f^{(0)}(\hat{U}_j^f)\overset{f^q}{\longrightarrow}\cdots \overset{f^q}{\longrightarrow}f^{qm_j^f}(\varphi_f^{(0)}(\hat{U}_j^f))=\varphi_f^{(1)}(\hat{U}_{j-1}^f)\overset{f^r}{\longrightarrow}\varphi_f^{(0)}(\hat{U}_{j-1}^f)\]
	is located in $\D(0,\rho)\cup\Omega_{f_0}^*$ for every $1\leq j\leq m^f$.
	Using a similar argument as that after the claim above, we can find an integer $n_i=n_i(f)>0$ and a multiple $s_i=s_i(f)$ of $r$ for each $U_i^f$,
	such that $f^{qn_i+s_i}(U_i^f)=\varphi_f^{(0)}(\hat {U}_0^f)$ and
	\begin{equation}\label{eq:66}
		\bigcup_{k=0}^{qn_i+s_i}f^k(U_i^f)
		\subset\left(\,\bigcup_{k=0}^{q-1}f^k(\D(0,\rho))\,\right)\bigcup\,\Omega_{f_0}^*\subset \De
	\end{equation}
	
	Let $\t U^f$ denote the component of $\pi^{-1}(V^f)$ that contains $\t w_0$. Then $\varphi_f(\t U^f)=U^f$.
	By the definition of $\hat{U}_0^f$ and the choice of $n_0$ before Assumption 1,  we have $T_{n_0}(\hat{U}_0^f)=\t{U}^f$ for $f\in\NNN\cap \FFF_2$ (shrinking $\NNN$ if necessary). It means that the disk $f^{qt}(\varphi_f^{(0)}(\hat{U}_0^f))$ contains the point $z_{n_0-t}$ for every $0\le t\leq n_0$ and $f^{qn_0}(\varphi_f^{(0)}(\hat{U}_0^f))=U^f$. Set $l_i:=(n_i+n_0)q+s_i$, then $f^{l_i}(U_i^f)=U^{f}$. 
	
	By Assumption 1 and shrinking $\NNN$,  we obtain that
	\begin{equation}\label{eq:77}
		\varphi_f^{(0)}(\hat{U}_0^f),\ldots, f^{qn_0}(\varphi_f^{(0)}(\hat{U}_0^f))=U^f\emph{ are pairwise disjoint and contained in $\De$}
	\end{equation}
	for all $f\in\NNN\cap\FFF_2$. Then statement (1)  follows from \eqref{eq:66} and \eqref{eq:77}.
	
	This point also implies that the length of the orbit for any point in $X_f$ is larger than $n_0>2M$. Then statement (3) holds.
\end{proof}

Suppose  that $f\in{\rm Rat}_d$ has a fixed point  $\zeta_f$ whose multiplier has the form \eqref{eq:good}. 
Suppose further that   $\zeta_f$ is attracting (see \eqref{attracting-condition}).
We denote by $A_f^*(\zeta_f)$ the immediate attracting basin of $f$ at the fixed point $\zeta_f$.

\begin{lem}\label{lem:on-boundary}
	Let $f_0$ and $\FFF_2$ satisfy the conditions in Theorem \ref{thm:shishikura}, and let $\Omega^*_{f_0}$ denote the immediate parabolic basin of $f_0$ at $\zeta$.
	Suppose that  $\zeta_f$ is $f$-attracting  and that $\ov{\Omega_{f_0}^*}$ is contained in a domain $\De$ which satisfies
	\begin{equation}\label{eq:keyproperty1}
		\De\cap \big(f^{-1}(A_f^*(\zeta_f))\setminus A_f^*(\zeta_f)\big)=\emptyset
	\end{equation}
	for all $f\in\FFF_2$.
	Then given any $\epsilon,M>0$, there exists a neighborhood $\NNN$ of $f_0$, such that any map $f\in\NNN\cap\FFF_2$ has a hyperbolic set
	$X_f\subset\partial A_f^*(\zeta_f)$ with Hausdorff dimension larger than $2-\epsilon$ and containing no    periodic points of $f$ with periods less than $M$.
\end{lem}

\begin{proof}
	By Lemma \ref{lem:orbit}, there exists a neighborhood $\NNN$ of $f_0$ fulfilling that, any $f\in\NNN\cap\FFF_2$ has a repelling system $(f,U^f,\{U_i^f\})$ and a hyperbolic set $X_f$ generated from it such that the following properties hold:
	\begin{enumerate}
		\item each orbit \, $U_i^f\overset{f}{\to}\cdots\overset{f}{\to} f^{l_i}(U_i^f)=U^{f}$
		is located in $\De$,\vspace{1pt}
		\item ${\rm Hdim}(X_f)>2-\epsilon$, and  \vspace{1pt}
		\item any periodic point in $X_f$ has period larger than $M$.
	\end{enumerate}
	Therefore, it is enough to prove that $X_f\subset \partial A_f^*(\zeta_f)$ for $f\in\NNN\cap\FFF_2$.
	
	Note that $U^{f_0}$ (given by Section \ref{sec:construction} ) contains the parabolic fixed point $\zeta$. So it also contains repelling periodic points of $f_0$.
	Since $U^f$ converge to $U^{f_0}$, it follows that $U^f$ contains the attracting fixed point $\zeta_f$ and repelling periodic points of $f$ for every  $f\in\NNN\cap\FFF_2$ (shrinking $\NNN$ if necessary). Thus the intersection of  $U^f$ and $\partial A_f^*(\zeta_f)$ is non-empty.

	For each $U_i^f$, we have $f^{l_i-1}(U_i^f)\cap \partial f^{-1}(A_f^*(\zeta_f))\not=\emptyset$ since $f^{l_i}(U_i^f)=U^f$ intersects $\partial A_f^*(\zeta_f)$.
	By condition \eqref{eq:keyproperty1} and statement (1) above, it follows that $f^{l_i-1}(U_i^f)\cap \partial A_f^*(\zeta_f)\not=\emptyset$.
	Inductively using this argument, we can deduce that the non-escaping set of the repelling system $(f^q,U^f,\{U_i^f\})$ is contained in $\partial A_f^*(\zeta_f)$. Hence $X_f\subset \partial  A^*_f(\zeta_f)$.
\end{proof}

\subsection{Boundary dimension for bounded attracting domains }

In this part, we will apply Lemma \ref{lem:on-boundary} to 
prove Theorem \ref{thm:boundary-dimension0}.   
The idea of the proof is to construct a sequence of quasiconformal deformation of $f$ such that these maps and their limit map satisfy the property of Lemma \ref{lem:on-boundary}.



The following fact  follows from Proposition \ref{RY}.

 \begin{ft}  \label{pro:separation}
 	Let $f$ be a polynomial. Then any cut-point   
 	in the boundary of a parabolic  or bounded attracting domain of $f$ is iterated to either the set of critical points, or the set of  periodic cut-points in boundaries of a parabolic  or bounded attracting  domains. These two sets are both finite.
\end{ft}


We first fix some notations. Let $f$ be a subhyperbolic polynomial. By Fact \ref{pro:separation}, we have a maximum $N(f)$ of the periods for all
\begin{itemize}
	\item periodic Fatou components of $f$,
	\item periodic postcritical points, and
	\item periodic cut-points of $J(f)$ on  boundaries of bounded attracting Fatou components.
\end{itemize}

Let $O_1(f),\ldots,O_m(f)$ be a collection of marked geometrically-attracting cycles with lengthes $\ell_1,\ldots,\ell_m$, respectively. Suppose that each of these cycles attracts only one critical point. Choose different  integers $q_1,\cdots, q_m$  so that
$\ell_kq_k>N(f)$ for every $k=1,\ldots,m$.

\begin{lem}\label{lem:separate}
	Let $\{f_n\}_{n\geq1}$ be a sequence of quasi-conformal deformation of $f$, such that the deformation occurs only at the attracting basin of $O_k(f),k=1,\ldots,m$, and the multiplier of $O_k(f_n)$ converges to
	${\rm exp}(2\pi i /q_k)$  as $n\to\infty$.
	Then by taking a subsequence, $f_n$ converges to a geometrically-finite polynomial $g_0$ with the following properties:
	\begin{enumerate}
		\item the parabolic cycles of $g_0$ are exactly $O_k(g_0):=\lim_{n\to\infty} O_k(f_n),k=1,\ldots,m,$ each of which attracts only one critical point of $g_0$;
		\item for every $k=1,\ldots,m$, the parabolic cycle $O_k(g_0)$ has length $\ell_k$ and multiplier $e^{2\pi i/q_k}$;
		\item for every $k=1,\ldots,m$, the boundary of the immediate basin of the parabolic cycle $O_k(g_0)$ contains no critical points of $g_0$;
		\item if $z$ is a repelling $f$-periodic point with period no more than $N(f)$, then its deformation $z_n$ at $f_n$ converges to a repelling $g_0$-periodic point with the same period.
	\end{enumerate}
\end{lem}

\begin{proof}
	Let $\psi_n:{\mathbb C}\to{\mathbb C}$ denote the quasiconformal mapping which conjugates $f$ to $f_n$ and is  normalized by $\psi_n(z)/z\to1$ as $z\to\infty$.
	Note that $\{f_n\}$ is cocompact  in $\mathcal P_d$ (see \cite[Chapter 8]{DH} or \cite{DP}). So we can assume that $f_n\rightarrow g_0\in\PPP_d$.

	
	Set $Y:=F(f)\setminus \bigcup_{k=1}^m A_f(O_k(f))$, where $A_f(O_k(f))$ is the whole attracting basin of $O_k(f)$. The restriction $\psi_n|_Y$ is a conformal conjugacy between $f:Y\to Y$ and $f_n:Y_n\to Y_n$ with $Y_n:=\psi_n(Y)$.\vspace{3pt}

	\noindent\emph{$\bullet$ The attracting domains of $g_0$.}\vspace{3pt}

	Let $z\in\mathbb C$ be an attracting point of $f$ not in $O_1(f),\ldots, O_m(f)$, with period $l$.  Set $z_n:=\psi_n(z)$. Since $f_n^l(z_n)=z_n$ and the multiplier $(f^l_n)'(z_n)\in \mathbb D$ is constant, the point $z_0:=\lim_{n\to\infty}z_n$ is an attracting  point of $g_0$ with period $l$.  By Lemma \ref{carathodory-convergence}, we have the kernel convergence:
	$(U_{f_n}(z_n), z_n)\rightarrow (U_{g_0}(z_0), z_0)$.
	By choosing a subsequence and by Carath\'eodory's Kernel Convergence Theorem, the conformal maps $\psi_n|_{U_f(z)}$'s converge locally uniformly to a conformal map $\psi_0: U_f(z)\rightarrow U_{g_0}(z_0)$, which conjugate $f|_{U_f(z)}$ to $g_0|_{U_{g_0}(z_0)}$.  It follows that $\psi_n|_{Y}$ converges  a conformal map (we still denote by) $\psi_0: Y\rightarrow Y_0$ with $Y_0:=\psi_0(Y)$.
	This map $\psi_0$ gives a conformal conjugacy between $f:Y\to Y$ and $g_0: Y_0\to Y_0$.  It is worth noting that $Y_0$ is the union of all attracting basins of $g_0$. Otherwise, let $w$ be an attracting periodic point of $g_0$ not in $Y_0$. Then it can be perturbed to an attracting periodic point $w_n$ of $f_n$ not in $Y_n$. Hence $w_n\in O_k(f_n)$ for some $k=1,\ldots,m$. However $O_k(f_n)$ converges to a parabolic cycle of $g_0$, a contradiction.\vspace{3pt}

	\noindent \emph{$\bullet$ The parabolic cycles of $g_0$.}\vspace{3pt}
	
	Fix a $k\in \{1,\ldots,m\}$ and a point $x\in O_k(f)$. Set $x_n:=\psi_n(x)$ and denote $x_0=\lim_{n\to\infty} x_n$. Then we have $g^{\ell_k}_0(x_0)=x_0$ and
	$(g_0^{\ell_k})'(x_0)=e^{2\pi i/q_k}$. By  the Implicity Function Theorem, for every large $n$, there is exactly one point $y_n$ such that $f_n^{\ell_k}(y_n)=y_n$ and $y_n\to x_0$ as $n\to\infty$. So $x_n=y_n$. This implies that the parabolic cycle $O_k(g_0)$ has length $\ell_k$ and multiplier $e^{2\pi i/q_k}$ {(otherwise, a point in $O_k(f_n)$ other than $x_n$ would converge to $x_0$, a contradiction.)}.
	As a consequence, any immediate parabolic basin associated to $O_k(g_0)$ has period $\ell_k\cdot q_k>N(f)$. Moreover, the parabolic cycles $O_1(g_0),\ldots,O_m(g_0)$ are pairwise disjoint, since they have pairwise different multipliers.
	
	According to the discussion in the last step, each critical point of $f$ in $Y$ deforms to a critical point of $g_0$ in an attracting domain. Since any critical point in $J(f)$ is preperiodic, it deforms to a preperiodic critical point of $g_0$ in the Julia set (otherwise, the critical orbit would eventually meet the center of some Siegel disk,  implying the existence of a wandering {critical point in the Julia set}. Contradiction!). The remaining $m$ critical points of $f$ are attracted by $O_1(f),\ldots,O_m(f)$ respectively. They deform to $m$ critical points of $g_0$, which are only able to be located in parabolic basins such that each parabolic cycle of $O_1(g_0),\ldots,O_m(g_0)$ attracts exactly one critical point. Thus $g_0$ is geometrically-finite.\vspace{3pt}

	\noindent $\bullet$ \emph{The  periodic points of $g_0$ with periods no more than $N(f)$.}\vspace{3pt}
	
	For each $k\in\{1,\ldots,m\}$, since $O_k(g_0)$ attracts only one critical point, any point $z_0\in O_k(g_0)$ belongs to the boundaries of $q_k$ parabolic basins, which forms a cycle under the iteration of $g_0^{\ell_k}$. Then
	\[g_0^{\ell_k}(z)=z_0+e^{2\pi i/q_k}(z-z_0)+a(z-z_0)^{q_k+1}\big(1+O(z-z_0)\big),\ z\to z_0.\]
	
	Thus, when perturbing $g_0$ to $f_n$, the cycle $O_k(g_0)$ splits into the attracting cycle $O_k(f_n)$ and a repelling cycle of $f_n$ with length $\ell_k\cdot q_k>N(f)$. It means that every repelling periodic point of $f$ with period $\leq N(f)$ deforms to a repelling periodic point of $g_0$. 
	
	In particular, a periodic postcritical point $b$ of $f$   deforms to a repelling periodic point $b_0$ of $g_0$, with period $\leq N(f)$. This $b_0$ is not  on the boundary of  any immediate parabolic basin of $O_k(g_0)$ because otherwise, its period would  $\geq \ell_k\cdot q_k>N(f)$.  This implies  that  the boundary of any immediate parabolic basin of $O_k(g_0)$ contains no critical points.
	
	Now we have checked statements (1)-(4), and the lemma is proved.
\end{proof}

\begin{proof}[Proof of Theorem \ref{thm:boundary-dimension0}]
	We follow the notations
	$\ell_k,q_k\,(k=1,\ldots,m)$ and $N(f)$ defined before Lemma \ref{lem:separate}. Let $\{f_n=\psi_n\circ f\circ\psi_n^{-1}\}_{n\geq1}$ be a sequence of quasi-conformal deformation of $f$ such that the deformation occurs only at the attracting basin of $O_k(f)$ ($k=1,\ldots,m$) and the multiplier of $O_k(f_n)$ has the form ${\rm exp}(2\pi i (1+\alpha_k(n))/q_k)$, where
	\begin{equation}\label{arithmetic-condition}
		\alpha_k(n)=\frac{1}{a_k(n)-\dfrac{1}{b_k(n)+\nu_k(n)}}
	\end{equation}
	with $\N\ni a_k(n),b_k(n)\to\infty$ and $\nu_k(n)\to -1$ as $n\to\infty$.

	According to Lemma \ref{lem:separate}, if taking a subsequence, we have $f_n\to f_0\in\PPP_d$  as $n\to\infty$, such that each attracting cycle $O_k(f_n)$ of length $\ell_k$ ($k=1,\ldots,m)$ converges to a parabolic cycle $O_k(f_0)$  with length $\ell_k$, which attracts only one critical point of $f_0$.

	Fix a $k\in\{1,\ldots,m\}$ and a point $z\in O_k(f)$. For each $n\geq1$, denote $z_n=\psi_n(z)$ and set $z_0=\lim_{n\to\infty} z_n$. Then $g_0:=f^{\ell_k}_0$ and $\FFF_2:=\{g_n:=f^{\ell_k}_n, n\geq 1\}$ satisfy the conditions of Theorem \ref{thm:shishikura} with $\zeta=z_0$ and $\zeta_f=z_n$.

	Let $\Omega_{g_0}^*$ denote the immediate parabolic basin of $g_0$ at $z_0$, and let $A^*$ denote the immediate attracting basin of $z$ for $g=f^{\ell_k}$. For every $n\geq1$, set $A_{n}^*:=\psi_n(A^*)$. Note that $N(f)=N(f_n)$ for all $n\geq1$.
	The  theorem follows immediately from
	Lemma \ref{lem:on-boundary} if we can verify the following condition for all large $n$:
	\begin{equation}\label{eq:disjoint1}
		\exists\text{ domain }\De\supset \ov{\Omega_{g_0}^*},\ \text{s.t.}\ \De\cap \big(g_n^{-1}(A_n^*)\setminus A_n^*\big)=\emptyset.
	\end{equation}
	
	To check condition \eqref{eq:disjoint1},  note first that there are cut-points $x_1,\ldots,x_r\subset\partial A^*$  and ray pairs
	$$R_g(\theta_i,\theta_i'):=R_{g}(\theta_i)\bigcup\{x_i\}\bigcup R_{g}(\theta_i'),\ i=1,\ldots, r,$$
	such that 
	the component $V$ of $\mathbb C\setminus \bigcup_{i=1}^rR_g(\theta_i,\theta_i')$  containing $A^*$  is disjoint from  $g^{-1}(A^*)\setminus A^*$. 
	Since $f$ is subhyperbolic, by Fact \ref{pro:separation}, all $x_1,\ldots,x_r$ are pre-repelling (note that $x_k$ might be a critical point).

	By the definition of $N(f)$ and Lemma \ref{lem:separate} (4), the limit point $x_i(f_0):=\lim_{n\to\infty}\psi_n(x_i)$ is  a pre-repelling point of $f_0$ with period no more than $N(f)$ for every $i=1,\ldots,r$.
	Moreover, since any immediate parabolic basin associated to $O_k(f_0)$ has period  $\ell_k\cdot q_k>N(f)$ and the unique root on its boundary  is a parabolic point,  the points $x_1(f_0),\ldots,x_r(f_0)$ are  disjoint from $\ov{\Omega_{g_0}^*}$.
	
	By Lemma \ref{stability-e-r},  the ray pair $R_{g_n}(\theta_i,\theta_i')$ converges to the ray pair $R_{g_0}(\theta_i,\theta_i')$ in the Hausdorff topology for every $i=1,\ldots,r$. It follows that $V_n=\psi_n(V)$ converges to a Jordan domain $V_0$ with  $\partial V_0=\bigcup_{i=1}^rR_{g_0}(\theta_i,\theta_i')$. Since $z_n\in V_n$ and $$\ov{\Omega_{g_0}^*}\cap \partial V_0=\ov{\Omega_{g_0}^*}\cap\{x_1(f_0),\ldots,x_r(f_0)\}=\emptyset,$$
	we have $\ov{\Omega_{g_0}^*}\subset V_0$. Thus, there is a domain $\De\Subset V_0$ such that $\ov{\Omega_{g_0}^*}\subset \De$, and then
	$\De\subset V_n$ for every large $n$. 
	
	Since $V$ is disjoint from $g^{-1}(A^*)\setminus A^*$, it follows that $V_n$ is disjoint from $g_n^{-1}(A_n^*)\setminus A_n^*$ for every $n\geq1$. Hence condition \eqref{eq:disjoint1} holds, and the theorem is proved.
\end{proof}

\begin{rmk}
	Theorem \ref{thm:boundary-dimension0} is still true  even if $f$ has disconnected Julia set. In its proof,  the connectivity of $J(f)$ is only used to ensure the validity of following two results:
	\begin{itemize}
		\item [(1)] \emph{a landing property of external rays with rational angles;}
		\item [(2)] \emph{a combinatorial property of cut-points:} for any two distinct Fatou components $U,V$, there is a cut-point $x\in \partial U$ separating $U,V$, i.e., $U,V$ lie in distinct components of $K(f)\setminus\{x\}$.
	\end{itemize}
	
	In the case that  $J(f)$ is not connected,  some external rays are not smooth  but break up at iterated preimages of the escaping critical points of $f$. However, there is still a \emph{generalized landing theory of external rays}, see \cite{LP,PZ}. To guarantee the smoothness of  external rays with rational angles, we impose the following additional condition on $f$:
	\begin{equation}\label{eq:88}
		\emph{every escaping critical point corresponds to irrational external angles.}
	\end{equation}
	
	Under this condition, the statement (2) above  has a generalized form: \emph{if distinct Fatou components $U,V$ lie in the same component of $K(f)$, then they are separated by a cut-point $x\in\partial U$; otherwise, $U,V$ are separated by an arc in the basin of infinity.}
	
	By the same argument as Theorem \ref{thm:boundary-dimension0} but in place of results (1) and (2) with their corresponding generalized version mentioned above, we can prove that the conclusion of Theorem \ref{thm:boundary-dimension0} holds for any sub-hyperbolic polynomial which satisfies condition \eqref{eq:88}.

\end{rmk}

\section{Proof of the main theorems} \label{proof-main}

In this part, we shall prove Theorem \ref{lchd-gf-section1}.  By Theorem \ref{boundary-lc}, it is implied by the following   

\begin{thm}\label{lchd-gf}  	Let $\mathcal H\subset \mathcal C_d$ be a  non disjoint-type  hyperbolic component. 
	There is a geometrically finite map $g_0\in\partial\mathcal H$ with $m_{\mathcal H}$ parabolic cycles, 
	such that for any $1\leq l\leq d-1-m_{\mathcal H}$, 
	\begin{equation} \label{local-hd-g}
		{\rm H. dim}( \partial_{\mathcal A}\mathcal H\cap \mathfrak{M}_l, g_0)=2d-2.
	\end{equation}
\end{thm}


\begin{proof} Let  $D=\big((B_u^0, D_u^{\partial})\big)_{u\in V}\in \mathbf{Div}_{\rm spp}$ be  a generic  Misiurewicz  divisor satisfying the condition of Proposition \ref{criterion-H-adm}, so that 
	$$\sum_{u\in V}{\rm deg}(D_u^{\partial})=d-1-m_{\mathcal H}.$$
	Then  Proposition \ref{criterion-H-adm} yields that $D$ is $\mathcal H$-admissible. By Proposition \ref{divisor-singleton},  $I_{\Phi}(D)$ consists of an $\mathcal H$-admissible map, say $f$.  This $f$ is subhyperbolic, with $m_{\mathcal H}$ geometrically attracting cycles, say $O_1(f),\ldots,O_{m_{\mathcal H}}(f)$,  and each whole attracting basin contains exactly  one critical  point.
			Suppose the $m_\mathcal{H}$-attracting Fatou components containing critical points correspond to indices $v_1, \cdots, v_{m_\mathcal{H}}$.
	For $\rho\in \mathbb D$,  let 
	$$\beta_\rho(z)=z\frac{(1-\rho)z+\rho(1-\overline{\rho})}{(1-\overline{\rho})+\overline{\rho}(1-\rho)z}.$$
	Then $\beta_\rho$ is a Blaschke product of degree two, with $\beta_\rho(1)=1$ and $\beta'_\rho(0)=\rho$.
	The $B$-factor of $D=\big((B_u^0, D_u^{\partial})\big)_{u\in V}$  can be written as
	$$B_{v_k}^0=\beta_{\rho_k^0}, \ 1\leq k\leq m_{\mathcal H};  \ B^0_v(z)=z, \ v\in V\setminus\{v_1, \cdots, v_{m_{\mathcal H}}\},$$
where $\rho_k^0$ is the multiplier of $f$ at the $k$-th cycle $O_k(f)$.

Since $D=\big((B_u^0, D_u^{\partial})\big)_{u\in V}\in \mathbf{Div}_{\rm spp}$, by the Implicit Function Theorem, there is a continuous family of divisors  $\{D_{\boldsymbol{\rho}}\}_{\boldsymbol{\rho}\in \mathbb D^{m_{\mathcal H}}}$ so that $D_{\boldsymbol{\rho_0}}=D$ for $\boldsymbol{\rho_0}=(\rho^0_1, \cdots, \rho^0_{m_{\mathcal H}})$, and for each $\boldsymbol{\rho}=(\rho_1, \cdots, \rho_{m_{\mathcal H}}) \in \mathbb D^{m_{\mathcal H}}$,
\begin{itemize}
	\item $D_{\boldsymbol{\rho}}=\big((B_{u, {\boldsymbol{\rho}}}, D_{u, {\boldsymbol{\rho}}}^{\partial})\big)_{u\in V}\in \mathbf{Div}_{\rm spp}$;
	
	\item $B_{v_k, {\boldsymbol{\rho}}}=\beta_{\rho_k}, \ 1\leq k\leq  m_{\mathcal H}$;
	
	\item $B_{v, {\boldsymbol{\rho}}}(z)=z, \ v\in V\setminus\{v_1, \cdots, v_{m_{\mathcal H}}\}$.
	\end{itemize}
The divisors $\{D_{\boldsymbol{\rho}}\}_{{\boldsymbol{\rho}}\in (\mathbb D\setminus\{0\})^{m_{\mathcal H}}}$ satisfy the conditions of  Proposition \ref{criterion-H-adm}, hence they are $\mathcal H$-admissible. 

For each $n\geq 1$, let $f_n$ be the map  in the impression of  $D_{{\boldsymbol{\rho}}_n}$ with ${\boldsymbol{\rho}}_n=({\rm exp}(2\pi i (1+\alpha_k(n))/q_k))_{1\leq k\leq {m_{\mathcal H}}}$, where     $\alpha_k(n)$ satisfies \eqref{arithmetic-condition}. Then $\{f_n\}_{n\geq1}$ is a sequence of quasi-conformal deformations of $f$ with prescribed multipliers, and the deformation occurs  at the attracting basins of   $O_1(f),\ldots,O_{m_{\mathcal H}}(f)$.

By Lemma \ref{lem:separate},  the sequence $f_n$ converges to a geometrically finite map $g_0\in\partial\mathcal H$, with ${m_{\mathcal H}}$ parabolic cycles.
Now, for any $\epsilon>0$, by Theorem \ref{thm:boundary-dimension0},  there is  an integer $n>0$ so that $f_n\in \mathcal N_\epsilon(g_0)$, and 
$${\rm H.dim}(\partial U_{f_n, v_k})\geq 2-\epsilon, \  \forall 1\leq k\leq {m_{\mathcal H}}.$$

Note that the vector  $\mathbf n(f_n)=(n_v(f_n))_{v\in V}$ satisfies $|\mathbf n(f_n)|= d-1-m_{\mathcal H}$, where $n_v(f_n)=\#({\rm Crit}(f_n)\cap \partial U_{f_n,v})$ for $v\in V$.  
For any $1\leq l\leq d-1-m_{\mathcal H}$,  by Theorem \ref{local-hausdorff-dim-J} and taking    $\mathbf 0\leq \mathbf n\leq \mathbf n(f_n)$ with $|\mathbf n|=l$,  we have 
$${\rm H.dim}(  \partial_{\mathcal A}\mathcal H\cap  \mathfrak{M}_l, f_n)\geq 2d-2- l\epsilon.$$
It follows that 
$${\rm H.dim}( \partial_{\mathcal A}\mathcal H\cap \mathfrak{M}_l\cap \mathcal N_\epsilon(g_0))\geq 2d-2- l\epsilon.$$
Since $\epsilon>0$ is arbitrary, we have  ${\rm H. dim}(\partial_{\mathcal A}\mathcal H\cap \mathfrak{M}_l, g_0)=2d-2.$
	\end{proof}


For each $f\in {\mathcal H}\cup \partial_{\rm reg}\mathcal H$, let $\mathbf{F}(f)$ be the collection of all bounded attracting Fatou components. Let
$${\rm H.dim}_B(f)=\min_{U\in \mathbf{F}(f)}{\rm H.dim}(\partial U).$$

The following result is a by-product of the proof of Theorem \ref{lchd-gf}.

\begin{pro} \label{sup-hb-Fatou} Let $\mathcal H\subset \mathcal C_d$ be a  non disjoint-type  hyperbolic component,
  then for any
	$1\leq l\leq d-1-m_\mathcal{H}$, 
	$$\sup_{g\in \mathcal{H}}{\rm H.dim}_B(g)=\sup_{g\in   \partial\mathcal H\cap \mathfrak{M}_l}{\rm H.dim}_B(g)=2.$$
\end{pro}
\begin{proof} We use the same notations as that in  the proof of Theorem \ref{lchd-gf}. By the proof Theorem \ref{lchd-gf}, there exist a sequence of subhyperbolic polynomials $\{f_n\}_{n\geq 1}\subset \partial_{\mathcal A}\mathcal H\cap \mathfrak{M}_{d_0}$ with $d_0=d-1-m_{\mathcal H}$, approaching   a geometrically finite polynomial $g_0\in\partial\mathcal H$, with $m_{\mathcal H}$ parabolic cycles.  Note that the multiplier vector of $f_n$ is  $({\rm exp}(2\pi i (1+\alpha_k(n))/q_k))_{1\leq k\leq m_\mathcal{H}}$, where    $\alpha_k(n)$ satisfies \eqref{arithmetic-condition} (this is the condition  given in the proof of Theorem \ref{thm:boundary-dimension0}). 
	 By Theorem \ref{thm:boundary-dimension0},  for any $\epsilon>0$,   there is a large $n$, and  $f_n^{\ell_k}$-hyperbolic sets $X_k\subset \partial_0 U_{f_n, v_k}$ for $1\leq k\leq m_{\mathcal H}$, so that 
	$${\rm H.dim}(X_k)\geq 2-\epsilon, \  \forall 1\leq k\leq m_{\mathcal H}.$$
	Note that these hyperbolic sets move holomorphically in a neighborhood $\mathcal N(f_n)$ of $f_n$, so that each $g\in \mathcal N(f_n)$ has exactly $m_{\mathcal H}$ attracting cycles,  and the holomorphic motion $h(g, X_k)$ of $X_k$ satisfies $h(g, X_k)\subset \partial_0 U_{g, v_k}$ (see \eqref{hyp-0-u}). Choose $r>0$ small so that $\mathcal N_{r}(f_n)\subset \mathcal N(f_n)$, and for each $g\in \mathcal N_{r}(f_n)$,  
	$${\rm H.dim}(h(g, X_k))\geq {\rm H.dim}(X_k)-\epsilon \geq 2-2\epsilon, \  \forall 1\leq k\leq m_{\mathcal H}.$$
	
	This implies  $\sup_{g\in \mathcal{H}\cap  \mathcal N_{r}(f_n)}{\rm H.dim}_B(g)\geq 2-2\epsilon$.    It  follows that $\sup_{g\in \mathcal{H}}{\rm H.dim}_B(g)=2$.

For each
$1\leq l\leq d-1-m_\mathcal{H}$, 	by Theorem \ref{local-perturbation1}, there is an  $\mathcal H$-admissible map $g_n\in  \mathcal N_{r}(f_n)\cap \partial_{\mathcal A}\mathcal H\cap \mathfrak{M}_l$.
It satisfies that 	${\rm H.dim}_B(g_n)\geq 2-2\epsilon$.  Since $\epsilon>0$ is arbitrary, we also get $\sup_{g\in   \partial_{\mathcal A}\mathcal H\cap \mathfrak{M}_l}{\rm H.dim}_B(g)=2$.
	\end{proof}


\end{document}